\documentclass[leqno]{article}
\usepackage[T1]{fontenc}
\usepackage[utf8]{inputenc}
\usepackage{float}
\usepackage{upgreek}
\usepackage{amsfonts,amssymb,amsthm}

\usepackage{etextools}
\usepackage{mathtools} 
\usepackage{textcomp}
\usepackage{relsize}
\usepackage{ifthen}
\usepackage{mathabx}
\usepackage{eucal}
\usepackage{indentfirst}
\usepackage{tikz}
\usetikzlibrary{calc,positioning}
\usetikzlibrary{arrows,decorations.pathmorphing,decorations.markings,backgrounds,positioning,fit,petri,arrows.meta} 
\usetikzlibrary{patterns, patterns.meta,math}
\usepackage[scr]{rsfso}
\usepackage{graphics}
\usepackage{bbm} 
\newcommand{\aset}{\text{$A$-Set}}

\newcommand{\xspace}{\mathcal{X}}
\newcommand{\pid}{\mathfrak{p}}
\newcommand{\oneideal}{\mathfrak{a}}
\newcommand{\sk}{\text{Sk}}
\newcommand{\hookedrightarrow}[1]{ \tikz	\draw[{Hooks[right,length=5,width=6]}->] (0,0) -- (#1cm,0cm);	\; }
\newcommand{\ideal}{\mathscr{I}}
\newcommand{\myemph}[1]{{\it{\bf{#1}}}}

\newcommand{\threeover}[3]{
				\begin{array}[h]{c}
								\scriptstyle #1 \vspace{-.25cm}\\
								\scriptstyle #2 \vspace{-.20cm}\\
								\scriptstyle #3
				\end{array}
}
\newcommand{\bigbit}{\text{Bit}}
\newcommand{\rdarrow}[2]{   
				\begin{tikzpicture}[baseline]
								\draw[->] (0.3cm,0.05cm)--(1cm,0.05cm);
								\draw[->] (0.3cm,0.15cm)--(1cm,0.15cm);
								\node at (0.7cm,-0.2cm) [] {$#1$};
								\node at (0.7cm,0.39cm) [] {$#2$};
\end{tikzpicture}}

\newcommand{\midLarr}[1]{
				\begin{tikzpicture}[baseline]
								\draw[-] (0.3cm,0.05cm)--(1cm,0.05cm);
							   \draw[decorate,decoration={markings,mark=at position 0.5 with
											{\arrow[color=black]{<}}}] (0.3cm,0.05cm)--(1cm,0.05cm);
								\node at (0.6cm,0.3cm) [] {$#1$};
				\end{tikzpicture}
			}
\newcommand{\midarrow}[4]{ 
				\begin{tikzpicture}[baseline]
								\draw[-] ($(0.3cm,#4+0.05cm)$)--($(0.3cm+#1,#4+0.05cm)$);
							   \draw[decorate,decoration={markings,mark=at position 0.5 with
								 {\arrow[color=black]{<}}}] ($(0.3cm,#4+0.05cm)$)--($(0.3cm+#1,#4+0.05cm)$);
								 \node at ($(0.3cm+0.5*#1,#2)$) [] {$#3$};
				\end{tikzpicture}
			}
\newcommand{\bbDelta}{%
  \Delta\mkern-12mu\Delta%
}

\usepackage{anyfontsize}
\usepackage{stmaryrd}
\newcommand{\BDelta}{\bbDelta}

\newcommand{\myhom}{\text{\normalfont{Hom}}}

\newcommand{\xtworightarrow}[2][]{%
				  \xrightarrow[#1]{#2}\mathrel{\mkern-14mu}\rightarrow
	}

\newcommand{\plusprimeideal}{\mathfrak{p}} 
\newcommand{\step}{\item[(\addtocounter{equation}{1}\theequation) ]}
\newcommand{\op}{\text{op}}
\makeatletter
\newcommand*{\rom}[1]{\expandafter\@slowromancap\romannumeral #1@}
\makeatother
\newcommand{\characteristic}{{\mathbf{\mathbbm{1}}}}
\newcommand{\unit}{{\mathbf{1}}}

\newcommand{\bigo}{\mathcal{O} }

\newcommand{\Q}{\mathbb{Q}}
\newcommand{\R}{\mathbb{R}}
\newcommand{\C}{\mathbb{C}}
\newcommand{\N}{\mathbb{N}}

\newcommand{\kernel}{\mbox{Ker}}

\makeatletter
\DeclareRobustCommand{\properideal}{\mathrel{\text{$\m@th\proper@ideal$}}}
\newcommand{\proper@ideal}{%
  \ooalign{$\lneq$\cr\raise.22ex\hbox{$\lhd$}\cr}%
}
\makeatother
\newcommand{\lideal}{\lhd}

\newcommand{\Z}{\mathbb{Z}}

\newcommand{\colim}{\mathop{\text{colim}}}
\newcommand{\colimrightarrow}{\mathop{\underrightarrow{\text{colim}}}}

\newcommand{\limrightarrow}{\mathop{\underrightarrow{\text{lim}}}}
\newcommand{\limleftarrow}{\mathop{\underleftarrow{\text{lim}}}}

\newcommand{\cocolim}{\mathop{\text{(co)lim}}}

\newcommand{\mymod}{\text{\normalfont-mod}}
\newcommand{\biop}{\mathscr{P}}
\newcommand{\bioq}{\mathscr{Q}}
\newcommand{\bior}{\mathscr{R}}
\newcommand{\circleftarrow}{\,\overleftarrow{\circ}\,}
\newcommand{\circrightarrow}{\,\overrightarrow{\circ}\,}

\DeclareMathSymbol{\LLCurly}{\mathrel}{mathb}{"CE}
\DeclareMathSymbol{\ggcurly}{\mathrel}{mathb}{"CF}

\newtheorem*{claim*}{Claim}
\newtheorem{theorem}{Theorem}
\newtheorem{lemma}{Lemma}
\newtheorem{definition}[equation]{Definition} 
\newtheorem{definition*}{Definition} 
\newtheorem{remark}[equation]{Remark}
\newtheorem*{remark*}{Remark}

\newtheorem{corollary}{Corollary}
\newtheorem*{corollary*}{Corollary}
\newtheorem{proposition}{Proposition}
\numberwithin{equation}{section}
\title{Homotopy and Arithmetic \\  \large From bi-operads to generalized schemes}
\author{Shai Haran \\ {haran@technion.ac.il} }
\date{}
\begin{document}
\maketitle
\begin{abstract}
				We define the concept of a Bi-operad. We develop the homotopy theory of Bital-Sets
				and of $\infty$-Bi-operad. We develop a geometry of generalized schemes based on
				the spectra of distributive monochromatic bi-operads.
\end{abstract}
\section{Introduction}
It was Dan Kan who developed the language of \myemph{Categories}, and of \break
\myemph{Simplicial Sets} based on finite-linearly-ordered sets, to give a
combinational approach to homotopy.  Dan Quillen  axiomatized this approach in
his language of \myemph{Model Categories}. \myemph{Operads} were discovered by
Boardman and Vogt in connection with homotopy invariant algebraic structures, and by May
in connection with iterated loop spaces. Ieke Moerdijk (with collaborators I. Weiss, C.
Berger, D.C. Cisinski, G. Heuts, V. Hinich, E. Hoffbeck) generalized categories and operads
to a theory of ``\myemph{Coloured-Operads}'', leading to a generalization of simplicial
sets to ``\myemph{Dendroided Sets}'', based on finite trees, and developed the homotopy
theory of \myemph{$\infty$Operads} generalizing \myemph{$\infty$Categories}. \\
While categories and simplicial sets have a dual opposite, there is no opposite of an operad. 
Some algebraic structures (such as Hopf-algebras and bi-algebras) are not controlled by an
operad since they involve both operations and ``co-operations''. The purpose of this paper
is to define a notion of ``Bi-Operad'', having both operations and co-operations, and of
``Bital-Sets'', based on finite - ``Bi-trees'', and sketch the homotopy theory of
$\infty$Bi-Operads generalizing the Dendroidal and the Simplicial stories. \\
Having an opposite for
Bi-operads and Bital-Sets we can speak of the self-dual Bi-operads and Bital-sets having an
 involution. Finally, returning to the classical ``monochromatic'' case of having
one colour or object, we can speak of the distributive Bi-operads. \\
It was Alexander Grothendieck who develop the modern language of algebraic geometry, the \myemph{schemes},
based on the \myemph{spectra} of commutative-rings. But commutative rings are Not adequate
for \myemph{Arithmetical Geometry} (The ``integers'' at a real or complex place of a
number field do  not form a sub-ring; the integers $\Z$ are the initial object of rings
and so $\text{spec}(\Z)$ is the final object of Grothendieck's geometry, rather than the
\myemph{absolute point} $\text{spec}(\mathbb{F}_{1})=\left\{ (0) \right\}$, where
$\mathbb{F}_{1}$ the initial object of distributive  bi-operads is the \myemph{``Field with one
element''}; the \myemph{Arithmetical surface} $\text{spec}(\Z)\prod
\text{spec}(\Z)=\text{spec}\left( \Z\otimes\Z \right)=\text{spec}(\Z)$ reduces to its diagonal
\cite{haran1989}) The self-dual distributive bi-operads are essentially the  \myemph{``Generalized-Rings''} of
\cite{MR3605614}, and are the building blocks of an \myemph{Arithmetical Geometry}, generalizing
Grothendieck's theory of \myemph{Schemes}. \vspace{.1cm}\\
As it turns out, to do geometry following
Grothendieck's paradigm, one does not need an involution (although in arithmetic, the real
and complex ``integers'' prefer the self-dual $\ell_2$-metric). We lay the foundation of
\myemph{generalized schemes} based on the spectra of distributive bi-operads. This new
language of geometry is indeed much simpler than the one of \cite{MR3605614} (where we had to keep the
involution), and it also shows the way to a \myemph{``Quantum-geometry``} (where vectors
and co-vectors are Not the same). \\
Thus the theory of Bi-operads and Bital-sets give the homotopy
language needed for \myemph{Derived Arithmetical Geometry} and for
\myemph{Arithmetical Stacks}. The \myemph{Arithmetical Surface} $\text{spec}(\Z)\prod
\text{spec}(\Z)=\text{spec}(\Z\coprod\Z)$ with $\Z\coprod \Z$ the sum in the category of
distributive Bi-operads (and its compactification as a pro-object of generalized schemes,
as in \cite{MR2330442}) is our Jungentraum of  \cite{haran1989} \vspace{.2cm}\\
\begin{center}
\colorbox{gray!10}{
\begin{minipage}[H]{8cm}
\begin{quote}
				``Mathematics should not be complicated! You just set the stage right - 
				then you sail away...'' 
\end{quote} 
\flushright Raoul Bott \\
\
\end{minipage}
}
\end{center}
\begin{center}
\colorbox{gray!10}{
\begin{minipage}[H]{8cm}
\begin{quote}
``What's important is the point-of-view... the theorems will follow!''
\end{quote} 
\flushright 
 Barry Mazur \\
\scriptsize{lecture on the rule of the point at infinity in \\ ``Algebra and Geometry'', Cambridge
1993.}  \
\end{minipage}
}
\end{center}
The content of this report is as follows: 
\begin{itemize}
				\item[1. ] \underline{Introduction}
				\item[2. ] \underline{Bio$\equiv$Bi-operands} We define Bi-operads as a pair of
								colored symmetric operads with the same objects, and with mutual actions
								on each other. 
				\item[3. ] \underline{Bit$\equiv$Bi-trees} We define the Bi-trees as finite
								oriented graphs, with no closed paths, and such that for any vertex
								$v$, $\# d_{-}(v)=1$ or $\# d_{+}(v)=1$. 
				\item[4. ] \underline{Formulas and $1$-sub-bits} We define our ``formulas''
								inductively as bi-trees with compatible planar structures. 
				\item[5. ] \underline{The free Bio} We describe the free bio on a collection
								$\mathcal{C}$ as the $\mathcal{C}$-labeled formulas. 
				\item[6. ] \underline{Bios via generators and relations} We show Bio is complete
								and co-complete. 
				\item[7. ] \underline{Bipo$\equiv$Bi-Pre-Order} We describe the (co-)complete
								closed symmetric monoidal structure on Bipo $\equiv$ the  full subcategory
								of Bio with objects the bios having at most one point in every hom-set. 
				\item[8. ] \underline{Natural transformation} We describe the internal hom
								structure on Bio. 
				\item[9. ] \underline{Bilinear maps and tensor products} We describe the closed
								monoidal structure on Bio.
				\item[10. ] \underline{From Bits to Bios} We associate with every bit a bio and we
								make Bit into a category by defining the ``flat'' maps of bits. 
				\item[11. ] \underline{Degeneracies and face maps} We show these maps generate
								Bit; the face maps are either ``outer'' ($=$ elimination of an edge), or
								``inner'' (= the contraction of ``composition'' or ``action''
								edge(s)).
				\item[12. ] \underline{The category of Bital sets bSet} We describe the category
								$\text{bSet}\equiv\text{Set}^{\text{Bit}^{\text{op}}}$ of pre-sheaves on
								Bit, and the resulting Kan adjunction 
								\begin{tikzpicture}[baseline=4.5mm]
												\draw [<-, color=black]  plot [smooth] coordinates {(5mm,7mm)(8mm,8mm)(11mm,7mm)};
												\node at (1mm,5.5mm) {bSet};
												\draw [->, color=black]  plot [smooth] coordinates {(5mm,4mm)(8mm,3mm)(11mm,4mm)};
												\node at (14mm,5.5mm) {Bio};
								\end{tikzpicture}.
				\item[13. ] \underline{The symmetric ``monoidal'' structure on bSet} We describe
								the closed, symmetric, and not quite associative, tensor structure on
								bSet.
				\item[14. ] \underline{The image of $N:\text{Bio}\hookrightarrow bSet$} The nerve
								functor $N$ gives a full and faithfull embedding of Bio onto the
								strict-Kan Bital Sets.
				\item[15. ] \underline{The homotopy functor $\tau=\text{ho}: \infty\text{Bio}\to
								\text{Bio}$} For an $\infty\text{Bio}$ $X\in \text{bSet}$ satisfying the
								(inner) Kan condition, the associated bio $\tau(X)$ has a homotopical
								description. 
				\item[16. ] \underline{Normal monomorphisms} We describe following \cite{MR2294028},
								\cite{MR4222648} the cofibrations of bSet, and their interaction with the monoidal
								structure. 
				\item[17. ] \underline{Quillen Model structure. Ein M{\" a}rchen} Following
								\cite{MR4222648} we describe the model structures on $\text{sBio}_{C}$, $C$ a fixed
								set of objects, and the conjectural model structures on simplicial bios
								$\text{sBio}$, and on bital sets $\text{bSet}$ (these should be equivalent
								via a modification of the Berger-Moerdijk-Boardman-Vogt resolution). 
				\item[18. ] \underline{Self-duality} We show the bios $\biop$, with an involution
								$\biop\cong\biop^{\text{op}}$, $X\mapsto X^{t}$, can be also described as
								operads with a ``contraction'' operaton, making contact with the language
								of \cite{MR3605614}.
				\item[19. ] \underline{Rigs} For a rig $A$ ($=$ associative ring without
								``negatives'', i.e. no inverses for addition) we associate a bio
								$\biop_{A}$ with one object, giving a full and faithful embedding
								$\text{Rig}\hookrightarrow \text{Bio}_{ \left\{ \ast \right\} }$. 
				\item[20. ] \underline{The $\ell_{p}$ bio} For $p\in [1,\infty]$ we define the
								``$\ell_{p}$-sub-bio'' of the bios $\biop_{\R}$ and $\biop_{\C}$; for
								$p=2$ these are $\Z_{\R}$ and $\Z_{\C}$, the real and complex
								``integers''. 
				\item[21. ] \underline{Distributive bios} These are the monochromatic bios
								$\biop$ where the actions and compositions interchange; the monoid
								$\biop(1)$ is commutative and central. The association $A\mapsto
								\biop_{A}$ gives a full and faithful embedding of commutative rings into
								distributive bios under $\mathbb{F}_{1}$ 
								\begin{equation*}
												C\text{Ring}\hookrightarrow \text{CBio}.
								\end{equation*}			
				\item[22. ] \underline{The spectrum} We associate with a distributive bio
								$\biop\in \text{CBio}$, a compact $\text{sober}\equiv \text{Zariski}$,
								topological space $\text{spec}(\biop)$. 
				\item[23. ] \underline{The structure sheaf $\mathscr{O}_{\biop}$} We describe the
								sheaf of bios over $\text{spec}(\biop)$, $\biop\in \text{CBio}$.
				\item[24. ] \underline{Generalized Schemes} Following Grothendieck \cite{MR3075000} we
								describe the generalized schemes as spaces with a sheaf of distributive
								bios with local stalks, that are locally affine. 
				\item[25. ] \underline{A$\text{-sets}$} For $A\in \text{CBio}$ we define the
								category of $A\text{-sets}$. For $A\in C\text{Ring}$ this is the category
								of $A$ modules. 
				\item[26. ] \underline{Commutativity} We discuss commutativity of
								$A\text{-sets}$
				\item[27. ] \underline{Symmetric monoidal structure} We describe the closed
								monoidal (model) structure on $CA\text{(s)Set}$-the category of
								commutative $A$-(simplicial) sets.
				\item[28. ] \underline{Symmetric sequences} We describe the Category
								$\Sigma(A)\equiv \left( CA\text{-sSet} \right)^{\text{iso}(\text{Fin})}$of
								symmetric sequences, and the operadic Fourier transform. 
				\item[29. ] \underline{The sphere spectrum} Following \cite{MR1860878},
								\cite{MR1695653}, we describe
								the sphere spectrum $S_{F}^{.}$ as the commutative-monoid-object of
								$\Sigma(\mathbb{F})$, giving rise to the category
								$S_{A}^{.}\text{-mod}\subseteq \Sigma(A)$ of ``symmetric spectra``, with
								its stable monoidal model structure. 
				\item[30. ] \underline{Quasi-coherent $\mathscr{O}_{X}$-modules} We describe the
								category of $\mathscr{O}_{X}\text{-modules}$, $X$ a generalized scheme,
								and its full subcategory $q.c.\mathscr{O}_{X}\text{-modules}$ of
								''quasi-coherent`` $\mathscr{O}_{X}\text{-modules}$.
				\item[31. ] \underline{The global derived category} Following \cite{MR2522659}, we can
								describe the derived category of quasi-coherent simplicial
								$\mathscr{O}_{X}\text{-modules}$ as the homotopy category of the
								$\infty\text{-category}$ of cartesian sections of the inner fibration
								\break
								$\text{Aff}/_{X}\ltimes S_{A}^{.}\text{-mod}\to \text{Aff}/_{X}$. 
\end{itemize}
\begin{center}
\colorbox{gray!10}{
\begin{minipage}[H]{10cm}
\begin{quote}
				``At all events, I have exceeded my commission and been seduced into describing
				things as they may be and, as seems to me at present, are likely to be... The
				views of this report are in any case not peculiar mine. I have simply fused my own
				observations and reflections with ideas of others and with commonly accepted
				tenets.''
\end{quote} 
\flushright 
R.P. Langlands \\
\scriptsize{}  \
\end{minipage}
}
\end{center}
\section{Bio $\equiv$ Bi-operads}
\begin{definition}
				A Bio ($\equiv$ Bi-Operad) $\biop=(\biop^{-},\biop^{o},\biop^{+})$ is a pair
				$\biop^{-}, \biop^{+}$ of (coloured, symmetric) operads with the same set of
				``colours'' or ``objects'' $\biop^{o}$, together with mutual actions of
				$\biop^{-}$, $\biop^{+}$ on each other.
				\label{def:bio}
\end{definition}
Explicitly, for $c_0,c_1,\dots,c_n\in\biop^{o}$ we have sets (simplicial sets, etc.) of
(co)-operations $\biop^{-}(c_0;c_1,\dots,c_n)$ and $\biop^{+}(c_1,\dots,c_n;c_0)$, with
actions of the symmetric group $S_n$, given for $\sigma\in S_n$ 
\begin{equation}
				\begin{array}[H]{ll}
								\biop^{-}(c_0;c_{\sigma(1)}\dots c_{\sigma(n)}) 
								 \;\begin{tikzpicture}
								 				\draw (0,0) circle [radius=0cm];
								 				\draw [{To[length=2mm, width=2mm]}-] (0,0) -- (1,0);
								 \end{tikzpicture}\; 
								\biop^{-}(c_0;c_1\dots c_n) \qquad\qquad  &P\sigma 
								\;\begin{tikzpicture}
												\draw (0,0) circle [radius=0cm];
												\draw [{To[length=2mm, width=2mm]}-|] (0,0) -- (1,0);
								\end{tikzpicture}\; P \\\\
								\biop^{+}(c_1\dots c_n;c_0)
								\;\begin{tikzpicture}
												\draw (0,0) circle [radius=0cm];
												\draw [-{To[length=2mm, width=2mm]}] (0,0) -- (1,0);
								\end{tikzpicture} \; 
								\biop^{+}(c_{\sigma^{-1}(1)}\dots c_{\sigma^{-1}}(n);c_0)
						&
							P
							\;\begin{tikzpicture}
												\draw (0,0) circle [radius=0cm];
												\draw [|-{To[length=2mm, width=2mm]}] (0,0) -- (1,0);
								\end{tikzpicture} \;
								\sigma P
				\end{array}
				\label{eq:2.2}
\end{equation}
We have units $\unit_{c}^{-}\in\biop^{-}(c;c)$ and $\unit^{+}_{c}\in \biop^{+}(c;c)$,
$(n=1)$. \\ 
The compositions are given for each word
\begin{equation*}
				\overline{c}=c_{11}\dots c_{1{m_1}},c_{21}\dots c_{2{m_2}}, \dots
				,c_{n1}\dots c_{n{m_n}} \in \left( \biop^{\circ} \right)^{\sum m_{i}}
\end{equation*}
partitioned into the $n$ words
\begin{equation*}
				\overline{c_1}=c_{11}\dots c_{1{m_{1}}}\; , \; \overline{c_2}=c_{21}\dots
				c_{2{m_2}}\; , \dots \; ,\; \overline{c_{n}}=c_{n1}\dots c_{n{m_{n}}}
\end{equation*}
and for $b_0,b_1,\dots , b_n\in \biop^{o}$,
\begin{equation}
				\begin{array}[H]{l}
								o:\biop^{-}(b_0;b_1\dots b_n)\times \left[
								\biop^{-}(b_1;\overline{c_1})
								\times\dots\times\biop^{-}(b_n,\overline{c}_n) \right] 
								\xrightarrow{\qquad} \biop^{-}(b_0;\overline{c}) \\\\
								o:\left[
								\biop^{+}(\overline{c_1};b_1)\times\dots\times\biop^{+}(\overline{c}_n;b_n)
				\right] \times \biop^{+}(b_1\dots b_n;b_0)\xrightarrow{\qquad} 
								\biop^{+}(\overline{c};b_0) 
				\end{array}
				\label{eq:2.3}
\end{equation}
These compositions are assumed to be \myemph{associative}, 
\begin{equation}
				\begin{array}[H]{rll}
				P\circ\big(P_{i}\circ(P_{ij})\big) &= \big(P\circ(P_i)\big)\circ(P_{ij}) \qquad &
								P,P_i,P_{ij}\in \biop^{-} \\\\
								(P_{ij})\circ\big((P_j)\circ P\big) &=  \big( (P_{ij})\circ P_j\big)\circ
								P
							&	P_{ij},P_j, P\in \biop^{+}
				\end{array}
				\label{eq:2.4}
\end{equation}
\myemph{unital},
\begin{equation}
				\begin{array}[H]{ll}
								\unit_{c_0}^{-}\circ  P = P = P\circ(\unit_{c_{i}}^{-}) \qquad  & P \in
				\biop^{-}(c_0;c_1\dots c_n) \\\\
				\left( \unit_{c_i}^{+} \right)\circ P = P = P \circ\unit_{c_0}^{+}  &P\in\biop^{+}(c_1\dots c_n;c_0)
				\end{array}
				\label{eq:2.5}
\end{equation}
and $S_n\times \left[ S_{m_1}\times \dots \times S_{m_{n}} \right]\subseteq
S_{m_1+\dots+m_n}$ \myemph{covariant}. The mutual actions are given for such words
$\overline{c}=\overline{c}_{1}\dots\overline{c}_{n}$, 
				$\overline{c}_{i}=\left( c_{i1}\dots c_{i,{m_i}}\right), b_0,b_1\dots b_n\in
				\biop^{o}$, by maps
\begin{equation}
				\begin{array}[H]{lll}
				\circrightarrow: \left[ \biop^{-}(b_1;\overline{c}_1)
				\times\dots\times\biop^{-}(b_n;\overline{c}_n)\right]\times\biop^{+}(\overline{c};b_0)\xrightarrow{\qquad}
				\biop^{+}(b_1\dots b_{n};b_0) \\\\
				\circleftarrow: \biop^{-}(b_0;\overline{c})\times \left[
				\biop^{+}(\overline{c}_{1};b_1)\times \dots \times \biop^{+}(\overline{c}_{n};b_n)
\right]\xrightarrow{\qquad} \biop^{-}(b_0;b_1\dots b_n)
				\end{array}
				\label{eq:2.6}
\end{equation}
These actions are assumed to be \myemph{associative} 
\begin{equation}
				\big(P\circleftarrow (P_{ij})\big) \circleftarrow (P_j) = P \circleftarrow \big(
								(P_{ij})\circ P_j\big) \quad , \quad
								(P_i)\circrightarrow\big((P_{ij})\circrightarrow
								P\big)=\big(P_i\circ (P_{ij})\big)\circrightarrow P
								\label{eq:2.7}
\end{equation}
\myemph{unital}, 
\begin{equation}
				P\circleftarrow (\unit_{c_i}^{+})= P \quad , \quad \left( \unit_{c_i}^{-}
				\right)\circrightarrow P = P
				\label{eq:2.8}
\end{equation}
$S_n$-\myemph{covariant}, and $S_{m_1}\times\dots\times
S_{m_n}$-\myemph{invariant}. \vspace{.1cm}\\
We will also assume that these actions are ``linear'' in all possible ways: 
\begin{equation}
				\big(P\circ(P_i)\big)\circleftarrow(P_{ij}) =
				P\circ\big(P_i\circleftarrow(P_{ij})\big)\quad , \quad (P_{ij})\circrightarrow
				\big((P_j)\circ P\big) = \big( (P_{ij})\circrightarrow P_{j}\big)\circ P
				\label{eq:2.9}
\end{equation}
and 
\begin{equation}
				P\circleftarrow\big((P_{ij})\circrightarrow P_j\big) = \big( P\circ
								(P_{ij})\big)\circleftarrow (P_j) \quad , \quad \big( P_i\circleftarrow
								(P_{ij})\big)\circrightarrow P = (P_i)\circrightarrow \big((P_{ij})\circ
												P\big)
				\label{2.10}
\end{equation}
A map of bios $\varphi: \biop\to \bioq$ is a pair of functors of operads
$\varphi^{-}:\biop^{-}\to \bioq^{-} $ and $\varphi^{+}:\biop^{+}\to \bioq^{+}$, with the same map
on objects $\varphi^{\circ}:\biop^{\circ}\to \bioq^{\circ}$ and preserving the natural actions 
\begin{equation}
				\varphi^{-}\big( P\circleftarrow (P_i)\big)= \varphi^{-}(P)\circleftarrow
				\big(\varphi^{+}(P_i)\big) \qquad 
				\varphi^{+}\big((P_i)\circrightarrow
				P\big)=\big(\varphi^{-}(P_i)\big)\circrightarrow \varphi^{+}(P)
				\label{eq:2.11}
\end{equation}
(as well as preserving the compositions, units, and $S_n$-actions). \\
Thus we have a
category \myemph{Bio}. 
\begin{remark}
				Forgetting the $S_n$-actions we get the notion of \myemph{planar}, or
				non-symmetric bio, and we have the category $p$Bio. There is a forgetful functor $U:\text{Bio}\to 
				p\text{Bio}$, forgetting the $S_n$-actions, and this functor has a left adjoint
				$\Sigma :  p\text{Bio}\to\text{Bio}$, with free $S_n$-actions 
				\begin{equation*}
								\textstyle\sum \displaystyle \biop^{-}(c_0;c_1\dots c_n) := \coprod\limits_{\sigma\in
								S_n}\biop^{-}(c_0;c_{\sigma(1)}\dots c_{\sigma(n)})
				\end{equation*}
				We say the bio $\biop$ is \myemph{closed} if ($n=0$),
$\biop^{-}(b;\phi)=\left\{ 0_{b}^{-} \right\}$, $\biop^{+}(\phi;b)=\left\{
				0_{b}^{+}
\right\}$, reduce to a point for each $b\in \biop^{\circ}$; we say $\biop$ is
\myemph{open} if these sets are empty. We will be mainly interested in closed bios
$\biop$; the advantage being that every $n$ (co)-operation give rise to $m$ (co)-operations
for $m<n$: and for $I=(i_1,\ldots, i_m)\subseteq \left\{ 1,\dots , n \right\}$ we get maps
obtained by substituting $0_{b_i}^{-}$, resp. $0_{b_i}^{+}$, for $i\not\in I$, 
\begin{equation*}
				\biop^{-}(b_0; b_1\dots b_n)\to \biop^{-}(b_0;b_{i_1}\dots b_{i_m}) \quad , \quad
				\biop^{+}(b_1\dots b_{n};b_0)\to \biop^{+}(b_{i_1}\dots b_{i_m};b_0)
\end{equation*}
We will write Bio for the category of \myemph{closed} bios, and $\text{Bio}^{\phi}$ for
the category of \myemph{open} bios.
We can also fix the set of objects $C$, and we have the categories $\text{Bio}_{C}$
(resp. $p\text{Bio}_{C}$) with objects the closed bios $\biop$ with $\biop^{\circ}=C$, and maps of
bios $\varphi$ such that $\varphi^{\circ}\equiv \text{id}_C$.
				\label{remark:2.12}
\end{remark}
\begin{remark}
				For a bio $\biop$ we have the underlying categories $i^*\biop^{-}$ and
				$i^*\biop^{+}$, and for $f^{\pm}\in \biop^{\pm}(b;c)$ we have 
				\begin{equation}
								\overrightarrow{f^{-}}:= f^{-}\circrightarrow \unit^{+}_{c}\in
								\biop^{+}(b; c) \quad \text{and} \quad
								\overleftarrow{f^{+}}:=\unit^{-}_{b}\circleftarrow f^{+} \in
								\biop^{-}(b;c). 
								\label{eq:2.14}
				\end{equation}
				These are inverse isomorphisms 
				\begin{equation}
								\overset{\leftrightarrows}{f^{-}} = \unit_{b}^{-}\circleftarrow \left( f^{-}
												\circrightarrow \unit_{c}^{+}
								\right) = (\unit_{b}^{-}\circ f^{-})\circleftarrow
								\unit_{c}^{+} =
								f^{-}
								\label{eq:2.15}
				\end{equation}
				and similarly $\overset{\rightleftarrows}{f^{+}}=f^{+}$. For $g^{\pm}\in
				\biop^{\pm}(c; d)$ we have 
				\begin{equation}
								\begin{array}[H]{lll}
								\overrightarrow{f^{-}\circ g^{-}} &= (f^{-}\circ g^{-})\circrightarrow
								\unit_{d}^{+}  = f^{-}\circrightarrow (g^{-}\circrightarrow
								\unit_{d}^{+})  \\\\ 
								&= f^{-}\circrightarrow \big( \unit_{c}^{+}\circ
								(g^{-}\circrightarrow \unit_{d}^{+})\big)  
								= (f^{-}\circrightarrow \unit_{c}^{+})\circ (g^{-}\circrightarrow
								\unit_{d}^{+}) \\\\ 
								&= \overrightarrow{f^{-}}\circ \overrightarrow{g^{-}}
								\end{array}
								\label{eq:2.16}
				\end{equation}
				and similarly $\overleftarrow{f^{+}\circ g^{+}}=\overleftarrow{f^{+}}\circ
				\overleftarrow{g^{+}}$. \\
				Using these canonical isomorphisms we identify $i^{*}\biop^{-}$ and $i^{*}\biop^{+}$, and we
				write $\unit_{c}$ for either $\unit_{c}^{-}$ or $\unit_{c}^{+}$, $c\in
				\biop^{\circ}$. Note that for $P\in\biop^{-}(b;c_1\cdots c_n)$,
				$Q\in\biop^{+}(c_1\cdots c_n;c)$, this identifies $P\circleftarrow Q$ with
				$P\circrightarrow Q$.
				\label{remark:2.13}
\end{remark}
\begin{remark}
				Every bio $\biop= (\biop^{-},\biop^{\circ},\biop^{+})$ has an \myemph{opposite} bio
				\break
				$\biop^{\op}=(\biop^{+},\biop^{\circ},\biop^{-})$, with the same objects, and
				for $c_0,c_1,\dots, c_n\in \biop^{\circ}$, 
				\begin{equation}
								\begin{array}[H]{l}
								(\biop^{\op})^{-}(c_0;c_1\dots c_n):= \biop^{+}(c_1\dots c_n;c_0),
								\\\\
								(\biop^{\op})^{+}(c_1\dots c_n; c_0):= \biop^{-}(c_0;c_1\dots c_n)	
								\end{array}
								\label{eq:2.18} 
				\end{equation}
				and with the opposite compositions and actions, so 
				\begin{equation}
								P^{\op}\circleftarrow \left( P_i^{\op} \right):= \big(\left(
								P_i \right)\circrightarrow P\big)^{\op}\quad , \quad
								(P_i^{\op})\circrightarrow P^{\op}:= \big(P\circleftarrow
								(P_i)\big)^{\op}
								\label{eq:2.19}
				\end{equation}
				This is an involution, $(\biop^{\op})^{\op}=\biop$, of Bio extending the
				one on 
				\begin{equation*}
								\text{Cat}\;\; (=\text{the category of small categories})
				\end{equation*}
				\label{remark:2.17}
\end{remark}
\begin{remark}
				For a symmetric monoidal category $\mathcal{E}$, we can similarly define the
				category $\text{Bio}(\mathcal{E})$ of $\mathcal{E}$-enriched bios
				$\biop=(\biop^{-},\biop^{\circ},\biop^{+})$. Here for $c_0,c_1\dots c_n$ we have
				objects of $\mathcal{E}$, $\biop^{-}(c_0;c_1\dots c_n)$, $\biop^{+}(c_1\dots
				c_n;c_0)\in \mathcal{E}^{\circ}$, and the compositions, actions, and units, are
				given by maps in $\mathcal{E}$, replacing the Cartesian product $\times$ in Set by
				the monoidal product $\otimes$ in $\mathcal{E}$. \vspace{.1cm}\\ 
				For an $\mathcal{E}$-enriched
				operad $\biop\in \text{Oper}(\mathcal{E})$, we get the bio $j_{!} \biop\in
				\text{Bio}(\mathcal{E})$, with $j_{!}\biop:=(\biop,\biop^{\circ},i^{*}\biop)$, i.e. with the
				same operations, but with only unaray cooperations. This is left adjoint to the
				forgetfull functor $j^{*}:\text{Bio}(\mathcal{E})\to\text{Oper}(\mathcal{E})$,
				\break
				$j^{*}\biop=(\biop^{-},\biop^{\circ})$. There is a right adjoints $i_{*}$ 
				 given by
defining the missing $m$-array operations, $m>1$ to be a point. We thus have the
adjunctions 
\begin{equation}
				\begin{tikzpicture}[baseline=-2mm]
				\coordinate (A) at (0,0);
				\node at (0,-.1cm) [] {$\text{Cat}(\mathcal{E})$};
				\node at (1.8cm,0.35cm) [] {$i_{!}$};
				\draw[{Hooks[right,length=5,width=6]}->,color=black] (A)+(.6cm,0.15cm) -- (3cm,0.15cm);	
				\draw[arrows={<<-},color=black] (A)+(.6cm,-0.1cm) -- (3cm,-0.1cm);	
				\draw[->,color=black] (0.6cm,-0.35cm) to [bend right] (3cm,-0.35cm);	
				\node at (1.8cm,-0.3cm) [] {$i^{*}$};
				\node at (2.8cm,-.8cm) [] {$i_{*}$};
				\node at (3.8cm,-.1cm) [] {$\text{Oper}(\mathcal{E})$};
				\node at (5.8cm,0.35cm) [] {$j_{!}$};
				\draw[{Hooks[right,length=5,width=6]}->,color=black] (A)+(4.6cm,0.15cm) -- (7cm,0.15cm);	
				\draw[arrows={<<-},color=black] (A)+(4.6cm,-0.1cm) -- (7cm,-0.1cm);	
				\node at (5.8cm,-0.3cm) [] {$j^{*}$};
				\node at (7.8cm,-.1cm) [] {$\text{Bio}(\mathcal{E})$};
\end{tikzpicture}
\label{eq:2.21}
\end{equation}
and $i_{!},j_{!}$ are full and faithful embeddings. \vspace{.1cm}\\
For a symmetric monoidal category
$\mathcal{E}$, we have the bio $\hat{\mathcal{E}}$, with the same objects
$\mathcal{E}^{\circ} $ as the category $\mathcal{E}$, and with (co)-operations for
$c_0,c_1\dots c_n\in\mathcal{E}^\circ$
\begin{equation*}
				\hat{\mathcal{E}}^{-}(c_0;c_1\dots c_n):=\mathcal{E}(c_1\otimes\dots\otimes
				c_n,c_0) \quad , \quad \hat{\mathcal{E}}^{+}(c_1\dots
				c_n;c_0):=\mathcal{E}(c_0,c_1\otimes\dots\otimes c_n)
\end{equation*}
with the obvious compositions and mutual actions (again given by compositions); this gives
an embedding 
\begin{equation*}
				\text{Sym}\;\tikz \draw[{Hooks[right,length=5,width=6]}->,color=red] (0cm,0cm) -- (2cm,0cm);\;
				\text{Bio} \quad , \quad \mathcal{E} \;\tikz \draw[arrows={|->},color=red] (0cm,0cm) --
				(2cm,0cm);\;\hat{\mathcal{E}}
\end{equation*}
Forgetting the $S_n$-actions we get a similar embedding of (non-symmetric) monoidal
categories into planar bios 
\begin{equation}
				M\text{Cat}\;
				\tikz \draw[{Hooks[right,length=5,width=6]}->] (0,0)--(1cm,0);
				\;p\text{Bio} \quad , \quad \mathcal{E}\; \tikz \draw[|->] (0,0)--(1cm,0);
				\; \hat{\mathcal{E}}.
				\label{eq:2.23}
\end{equation}
If the symmetric monoidal category $\mathcal{E}$ is \myemph{closed}, so that we have
an internal hom functor
\begin{equation}
				\mathcal{E}^{\op}\times\mathcal{E} \;\tikz \draw[arrows={->},color=red]
				(0cm,0cm) -- (2cm,0cm);\;	\mathcal{E} \quad , \quad (b,c) \;\tikz
				\draw[arrows={|->},color=red] (0cm,0cm) -- (2cm,0cm);\; c^{b}
				\label{eq:2.24}
\end{equation}
and adjunction  $\mathcal{E}(b_1\otimes b_2,c)=\mathcal{E}(b_1,c^{b_2})$, then we can view
$\hat{\mathcal{E}}$ as an $\mathcal{E}$-enriched bio 
\begin{equation}
				\hat{\mathcal{E}}^{-}(c_0;c_1\dots c_n):= c_0^{c_1\otimes \dots \otimes c_n} \quad ,
				\quad \hat{\mathcal{E}}^{+}(c_1\dots c_n;c_0):=(c_1\otimes \dots \otimes
				c_n)^{c_0}
				\label{eq:2.25}
\end{equation}
For the category of simplicial sets
$\mathcal{E}=\text{sSet}=(\text{Set})^{\bbDelta^{\text{op}}}$, we shall write \break
$\text{sBio}\equiv  \text{Bio}(\text{sSet})$ for simplicial bios. For the category of
(compactly generated)
topological spaces Top we shall write $\text{tBio}\equiv\text{Bio}(\text{Top})$ for
topological bios. 
\label{remark:2.20}
\end{remark}
\begin{definition}
				A collection
				$\mathcal{C}=(\mathcal{C}^{-},\mathcal{C}^{\circ},\mathcal{C}^{+})$ is a
				set $\mathcal{C}^{\circ}$, and for $c_0,c_1\dots c_n\in\mathcal{C}^{\circ}$ we
				have sets $\mathcal{C}^{-}(c_0;c_1\dots c_n)$ and $\mathcal{C}^{+}(c_1\dots
				c_n;c_0)$ and when $n=1$ we have $\mathcal{C}^{-}(c_0;c_1)\equiv
				\mathcal{C}^{+}(c_0;c_1)$, and for $c_{0}=c_{1}=c$ we have $\unit_{c}\in
				\mathcal{C}^{\pm}(c;c)$. 
				\label{def:2.26}
				\end{definition}
				A map of collections $\varphi:\mathcal{C}\to \mathcal{B}$ is a function
				$\varphi^{\circ}:\mathcal{C}^{\circ}\to \mathcal{B}^{\circ}$, and for $c_0,c_1\dots c_n\in
				\mathcal{C}^{\circ}$ functions $\varphi^{-}:\mathcal{C}^{-}(c_0;c_1\dots c_n)\to 
				\mathcal{B}^{-}\left( \varphi^{\circ}(c_0);\varphi^{\circ}(c_1)\dots
				\varphi^{\circ}(c_n) \right)$ and $\varphi^{+}:\mathcal{C}^{+}(c_1\dots
				c_n;c_0)\to \mathcal{B}^{+}\left(
				\varphi^{\circ}(c_1)\dots\varphi^{\circ}(c_n);\varphi^{0}(c_0) \right)$ with
				$\varphi^{+}=\varphi^{-}$ when $n=1$, and
				$\varphi^{\pm}(\unit_{c})=\unit_{\varphi^{\circ}(c)}$. Thus we have a category
				\myemph{Coll}, and a forgetful functor 
				\begin{equation}
								U:\text{pBio}\to\text{Coll}
								\label{eq:2.27}
				\end{equation}
				We shall describe next its left adjoint 
				\begin{equation}
								p\mathcal{F}:\text{Coll}\to \text{pBio}
								\label{eq:2.28}
				\end{equation}
				which gives the left adjoint of $U:\text{Bio}\to\text{pBio}\to\text{Coll}$, the
				\myemph{free bio} on $\mathcal{C}$  $\mathcal{F}_{\mathcal{C}}=\sum
				p\mathcal{F}_{\mathcal{C}}$.
\section{Bit $\equiv$ Bi-Trees} 
\begin{definition}
				A \myemph{graph}
				$ 
				\begin{tikzpicture}[baseline]
								\node at (0cm,0.1cm) [] {$ (G^{\unit}$};
								\node at (1.4cm,0.1cm) [] {$ G^{\circ})$};
								\draw[->] (0.3cm,0.05cm)--(1cm,0.05cm);
								\draw[->] (0.3cm,0.15cm)--(1cm,0.15cm);
								\node at (0.7cm,-0.1cm) [] {$\scriptstyle d_{-}$};
								\node at (0.7cm,0.35cm) [] {$\scriptstyle d_{+}$};
\end{tikzpicture}
				$ for us will always be finite and \break oriented, so is just given by two parallel functions of
				finite sets: $G^{1}$ the ``\myemph{edges}'' to $G^{\circ}$ the
				``\myemph{vertices}''. 
				\label{def:graph}
\end{definition}
				As a general rule we write our edges from right
				$\equiv$ positive to the \break 
				left $\equiv$ negative 
				\begin{equation}
								\begin{tikzpicture}
											\draw[decorate,decoration={markings,mark=at position 0.5 with
											{\arrow[color=black]{<}}}] (0,0)--(3,0);	
											\draw[-] (0,0)--(3,0);
											\node at (1.5cm,.3cm) [] {$e$};
							\node at (-.2cm,.3cm) [] {$d_{-}e$}; 
											\filldraw (-.2cm,0cm) circle [radius=.04cm];
											\filldraw (3.2cm,0cm) circle [radius=.04cm];
							\node at (3.2cm,.3cm) [] {$d_{+}e$}; 
								\end{tikzpicture}
								\label{eq:3.2}
				\end{equation}
				A map of graphs $\varphi:G\to H$ is a pair of functions $\varphi^{1}:G^{1}\to
				H^{1}$ and \break $\varphi^{\circ}:G^{\circ}\to H^{\circ}$ such that
				$d_{\pm}^{H}\circ\varphi^{1}=\varphi^{\circ}\circ d_{\pm}^{G}$. \vspace{.1cm}\\ 
				Thus we have a category of (finite, oriented) \myemph{Graph}. \vspace{.1cm}\\
				For $v\in G^{\circ}$, we write $ d_{-}v = d_{+}^{-1}(v)$, $d_{+}v =
				d_{-}^{-1}(v)$. \vspace{.1cm}\\
				For $G\in \text{Graph}$, $n,m\ge 0$, put $ C_{n,m}(G):= \left\{ v \in G^{\circ}, \# d_{-}v = n, \# d_{+}v = m
				\right\}$. \vspace{.1cm}\\
				The left (resp. right) ends of $G$, $C_{01}(G)$ (res. $C_{10}(G)$), are called
				\myemph{stumps}, and they are in bijection with the edges attached to them,
				which we will denote by 
				\begin{equation}
								d_{-}(G):= d_{-}^{-1}C_{01}(G)\xrightarrow{\;\;\underset{\sim}{d_{-}}\;\;}
								C_{01}(G) \quad , \quad d_{+}G:=d_{+}^{-1}C_{10}(G) \xrightarrow{\;\;\underset{\sim}{d_{+}}\;\;}
								C_{10}(G)
								\label{eq:3.3}
				\end{equation}
		We say $G$ is a $\unit$-Graph if 
		\begin{equation}
		C_{nm}(G)\not=\phi \;\; \Longrightarrow \;\; n=1 \text{ or } m=1
		\label{eq:3.4}
		\end{equation}
		Thus for a $\unit$-graph $G$
		\begin{equation}
						\begin{array}[H]{rll}
						G^{\circ} &= C_{01}(G) \coprod C^{-}(G) \coprod C^{+}(G) \coprod
						C_{10}(G) \\\\
						C^{-}(G) &= \bigcup_{m\ge 1} C_{1m}(G) \quad, \quad
						C^{+}(G) = \bigcup_{n\ge 1} C_{n1}(G)  
						\end{array}
						\label{eq:3.5}
		\end{equation}
Such a $\unit$-graph is made up of the basic negative and positive
$n$-\myemph{corollas}, $n\ge 0$, 
\begin{equation}
				\begin{tikzpicture}[baseline]
								\node at (-3cm,0.6cm) [] {$C_{n}^{-}$:}; 
								\filldraw (-2.4cm,-0.5cm) circle [radius=.04cm];
							\draw[decorate,decoration={markings,mark=at position 0.5 with
							{\arrow[color=black]{<}}}] (-2.2cm,-0.5cm)--(-1cm,-0.5cm);	
							\draw[-] (-2.2cm,-0.5cm)--(-1cm,-0.5cm);
							\node at (-1.8cm,-0.25cm) [] {$e_0$};
							\filldraw (-0.8cm,-0.5cm) circle [radius=.04cm];
							\draw[-] (-0.7cm,-0.45cm)--(0.3cm,0.25cm);
							\draw[decorate,decoration={markings,mark=at position 0.5 with
							{\arrow[color=black]{<}}}] (-0.7cm,-0.45cm)--(0.3cm,0.25cm);	
							\node  at (-0.26cm,0.1cm) {$e_1$};
							\node  at (-0.20cm,-0.7cm) {$e_n$};
							\filldraw (0.4cm,0.3cm) circle [radius=.04cm];
							\filldraw (0.4cm,-1.3cm) circle [radius=.04cm];
							\draw[-] (-0.7cm,-0.6cm)--(0.3cm,-1.25cm);
							\draw[decorate,decoration={markings,mark=at position 0.5 with
							{\arrow[color=black]{<}}}] (-0.7cm,-0.6cm)--(0.3cm,-1.25cm);	
							\filldraw (0.1cm,-0.8cm) circle [radius=.01cm];
							\filldraw (0.1cm,-0.5cm) circle [radius=.01cm];
							\filldraw (0.1cm,-0.2cm) circle [radius=.01cm];
							\node at (2cm,0.6cm) [] {$C_{n}^{+}$:};
							\filldraw (5.6cm,-0.5cm) circle [radius=.04cm];
							\draw[decorate,decoration={markings,mark=at position 0.5 with
							{\arrow[color=black]{>}}}] (5.5cm,-0.5cm)--(4.5cm,-0.5cm);
							\draw[-] (5.5cm,-0.5cm)--(4.5cm,-0.5cm);
							\filldraw (4.4cm,-0.5cm) circle [radius=.04cm];
							\filldraw (3.4cm,0.3cm) circle [radius=.04cm];
							\filldraw (3.4cm,-1.3cm) circle [radius=.04cm];
							\filldraw (3.6cm,-0.8cm) circle [radius=.01cm];
							\filldraw (3.6cm,-0.5cm) circle [radius=.01cm];
							\filldraw (3.6cm,-0.2cm) circle [radius=.01cm];
							\draw[-] (4.3cm,-0.45cm)--(3.5cm,0.2cm);
							\draw[decorate,decoration={markings,mark=at position 0.5 with
							{\arrow[color=black]{>}}}] (4.3cm,-0.45cm)--(3.5cm,0.2cm);
							\draw[-] (4.3cm,-0.6cm)--(3.5cm,-1.25cm);
							\draw[decorate,decoration={markings,mark=at position 0.5 with
							{\arrow[color=black]{>}}}] (4.3cm,-0.6cm)--(3.5cm,-1.25cm);
							\node  at (4.2cm,-1.1cm) {$e_n$};
							\node  at (4.2cm,0cm) {$e_1$};
							\node  at (5.1cm,-0.25cm) {$e_0$};
				\end{tikzpicture}
				\label{eq:3.6}
\end{equation}
A \myemph{path} in a graph $G$, ``from $v_m$, resp. $e_m$, to $v_0$, resp. $e_1$'', 
$e_i\in G^{1}$, $v_i\in G^{0}$, is given by 
\begin{equation}
				\begin{tikzpicture}[baseline]
				\tikzmath{
								let \x =1.2;
								let \y =0;
								let \w = 1;
								let \s = 0.3;
				};
				\def\mlarrow#1#2#3#4{
							\draw[-] #1--#2;
							\draw[decorate,decoration={markings,mark=at position 0.5 with {\arrow[color=black]{<}}}] #1--#2;
							\node at ($0.5*#1+0.5*#2+(0,0.3)$) {$#3$};
							\filldraw ($#1-(\s,0cm)$) circle [radius=.04cm];
							\node at ($#1-(\s,-0.3cm)$) {$#4$};
			};
		  \node at (0,0) {$\ell= \Big($};
			\mlarrow{(\x,\y)}{(\x+\w,\y)}{e_1}{v_0};
			\mlarrow{(\x+\w+2*\s,\y)}{(\x+2*\s+2*\w,\y)}{e_2}{v_1};
			\filldraw  (\x+3*\s+2*\w,\y) circle [radius=.04cm];
			\filldraw  (\x+4*\s+2*\w,\y) circle [radius=.01cm];
			\filldraw  (\x+5*\s+2*\w,\y) circle [radius=.01cm];
			\filldraw  (\x+6*\s+2*\w,\y) circle [radius=.01cm];
			\mlarrow{(\x+4*\w+\s,\y)}{(\x+\s+5*\w,\y)}{e_{m}}{};
			\filldraw  (\x+2*\s+5*\w,\y) circle [radius=.04cm];
			\node at (\x+2*\s+5*\w,\y+0.3cm) {$v_{m}$};
			\node at (\x+4*\s+5*\w,\y) {$\Big)$};
\end{tikzpicture} = (e_1,\cdots, e_m)
\label{eq:3.7}
\end{equation}
with $d_{-}e_i=v_{i-1}, d_{+}e_{i}=v_i$, $i=1\cdots m$. We say $G$ is \myemph{simple} if
there are No \myemph{closed} paths $(v_0=v_m)$. For simple $G$ we get a partial order
$\le $ on $G^{1}\coprod G^{\circ}$: 
\begin{equation}
				v,\;\text{resp. } e \le v^{\prime},\; \text{resp. } e^{\prime}\Leftrightarrow
				\exists \; \text{path from $v^{\prime}$, resp. $e^{\prime}$, to $v$, resp. e.}
				\label{eq:3.8}
\end{equation}
Every path in a simple graph $G$ extends to maximal path. \vspace{.1cm}\\ 
A path $\ell=(e_1,\cdots , e_m)$
is maximal iff $e_1\in d_{-}G$, $e_{m}\in d_{+}G$. 
\begin{definition}
				A \myemph{Bit} ($\equiv \text{Bi-Tree}$) $G$ is a simple $\unit$-graph. \\
				We write
				\begin{equation*}
								\begin{array}[H]{ll}
								\text{Bit}^{k}_{n,m} &=  \left\{ G\in Bit,\; \# d_{-}G=n,\; \# d_{+}G=m,
												\; \#
								G^{\circ}=n+m+k \right\}/\text{isom} \\\\
								\text{Bit}_{n,m} &= \bigcup_{k}\text{Bit}_{n,m}^{k}\quad , \quad
								\text{Bit}=\bigcup_{n,m}\text{Bit}_{n,m} \\\\
								\text{Bit}^{-}&= \bigcup_{m}\text{Bit}_{1,m}\quad , \quad
								\text{Bit}^{+}=\bigcup_{n}\text{Bit}_{n,1}
								\end{array}
				\end{equation*}
				A \myemph{planar structure} on a bit $G$ is a linear order on
				$d_{\pm}(v)$ for each $v\in G^{\circ}$, and a linear order on $d_{\pm}G$, so
				that we can write
				\begin{equation}
								\begin{array}[t]{ll}
								\ & d_{-}v = \left\{ 0_{v}\right\}\quad,\quad d_{+}v=\left\{ \unit_{1}(v),\cdots ,
								\unit_{m}(v) \right\} \quad , \quad v\in C^{-}(G) \\\\
								\text{resp. } &
												d_{-}v = \left\{ 0_1(v),\cdots , 0_{n}(v) \right\} \quad , \quad  d_{+}v=\left\{
												1_{v} \right\}\quad\, ,  \quad\, v\in C^{+}(G)
								\end{array}
								\label{eq:3.10}
				\end{equation}
				\begin{equation}
								\begin{array}[H]{ll}
								d_{-}(G)=\left\{ 0_1(G),\cdots, 0_{n}(G) \right\}, \qquad
								 d_{+}G=\left\{ \unit_{1}(G),\ldots , \unit_{m}(G) \right\}. 
								\end{array}
								\label{eq:3.11}
				\end{equation}
				\label{def:3.9}
\end{definition}
				Thus we have the similarly defined sets of isomorphism classes preserving the
				planar structure $p\text{Bit}$, $p\text{Bit}^{-}$, $p\text{Bit}^{+}$. 
				\begin{remark}
								We only define the ``closed'' bits. One can similarly define the more
								general bits by keeping two subsets $\text{out}(G)\subseteq d_{-}G$ and
								$\text{in}(G)\subseteq d_{+}G$ of the ``open-stumps''.  
								\label{remark:3.12}
				\end{remark}
				\begin{definition}
								For $G=\big( G^{1}  \rdarrow{}{d_{\pm}} G^{\circ}\big)\in p \text{Bit}$,
								$\mathcal{C}=(\mathcal{C}^{-},\mathcal{C}^{\circ},\mathcal{C}^{+})\in \text{Coll}$, a \myemph{labeling}
								$f:G\to \mathcal{C}$ is given by functions $f^{1}:G^{1}\to \mathcal{C}^{\circ}$ and
								$f^{\pm}:C^{\pm}(G)\to \mathcal{C}^{\pm}$, such that for $v\in C^{-}(G)$, resp.
								$C^{+}(G)$, 
								\begin{equation*}
												\begin{array}[H]{ll}
												\ & f^{-}(v)\in \mathcal{C}^{-}\left( f^{1}(0_{v});f^{1}(\unit_{1}(v)),\cdots ,
												f^{1}(\unit_{m}(v))\right) \\\\
								\text{resp. }  & f^{+}(v)\in \mathcal{C}^{+}\left( f^{1}(0_{1}(v)),\cdots ,
												f^{1}(0_n(v));f^{1}(\unit_{v}) \right)
												\end{array}
								\end{equation*}
								\begin{equation*}
												f^{+}(v)=f^{-}(v) \; \text{for} \; v\in C_{11}(G).
								\end{equation*}
								\label{def:3.4}
				\end{definition}
			\section{Formulas and $\unit$-sub-Bits}
			Contemplating the following bits in $\bigbit_{2,2}$ we see their structure can be
			complicated: \\
			\begin{figure}[H]
							\centering
			\begin{tikzpicture}[baseline]
							\tikzmath{
											let \buff=0.1cm;
											let \fdist=0.7cm;
											let \sdist=0.7cm;
											let \down=-0.5cm;
							}

							\def\hrarrow#1#2#3{
							   \draw[decorate,decoration={markings,mark=at position #3 with
											{\arrow[color=black]{>}}}] #1--#2;
								 \draw[-] #1--#2;
				 };
				 \coordinate (a) at (0,-0.5cm);
				 \coordinate (b) at (\fdist,-0.5cm);
				 \coordinate (c) at ($(b)+(\sdist,0.5cm)$);
				 \coordinate (d) at ($(b)+(\sdist,-0.5cm)$);
				 \coordinate (e) at ($(a)+(-\sdist,0.5cm)$);
				 \coordinate (f) at ($(a)+(-\sdist,-0.5cm)$);
				 \filldraw (a) circle [radius=0.04cm];
				 \filldraw (b) circle [radius=0.04cm];
				 \filldraw (c) circle [radius=0.04cm];
				 \filldraw (d) circle [radius=0.04cm];
				 \filldraw (e) circle [radius=0.04cm];
				 \filldraw (f) circle [radius=0.04cm];
				 \hrarrow{($(b)+(-\buff,0)$)}{($(a)+(\buff,0)$)}{0.5};
				 \hrarrow{($(c)+(-\buff,-\buff)$)}{($(b)+(\buff,\buff)$)}{0.5};
				 \hrarrow{($(d)+(-\buff,\buff)$)}{($(b)+(\buff,-\buff)$)}{0.5};
				 \hrarrow{($(a)+(-\buff,\buff)$)}{($(e)+(\buff,-\buff)$)}{0.5};
				 \hrarrow{($(a)+(-\buff,-\buff)$)}{($(f)+(\buff,\buff)$)}{0.5};
			\end{tikzpicture} \qquad
			\begin{tikzpicture}[baseline]
							\tikzmath{
											let \buff=0.1cm;
											let \fdist=1.5cm;
											let \sdist=0.8cm;
											let \down=-0.6cm;
							}

							\def\hrarrow#1#2#3{
							   \draw[decorate,decoration={markings,mark=at position #3 with
											{\arrow[color=black]{>}}}] #1--#2;
								 \draw[-] #1--#2;
				 };
										\coordinate (a) at (0,-0.2cm);
										\coordinate (b) at ($(a)+(\sdist,0)$);
										\coordinate (c) at ($(b)+(\fdist,0)$);
										\coordinate (ad) at ($(0,-0.2cm)+(0,\down)$);
										\coordinate (bd) at ($(ad)+(\fdist,0)$);
										\coordinate (cd) at ($(bd)+(\sdist,0)$);
										\filldraw (a) circle [radius=0.04cm];
										\filldraw (b) circle [radius=0.04cm];
										\filldraw (c) circle [radius=0.04cm];
										\filldraw (ad) circle [radius=0.04cm];
										\filldraw (bd) circle [radius=0.04cm];
										\filldraw (cd) circle [radius=0.04cm];
										\hrarrow{($(c)+(-\buff,0)$)}{($(b)+(\buff,0)$)}{0.5};
										\hrarrow{($(b)+(-\buff,0)$)}{($(a)+(\buff,0)$)}{0.5};
										\hrarrow{($(cd)+(-\buff,0)$)}{($(bd)+(\buff,0)$)}{0.5};
										\hrarrow{($(bd)+(-\buff,0)$)}{($(ad)+(\buff,0)$)}{0.5};
										\hrarrow{($(bd)+(-\buff,\buff)$)}{($(b)+(\buff,-\buff)$)}{0.5};
			\end{tikzpicture} \qquad
							\begin{tikzpicture}[baseline, scale=1]
							\tikzmath{
											let \ax=0;
											let \ay=0;
											let \buff=0.1;
											let \down=-1;
											let \mid=3cm;
											let \mlen=1cm;
							};
							\def\hrarrow#1#2#3{
							   \draw[decorate,decoration={markings,mark=at position #3 with
											{\arrow[color=black]{>}}}] #1--#2;
								 \draw[-] #1--#2;
				 };
										\coordinate (a) at (\ax,\ay);
										\coordinate (b) at ($(-\mlen,0)+(-\buff,0)$);
										\coordinate (c) at ($(-\mid,0)+(-\buff,0)$);
										\coordinate (d) at ($(c)+(-\mlen,0)$);
										\coordinate (ad) at (\ax,\down);
										\coordinate (bd) at ($(-\mlen,\down)+(-\buff,0)$);
										\coordinate (cd) at ($(-\mid,\down)+(-\buff,0)$);
										\coordinate (dd) at ($(c)+(-\mlen,\down)$);
										\filldraw (a) circle  [radius=0.04cm];
										\hrarrow{(-.2cm,0)}{(-\mlen,0)}{0.5};
										\filldraw (b) circle [radius=0.04cm];
										\filldraw (c) circle [radius=0.04cm];
										\hrarrow{($(b)+(-.2cm,0)$)}{($(\buff,0)+(c)$)}{0.5};
										\filldraw (d) circle [radius=0.04cm];
										\hrarrow{($(c)+(-.2cm,0)$)}{($(\buff,0)+(d)$)}{0.5};

										\filldraw (ad) circle  [radius=0.04cm];
										\hrarrow{($(ad)+(-\buff,0)$)}{($(ad)+(-\mlen,0)$)}{0.5};
										\filldraw (bd) circle [radius=0.04cm];
										\filldraw (cd) circle [radius=0.04cm];
										\hrarrow{($(bd)+(-.2cm,0)$)}{($(\buff,0)+(cd)$)}{0.5};
										\filldraw (dd) circle [radius=0.04cm];
										\hrarrow{($(cd)+(-.2cm,0)$)}{($(\buff,0)+(dd)$)}{0.5};
										\hrarrow{($(bd)+(-\buff,\buff)$)}{($(\buff,-\buff)+(c)$)}{0.7};
										\hrarrow{($(b)+(-\buff,-\buff)$)}{($(\buff,\buff)+(cd)$)}{0.7};
										\filldraw (1.2cm,-.5cm) circle [radius=0.04cm];
										\filldraw (0.9cm,-.5cm) circle [radius=0.04cm];
										\filldraw (0.6cm,-.5cm) circle [radius=0.04cm];
			\end{tikzpicture}
			\caption{Bits in $\text{Bit}_{2,2}$}
							\label{fig:3}
			\end{figure}
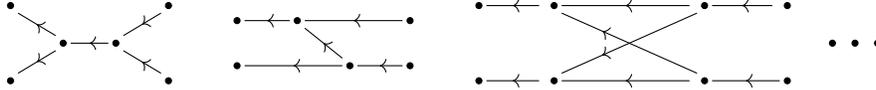 \ \\
			We define the subset $\text{fBit}^{\pm}$ of $p\bigbit^{\pm}$ of ``formulas'' having
			\myemph{compatible} linear orders on the $d_{\pm}(v)$'s and on $d_{\pm}(G)$, by an
			inductive definition obtained by grafting of bits: 
			\begin{equation}
							\begin{array}[H]{l}
							\text{fBit}_{1,1}^{\circ}= \left\{ I=\left( \overset{0}{\cdot}\leftarrow
							\overset{1}{\cdot} \right) \right\}	 \\\\
							\text{fBit}_{1,m}^{\unit}=\left\{ C_{m}^{-} \right\}\quad , \quad
							\text{fBit}_{n,1}^{1}=\left\{ C_n^{+} \right\}
							\end{array}
							\label{eq:4.1}
			\end{equation}
			There are two inductive steps given by compositions and actions: \\
			For $G\in
			\text{fBit}_{1,m}^{k}$, $d_{+}G=\left\{ e_1,\cdots , e_{m} \right\}$, $i\in\left\{ 1,\cdots
							, m
			\right\}$, $n\ge 1$ we have
			\begin{equation}
							G \circ_{i} C_{n}^{-}\in \bigbit_{1,n+m-1}^{k+1};
							\label{eq:4.2}
			\end{equation}
			Picturially it is given by 
			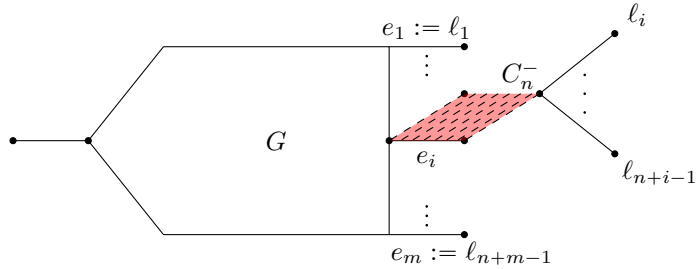
\begin{figure}[H]
							\centering
							\begin{tikzpicture}
												\tikzmath{
																let \lBpoly=3;
																let \hBpoly=2.5;
																let \bAng = 1;
																let \relHeight = 0;
																let \simConst= 0.5 ;
																let \padding = 0;
																let \step=0.15;
																let \seg=1;
																let \spc =5.5*\seg ;
																let \over=0.25;
												};
												\def\Poly#1#2#3#4#5{ 
																				\draw[-]
																								#1--($#1+(#2,0)$)--($#1+(#2,-#3)$)--($#1+(0,-#3)$)--($#1-(#4,0.5*#3)$)--#1;
																				\node at ($#1+(0.5*#2,-0.5*#3)$) {#5};
															};
												\coordinate (A) at (0,0); 
												\Poly
																{(A)}
																{\lBpoly}
																{\hBpoly}
																{\bAng}{$G$};
																\coordinate (URC) at ($(A)+(\lBpoly,0)$);	
																\coordinate (ULC) at ($(A)$);	
																\coordinate (LRC) at ($(A)+(\lBpoly,-\hBpoly)$);	
																\coordinate (LLC) at ($(A)+(0,-\hBpoly)$);	
																\coordinate (LTP) at ($(A)-(\bAng,0.5*\hBpoly)$);
																\coordinate (MID) at ($(A)+(\lBpoly,-0.5*\hBpoly)$);

																\coordinate (urc) at ($(A)+(\spc,0)+\simConst*(\lBpoly,0)$);	
																\coordinate (ulc) at ($(A)+(\spc,0)$);	
																\coordinate (lrc) at ($(A)+(\spc+\simConst*\lBpoly,-\simConst*\hBpoly)$);	
																\coordinate (llc) at ($(A)+(\spc,-\simConst*\hBpoly)$);	
																\coordinate (ltp) at ($(A)+(\spc,0)-\simConst*(\bAng,0.5*\hBpoly)$);

																\filldraw ($(LTP)-(\padding,0)$) circle [radius=0.04cm];
																\draw[-] ($(LTP)-(2*\padding,0)$)--($(LTP)-(2*\padding+\seg,0)$);
																\filldraw ($(LTP)-(2*\padding+\seg,0)-(\padding,0)$) circle [radius=0.04cm];

																\draw[-] ($(URC)+(\padding,0)$)--($(URC)+(\padding+\seg,0)$);
																\node at ($(URC)+(\padding+0.5*\seg,\over)$) {$e_{1}:=\ell_1$};
																\filldraw ($(URC)+(2*\padding+\seg,0)$) circle [radius=0.04cm];

																\draw[-] ($(LRC)+(\padding,0)$)--($(LRC)+(\padding+\seg,0)$);
																\filldraw ($(LRC)+(2*\padding+\seg,0)$) circle [radius=0.04cm];
																\node at ($(LRC)+(\padding+1.1*\seg,-\over)$)
																{$e_{m}:=\ell_{n+m-1}$};
																\filldraw (MID) circle [radius=0.04cm];
																\draw[-] ($(MID)+(\padding,0)$)--($(MID)+(\seg+\padding,0)$);
																\coordinate (A3) at ($(MID)+(\seg+2*\padding,0)$);
																\filldraw  (A3) circle  [radius=0.04cm];
																\coordinate (A1) at ($(MID)+(\seg+2*\padding,0.25*\hBpoly)$);
																\filldraw (A1) circle  [radius=0.04cm];
																\coordinate (A2) at ($(ltp)+(-\padding,0)$);
																\filldraw (A2) circle [radius=0.04cm];
																\node at ($(A2)+(-\over,\over)$) {$C_{n}^{-}$};
																\fill[color=red, opacity=.4] (MID)--(A1)--(A2)--(A3)--(MID);
																\foreach \i in {0,1,2,3,4,5,6} {
																				\draw[dashed] ($(A1)+(\i*\step,0)$)--($(MID)+(\i*\step,0)$);
														};
														\node at ($0.5*(MID)+0.5*(A3)+(0,-\over)$) {$e_{i}$};
														\draw[dashed] (A2)--(A3);

																\filldraw ($(URC)+(0.5*\seg,-0.1*0.5*\hBpoly)$) circle [radius=0.01cm];
																\filldraw ($(URC)+(0.5*\seg,-0.2*0.5*\hBpoly)$) circle [radius=0.01cm];
																\filldraw ($(URC)+(0.5*\seg,-0.3*0.5*\hBpoly)$) circle [radius=0.01cm];

																\filldraw ($(LRC)+(0.5*\seg,0.1*0.5*\hBpoly)$) circle [radius=0.01cm];
																\filldraw ($(LRC)+(0.5*\seg,0.2*0.5*\hBpoly)$) circle [radius=0.01cm];
																\filldraw ($(LRC)+(0.5*\seg,0.3*0.5*\hBpoly)$) circle [radius=0.01cm];
																\coordinate (li) at ($(A2)+(\seg,0.8*\seg)$);
																\filldraw (li) circle [radius=0.04cm];
																\coordinate (lid) at ($(A2)+(\seg,-0.8*\seg)$);
																\filldraw (lid) circle [radius=0.04cm];
																\draw[-] (A2)--(li);
																\draw[-] (A2)--(lid);
																\node at ($(lid)+(0.6cm,-\over)$) {$\ell_{n+i-1}$};
																\node at ($(li)+(0.3cm,\over)$) {$\ell_i$};
																\coordinate (ddots) at ($(A2)+(0.6cm,0)$);
																\filldraw (ddots) circle [radius=0.01cm];
																\filldraw ($(ddots)+(0,\over)$) circle [radius=0.01cm];
																\filldraw ($(ddots)+(0,-\over)$) circle [radius=0.01cm];

							\end{tikzpicture}
							\caption{Inductive composition step}
							\label{fig:4}
			\end{figure} 
			\noindent For $G\in\text{fBit}_{1,m}^{k}$, $m\ge n\ge 1$, $j\le m-n+1$ we obtain
			\begin{equation}
							G \threeover{\leftarrow}{\circ}{j}
							C_{n}^{+}\in \text{Bit}_{1,m-n+1}^{k+1}
							\label{eq:4.3}
			\end{equation}
			pictorially it is given by  \\
			\begin{figure}[H]
							\centering
							\begin{tikzpicture}
												\tikzmath{
																let \lBpoly=3;
																let \hBpoly=2.5;
																let \bAng = 1;
																let \relHeight = 0;
																let \simConst= 0.5 ;
																let \padding = 0;
																let \step=0.15;
																let \seg=1;
																let \spc =5.5*\seg ;
																let \over=0.25;
												};
												\def\Poly#1#2#3#4#5{ 
																				\draw[-]
																								#1--($#1+(#2,0)$)--($#1+(#2,-#3)$)--($#1+(0,-#3)$)--($#1-(#4,0.5*#3)$)--#1;
																				\node at ($#1+(0.5*#2,-0.5*#3)$) {#5};
															};
												\coordinate (A) at (0,0); 
												\Poly
																{(A)}
																{\lBpoly}
																{\hBpoly}
																{\bAng}{$G$};
																\coordinate (URC) at ($(A)+(\lBpoly,0)$);	
																\coordinate (ULC) at ($(A)$);	
																\coordinate (LRC) at ($(A)+(\lBpoly,-\hBpoly)$);	
																\coordinate (LLC) at ($(A)+(0,-\hBpoly)$);	
																\coordinate (LTP) at ($(A)-(\bAng,0.5*\hBpoly)$);
																\coordinate (fMID) at ($(A)+(\lBpoly,-0.3*\hBpoly)$);
																\coordinate (sMID) at ($(A)+(\lBpoly,-0.7*\hBpoly)$);
																\coordinate (MID) at ($(A)+(\lBpoly,-0.5*\hBpoly)+(2*\seg,0)$);

																\coordinate (urc) at ($(A)+(\spc,0)+\simConst*(\lBpoly,0)$);	
																\coordinate (ulc) at ($(A)+(\spc,0)$);	
																\coordinate (lrc) at ($(A)+(\spc+\simConst*\lBpoly,-\simConst*\hBpoly)$);	
																\coordinate (llc) at ($(A)+(\spc,-\simConst*\hBpoly)$);	
																\coordinate (ltp) at ($(A)+(\spc,0)-\simConst*(\bAng,0.5*\hBpoly)$);

																\filldraw ($(LTP)-(\padding,0)$) circle [radius=0.04cm];
																\draw[-] ($(LTP)-(2*\padding,0)$)--($(LTP)-(2*\padding+\seg,0)$);
																\filldraw ($(LTP)-(2*\padding+\seg,0)-(\padding,0)$) circle [radius=0.04cm];

																\draw[-] ($(URC)+(\padding,0)$)--($(URC)+(\padding+\seg,0)$);
																 \draw[decorate,decoration={markings,mark=at position 0.5 with
																				 {\arrow[color=black]{<}}}] ($(URC)+(\padding,0)$)--($(URC)+(\padding+\seg,0)$);

																\node at ($(URC)+(\padding+0.5*\seg,\over)$) {$e_{1}:=\ell_1$};
																\filldraw ($(URC)+(2*\padding+\seg,0)$) circle [radius=0.04cm];

																\draw[-] ($(LRC)+(\padding,0)$)--($(LRC)+(\padding+\seg,0)$);
																 \draw[decorate,decoration={markings,mark=at position 0.5 with
																				 {\arrow[color=black]{<}}}] ($(LRC)+(\padding,0)$)--($(LRC)+(\padding+\seg,0)$);
																\filldraw ($(LRC)+(2*\padding+\seg,0)$) circle [radius=0.04cm];
																\node at ($(LRC)+(\padding+1.1*\seg,-\over)$) {$e_{m}:=\ell_{m-n+1}$};
																\filldraw (fMID) circle [radius=0.04cm];
																\filldraw (sMID) circle [radius=0.04cm];
																\filldraw (MID) circle [radius=0.04cm];
																\coordinate (lastMID) at ($(MID)+(1.5*\seg,0)$);
																\filldraw (lastMID) circle [radius=0.04cm];
																\draw[-] (lastMID)--(MID);
																 \draw[decorate,decoration={markings,mark=at position 0.5 with
																				 {\arrow[color=black]{>}}}] (lastMID)--(MID);
																 \draw (MID) .. controls +(down:0.8cm) and +(left:0cm) ..  (sMID);
																 \draw [decorate,decoration={markings,mark=at position 0.3 with
																 {\arrow[color=black]{>}}}](MID) .. controls +(down:0.8cm) and +(left:0cm) ..  (sMID);
																 \draw [decorate,decoration={markings,mark=at position 0.3 with
																 {\arrow[color=black]{>}}}](MID) .. controls +(up:0.8cm) and +(left:0cm) ..  (fMID);
																 \draw (MID) .. controls +(up:0.8cm) and +(left:0cm) ..  (fMID);
																 \node at ($(fMID)+(0.7*\seg,\over)$) {$e_{j}$};
																 \node at ($(sMID)+(0.5*\seg,0.6*\over)$) {$e_{j+n-1}$};
																 \node at ($(lastMID)+(-0.45\seg,\over)$) {$\ell_{j}$};

							\end{tikzpicture}
							\caption{Inductive action step}
							\label{fig:5}
			\end{figure}
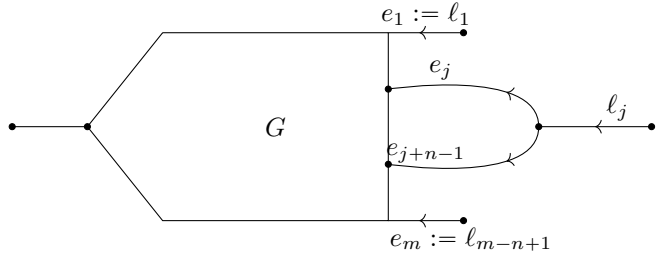 \ \\
			Elements of $\text{fBit}_{n,1}^{k}$, and $\text{fBit}^{+}$, have a similar inductive
			definition. \\ 
			For $G\in\bigbit$, a subgraph
			$E=\big(E^{\unit}\rdarrow{}{d_{\pm}}
			E^{\circ}\big)$, $E^{i}\subseteq G^{i}$, $d_{\pm}^{E}=d_{\pm}^{G}|_{E^{\unit}}$, and
			such that $E$ by itself is a bit will be called a \myemph{sub-bit} of $G$, and we let
			$\bigbit_{n,m}^{k}(G)$  denote the sub-bits of $G$ that are in
			$\bigbit_{n,m}^{k}$. In general, the collection
			$\bigbit_{n,m}(G)=\bigcup_{k}\bigbit_{n,m}^{k}(G)$ are hard to describe, but the
			sets $\bigbit_{1,m}^{k}(G)$ and $\bigbit_{n,1}^{k}(G)$ have a similar
			\myemph{inductive} definition. Beginning with 
			\begin{equation}
							\begin{array}[H]{ll}
							\bigbit_{1,1}^{\circ}(G) &\equiv G^{\unit} \\\\
							\bigbit_{1,m}^{\unit}(G)&\equiv \left\{ C_{i_{1}\ldots i_{m}}^{-}(v); v\in
							C_{1,m^{\prime}}(G), \left\{ i_{1},\cdots , i_{m} \right\}\subseteq \left\{1\cdots m^{\prime}\right\} \right\} \\\\
							\bigbit_{n,1}^{\unit} (G) &\equiv \left\{ C_{j_{1}\ldots j_{n}}^{+}(v); v\in
							C_{n^{\prime},1}(G), \left\{ j_{1}\cdots j_{n} \right\}\subseteq \left\{1\cdots n^{\prime}\right\} \right\}
							\end{array}
							\label{eq:4.4}
			\end{equation}
			and using the compositions and actions graftings, for $E\in\bigbit_{1,m}^{k}(G)$,
			(and similarly for $E\in\bigbit_{n,1}^{k}(G)$): \\ 
			For $e_{i}\in d_{+}E$ with $v=d_{+}e_{i}\in C_{1,n^{\prime}}(G)$, we have for a choice of $n\le n^{\prime}$ edges
			in $d_{+}v$, 
			\begin{equation}
							E\circ_{i} C_{n}^{-}(v)\in \bigbit_{1,m+n-1}^{k+1}(G).
							\label{eq:4.5}
			\end{equation}
			For $D\subseteq d_{+}E\cap d_{-}v$, $v\in C_{n^{\prime},1}(G)$, $n^{\prime}\ge n
			=\# D$, we have 
			\begin{equation}
							E 
											\begin{array}[h]{l}
															\leftarrow \vspace{-.27cm}\\
											\,\circ  \vspace{-.23cm}\\
											\scriptstyle D
											\end{array}
							C_{n}^{+}(v)\in
							\bigbit_{1,m-n+1}^{k+1}(G).
							\label{eq:4.6}
			\end{equation}
			\begin{remark}
							This is the very reason we prefer to work with ``bi-operads'', and not with
							the more general ``properads'' having operations $\biop
							(b_1\cdots b_{n};c_1\cdots c_m)$ with any output $n\ge 0$ and any input
							$m\ge 0$. \vspace{.1cm}\\
							For  $G\in\bigbit$, we denote the sub-$\unit$-bits $E$ of $G$ together with
							linear order on $d_{\pm}E$, by
							\begin{equation*}
											\mathscr{S}_{G}^{-} = \bigcup_{k,m} \bigbit_{1,m}^{k}(G) \quad ,
											\quad  \mathscr{S}_{G}^{+} =
											\bigcup_{k,n}\bigbit_{n,1}^{k}(G). 
							\end{equation*}
							\label{remark:4.7}
			\end{remark}
			\section{The free Bio}
				We can now define $p\mathcal{F}_{\mathcal{C}}=\left(
				p\mathcal{F}_{\mathcal{C}}^{-},\mathcal{C}^{\circ},p\mathcal{F}_{\mathcal{C}}^{+} \right)$, 
				\begin{equation}
								p\mathcal{F}_{\mathcal{C}}^{\pm}:= \left\{ f:G\to \mathcal{C} \;\text{labeling},\; G\in
								\text{fBit}^{\pm} \right\}/\text{isom}
								\label{eq:5.1}
				\end{equation}
			  where ``isomorphisms'' should preserve the planar structure and $\mathcal{C}$
				labeling, and also allow ``degeneracies'': the replacement of an edge
				$e \in G^{1}$. by a path $\Big( 
				\begin{tikzpicture}[baseline]
				\tikzmath{
								let \x =1.2;
								let \y =0;
								let \w = 1;
								let \s = 0.3;
				};
				\def\mlarrow#1#2#3#4{
							\draw[-] #1--#2;
							\draw[decorate,decoration={markings,mark=at position 0.5 with {\arrow[color=black]{<}}}] #1--#2;
							\node at ($0.5*#1+0.5*#2+(0,0.3)$) {$#3$};
							\filldraw ($#1-(\s,0cm)$) circle [radius=.04cm];
							\node at ($#1-(\s,-0.3cm)$) {$#4$};
			};
			\mlarrow{(\x,\y)}{(\x+\w,\y)}{e_1}{v_0};
			\mlarrow{(\x+\w+2*\s,\y)}{(\x+2*\s+2*\w,\y)}{e_2}{v_1};
			\filldraw  (\x+3*\s+2*\w,\y) circle [radius=.04cm];
			\filldraw  (\x+4*\s+2*\w,\y) circle [radius=.01cm];
			\filldraw  (\x+5*\s+2*\w,\y) circle [radius=.01cm];
			\filldraw  (\x+6*\s+2*\w,\y) circle [radius=.01cm];
			\mlarrow{(\x+4*\w+\s,\y)}{(\x+\s+5*\w,\y)}{e_{m}}{};
			\filldraw  (\x+2*\s+5*\w,\y) circle [radius=.04cm];
			\node at (\x+2*\s+5*\w,\y+0.3cm) {$v_{m}$};
\end{tikzpicture} 
				\Big)$ with labeling $f^{1}(e_{i})\equiv f^{1}(e)$, and
				$f^{\pm}(v_{i})=\unit_{f^{1}(e)}$, $i=1,\cdots , m-1$, $v_{0}\equiv d_{-}e,
				v_{m}\equiv d_{+}e$. \\
				Thus for $c_0,c_1,\cdots, c_n\in \mathcal{C}^{0}$, 
				\begin{equation}
								\begin{array}[H]{l}
								p \mathcal{F}_{\mathcal{C}}^{-}(c_{0};c_{1}\cdots c_{n})= 
								\left\{ (f:G\to \mathcal{C})\in p {\mathcal{F}^{-}_{\mathcal{C}}}, G\in
								\text{fBit}^{-}, f^{1}(0_G)=c_0,
				f^1\big(\unit_{i}(G)\big)=c_{i}\right\} \\\\
								p\mathcal{F}_{\mathcal{C}}^{+}(c_{1}\cdots c_{n};c_{0})=\left\{ (f:G\to
												\mathcal{C})\in
												p\mathcal{F}_{\mathcal{C}}^{+}, G\in \text{fBit}^{+},f^{1}(\unit_{G})=c_0,
								f^{1}(0_{i}(G))=c_{i} \right\}.
								\end{array}
								\label{eq:5.2}
				\end{equation}
				The composition in $p\mathcal{F}_{\mathcal{C}}^{\pm}$ are given by ``grafting'' of bits. For
				example, for 
				\begin{equation*}
				P=\left\{ f:G\to \mathcal{C} \right\}\in p\mathcal{F}^{-}_{\mathcal{C}}(b_0;b_1\cdots
				b_n)\quad , \quad P^{\prime}=\left\{ f^{\prime}:G^{\prime}\to C \right\}\in p
				\mathcal{F}^{-}_{\mathcal{C}}(b_{i};c_{1}\cdots c_{m}),
				\end{equation*}
				we have
				\begin{equation}
				P\circ_{i}P^{\prime}:=\left\{ f\cup f^{\prime}: G\coprod G^{\prime}/\approx \; \to
								\mathcal{C}
				\right\}\in p\mathcal{F}_{\mathcal{C}}^{-}(b_0; b_1\cdots b_{i-1}c_{1}\cdots
				c_{m}b_{i+1}\cdots b_{n})
				\label{eq:5.3}
				\end{equation}
				where $\approx$ means that we identify
				$\unit_{i}(G)$ with $0(G^{\prime})$; \vspace{.1cm}\\ 
				\big(and identify $d_{\pm}\unit_{i}(G)\approx d_{\pm} 0(G^{\prime})$\big); pictorially:  
				\begin{figure}[H]
								\centering
								\begin{tikzpicture}[baseline]
												\tikzmath{
																let \lBpoly=3;
																let \hBpoly=2.5;
																let \bAng = 1;
																let \relHeight = 0;
																let \simConst= 0.5 ;
																let \padding = 0;
																let \step=0.15;
																let \seg=1;
																let \spc =5.5*\seg ;
																let \over=0.25;
												};
												\def\Poly#1#2#3#4#5{ 
																				\draw[-]
																								#1--($#1+(#2,0)$)--($#1+(#2,-#3)$)--($#1+(0,-#3)$)--($#1-(#4,0.5*#3)$)--#1;
																				\node at ($#1+(0.5*#2,-0.5*#3)$) {#5};
															};
												\coordinate (A) at (0,0); 
												\Poly
																{(A)}
																{\lBpoly}
																{\hBpoly}
																{\bAng}{$P$};
												\Poly
																{($(A)+(\spc,\relHeight)$)}
																{\simConst*\lBpoly}
																{\simConst*\hBpoly}
																{\simConst*\bAng}
																{$P^{\prime}$};
																\coordinate (URC) at ($(A)+(\lBpoly,0)$);	
																\coordinate (ULC) at ($(A)$);	
																\coordinate (LRC) at ($(A)+(\lBpoly,-\hBpoly)$);	
																\coordinate (LLC) at ($(A)+(0,-\hBpoly)$);	
																\coordinate (LTP) at ($(A)-(\bAng,0.5*\hBpoly)$);
																\coordinate (MID) at ($(A)+(\lBpoly,-0.5*\hBpoly)$);

																\coordinate (urc) at ($(A)+(\spc,0)+\simConst*(\lBpoly,0)$);	
																\coordinate (ulc) at ($(A)+(\spc,0)$);	
																\coordinate (lrc) at ($(A)+(\spc+\simConst*\lBpoly,-\simConst*\hBpoly)$);	
																\coordinate (llc) at ($(A)+(\spc,-\simConst*\hBpoly)$);	
																\coordinate (ltp) at ($(A)+(\spc,0)-\simConst*(\bAng,0.5*\hBpoly)$);

																\filldraw ($(LTP)-(\padding,0)$) circle [radius=0.04cm];
																\draw[-] ($(LTP)-(2*\padding,0)$)--($(LTP)-(2*\padding+\seg,0)$);
																\node at ($(LTP)-(2*\padding+0.5*\seg,-\over)$)  {$0(G)$} ;
																\filldraw ($(LTP)-(2*\padding+\seg,0)-(\padding,0)$) circle [radius=0.04cm];

																\draw[-] ($(URC)+(\padding,0)$)--($(URC)+(\padding+\seg,0)$);
																\node at ($(URC)+(\padding+0.5*\seg,\over)$) {$\unit_1(G)$};
																\filldraw ($(URC)+(2*\padding+\seg,0)$) circle [radius=0.04cm];

																\draw[-] ($(LRC)+(\padding,0)$)--($(LRC)+(\padding+\seg,0)$);
																\filldraw ($(LRC)+(2*\padding+\seg,0)$) circle [radius=0.04cm];
																\node at ($(LRC)+(\padding+0.5*\seg,-\over)$) {$\unit_n(G)$};

																\filldraw (MID) circle [radius=0.04cm];
																\draw[-] ($(MID)+(\padding,0)$)--($(MID)+(\seg+\padding,0)$);
																\coordinate (A3) at ($(MID)+(\seg+2*\padding,0)$);
																\filldraw  (A3) circle  [radius=0.04cm];
																\coordinate (A1) at ($(MID)+(\seg+2*\padding,0.25*\hBpoly)$);
																\filldraw (A1) circle  [radius=0.04cm];
																\coordinate (A2) at ($(ltp)+(-\padding,0)$);
																\filldraw (A2) circle [radius=0.04cm];
																\fill[color=red, opacity=.4] (MID)--(A1)--(A2)--(A3)--(MID);
																\foreach \i in {0,1,2,3,4,5,6} {
																				\draw[dashed] ($(A1)+(\i*\step,0)$)--($(MID)+(\i*\step,0)$);
														};
														\node at ($0.5*(A2)+0.5*(A1)+(0,\over)$) {$0(G^{\prime})$};
														\node at ($0.5*(MID)+0.5*(A3)+(0,-\over)$) {$\unit_i(G)$};
														\draw[dashed] (A2)--(A3);
																\filldraw ($(urc)+(\padding,0)$) circle [radius=0.04cm];
																\draw[-] ($(urc)+(\padding,0)$)--($(urc)+(\padding+\seg,0)$);
																\filldraw ($(lrc)+(\padding,0)$) circle [radius=0.04cm];
																\draw[-] ($(lrc)+(\padding,0)$)--($(lrc)+(\padding+\seg,0)$);
																\filldraw ($(lrc)+(2*\padding+\seg,0)$) circle [radius=0.04cm];
																\filldraw ($(urc)+(2*\padding+\seg,0)$) circle [radius=0.04cm];
																\node at ($(lrc)+(\padding+0.7*\seg,-\over)$) {$\unit_m(G^{\prime})$};
																\node at ($(urc)+(\padding+0.6*\seg,\over)$) {$\unit_1(G^{\prime})$};

																\draw[-] ($(ltp)+(-2*\padding,0)$)-- ($(MID)+(\seg+3*\padding,0.25*\hBpoly)$) ;

																\filldraw ($(URC)+(0.5*\seg,-0.1*0.5*\hBpoly)$) circle [radius=0.01cm];
																\filldraw ($(URC)+(0.5*\seg,-0.2*0.5*\hBpoly)$) circle [radius=0.01cm];
																\filldraw ($(URC)+(0.5*\seg,-0.3*0.5*\hBpoly)$) circle [radius=0.01cm];

																\filldraw ($(LRC)+(0.5*\seg,0.1*0.5*\hBpoly)$) circle [radius=0.01cm];
																\filldraw ($(LRC)+(0.5*\seg,0.2*0.5*\hBpoly)$) circle [radius=0.01cm];
																\filldraw ($(LRC)+(0.5*\seg,0.3*0.5*\hBpoly)$) circle [radius=0.01cm];


								\end{tikzpicture}
								\caption{Grafting composition}
								\label{fig:1}
				\end{figure}
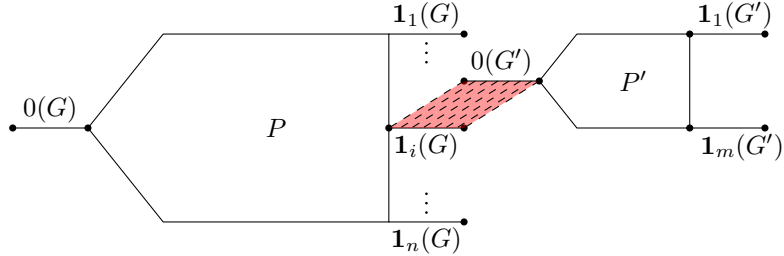 
				\noindent Note the new linear order on $d_{+}(P\circ_{i} P^{\prime})$. This composition on
				$p\mathcal{F}_{\mathcal{C}}^{-}$, and the similar one on $p\mathcal{F_{\mathcal{C}}^{+}}$, are
				associative;
				$p\mathcal{F}_{\mathcal{C}}^{-}(c;c^{\prime})\equiv
				p\mathcal{F}_{\mathcal{C}}^{+}(c;c^{\prime})$
				$n=1$; and we have the units
				$\unit_{c}:=\left\{ f\equiv f^{1}\equiv c:I\to \mathcal{C} \right\}$,
				with the \myemph{unit bit} $I\in \text{Bit}^{\circ}_{1,1}$, 
				\begin{equation}
								I:= \left( I^{1}=\left\{ e \right\}\rdarrow{}{d_{\pm}}\;I^{\circ}=\left\{ 0,1
								\right\}\right)\equiv \left( \overset{0}{\bullet} \; \midLarr{e}\; \overset{1}{\bullet} \right)
								\label{eq:5.4}
				\end{equation}
				Note that $C^{\pm}(I)=\phi$, so $f$ is just the value $f^{1}(e)=c$.
				\vspace{.1cm}\\ 
				Thus $p\mathcal{F}_{\mathcal{C}}^{\pm}$ are operads, with the same objects $\mathcal{C}^{\circ}$, and the
				same underlying category. \vspace{.1cm}\\ 
				Similarly, the natural actions of
				$p\mathcal{F}_{\mathcal{C}}^{-}$ and $p\mathcal{F}_{\mathcal{C}}^{+}$ on each other are given by
				grafting $\mathcal{C}$-labeled bits. For example, for 
				\begin{equation*}
								\begin{array}[H]{ll}
												P=\left\{ f:G\to \mathcal{C} \right\}\in p\mathcal{F}_{\mathcal{C}}^{-}(b_0;b_1\cdots
												b_{i-1}c_{1}\cdots c_{m}b_{i+1}\cdots b_{n}) \\\\
												P^{\prime}=\left\{ f^{\prime}:G^{\prime}\to \mathcal{C} \right\} \in
								p\mathcal{F}_{\mathcal{C}}^{+}(c_{1}\cdots c_{m}; b_{i})
								\end{array}
				\end{equation*}
				we get 
				\begin{equation}
								P \circleftarrow P^{\prime}:= \left\{ f\cup f^{\prime}\; :\;
								G \coprod G^{\prime}/\approx\; \to \mathcal{C}\right\}\in
								p\mathcal{F}_{\mathcal{C}}^{-}(b_{0};b_{1}\cdots b_{i}\cdots b_{n})
								\label{eq:5.5}
				\end{equation}
				where $\approx$ means that we identify $\unit_{i-1+j}(G)\approx 0_{j}(G^{\prime})$;
				pictorially, \\
				\begin{figure}[H]
								\centering
							\begin{tikzpicture}
												\tikzmath{
																let \lBpoly=3.5;
																let \hBpoly=3;
																let \bAng = 1;
																let \relHeight = -0.25*\hBpoly;
																let \simConst= 0.45 ;
																let \padding = 0;
																let \seg=1;
																let \spc = \lBpoly+2*\seg;
																let \over=0.25;
												};
												\def\Poly#1#2#3#4#5{ 
																				\draw[-]
																								#1--($#1+(#2,0)$)--($#1+(#2,-#3)$)--($#1+(0,-#3)$)--($#1-(#4,0.5*#3)$)--#1;
																				\node at ($#1+(0.5*#2,-0.5*#3)$) {#5};
															};

												\def\smallPoly#1#2#3#4#5{ 
																				\draw[-]
																				($#1+(#2,-#3)$)--($#1+(0,-#3)$)--#1--($#1+(#2,0)$)--($#1+(#2+#4,0)-(0,0.5*#3)$)--($#1+(#2,-#3)$);

																				\node at ($#1+(0.5*#2,-0.5*#3)$) {#5};
											};
										  \coordinate (A) at (0,0); 
											\Poly
																{(A)}
																{\lBpoly}
																{\hBpoly}
																{\bAng}{$P$};
											\smallPoly 
																{($(A)+(\spc,\relHeight)$)}
																{\simConst*\lBpoly}
																{\simConst*\hBpoly}
																{\simConst*\bAng}
																{$P^{\prime}$};
												\coordinate (ulc) at ($(A)+(\spc,\relHeight)$);
												\coordinate (urc) at ($(A)+(\spc,\relHeight)+(\simConst*\lBpoly,0)$);
												\coordinate (lrc) at ($(A)+(\spc,\relHeight)+(\simConst*\lBpoly,\simConst*\hBpoly)$);
												\coordinate (llc) at ($(A)+(\spc,\relHeight)+(0,-\simConst*\hBpoly)$);

												\coordinate (mid) at
												($(A)+(\spc,\relHeight)+(\simConst*\lBpoly+\simConst*\bAng,0)-(0,0.5*\simConst*\hBpoly)$);

												\coordinate (URC) at ($(A)+(\lBpoly,0)$);	
												\coordinate (ULC) at ($(A)$);	
												\coordinate (LRC) at ($(A)+(\lBpoly,-\hBpoly)$);	
												\coordinate (LLC) at ($(A)+(0,-\hBpoly)$);	
												\coordinate (LTP) at ($(A)-(\bAng,0.5*\hBpoly)$);
												\coordinate (MID) at ($(A)+(\lBpoly,-0.5*\hBpoly)$);

												\draw[-] (mid)--($(mid)+(\seg,0)$);
												\filldraw (mid) circle [radius=0.04cm];
												\filldraw ($(mid)+(\seg+\padding,0)$) circle [radius=0.04cm];
												\node at ($(mid)+(0.5*\seg,0)+(0,\over)$) {$\scriptstyle \unit(G^{\prime})$};
												\tikzmath{ let \segprop=0.9; };
												\filldraw (llc) circle [radius=0.04cm];
												\filldraw (ulc) circle [radius=0.04cm];

												\coordinate (A2) at ($(ulc)+(-\padding-\segprop*\seg,-0.128*\hBpoly)$);
												\coordinate (A1) at ($(URC)+(0,-0.375*\hBpoly)$);
												\coordinate (A3) at (ulc);
												\coordinate (A4) at ($(ulc)+(-\padding-\segprop*\seg,0)$);
												\fill[color=red,opacity=0.3] (A1)--(A2)--(A3)--(A4);
												\draw[-] (A1)--(A2)--(A3)--(A4)--(A1);
												\node at ($(A4)+(0.5*\seg,+\over)$) {$\scriptstyle 0_1(G^{\prime})$};
												\node at ($(llc)+(-0.5*\seg,+\over)$) {$\scriptstyle 0_m(G^{\prime})$};
												\filldraw ($(llc)+(-0.5*\seg,0.5cm)$) circle [radius=0.01cm];
												\filldraw ($(llc)+(-0.5*\seg,0.6cm)$) circle [radius=0.01cm];
												\filldraw ($(llc)+(-0.5*\seg,0.7cm)$) circle [radius=0.01cm];

												\filldraw (A1) circle [radius=0.04cm];
												\filldraw  (A2) circle [radius=0.04cm];
												\filldraw (A4) circle [radius=0.04cm];
												\coordinate (B2) at ($(ulc)+(-\padding-\segprop*\seg,-4.5*0.125*\hBpoly)$);
												\coordinate (B1) at ($(URC)+(0,-6.5*0.125*\hBpoly)$);
												\coordinate (B3) at (llc);
												\coordinate (B4) at ($(llc)+(-\padding-\segprop*\seg,0)$);
												\foreach \i in {1,2,3} {
																\draw[dashed]
																($(A1)+(\i*0.25*\seg+\i*0.25*\padding,0)$)--($(A4)+(\i*0.25*\seg+\i*0.25*\padding,0)$);
																\draw[dashed]
																($(B1)+(\i*0.25*\seg+\i*0.25*\padding,0)$)--($(B4)+(\i*0.25*\seg+\i*0.25*\padding,0)$);
												};
												\filldraw (B2) circle [radius=0.04cm];
												\filldraw (B1) circle [radius=0.04cm];
												\filldraw (B4) circle [radius=0.04cm];
												\draw[-] (B1)--(B2)--(B3)--(B4)--(B1);
												\fill[color=red,opacity=0.3] (B1)--(B2)--(B3)--(B4)--(B1);

												\filldraw ($(LTP)-(\padding,0)$) circle [radius=0.04cm];
												\draw[-] ($(LTP)-(2*\padding,0)$)--($(LTP)-(2*\padding+\seg,0)$);
																\node at ($(LTP)-(2*\padding+0.5*\seg,-\over)$)
																{$\scriptstyle 0(G)$} ;
																\filldraw ($(LTP)-(2*\padding+\seg,0)-(\padding,0)$) circle [radius=0.04cm];
												\draw[-] ($(URC)+(\padding,0)$)--($(URC)+(\padding+0.5*\seg,0)$);
												\draw[-] ($(URC)+(\padding,-0.25*\hBpoly)$)--($(URC)+(\padding+0.5*\seg,-0.25*\hBpoly)$);
												\draw[-] ($(LRC)+(\padding,0)$)--($(LRC)+(\padding+0.5*\seg,0)$);
												\node at ($(URC)+(\padding+0.5*\seg,\over)$) {$\scriptstyle \unit_1(G)$};
												\node at ($(URC)+(\padding+0.5*\seg,-0.530*\hBpoly+\over)$)
												{$\scriptstyle \unit_i(G)$};
												\filldraw ($(URC)+(2*\padding+0.5*\seg,0)$) circle [radius=0.04cm];
												\filldraw ($(LRC)+(2*\padding+0.5*\seg,0)$) circle [radius=0.04cm];
												\filldraw ($(LRC)+(0,0.75*\hBpoly)+(2*\padding+0.5*\seg,0)$) circle [radius=0.04cm];

												\filldraw ($(URC)+(0.3*\seg,-0.1*0.5*\hBpoly)$) circle [radius=0.01cm];
												\filldraw ($(URC)+(0.3*\seg,-0.2*0.5*\hBpoly)$) circle [radius=0.01cm];
												\filldraw ($(URC)+(0.3*\seg,-0.3*0.5*\hBpoly)$) circle [radius=0.01cm];

												\filldraw ($(LRC)+(0.3*\seg,0.05*0.5*\hBpoly)$) circle [radius=0.01cm];
												\filldraw ($(LRC)+(0.3*\seg,0.1*0.5*\hBpoly)$) circle [radius=0.01cm];
												\filldraw ($(LRC)+(0.3*\seg,0.15*0.5*\hBpoly)$) circle [radius=0.01cm];

												\filldraw ($(LRC)+(0.3*\seg,0.33*\hBpoly)$) circle [radius=0.01cm];
												\filldraw ($(LRC)+(0.3*\seg,0.3*\hBpoly)$) circle [radius=0.01cm];
												\filldraw ($(LRC)+(0.3*\seg,0.27*\hBpoly)$) circle [radius=0.01cm];
							\end{tikzpicture}
								\caption{Grafting action}
								\label{fig:2}
				\end{figure}
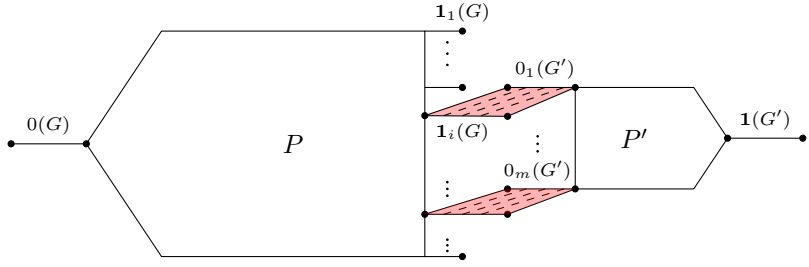 \ \\
			The axioms for a planar bio are easily checked. The unit of adjunction \break
			$\varepsilon : \mathcal{C}\to U\mathcal{F}_{\mathcal{C}}$ is easy to describe:
			it is the identity on objects $\mathcal{C}^{0}$, and it takes
			$P\in \mathcal{C}^{-}(c_0;c_1\cdots c_n)$ (resp. $\mathcal{C}^{+}(c_1\cdots
			c_n;c_0)$), to the negative (resp. positive) $n$-corolla with  $\mathcal{C}$ labeling given
			by $f^\pm(v)=P$, $f^{1}(e_{i})=c_i$, pictorially
			
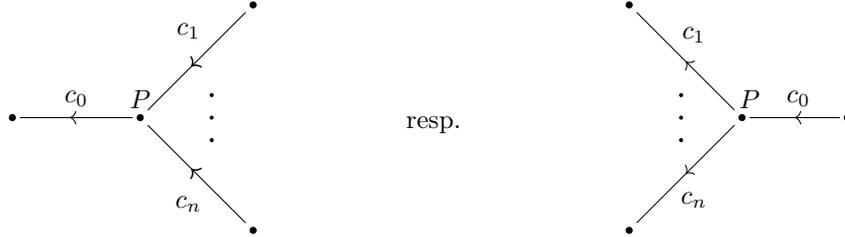
\begin{equation}
							\begin{tikzpicture}[baseline=0mm]
		 							\tikzmath{
		 												let \xstart=0;
		 												let \ystart=0;
		 												let \dist=1.5cm;
		 												let \width=1.5cm;
		 												let \padding=0.1cm;
		 												let \over=0.25cm;
		 								};
		 								\coordinate (start) at (\xstart,\ystart);
		 								\coordinate (P) at ($(start)+(\dist,0)$) ;
		 								\coordinate (P1) at ($(start)+(2*\dist,\width)$) ;
		 								\coordinate (P2) at ($(start)+(2*\dist,-\width)$) ;
		 								\draw[-] (start)--(P);
		 								\draw[decorate,decoration={markings,mark=at position 0.5 with {\arrow[color=black]{<}}}] (start)--(P);

		 								\draw[-] ($(P)+(2*\padding,\padding)$)--($(P1)+(0,-\padding)$);
		 								\draw[decorate,decoration={markings,mark=at position 0.5 with
		 								{\arrow[color=black, line width=0.2mm]{<}}}]
										($(P)+(2*\padding,\padding)$)--($(P1)+(0,-\padding)$);

										\node at ($0.5*(P)+0.5*(P1)+(-0.1*\padding,4*\padding)$) {$c_1$};

										\node at ($0.5*(P)+0.5*(P2)+(-0.1*\padding,-4*\padding)$) {$c_n$};

		 								\draw[-]($(P)+(2*\padding,-\padding)$)--($(P2)+(0,+\padding)$); 
		 								\draw[decorate,decoration={markings,mark=at position 0.5 with
		 								{\arrow[color=black, line width=0.2mm]{<}}}]
										($(P)+(2*\padding,-\padding)$)--($(P2)+(0,+\padding)$);
										\filldraw ($(P)+(0.7*\dist,0)$) circle [radius=0.02cm];
										\filldraw ($(P)+(0.7*\dist,-0.3cm)$) circle [radius=0.02cm];
										\filldraw ($(P)+(0.7*\dist,+0.3cm)$) circle [radius=0.02cm];

										\node at ($0.5*(start)+0.5*(P)+(0,\over)$) {$c_0$};
										\node at ($(P)+(\padding,\over)$) {$P$};
			 							\filldraw ($(start)+(-\padding,0)$) circle [radius=0.04cm];
			 							\filldraw ($(P)+(\padding,0)$) circle [radius=0.04cm];
			 							\filldraw ($(P1)+(\padding,0)$) circle [radius=0.04cm];
			 							\filldraw ($(P2)+(\padding,0)$) circle [radius=0.04cm];

		 							\tikzmath{
		 												let \xstart2=11;
		 												let \ystart2=0;
		 												let \dist2=1.3cm;
		 												let \width2=1.3cm;
		 												let \padding2=-0.1cm;
		 												let \over2=0.25cm;
		 								};
										\def\decl#1#2{
														\draw[decorate,decoration={markings,mark=at position 0.5 with {\arrow[color=black]{<}}}] #1--#2;
										}
										\def\decr#1#2{
														\draw[decorate,decoration={markings,mark=at position 0.5 with {\arrow[color=black]{>}}}] #1--#2;
										}
		 								\coordinate (start2) at (\xstart2,\ystart2);
										\coordinate (PP) at ($(start2)-(\dist2,0)$);
										\filldraw (start2) circle [radius=0.04cm];
										\draw[-] ($(start2)+(\padding2,0)$)--(PP);
										\filldraw ($(PP)+(\padding2,0)$) circle [radius=0.04];
										\coordinate (AA2) at ($(PP)+(2*\padding2,\padding2)$);
										\coordinate (AA1) at ($(PP)+(2*\padding2,-\padding2)$);
										\draw[decorate,decoration={markings,mark=at position 0.5 with
										{\arrow[color=black]{>}}}] ($(start2)+(\padding2,0)$)--(PP);
										\coordinate (pp1) at ($(AA1)+(-\dist2,+\dist2)$);
										\coordinate (pp2) at ($(AA2)-(\dist2,+\dist2)$);
										\draw[-] (AA1)--($(pp1)$);	
										\decr{(AA1)}{($(pp1)$)};	
										\filldraw ($(pp1)+(\padding2,-\padding2)$) circle [radius=0.04];
										\filldraw ($(pp2)+(\padding2,\padding2)$) circle [radius=0.04];
										\draw[-] (AA2)--(pp2);	
										\decr{(AA2)}{($(pp2)$)};	

										\node at ($0.5*(PP)+0.5*(pp1)+(-0.1*\padding2,-4*\padding2)$) {$c_1$};
										\node at ($0.5*(PP)+0.5*(pp2)+(-0.1*\padding2,4*\padding2)$) {$c_n$};
										\node at ($0.5*(start2)+0.5*(PP)+(0,\over2)$) {$c_0$};
										\node at ($(PP)+(0,\over2)$) {$P$};

										\filldraw ($(PP)+(-0.7*\dist2,0)$) circle [radius=0.02cm];
										\filldraw ($(PP)+(-0.7*\dist2,-0.3cm)$) circle [radius=0.02cm];
										\filldraw ($(PP)+(-0.7*\dist2,+0.3cm)$) circle [radius=0.02cm];
										\node at ($(AA2)+(-4cm,0)$) {$\text{resp.}$};
			 				\end{tikzpicture}
							\label{eq:5.6}
			\end{equation}
		Given a planar bio $\biop=\left( \biop^{-},\biop^{\circ},\biop^{+} \right)$,
			and given an element 
			\begin{equation*}
							P=(f:G\rightarrow U\biop)\in p\mathcal{F}_{U\biop}^{-}(c_0;c_1\cdots
			c_n),
			\end{equation*}
				 we can define now $\mu^{-}(P)\in \biop^{-}(c_0;c_1\cdots c_n)$
			by induction:
			\begin{equation}
							\begin{array}[H]{l}
											\text{If } G\in\text{fBit}_{1,n}^{1} , \quad G=C^{-}(v), \quad
											\mu^{-}(P)=f^{-}(v)\in\biop^{-}(c_0;c_1\cdots c_n). \\\\
											\text{If}\; G\in\text{fBit}_{1,n}^{k+1}, \; \text{then either}\;
											G=G^{\prime} \circ_{i} 
												 C^{-}_{\ell}(v) \; \text{or}\; 
											G=G^{\prime}\hspace{-.2cm}
											\begin{array}[h]{c}
															\;\leftarrow\vspace{-.25cm}\\ \circ \vspace{-.2cm}\\
															\scriptstyle{j}
											\end{array} \textstyle C_{\ell}^{+}(v)
							\end{array}
							\label{eq:5.7}
			\end{equation}
			with $G^{\prime}\in\bigbit_{1,n^{\prime}}^{k}$, and we can define
			\begin{equation}
							\mu^{-}(P)=\mu^{-}(P^{\prime}) \circ_{i} f^{-}(v)\quad \text{or
							respectively}\quad
							\mu^{-}(P)=\mu^{-}(P^{\prime})\circleftarrow_{\hspace{-1.6mm}j} f^{+}(v).
							\label{eq:5.8}
			\end{equation}
			It is easy to see that $\mu^{-}$ is well define (independent of the order of the
			operations we used in the induction steps), and similarly defining inductively
			$\mu^{+}:p\mathcal{F}^{+}_{U\biop}(c_1\cdots c_n; c_0)\rightarrow\biop^{+}(c_1\cdots
			c_n; c_0)$ we get  the co-unit map
			\begin{equation*}
							\mu: p\mathcal{F}_{U\biop}\rightarrow \biop
			\end{equation*}
			\section{Bios via generators and relations}
			\begin{definition}
							For $\biop=(\biop^{+},\biop^{\circ},\biop^{-})\in\text{Bio}$, an
							\myemph{ideal} $\ideal\lideal\biop$ is a collection of equivalence
							relations $\approx$ on the sets $\biop^{-}(c_0;c_1\cdots c_n)$ and
							$\biop^{+}(c_1\cdots c_n;c_0)$ for all $c_0,c_1,\cdots , c_n\in
							\biop^{\circ}$, respecting the compositions, mutual actions, and
							$S_n$-actions, so that we get a quotient bio $\biop/\ideal$, with the same
							objects $\left( \biop/\ideal \right)^{\circ}=\biop^{\circ}$, 
							\begin{equation}
											\begin{array}[H]{l}
											\left( \biop/\ideal \right)^{-}(c_0;c_1\cdots
											c_n)=\biop^{-}(c_0;c_1\cdots c_n)/\approx \quad , \\\\
											\left(
															\biop/\ideal
											\right)^{+}(c_1\cdots c_n;c_0)=\biop^{+}(c_1\cdots c_n; c_0)/\approx ,
											\end{array}
											\label{eq:6.2}
							\end{equation}
							and a canonical surjective map of bios
							$\pi_{\ideal}:\biop\xtworightarrow{\quad}
							\biop/\ideal$ identity on objects. \\
							Given any map $\varphi\in\text{Bio}_{C}(\biop,\bioq)$ with
			$\varphi^{\circ}=\text{id}_{C}$ we get an ideal
			$\ker(\varphi)\lideal\biop$ by $P\approx P^{\prime}$ iff
			$\varphi(P)=\varphi(P^{\prime})$, and canonical factorization of $\varphi$
			\begin{equation}
							\begin{tikzpicture}[baseline=-10mm]
		 							\tikzmath{
		 												let \xstart=0;
		 												let \ystart=0;
		 												let \dist=3.5cm;
		 												let \d=1.5cm;
		 												let \over=0.30cm;
		 												let \buff=0.25cm;
		 								};
		 								\coordinate (start) at (\xstart,\ystart);
		 								\coordinate (Q) at (,\ystart);
										\node at (start) {$\biop$};
										\draw[->] ($(start)+(\buff,0)$)--(\dist,0);
										\node at ($0.5*(\dist,\over)$) {$\varphi$};
										\node at ($(\dist+\buff,0)$) {$\bioq$};
										\draw[<-{Hooks[right,length=5,width=6]}] (\dist+\buff,-\buff) -- (\dist+\buff,-\d);	
									 \node at ($(\dist+\buff,-\d-\buff)$) {$\varphi(\biop)$};
								   \draw[->] (3*\buff,-\d-\buff) -- (\dist-\buff,-\d-\buff);	
									 \node at  (0,-\d-\buff) {$\biop/_{\ker(\varphi)}$};
									 \draw[->>] ($(start)-(0,\buff)$) -- (0,-\d);	
									 \node  at ($0.5*(3*\buff,-\d-\buff) +
									 0.5*(\dist-\buff,-\d-\buff)+(0,0.5*\over)$) {$\sim$} ;	
							\end{tikzpicture}
							\label{eq:6.3}
			\end{equation}
			\label{def:6.1}
			\end{definition}
			\begin{remark}
							Given ideals $\ideal_{\alpha}\lideal \biop$, their intersection is again
							an ideal \break  $\bigcap_{\alpha}\ideal_{\alpha}\lideal \biop$, with
							$P\approx P^{\prime} \mod
							\bigcap_{\alpha}\ideal_{\alpha}\Longleftrightarrow P\approx P^{\prime}
							\mod \ideal_{\alpha}$  for all $\alpha$. \\ 
							Thus given any collection of
							pairs $\left\{ (P_{j},P_{j}^{\prime}) \right\}$ with
							$P_{j},P_{j}^{\prime}\in \biop^{\pm}(c_0;c_i)$ we can speak of the ideal they
							generate $\ideal = \left\{ (P_{j},P_{j}^{\prime}) \right\}_{j\in J}=\bigcap
							\ideal_{\alpha}$ the intersection of all ideals containing all the
							$(P_{j},P^{\prime}_{j})$'s, and we have the canonical quotient map 
							\begin{equation*}
											\pi : \biop \xtworightarrow{\qquad}\biop/\ideal := \biop/\left\{
															P_j\sim P_{j^{\prime}}
											\right\}_{j\in J}
							\end{equation*}
							Thus we can construct bios by ``generators and relations'': the generating
							objects $C^{\circ}$, and the generating (co)-operations $\left\{ Q_i
							\right\}$ between them gives a collection $C=(C^{-},C^{\circ},C^{+})\in
							\text{Coll}$, the relations $\left\{ P_{j}\sim P_{j}^{\prime} \right\}$
							generate an ideal $\ideal\lideal \mathcal{F}_{C}$ as above, and we have the
							bio 
							\begin{equation}
											\mathbb{F}_{C^{\circ}}\left[ Q_{i} \right]/\left\{ P_{j}\sim
															P_{j}^{\prime}
											\right\} := \mathcal{F}_{C}/\ideal .
											\label{eq:6.5}
							\end{equation}
							\label{remark:6.4}
			\end{remark}
			\begin{corollary}
							The category $\text{Bio}$ is complete and co-complete
							\label{cor:1}
			\end{corollary}
			\begin{proof}
							All limits are created in Set - for $(\biop_{j})\in\text{Bio}^{J}$,
							\begin{equation}
											( \lim_{\overset{\leftarrow}{J}}\biop_{j} )^{\circ} :=
											\lim_{\overset{\leftarrow}{J}}\biop^{\circ}_{j}
											\label{eq:6.6}
							\end{equation} 
							and for $c_{0}=c_{0}^{(j)},\; c_{1}=c_{1}^{(j)}\cdots c_n=c_{n}^{(j)}\in
							\lim\limits_{\overset{\leftarrow}{J}}\biop_{j}^{\circ}$ we have 
							\begin{equation}
											( \lim_{\overset{\leftarrow}{J}}\biop_{j} )^{-} (c_0;c_1\cdots
											c_n) :=
											\lim_{\overset{\leftarrow}{J}}\left(\biop_{j}^{-}(c_0^{(j)};c_1^{(j)}\cdots
											c_n^{(j)})\right)
											\label{eq:6.7}
							\end{equation}
							and similarly for $( \lim_{\overset{\leftarrow}{J}}\biop_{j} )^{+} (c_1\cdots
							c_n; c_0)$. The functor $\biop\mapsto \biop^{\circ}:\text{Bio}\to\text{Set}$
							preserves also co-limits, so let
							$C^{\circ}=(\colim\limits_{\overset{\rightarrow}{J}}\biop_{j})^{\circ}=\colim\limits_{\overset{\rightarrow}{J}}\biop_j^{\circ}$
							be the colimit in Set. The sets $\left\{ \biop_{j}^{\pm}(c_0;c_1\cdots c_n),
							j\in J, c_0,c_1\cdots c_n\in \biop_{j}^{\circ}\right\}$ induce a collection
							$C=(C^{-},C^{\circ},C^{+})$. The relations in each $\biop_{j},j\in J$, give
							relations on these generators. Also for arrows $\varphi^{\prime}:j\to
							j^{\prime}$, $\varphi^{''}:j\to j^{''}$ in $J$ with the same domain, and
							for any (co).operation $P\in\biop_{j}^{\pm}(c_0;c_1\cdots c_n)$ we have the
							relation $\varphi_{\ast}^{'}(P)\sim\varphi^{''}_{\ast}{(P)}$; these give all the
							relations and 
							\begin{equation}
											\colim\limits_{\overset{\rightarrow}{J}}
											\biop_{j}=\mathcal{F}_{C}/_{
											\left\{ \varphi^{'}_{\ast}(P)\sim \varphi^{''}_{\ast}(P) \right\} }
											\label{eq:6.8}
							\end{equation}
			\end{proof}
\section{Bipo $\equiv$ Bi-Pre-Orders}
			\begin{definition}
							For a set $C$ let $M(C)=\coprod_{n\ge 0}C^{n}$ denote the free (associative,
							unital) monoid on $C$; its elements are ``words''
							$\overline{c}=c_1\cdots c_n\in C^{n}$							  of ``length'' $n\ge 0$,
							including the unit: the empty word $\phi$ ($n=0$). \vspace{.1cm}\\
							A \myemph{planar Bipo}
							(= Bi-pre-order) is a set $C$ together with two relations \break $R^{-}\subseteq
							C\times M(C)$ and $R^{+}\subseteq M(C)\times C$; we write: 
							\begin{equation*}
							c_0\le c_1\cdots
							c_n\Leftrightarrow (c_0;c_1\cdots c_n)\in R^{-}\quad, \quad c_1\cdots c_n\le'
							c_0\Leftrightarrow (c_1\cdots c_n; c_0)\in R^{+},
							\end{equation*}
							such that 
							\begin{equation}
											c\le c \quad \text{and}\quad c\le' c \quad \text{for all}\quad c\in C.
											\label{eq:7.2}
							\end{equation}
						For $n=1$ we have
							\begin{equation}
											c_0\le c_1 \Longleftrightarrow c_0\le' c_1, 
															\label{eq:7.3}
							\end{equation}
							and we write $\le$ for $\le'$ with no danger of confusion.
							\begin{equation}
											\begin{array}[t]{l}
											b_0\le b_1\cdots b_n, \quad b_i\le \overline{c}_i = c_{i1}\cdots
											c_{i m_{i}}\Longrightarrow b_0\le \overline{c}_1\cdots
											\overline{c}_{n} \\\\
											 b_1\cdots b_n\le b_0, \quad  \overline{c}_i= c_{i1}\cdots
											c_{i m_{i}}\le b_i \Longrightarrow \overline{c}_1\cdots
											\overline{c}_{n} \le b_0
											\end{array}
											\label{eq:7.4}
							\end{equation}
							\begin{equation}
											\begin{array}[t]{l}
															b_0\le \overline{c}_1\cdots \overline{c}_{n},\quad
															\overline{c}_{i}=c_{i1}\cdots c_{im_{i}}\le
															b_{i}\Longrightarrow b_0\le b_1\cdots b_n \\\\
															\overline{c}_1\cdots \overline{c}_{n}\le b_0, \quad b_i\le
															\overline{c}_{i} = c_{i1}\cdots c_{im_{i}}
															\Longrightarrow b_1\cdots b_n\le b_0
											\end{array}
											\label{eq:7.5}
							\end{equation}
							(Notice that (\ref{eq:7.3}) follows from (\ref{eq:7.2}), (\ref{eq:7.4}) and
											(\ref{eq:7.5})).
							\label{def:7.1}
			\end{definition}
			A map of planar bipo $\varphi:C\to C'$ is a map of sets such that 
			\begin{equation}
							\begin{array}[H]{ll}
											c_0\le c_1\cdots c_n \Longrightarrow \varphi(c_0)\le
											\varphi(c_1)\cdots \varphi(c_n) \\\\
											c_1\cdots c_n\le c_0 \Longrightarrow \varphi(c_1)\cdots
											\varphi(c_n)\le \varphi(c_0)
							\end{array}
							\label{eq:7.6}
			\end{equation}
			Thus we have a category \myemph{pBipo}. \\ 
			A planar bipo $C\in p\text{Bipo}$ is
			called \myemph{symmetric} if for $c_0,c_1\cdots c_n\in C$, 
			\begin{equation}
							\begin{array}[H]{l}
							c_0\le c_1 \cdots c_n \Longleftrightarrow c_0\le c_{\sigma(1)}\cdots
							c_{\sigma(n)} \quad \text{any $\sigma\in S_n$} \\\\
						 c_1 \cdots c_n \le 	c_0 \Longleftrightarrow c_{\sigma(1)}\cdots
							c_{\sigma(n)} \le  c_0 \quad \text{any $\sigma\in S_n$} 
							\end{array}
							\label{eq:7.7}
			\end{equation}
			i.e. the relations $\le$, $\le'$ are induced from similar relations on the free
			\break
			\myemph{commutative} monoid on $C$, 
			\begin{equation}
			\text{CM}(C)=\coprod\limits_{n\ge 0} C^{n}/_{S_n}
			\label{eq:3.12}
			\end{equation}
			via the cannonical map $\pi: M(C)\xtworightarrow{\quad}\text{CM}(C)$. \\
			We let $\text{Bipo}\subseteq
			p\text{Bipo}$ denote the full subcategory of symmetric - bipos, and we call its
			objects simply ``bipos''. \vspace{.1cm}\\ 
			We say a bipo $C\in \text{Bipo}$ is \myemph{fermionic}
			if for $c_0,c_1\cdots c_n\in C$, 
			\begin{equation}
							\begin{array}[H]{l}
											c_0\le c_1\cdots c_n \Longrightarrow c_i\not= c_j \; \text{for
											$i\not= j$,  $i,j>0$.}\\\\
											c_1\cdots c_n\le c_0 \Longrightarrow c_i\not= c_j \; \text{for
											$i\not= j$, $i,j>0$.} 
							\end{array}
							\label{eq:7.9}
			\end{equation}
			We say $C$ is \myemph{simple} if,
			\begin{equation}
							c_0\le c_1\cdots c_n \;\; \text{and}\;\; c_1\cdots c_n\le
			c_0\Rightarrow c_0=c_1=\cdots = c_n.			 
			\label{eq:7.10}
			\end{equation}
			 We say $C$ is \myemph{closed} if,
			 \begin{equation}
							c_0\le c_1\cdots c_n, \;\;\; n>1\Rightarrow c_0\le c_1\cdots c_{n-1}
							\label{eq:7.11}
			 \end{equation}
			and similarly if
			\begin{equation*}
			c_1\cdots c_n\le c_0 \quad,\quad n>1\Rightarrow c_1\cdots c_{n-1}\le c_0.
			\end{equation*}
			Given $\varphi_{0},\varphi_{1},\cdots, \varphi_{n}\in\text{Bipo}(C,D)\equiv D^{C}$ define 
			\begin{equation}
							\begin{array}[H]{l}
											\varphi_0\le \varphi_1 \cdots \varphi_n \Longleftrightarrow
											\varphi_{0}(c)\le_{D}\varphi_{1}(c)\cdots \varphi_n(c) \quad
											\text{for all $c\in C$} \\\\
											 \varphi_1 \cdots \varphi_n \le \varphi_0 \Longleftrightarrow
											\varphi_{1}(c)\cdots \varphi_n(c)  \le_{D} \varphi_{0}(c)\quad
											\text{for all $c\in C$} 
							\end{array}
							\label{eq:7.12}
			\end{equation}
			This makes $D^{C}$ into a bipo, the ``internal Hom''. \vspace{.1cm}\\
			Given bipos $\le_{i}$ on the same underlying set $C$, the intersection
			$\bigcap\limits_{i}\le_i$ is again a bipo, $c_0\le c_1\cdots c_n\Leftrightarrow
			c_0\le_{i} c_1\cdots c_n$ for all $i$. Thus we can talk about the bipo generated by
			certain relations 
			\begin{equation*}
							\left\{ c_0^{(i)}\le c_{1}^{(i)}\cdots c_{n_i}^{(i)} \right\}_{i\in
							J^{-}}\cup
							\left\{ c_1^{\prime (i)} \cdots c_{n_i}^{\prime (i)}\le c_0^{\prime (i)} \right\}_{i\in
							J^{+}}
			\end{equation*}
			We let $B\otimes C$ denote the bipo on $B\times C$, for $B,C\in \text{Bipo}$,
			generated by the following relations
			\begin{equation}
							\begin{array}[H]{l}
											(b,c_0)\le (b,c_1)\cdots (b,c_n)\quad \text{all $b\in B$,
											$c_0\le c_1\cdots c_n$ in C} \\\\
											(b_0,c)\le (b_1,c)\cdots (b_n,c) \quad \text{all $c\in C$,
											$b_0\le b_1\cdots b_n$ in $B$} \\\\
											(b,c_1)\cdots (b,c_n)\le (b,c_0) \quad \text{all $b\in B$,
											$c_1\cdots c_n\le c_0$ in $C$} \\\\
											(b_1,c)\cdots (b_n,c)\le (b_0,c) \quad \text{all $c\in C$,
											$b_1\cdots b_n\le b_0$ in $B$}. 
							\end{array}
							\label{eq:7.13}
			\end{equation}
			\begin{proposition}
							We have the adjunction
							\begin{equation*}
											\text{Bipo}(B,D^{C}) \equiv \text{Bipo}(B\otimes C,D)
							\end{equation*}
							via the usual Set adjunction $\varphi_{b}(c)\leftrightarrow
							\varphi(b,c)$, and thus \text{Bipo} is complete, co-complete, closed
							symmetric monoidal category. 
							\label{prop:1}
			\end{proposition}
			\begin{proof}
							Limits in Bipo are created in Set. For co-limits
							$\colim\limits_{\xrightarrow{\quad}} C_j$, we take the co-limit in Set
							of the underlying sets $C_j$, and than the bipo generated by all the
							relations in the $C_j$'s. The functor $B,C\mapsto B\otimes C$ is clearly
							symmetric monoidal, and it is closed because of the adjunction. 
			\end{proof}
							\noindent For a bipo $B$, we have the bio $\ell_{!}B$, with
							$\ell_{!}B^{\circ}:= B$, and 
							\begin{equation}
											\begin{array}[H]{l}
															\ell_{!}B^{-}(b_0;b_1\cdots b_n)=
															\left\{\begin{array}[H]{cl}
																							\left\{ \ast \right\} & b_0\le b_1\cdots b_n
																							\\\\
																							\phi & \text{otherwise}
															\end{array}\right. \\\\
															\ell_{!}B^{+}(b_1\cdots b_n; b_0)=
															\left\{\begin{array}[H]{cl}
																							\left\{ \ast \right\} &  b_1\cdots b_n \le b_0
																							\\\\
																							\phi & \text{otherwise}
															\end{array}\right.
											\end{array}
											\label{eq:7.14}
							\end{equation}
							This gives a full and faithful embedding $\ell_{!}:\text{Bipo}\subseteq
							\text{Bio}$, where the image are the bios having at most one point in every
							hom - set (just like the pre-orders are the categories having at most one
							point in every hom-set). The functor $\ell_{!}$ has a left adjoint
							$\ell^{\ast}:\text{Bio}\to \text{Bipo}$, for $\biop\in\text{Bio}$,
							$\ell^{\ast}\biop$ is the bipo with the same objects as $\biop$ and for
							$b_0,\cdots, b_n\in\biop^{\circ}$, 
							\begin{equation}
											\begin{array}[H]{l}
															b_0\le_{\ell^{\ast}\biop} b_{1}\cdots b_n \Longleftrightarrow
															\biop^{-}(b_0;b_1\cdots b_n)\not = \phi \\\\
															b_1\cdots b_n\le_{\ell^{\ast}\biop} b_0 \Longleftrightarrow
															\biop^{+}(b_1\cdots b_n;b_0)\not = \phi.
											\end{array}
											\label{eq:7.15}
							\end{equation}
							\
\section{Natural transformations}
\begin{definition}
				Given $\varphi_0,\varphi_1\cdots\varphi_n\in\text{Bio}(\bioq,\bior)$ a
				\myemph{natural transformation} $\alpha^{-}\in \bior^{\bioq}(\varphi_0;\varphi_1\cdots
				\varphi_{n})$ is a map $\bioq^{\circ}\ni
				q\mapsto\alpha^{-}_{q}\in \bior^{-}(\varphi_0(q);\varphi_1(q)\cdots \varphi_n(q))$
				such that for $Q\in \bioq^{-}(q_0;q_1\cdots q_m)$ we  have \myemph{interchange}:
				\begin{equation}
								\varphi_{0}^{-}(Q)\circ \left( \alpha^{-}_{q_{j}} \right)\circ\sigma_{n,m}
								= \alpha_{q_{0}}^{-}\circ \left( \varphi_i^{-}(Q) \right)
								\label{eq:8.2}
				\end{equation}
				with $\sigma_{n,m}\in S_{n\cdot m}$ the ``rows-columns-interchange''
				\begin{equation}
								\sigma_{n,m}(a\cdot n+b+1) = b\cdot m+a+1, \quad 0\le a< m, \quad 0\le b < n
								\label{eq:8.3}
				\end{equation}
				Similarly $\alpha^{+}\in \bior^{\bioq}(\varphi_1\cdots \varphi_n;\varphi_0)$ is a map
				$\bioq^{\circ}\ni q\mapsto \alpha_{q}^{+}\in \bior^{+}(\varphi_1(q)\cdots
				\varphi_n(q);\varphi_0(q))$ such that for $Q\in \bioq^{+}(q_1\cdots q_m;q_0)$ we have
				interchange:
				\begin{equation}
								\sigma_{n,m}\circ (\alpha_{q_{j}}^{+})\circ\varphi^{+}_{0}(Q)=\left(
												\varphi_{i}^{+}(Q)
								\right)\circ\alpha_{q_0}^{+}.
								\label{eq:8.4}
				\end{equation}
				\label{def:8.1}
\end{definition}
We have composition and mutual actions of natural transformations given pointwise 
\begin{equation}
				\begin{array}[H]{l}
				\left( \alpha^{-}\circ (\alpha_{i}^{-}) \right)_{q} = \alpha^{-}_{q}\circ \left(
				(\alpha_{i}^{-})_q \right)\quad , \quad  
\left(  (\alpha_{i}^{+})  \circ \alpha^{+} \right)_{q} =  \left(
				(\alpha_{i}^{+})_q \right) \circ \alpha^{+}_{q}\\\\
				\left( \alpha^{-}\circleftarrow (\alpha_{i}^{+}) \right)_{q} =
				\alpha^{-}_{q}\circleftarrow \left(
				(\alpha_{i}^{+})_q \right)\quad , \quad  
\left(   (\alpha_{i}^{-}) \circrightarrow \alpha^{+} \right)_{q} =
\left( (\alpha_{i}^{-})_{q}\right)  \circrightarrow  \alpha^{+}_{q}.
				\end{array}
				\label{eq:8.5}
\end{equation}
These make $\bior^{\bioq}$ into a bio with objects $\left( \bior^\bioq
\right)^{\circ}=\text{Bio}(\bioq,\bior)$, and gives the internal hom functor
\begin{equation}
				\left( \text{Bio} \right)^{\op}\times \text{Bio}\to \text{Bio}\quad , \quad
				(\bioq,\bior)\mapsto \bior^{\bioq}.
				\label{eq:8.6}
\end{equation}
\section{Billinear maps and tensor products}
Fix $\varphi\in\text{Bio}(\biop,\bior^{\bioq})$. On objects $\varphi^{\circ}$ gives for each
$p\in\biop^{\circ}$ an object $\varphi^{\circ}(p)$ of $\bior^{\bioq}$, i.e.
$\varphi^{\circ}(p)\in \text{Bio}(\bioq,\bior)$, which on objects gives for $q\in
\mathscr{Q}^{\circ}$ the
object $\varphi^{\circ}(p,q)\in \bior^{\circ}$, thus we have a function 
\begin{equation}
				\varphi^{\circ}:\biop^{\circ}\times \bioq^{\circ}\to \bior^{\circ}.
				\label{eq:9.1}
\end{equation}
For $Q\in \bioq^{\pm}$, $\varphi^{\circ}(p)$ associate $\varphi^{\pm}(p,Q)\in \bior^{\pm}$,
preserving composition and actions. For $P\in \biop^{-}(p_0;p_1\cdots p_n)$, (resp.
$P\in\biop^{+}(p_1\cdots p_n;p_0)$), $\varphi$ associates the natural transformation
\begin{equation}
				\begin{array}[H]{c}
\varphi^{-}(P)\in \bior^{\bioq}\big(\varphi^{\circ}(p_0);\varphi^{\circ}(p_1)\cdots
\varphi^{\circ}(p_n)\big), \\\\
\Big(\text{resp. } \varphi^{+}(P)\in
\bior^{\bioq}\big(\varphi^{\circ}(p_1)\cdots\varphi^{\circ}(p_n);\varphi^{\circ}(p_0)\big)\,\Big),
				\end{array}
				\label{eq:9.2}
\end{equation}
thus for $q\in Q^{\circ}$ we have $\varphi^{-}(P,q)\in \bior^{-}\big(\varphi^{\circ}(p_0,q);\varphi^{\circ}(p_1,q),\cdots,
\varphi^{\circ}(p_n,q)\big)$, $\Big($resp. $\varphi^{+}(P,q)\in \bior^{+}\big(
\varphi^{\circ}(p_1,q)\cdots  \varphi^{\circ}(p_n,q),\varphi^{\circ}(p_0,q)\big)\Big)$ and
since the composition and actions are defined poinwise, for fixed $q\in \bioq^{\circ}$,
$P\mapsto\varphi(P,q)$ \break preserves compositions and actions. Moreover, we have interchange: 
for \break
$ Q\in \bioq^{-}(q_0;q_1\cdots q_m)$, $\Big($resp. $Q\in \bioq^{+}(q_1\cdots
q_m;q_0)\Big)$,
\begin{equation}
				\begin{array}[H]{l}
\varphi^{-}(p_0,Q)\circ\big(\varphi^{-}(P,q_{j})\big) = \varphi^{-}(P,q_0)\circ
\big(\varphi^{-}(p_i,Q)\big)\circ\sigma_{n,m}, \\
				\end{array}
				\label{eq:9.3}
\end{equation}
$\Big($resp. 
$\sigma_{n,m}\circ\big(\varphi^{+}(P,q_{j})\big)\circ\varphi^{+}(p_0,Q)=\big(\varphi^{+}(p_i,Q)\big)\circ\varphi^{+}(P,q_0)\Big)$.
\vspace{.1cm}\\
Writing $\varphi_{p}(-):=\varphi(p,-)$, $\psi_{q}(-):=\varphi(-,q)$, we have 
\begin{equation}
				\begin{array}[H]{l}
								\text{Bio}(\biop,\bior^{\bioq})\equiv \text{Bill}(\biop,\bioq;\bior) :\overset{\text{def}}{=}
								\\\\
								\left\{  \begin{array}[H]{l}
												(\varphi,\psi)\in\text{Bio}(\bioq,\bior)^{\biop^{\circ}}\times\text{Bio}(\biop,\bior)^{\bioq^{\circ}}\quad
												, \quad \varphi_{p}(q)=\psi_{q}(p)\;\text{for}\; p\in \biop^{\circ},\; q\in
												\bioq^{\circ} \\\\
												\varphi_{p_{0}}^{-}(Q)\circ \big(\psi_{q_j}^{-}(P)\big) =
												\psi_{q_0}^{-}(P)\circ\big(\varphi_{p_{i}}^{-}(Q)\big)\circ
												\sigma_{m,n}\quad , \quad  
												\begin{array}[t]{l}
												P\in\biop^{-}(p_0;p_1\cdots p_n)\; , \\
												Q\in \bioq^{-}(q_0;q_1\cdots q_m)
												\end{array}
												\\\\
												\sigma_{n,m}\circ\big(\psi_{q_j}^{+}(P)\big)\circ \varphi_{p_0}^{+}(Q)=
												\big(\varphi_{p_i}^{+}(Q)\big)\circ\psi_{q_{0}}^{+}(P) \quad,\quad 
												\begin{array}[t]{l}
												P\in\biop^{+}(p_1\cdots p_{n};p_0), \\
												Q\in \bioq^{+}(q_1\cdots q_m;q_0)
												\end{array}
								\end{array}\right\}
				\end{array}
				\label{eq:9.4}
\end{equation}
For fixed $\biop, \bioq$ the functor $\text{Bio}\to\text{Set}$,
$\bior\mapsto\text{Bill}(\biop,\bioq;\bior)$ is representable by a bio $\biop\otimes \bioq$, 
and we get the adjunction 
				\begin{equation}
								\text{Bio}(\biop,\bior^{\bioq})\equiv
								\text{Bill}(\biop,\bioq;\bior)\equiv\text{Bio}(\biop\otimes \bioq, \bior)
								\label{eq:9.5}
				\end{equation}
				Here $\biop\otimes \bioq$ is defined by generators and relations: 
				\begin{equation}
								\text{on objects} \quad \left( \biop\otimes \bioq \right)^{\circ}:=
								\biop^{\circ}\times \bioq^{\circ};
								\label{eq:9.6}
				\end{equation}
				For $q\in \bioq^{\circ}$ we have generators $\left\{ P\otimes q, P\in\biop^{\pm}
				\right\}$, and for $p\in \biop^{\circ}$ we have generators $\left\{ p\otimes Q,
								Q\in \bioq^{\pm}
				\right\}$; The generators $\left\{P\otimes q  \right\}_{P\in\biop^{\pm}}$,
				(resp $\left\{ p\otimes Q \right\}_{Q\in \bioq^{\pm}}$) satisfying the relations of
				$\biop$ (resp. of $\bioq$), so that for each $q\in \bioq^{\circ}$ (resp.
				$p\in\biop^{\circ}$) $j_{q}:\biop\rightarrow\biop\otimes\bioq$, $j_{q}(P)=P\otimes q$
				(resp. $j_{p}:\bioq\rightarrow\biop\otimes\bioq$, $j_{p}(Q)=p\otimes Q$)
				is a map of bios, and we have the interchange relations: 
				\begin{equation}
								\begin{array}[H]{l}
								(p_0\otimes Q)\circ (P\otimes q_{j}) = (P\otimes q_{0})\circ (p_{i}\otimes
								Q)\circ \sigma_{n,m}\quad , \quad 
								\begin{array}[t]{l}
								P\in \biop^{-}(p_0;p_1\cdots p_n)\;, \\
								Q\in \bioq^{-}(q_0,q_1\cdots q_m)
								\end{array}
								\\\\
								\sigma_{n,m}\circ (P\otimes q_{j})\circ (p_0\otimes Q) = (p_{i}\otimes
								Q)\circ (P\otimes q_0)\quad , \quad
								\begin{array}[t]{l}
								P\in \biop^{+}(p_1\cdots p_n; p_0)\; , \\
								Q\in \bioq^{+}(q_1\cdots q_m; q_0)
								\end{array}
								\end{array}
								\label{eq:9.7}
				\end{equation}
				The functor $\text{Bio}\times \text{Bio}\to \text{Bio}$, $(\biop,\bioq)\mapsto
				\biop\otimes \bioq$, makes Bio into a symmetric monoidal category with unit 
				\begin{equation}
				I: I^{\circ}=\left\{ e \right\}, \;\; I(e,e)=\left\{ \text{id}_{e} \right\}, \;\;
				I^{-}(e;\underbrace{e\cdots e}_{n})=\phi=I^{+}(\underbrace{e\cdots e}_{n};e)\;\;
				\text{for $n>1$; }
				\label{eq:9.8}
				\end{equation}
				so $\text{Bio}(I,\biop)\equiv \biop^{\circ}$, and $\text{Bio}/I\equiv
				\text{Cat}$. Thus we have, 
				\begin{theorem}
								Bio is a closed symmetric monoidal category. 
								\label{thm:1}
				\end{theorem}
				\begin{remark}
								We have the operads $(\biop\otimes \bioq)^{\pm}$, and there are map of operads 
								\begin{equation*}
												\biop^{-}\otimes_{BW} \bioq^{-}\longrightarrow (\biop\otimes
												\bioq)^{-} \quad \text{and}\quad
												\biop^{+}\otimes_{BW}\bioq^{+}\longrightarrow (\biop\otimes
												\bioq)^{+}
								\end{equation*}
								with $\otimes_{BV}$ the Boardman-Vogt tensor product; but these maps are
								not in general isomorphism!
								\label{remark:9.9}
				\end{remark}
\section{From Bits to Bios}
				Fix a bit $G=\big((G^{\unit})\rdarrow{}{d_{\pm}} G^{\circ}\big)$. Recall the
				$\unit$-sub bits $B$ of $G$, with linear order on $d_{\pm}B$,
				$\mathscr{S}_{G}^{\pm}$, and associate with $G$ the bio $\mathscr{S}_{G}$:
				\begin{equation}
								\begin{array}[H]{l}
												\mathscr{S}_{G}^{\circ}:=G^{\unit}\quad , \quad \text{for $c_0,c_1\cdots c_n\in
								G^{\unit}$}  \\\\
								\mathscr{S}_{G}^{-}(c_0;c_1\cdots c_n):= \left\{ B\in\mathscr{S}_{G}^{-},
								d_{-}B=\{c_0\}, d_{+}B=\left\{ c_1\cdots c_n \right\} \right\} \\\\
								\mathscr{S}_{G}^{+}(c_1\cdots c_n; c_0):= \left\{
												B \in\mathscr{S}_{G}^{+},d_{-}B=\left\{
								c_1\cdots c_n \right\}, d_{+}B=\left\{ c_0 \right\}\right\}
								\end{array}
								\label{eq:10.1}
				\end{equation}
				The compositions and actions are given by grafting sub-bits. The $S_n$-actions are
				given by reordering $d_{\pm}B$, and is free. Note that the associated bipo
				$\ell^{*}\mathscr{S}_{G}$, given by the set $G^{\unit}$ and 
				\begin{equation}
								\begin{array}[H]{l}
												c_0\le c_1\cdots c_n\Longleftrightarrow \exists
												B\in\mathscr{S}_{G}^{-} \quad , \quad d_{-}B=\left\{ c_0
												\right\} \quad , \quad d_{+}B=\left\{ c_1\cdots c_n \right\} \\\\
								c_{1}\cdots c_{n}\le c_{0} \Longleftrightarrow \exists B\in
								\mathscr{S}_{G}^{+}\quad , \quad d_{-}B=\left\{ c_1\cdots c_n \right\}\quad ,
								\quad d_{+}B=\left\{ c_0 \right\}
								\end{array}
								\label{eq:10.2}
				\end{equation}
				is simple, closed, and fermionic. \vspace{.1cm}\\
				A map
				of bios $f:\mathscr{S}_G\to \mathscr{S}_H$, $G,H\in\text{Bit}$, is given by 
				\begin{equation}
								\begin{array}[H]{l}
												f^{\unit}:G^{\unit}\to H^{\unit} \; \text{and} \;
												f^{\circ}:G^{\circ}\to \mathscr{S}^{\pm}_{G}\; 
												\text{with} \; d_{\pm}f^{\circ}(v)=f^{\unit}(d_{\pm}v)
												\; \text{for}\; v\in G^{\circ}.
								\end{array}
								\label{eq:10.3}
				\end{equation}
				We have the  following easy, 
				\begin{lemma}
								\hspace{.3cm} If  $\displaystyle H=[n]=\left\{
								\underset{\cdot}{v_0}\midarrow{1cm}{0.05cm}{e_1}{-0.2cm}\underset{\cdot}{v_1}
								\midarrow{1cm}{0.05cm}{e_2}{-0.2cm}\underset{\cdot}{v_2}
								\midarrow{1cm}{0.1cm}{}{-0.2cm}
								\begin{array}[b]{l}
												\vspace{-0.35cm} \ldots
								\end{array}
								\midarrow{1cm}{0.05cm}{e_n}{-0.2cm}\underset{\cdot}{v_n}
				\right\}$ \break 
				then \break
				$G=\coprod\limits_{\alpha\in Q}[m_{\alpha}]\; , \;
				f=\coprod\limits_{\alpha\in Q}f_{\alpha}\; , \;
				f_{\alpha}\in\bbDelta([m_{\alpha}],[n]). $
								\label{lema:1}
				\end{lemma}
				Indeed there can be no $m$-array, $m>1$, that map to unary, and the sub-bits of
				$H$ are just intervals. \vspace{.1cm}\\
				Fix the map $f:\mathscr{S}_{G}\to \mathscr{S}_{H}$, and fix $e\in H^{\unit}$. By the lemma we
				can write 
				\begin{equation}
								f^{-1}(e)=\left\{
												\underset{\cdot}{v_{0}^{\alpha}}\midarrow{1cm}{0.1cm}{e_1^{\alpha}}{-0.2cm}\underset{\cdot}{v_1^{\alpha}}
												\midarrow{1cm}{0.1cm}{e_2^{\alpha}}{-0.2cm}
								\begin{array}[b]{l}
												\vspace{-0.35cm} \ldots
								\end{array}
								\midarrow{1cm}{0.1cm}{e_{m_\alpha}^{\alpha}}{-0.2cm}\underset{\cdot}{v_{m_\alpha}^{\alpha}}
				\right\}_{\alpha\in V(e)} 
				\label{eq:10.4}
				\end{equation}
				a collection of paths in $G$ with
				\begin{equation*}
								f(e_{1}^{\alpha})=\cdots =
								f(e_{m_{\alpha}}^{\alpha})=e \quad , \quad v_{1}^{\alpha}\cdots
								v_{m_{\alpha}-1}^{\alpha}\in C_{11}(G). 
				\end{equation*}
				Note that we can split paths to shorter paths, but if we take the paths to be
				\myemph{maximal} this representation of $f^{-1}(e)$ is unique.
				\begin{definition}
								The map $f:\mathscr{S}_{G}\to \mathscr{S}_{H}$ is \myemph{flat} if for all $e\in
								H^{\unit}$, 
								\begin{equation}
												f^{-1}(e)= \left\{
												\underset{\cdot}{v_0^{\alpha}}\midarrow{1cm}{0.1cm}{e_1^{\alpha}}{-0.2cm}
								\begin{array}[b]{l}
												\vspace{-0.35cm} \ldots
								\end{array}
								\midarrow{1cm}{0.1cm}{e_{m_\alpha}^{\alpha}}{-0.2cm}\underset{\cdot}{v_{m_\alpha}^{\alpha}}
				\right\}_{\alpha\in V(e)}
				\label{eq:10.6}
								\end{equation}
								and we have  $e_{1}^{\alpha}\in d_{-}G$ (resp. $e_{m_{\alpha}}^{\alpha}\in
								d_{+}G$) for all $\alpha$ in $V(e)$, except perhaps for one index
								$\alpha=\alpha_{-}$  (resp. $\alpha=\alpha_{+}$);
								\label{def:10.5}
				\end{definition}
\noindent Note that if $e_{1}^{\alpha_{-}}\not\in d_{-}G$  (resp.
$e_{m_{\alpha_{+}}}^{\alpha_{+}}\not\in d_{+}G$)  then $e\not\in
d_{-}H$ (resp. $e\not\in d_{+}H$). It follows that for every path
$e=(e_1,\cdots, e_{n})$, $e_{i}\in H^{1}$, 
$v_i=d_{+}e_{i}=d_{-}e_{i+1}\in C_{11}$  for $i=1\cdots n-1$, the set
$f^{-1}(e)$ has a similar representation. Thus if $g:\mathscr{S}_{H}\to
\mathscr{S}_{K}$ is flat too, then $g\circ f:\mathscr{S}_{G}\to \mathscr{S}_{K}$ is flat, and so we have a
category \myemph{Bit} with maps $f:G\to H$ the \myemph{flat} maps
\begin{equation}
				f:\mathscr{S}_G\to \mathscr{S}_H,
				\label{eq:10.7}
\end{equation}
and a functor 
\begin{equation}
				\mathscr{S}:\text{Bit}\to \text{Bio} \quad , \quad G\mapsto \mathscr{S}_{G}.
				\label{eq:10.8}
\end{equation}
For $f\in \text{Bit}(G,H)$, $e\in H^{1}$, $f^{-1}(e)= 
\left\{  \underset{\cdot}{v_0^{\alpha}}\midarrow{1cm}{0.1cm}{e_1^{\alpha}}{-0.2cm}
								\begin{array}[b]{l}
												\vspace{-0.35cm} \ldots
								\end{array}
								\midarrow{1cm}{0.1cm}{e_{n_\alpha}^{\alpha}}{-0.2cm}\underset{\cdot}{v_{n_\alpha}^{\alpha}}\right\}_{\alpha\in
								V(e)} $,
we can identify all the $e_{i}^{\alpha}$'s to one edge $\tilde{e}$, with
$d_{-}\tilde{e}=v_{0}^{\alpha_{-}}$ (resp.
$d_{+}\tilde{e}=v_{n_{\alpha_{+}}}^{\alpha_{+}}$) if $\alpha_{-}$ (resp. $\alpha_{+}$)
exists, and otherwise $d_{\pm}\tilde{e}$ is a stump. Doing this for all $e\in H^{\unit}$ we
get a factorisation \\
\begin{equation}
				\begin{tikzpicture}[baseline]
								\tikzmath{
												let \spc=5cm;
												let \buff=0.4cm;
								};
								\node at (0.73*\spc,0.5mm) {$f=\tilde{f}\circ \pi$:};
								\node at ($(\spc+0,0)$) {$\mathscr{S}_{G}$};
								\coordinate (a) at ($(\spc+\buff,0)$);
								\coordinate (b) at ($(2*\spc-\buff,0)$);
								\node at ($(b)+(\buff,0)$) {$\mathscr{S}_{H}$}; 
								\node at ($0.5*(b)+0.5*(a)+(0,0.3cm)$) {$f$};
								\coordinate (c) at ($0.5*(b)+0.5*(a)+(0,-0.5*\spc)$);
								\draw[->] (a)--(b);
								\draw[->>] ($(a)+(-0.3*\buff,-0.7*\buff)$)--($(c)-(\buff,-\buff)$);
								\draw[arrows={<->[harpoon,swap,scale=1.5]}] ($(b)+(0.3*\buff,-0.7*\buff)$)--($(c)+(\buff,\buff)$);
								\node at (c) {$\mathscr{S}_{G}/\ker(f)$};
								\node at ($0.5*(a)+0.5*(c)+(-0.5cm,0)$) {$\pi$};
								\node at ($0.5*(b)+0.5*(c)+(0.5cm,0)$) {$\tilde{f}$};
				\end{tikzpicture} 
				\label{eq:10.9}
\end{equation}
where $\pi$  is flat by definition, and is surjective on edges, and $\tilde{f}$ is
injective on edges (hence also flat).
\begin{remark}
		The map $f:G\to H$ given pictorially 
			\begin{figure}[H]
							\centering
							\begin{tikzpicture}[baseline, scale=1]
							\tikzmath{
											let \ax=0;
											let \ay=0;
											let \buff=0.1;
											let \down=-1;
											let \mid=3cm;
											let \mlen=1cm;
							};
							\def\hrarrow#1#2#3{
							   \draw[decorate,decoration={markings,mark=at position #3 with
											{\arrow[color=black]{>}}}] #1--#2;
								 \draw[-] #1--#2;
				 };
										\coordinate (a) at (\ax,\ay);
										\coordinate (b) at ($(-\mlen,0)+(-\buff,0)$);
										\coordinate (c) at ($(-\mid,0)+(-\buff,0)$);
										\coordinate (d) at ($(c)+(-\mlen,0)$);
										\coordinate (ad) at (\ax,\down);
										\coordinate (bd) at ($(-\mlen,\down)+(-\buff,0)$);
										\coordinate (cd) at ($(-\mid,\down)+(-\buff,0)$);
										\coordinate (dd) at ($(c)+(-\mlen,\down)$);
										\filldraw (a) circle  [radius=0.04cm];
										\hrarrow{(-.2cm,0)}{(-\mlen,0)}{0.5};
										\filldraw (b) circle [radius=0.04cm];
										\filldraw (c) circle [radius=0.04cm];
										\hrarrow{($(b)+(-.2cm,0)$)}{($(\buff,0)+(c)$)}{0.5};
										\filldraw (d) circle [radius=0.04cm];
										\hrarrow{($(c)+(-.2cm,0)$)}{($(\buff,0)+(d)$)}{0.5};

										\filldraw (ad) circle  [radius=0.04cm];
										\hrarrow{($(ad)+(-\buff,0)$)}{($(ad)+(-\mlen,0)$)}{0.5};
										\filldraw (bd) circle [radius=0.04cm];
										\filldraw (cd) circle [radius=0.04cm];
										\hrarrow{($(bd)+(-.2cm,0)$)}{($(\buff,0)+(cd)$)}{0.5};
										\filldraw (dd) circle [radius=0.04cm];
										\hrarrow{($(cd)+(-.2cm,0)$)}{($(\buff,0)+(dd)$)}{0.5};
										\hrarrow{($(bd)+(-\buff,\buff)$)}{($(\buff,-\buff)+(c)$)}{0.7};
										\hrarrow{($(b)+(-\buff,-\buff)$)}{($(\buff,\buff)+(cd)$)}{0.7};
										
										\node at ($0.5*(a)+0.5*(b)+(0,3*\buff)$) {$d_1$};
										\node at ($0.5*(b)+0.5*(c)+(0,3*\buff)$) {$b_1$};
										\node at ($0.5*(c)+0.5*(d)+(0,3*\buff)$) {$a_1$};
										\node at ($0.5*(c)+0.5*(bd)+(-6*\buff,1.5*\buff)$) {$c_1$};
										\node at ($0.1*(b)+0.9*(cd)+(\buff,2.5*\buff)$) {$c_2$};

										\node at ($0.5*(ad)+0.5*(bd)+(0,-3*\buff)$) {$d_2$};
										\node at ($0.5*(bd)+0.5*(cd)+(0,-3*\buff)$) {$b_2$};
										\node at ($0.5*(cd)+0.5*(dd)+(0,-3*\buff)$) {$a_2$};
										\coordinate (G) at ($(dd)+(\down,-0.5*\down)$);
										\node at (G) {$G$};
										\coordinate (H) at ($(G)+(0,2*\down)$);
										\draw[->] ($(G)-(0,0.4cm)$)--(H);
										\node at ($(H)-(0,0.4cm)$) {$H$};
										\coordinate (I) at ($(dd)+(0,1.7*\down)$);
										\coordinate (J) at ($(cd)+(0,1.7*\down)$);
										\coordinate (K) at ($(bd)+(0,1.7*\down)$);
										\coordinate (L) at ($(ad)+(0,1.7*\down)$);
										\filldraw (I) circle [radius=0.04cm];
										\filldraw (J) circle [radius=0.04cm];
										\filldraw (K) circle [radius=0.04cm];
										\filldraw (L) circle [radius=0.04cm];
										\hrarrow{($(-\buff,0)+(J)$)}{($(\buff,0)+(I)$)}{0.5};
										\node at ($0.5*(-\buff,0)+0.5*(J)+0.5*(\buff,0)+0.5*(I)+(0,2*\buff)$) {$a$};
										\draw[-] ($(K)+(-\buff,\buff)$) to [bend right] ($(J)+(\buff,\buff)$);
										\node at ($0.5*(-\buff,0)+0.5*(K)+0.5*(\buff,0)+0.5*(L)+(0,2*\buff)$) {$d$};
										\draw[decorate,decoration={markings,mark=at position 0.5 with {\arrow[color=black]{>}}}] ($(K)+(-\buff,\buff)$) to [bend right] ($(J)+(\buff,\buff)$);
										\node at ($0.5*(K)+0.5*(-\buff,10*\buff)+0.5*(J)+0.5*(\buff,\buff)$) {$b$};
										\draw[-] ($(K)+(-\buff,-\buff)$) to [bend left] ($(J)+(\buff,-\buff)$);
										\node at ($0.5*(K)+0.5*(-\buff,-\buff) + 0.5*(J)+0.5*(\buff,-10*\buff)$)
										{$c$};
										\draw[decorate,decoration={markings,mark=at position 0.5 with
										{\arrow[color=black]{>}}}] ($(K)+(-\buff,-\buff)$) to [bend left]
										($(J)+(\buff,-\buff)$);
										\hrarrow{($(-\buff,0)+(L)$)}{($(\buff,0)+(K)$)}{0.5};
			\end{tikzpicture}
			\caption{The map $f:G\to H$}
							\label{fig:6}
			\end{figure} \ \\
			is both a map of graphs, $f\in \text{Graph}(G,H)$, and is a map of bios,
			$f\in\text{Bio}\left( \mathscr{S}_{G},\mathscr{S}_{H} \right)$, but is Not flat, hence is not a
			map of bits: $f\not\in \text{Bit}(G,H)$.
			\label{remark:10.10}
\end{remark}
\section{Degeneracies and face maps}
For $G\in \text{Bit}$, $e\in G^{1}$, we can form the bit map 
\begin{equation}
				s_{e}=s_{_e}^{G}: G_{e}\overset{\text{def}}{=} \left(
								(G^{1}\setminus \left\{ e \right\}\cup \left\{ e_{-},e_{+}
								\right\}\rdarrow{}{d_{\pm}} G^{\circ}\cup \left\{ v_{e}
\right\})\right) \xtworightarrow{\qquad }{\ } G 
\label{eq:11.1}
\end{equation}
with $d_{-}e_{-}=d_{-}e$, $d_{+}e_{-}=v_{e}=d_{-}e_{+}$, $d_{+}e_{+}=d_{+}e$
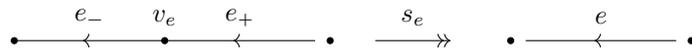
\begin{center}
\begin{tikzpicture}[baseline]
				\tikzmath{
								let \x =1.2;
								let \y =0;
								let \w = 2;
								let \s = 0.2;
				};
				\def\mlarrow#1#2#3#4#5{
							\draw[-] #1--#2;
							\draw[decorate,decoration={markings,mark=at position 0.5 with {\arrow[color=black]{<}}}] #1--#2;
							\node at ($0.5*#1+0.5*#2+(0,0.3)$) {$#3$};
							\filldraw ($#1-(#5,0cm)$) circle [radius=.04cm];
							\node at ($#1-(0,-0.3cm)$) {$#4$};
			};
			\node at  (\x+\w,\y+0.3cm) {$v_{e}$};
			\mlarrow{(\x,\y)}{($(\x+\w,\y)$)}{e_{-}}{}{0}; 
			\mlarrow{(\x+\w,\y)}{($(\x+2*\w,\y)$)}{e_{+}}{}{0};
			\filldraw  (\x+2*\w+\s,\y) circle [radius=.04cm];
			\draw[->>]  (\x+2*\w+4*\s,\y)--(\x+2.5*\w+4*\s,\y);
			\node at ($0.5*(\x+2*\w+4*\s,\y)+0.5*(\x+2.5*\w+4*\s,\y)+(0,0.3cm)$)
			{$s_{e}$};
			\mlarrow{(\x+3*\w+4*\s,\y)}{(\x+4*\w+4*\s,\y)}{e}{}{\s};
			\filldraw ($(\x+4*\w+4*\s,\y)+(\s,0)$) circle [radius=0.04cm];
\end{tikzpicture} 
\end{center}
i.e. we add
the vertex $v_{e}$ in the middle of the edge $e$, splitting $e$ to the two halves
$e_{-}$ and $e_{+}$, 
$
s(e^{\prime})= 
\left\{ \begin{array}[H]{lc}
								e & e^{\prime}=e_{\pm} \\
								e^{\prime} & \text{otherwise}
\end{array}\right.
$.
We call such a map an ``elementary degeneracy''. \\
It will be convenient to  call
``elementary degeneracy'' also for the  maps 
\begin{equation}
				\begin{array}[H]{c}
								s_{)e(}=s_{_{)e(}}^{G}: G_{)e(}	\overset{\text{def}}{=} \left(
				(G^{1}\setminus \left\{ e \right\})\cup \left\{ e_{-}, e_{+}
				\right\}\,\rdarrow{}{d_{\pm}}\; G^{\circ}\cup \left\{ v_{e}^{-}, v_{e}^{+}
\right\}  \right)\xtworightarrow{\quad}{}  G \\\\
d_{-}e_{-} = d_{-}e\quad , \quad d_{+}e_{-}=v_{e}^{-}\quad , \quad d_{-}e_{+}=v_{e}^{+}\quad ,
\quad d_{+}e_{+}= d_{+} e 
				\end{array}
				\label{eq:11.2}
\end{equation}
\begin{figure}[H]
\begin{tikzpicture}[baseline]
				\tikzmath{
								let \x =1.2;
								let \y =0;
								let \w = 1.5;
								let \s = 0.2;
				};
				\def\mlarrow#1#2#3#4#5{ 
							\draw[-] #1--#2;
							\draw[decorate,decoration={markings,mark=at position 0.5 with {\arrow[color=black]{<}}}] #1--#2;
							\node at ($0.5*#1+0.5*#2+(0,0.3)$) {$#4$};
							\filldraw ($#1-(0.1cm,0cm)$) circle [radius=.04cm];
							\filldraw ($#2+(0.1cm,0cm)$) circle [radius=.04cm];
							\node at ($#1-(0.1cm,-0.3cm)$) {$#3$};
							\node at ($#2+(0.1cm,0.3cm)$) {$#5$};
			};
			\coordinate (a) at (\x,\y);
			\mlarrow{(a)}{($(a)+(\w,0)$)}{}{e_{-}}{v_{e}^{-}};
			\coordinate (b) at ($(\x,\y)+(\w+7*\s,0)$);
			\mlarrow{(b)}{($(b)+(\w,0)$)}{v_{e}^{+}}{e_{+}}{};
			\coordinate (c) at ($(b)+(\w+7*\s,0)$);
			\draw[->>]  (c)--($(c)+(\w,0)$);
			\node  at ($0.5*(c)+0.5*($(c)+(\w,0)$) +(0,0.25cm)$) {$s_{_{)e(}}$};
			\coordinate (d) at ($(c)+(\w+7*\s,0)$);
			\mlarrow{(d)}{($(d)+(\w,0)$)}{}{e}{};
\end{tikzpicture}
				\centering
				\caption{Composition map of two degeneracies and one face map}
				\label{fig:7}
\end{figure}
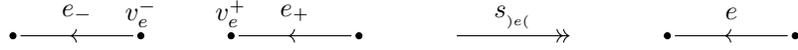 \ \\
For an edge $e\in G^{1}$, we may add another copy $e^{\prime}$ to $G$ with
$v_{\pm}^{\prime}=d_{\pm}e^{\prime}$ both stumps, and identifying $e$ and $e^{\prime}$ we
get the ``elementary degeneracy''
\begin{equation}
				s_{{\widecheck{e}}} = s_{{\widecheck{e}}}^{^G}:
				G_{{{\widecheck{e}}}}\overset{\text{def}}{=} \left( G^{1}\cup \left\{
												e^{\prime} \right\} \rdarrow{}{d_{\pm}}G^{\circ}\cup \left\{
																v_{\pm}^{\prime}
								\right\}\right)\xtworightarrow{\qquad}{} G
								\label{eq:11.3}
\end{equation}
Note that $G_{e}$ $G_{_{)e(}}$, $G_{_{\widecheck{e}}}$ have one more edge than $G$. \\
For $f\in \text{Bit}(G,H)$ clearly $\pi\in\text{Bit}\left( G,G/\ker(f) \right)$ is a
composition of such elementary degeneracies $s_{_{e}}$, $s_{_{)e(}}$ and
$s_{_{\widecheck{e}}}$. \\
Given $G\in \text{Bit}$, any $e\in G^{1}$, we get a bit by
eliminating $e$ 
\begin{equation}
				\partial_{e}=\partial_{e}^{G}: G^{e}\overset{\text{def}}{=} \left(
								G^{1}\setminus\left\{ e \right\}\rdarrow{}{d_{\pm}} G^{\circ}\cup \left\{
								v_{\ell} \right\}_{\ell} 
				\right)\hookedrightarrow{1} G
				\label{eq:11.4}
\end{equation}
here the vertices $\left\{ v_{\ell} \right\}$ are added to keep $G^{e}$ a bit; pictorially
we have four such $0$-faces: 
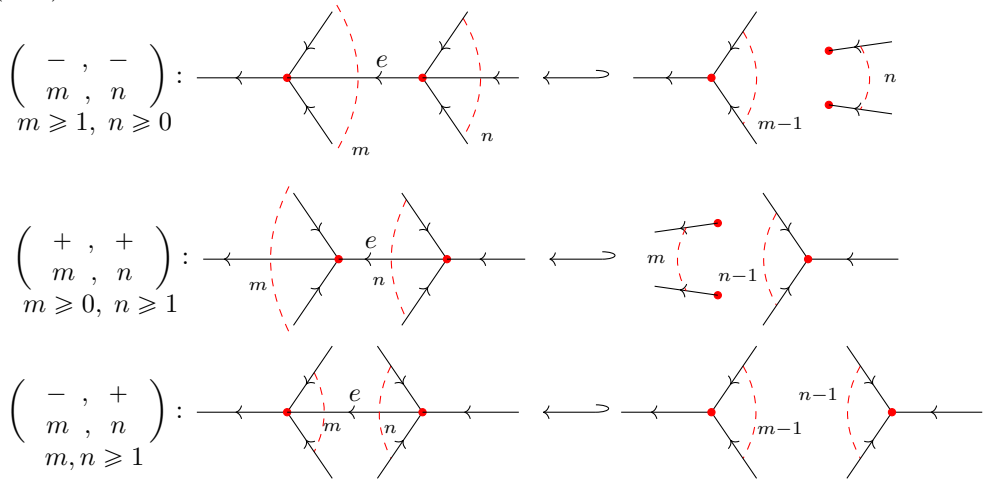
\begin{equation}
				\begin{array}[H]{c}
								\left(\; \begin{array}[H]{l}
												- \;\; , \;\; - \\
												m \;\; , \;\; n 
				\end{array}\;\right): \\
				m\ge 1, \; n\ge 0
				\end{array}
				\begin{tikzpicture}[baseline,scale=0.8]
				\tikzmath{
								let \w = 1.5cm;
								let \buff = 0.2;
								let \y = 0.35cm;
				};

			\def\mline#1#2#3{
			 \draw[decorate,decoration={markings,mark=at position #3 with {\arrow[color=black]{<}}}] #1--#2;
			 \draw  #1--#2;
			 };
			 \def\plt#1#2#3#4#5#6#7{  
							\coordinate (t1) at ($#1$);
							\coordinate (t2) at ($#2$);
							\coordinate (t3) at ($#3*(t2)+(t1)-#3*(t1)$);
							\mline{(t1)}{(t2)}{#3};
							\node at  ($(t3)+(0,1.4*\buff)$) {$#4$};
							\coordinate (u1) at ($#2+(#5,0.5*\w+\y)$);
							\coordinate (u2) at ($#2+(#5,-0.5*\w-\y)$);
							\def\d{#6};
							\draw[-,color=red, dashed] ($\d*(u1)+#2-\d*#2$) to [bend left] ($\d*(u2)+#2-\d*#2$);
							\node at ($\d*(u2)+#2-\d*#2+(2*\buff,0)$) {#7};
							\mline{#2}{(u1)}{0.5};
							\mline{#2}{(u2)}{0.5};
							\filldraw[color=red] #2 circle [radius=.06cm];
			};
				\coordinate (a) at (0,\y);
				\coordinate (b) at ($(a)+(\w,0)$);
				\coordinate (c) at ($(b)+(1.5*\w,0)$);
				\coordinate (d) at ($(c)+(1.2*\w,0)$);
				\coordinate (d1) at ($(d)-(\buff,0)$);
				\coordinate (d2) at ($(d)+(\buff,0)$);
				\coordinate (e) at ($(d)+(1*\w,0)$);
				\coordinate (e1) at ($(e)-(\buff,0)$);
				\coordinate (e2) at ($(e)+(\buff,0)$);
				\coordinate (f) at ($(e)+(1*\w,0)$);
				\coordinate (f1) at ($(f)-(\buff,0)$);
				\coordinate (f2) at ($(f)+(\buff,0)$);
				\plt{(a)}{(b)}{0.5}{}{0.5*\w+0*\w}{1.1}{$\scriptstyle m$};
				\plt{(b)}{(c)}{0.7}{e}{0.5*\w}{0.9}{$\scriptstyle n$};
				\mline{(c)}{(d1)}{0.8}; 
		    \draw[<-{Hooks[left,length=5,width=6]}] (d2) -- (e1);	
				\plt{(e2)}{(f)}{0.5}{}{0.5*\w}{0.7}{$\quad \scriptstyle m-1$};
				\coordinate (y1) at ($(f)+(1.3*\w,0.3*\w)$);
				\coordinate (y2) at ($(f)+(1.3*\w,-0.3*\w)$);
				\coordinate (z1) at ($(f)+(2*\w,0.4*\w)$);
				\coordinate (z2) at ($(f)+(2*\w,-0.4*\w)$);
				\filldraw[color=red] (y1) circle [radius=.06cm]; 
				\filldraw[color=red] (y2) circle [radius=.06cm]; 
				\mline{(y1)}{(z1)}{0.5};
				\mline{(y2)}{(z2)}{0.5};
				\def\curvstart{0.5};
				\draw[-,color=red, dashed] ($\curvstart*(z1)+(y1)-\curvstart*(y1)$) to [bend left] ($\curvstart*(z2)+(y2)-\curvstart*(y2)$);
				\node at ($0.5*($\curvstart*(z2)+(y2)-\curvstart*(y2)+ (3*\buff,0)$) +
				0.5*($\curvstart*(z1)+(y1)-\curvstart*(y1)$)+(\buff,0)$) {$\scriptstyle n$};
\end{tikzpicture}
\label{eq:11.5}
\end{equation}
\
\begin{equation*}
				\begin{array}[H]{c}
								\left(\; \begin{array}[H]{l}
												+ \;\; , \;\; + \\
												m \;\; , \;\; n 
				\end{array}\;\right): \\
				m\ge 0, \; n\ge 1
				\end{array}
				\begin{tikzpicture}[baseline,scale=0.8]
				\tikzmath{
								let \w = 1.5cm;
								let \buff = 0.2;
								let \y = 0.35cm;
				};

			\def\mline#1#2#3{
			 \draw[decorate,decoration={markings,mark=at position #3 with {\arrow[color=black]{<}}}] #1--#2;
			 \draw  #1--#2;
			 };
			 \def\plt#1#2#3#4#5#6#7#8{  
							\coordinate (t1) at ($#1$);
							\coordinate (t2) at ($#2$);
							\coordinate (t3) at ($#3*(t2)+(t1)-#3*(t1)$);
							\mline{(t1)}{(t2)}{#3};
							\node at  ($(t3)+(0,1.4*\buff)$) {$#4$};
							\coordinate (u1) at ($#2+(#5,0.5*\w+\y)$);
							\coordinate (u2) at ($#2+(#5,-0.5*\w-\y)$);
							\def\d{#6};
							\draw[-,color=red, dashed] ($\d*(u1)+#2-\d*#2$) to [bend left=#8] ($\d*(u2)+#2-\d*#2$);
							\node at ($\d*(u2)+#2-\d*#2-(\buff,-0.7*\d)$) {#7};
							\mline{#2}{(u1)}{0.5};
							\mline{#2}{(u2)}{0.5};
							\filldraw[color=red] #2 circle [radius=.06cm];
			};
			 \def\pltt#1#2#3#4#5#6#7#8{  
							 \mline{#2}{($#2+(\w,0)$)}{0.5};
							\coordinate (t1) at ($#1$);
							\coordinate (t2) at ($#2$);
							\coordinate (t3) at ($#3*(t2)+(t1)-#3*(t1)$);
							\node at  ($(t3)+(0,1.4*\buff)$) {$#4$};
							\coordinate (u1) at ($#2+(#5,0.5*\w+\y)$);
							\coordinate (u2) at ($#2+(#5,-0.5*\w-\y)$);
							\def\d{#6};
							\draw[-,color=red, dashed] ($\d*(u1)+#2-\d*#2$) to [bend left=#8] ($\d*(u2)+#2-\d*#2$);
							\node at ($\d*(u2)+#2-\d*#2-(\buff,-0.7*\d)$) {#7};
							\mline{#2}{(u1)}{0.5};
							\mline{#2}{(u2)}{0.5};
							\filldraw[color=red] #2 circle [radius=.06cm];
			};
				\coordinate (a) at (0,\y);
				\coordinate (b) at ($(a)+(1.5*\w,0)$);
				\coordinate (c) at ($(b)+(1.2*\w,0)$);
				\coordinate (d) at ($(c)+(1*\w,0)$);
				\coordinate (d1) at ($(d)-(\buff,0)$);
				\coordinate (d2) at ($(d)+(\buff,0)$);
				\coordinate (e) at ($(d)+(1*\w,0)$);
				\coordinate (e1) at ($(e)-(\buff,0)$);
				\coordinate (e2) at ($(e)+(\buff,0)$);
				\coordinate (f) at ($(e)+(1*\w,0)$);
				\coordinate (f1) at ($(f)-(\buff,0)$);
				\coordinate (f2) at ($(f)+(\buff,0)$);
				\plt{(a)}{(b)}{0.2}{}{-0.5*\w}{1.1}{$\scriptstyle \hspace{-.5cm} m$}{-0.9cm};
				\plt{(b)}{(c)}{0.3}{e}{-0.5*\w}{0.9}{$\scriptstyle \hspace{-.4cm} n$}{-0.9cm};
				\mline{(c)}{(d1)}{0.5}; 
		    \draw[<-{Hooks[left,length=5,width=6]}] (d2) -- (e1);	

				\coordinate (y1) at ($(e)+(.3*\w,0.3*\w)$);
				\coordinate (y2) at ($(e)+(.3*\w,-0.3*\w)$);
				\coordinate (z1) at ($(e)+(1*\w,0.4*\w)$);
				\coordinate (z2) at ($(e)+(1*\w,-0.4*\w)$);
				\filldraw[color=red] (z1) circle [radius=.06cm]; 
				\filldraw[color=red] (z2) circle [radius=.06cm]; 
				\mline{(y1)}{(z1)}{0.5};
				\mline{(y2)}{(z2)}{0.5};
				\def\curvstart{0.5};
				\draw[-,color=red, dashed] ($\curvstart*(z1)+(y1)-\curvstart*(y1)$) to [bend
				left=-1cm] ($\curvstart*(z2)+(y2)-\curvstart*(y2)$);
				\node at ($0.5*($\curvstart*(z2)+(y2)-\curvstart*(y2)+ (3*\buff,0)$) +
				0.5*($\curvstart*(z1)+(y1)-\curvstart*(y1)$)+(-4*\buff,0)$) {$\scriptstyle m$};

				\coordinate (f) at ($(e)+(1*\w,0)$);
				\coordinate (e) at ($(f)+(1*\w,0)$);
				\pltt{(f)}{(e)}{0.4}{}{-0.5*\w}{0.7}{$\scriptstyle  \\ \hspace{-.7cm}  n-1$}{-1cm};
\end{tikzpicture}
\end{equation*}
\begin{equation*}
				\begin{array}[H]{c}
								\left(\; \begin{array}[H]{l}
												- \;\; , \;\; + \\
												m \;\; , \;\; n 
				\end{array}\;\right): \\
				m,n\ge 1
				\end{array}
				\begin{tikzpicture}[baseline,scale=0.8]
				\tikzmath{
								let \w = 1.5cm;
								let \buff = 0.2;
								let \y = 0.35cm;
				};

			\def\mline#1#2#3{
			 \draw[decorate,decoration={markings,mark=at position #3 with {\arrow[color=black]{<}}}] #1--#2;
			 \draw  #1--#2;
			 };
			 \def\plt#1#2#3#4#5#6#7#8{  
							\coordinate (t1) at ($#1$);
							\coordinate (t2) at ($#2$);
							\coordinate (t3) at ($#3*(t2)+(t1)-#3*(t1)$);
							\mline{(t1)}{(t2)}{#3};
							\node at  ($(t3)+(0,1.4*\buff)$) {$#4$};
							\coordinate (u1) at ($#2+(#5,0.5*\w+\y)$);
							\coordinate (u2) at ($#2+(#5,-0.5*\w-\y)$);
							\def\d{#6};
							\draw[-,color=red, dashed] ($\d*(u1)+#2-\d*#2$) to [bend left=#8] ($\d*(u2)+#2-\d*#2$);
							\node at ($\d*(u2)+#2-\d*#2-(\buff,-0.7*\d)$) {#7};
							\mline{#2}{(u1)}{0.5};
							\mline{#2}{(u2)}{0.5};
							\filldraw[color=red] #2 circle [radius=.06cm];
			};
			\def\pltt#1#2#3#4#5#6#7#8{  
							 \mline{#2}{($#2+(\w,0)$)}{0.5};
							\coordinate (t1) at ($#1$);
							\coordinate (t2) at ($#2$);
							\coordinate (t3) at ($#3*(t2)+(t1)-#3*(t1)$);
							\node at  ($(t3)+(0,1.4*\buff)$) {$#4$};
							\coordinate (u1) at ($#2+(#5,0.5*\w+\y)$);
							\coordinate (u2) at ($#2+(#5,-0.5*\w-\y)$);
							\def\d{#6};
							\draw[-,color=red, dashed] ($\d*(u1)+#2-\d*#2$) to [bend left=#8] ($\d*(u2)+#2-\d*#2$);
							\node at ($\d*(u2)+#2-\d*#2-(\buff,-1.5*\d)$) {#7};
							\mline{#2}{(u1)}{0.5};
							\mline{#2}{(u2)}{0.5};
							\filldraw[color=red] #2 circle [radius=.06cm];
			};
				\coordinate (a) at (0,\y);
				\coordinate (b) at ($(a)+(1*\w,0)$);
				\coordinate (c) at ($(b)+(1.5*\w,0)$);
				\coordinate (d) at ($(c)+(1.2*\w,0)$);
				\coordinate (d1) at ($(d)-(\buff,0)$);
				\coordinate (d2) at ($(d)+(\buff,0)$);
				\coordinate (e) at ($(d)+(1*\w,0)$);
				\coordinate (e1) at ($(e)-(\buff,0)$);
				\coordinate (e2) at ($(e)+(\buff,0)$);
				\coordinate (f) at ($(e)+(1*\w,0)$);
				\coordinate (f1) at ($(f)-(\buff,0)$);
				\coordinate (f2) at ($(f)+(\buff,0)$);
				\plt{(a)}{(b)}{0.5}{}{0.5*\w}{0.6}{$\scriptstyle \hspace{0.8cm} m$}{0.9cm};
				\plt{(b)}{(c)}{0.5}{e}{-0.5*\w}{0.7}{$\scriptstyle \hspace{.3cm} n$}{-0.9cm};
				\mline{(c)}{(d1)}{0.5}; 
		    \draw[<-{Hooks[left,length=5,width=6]}] (d2) -- (e1);	


				\plt{(e)}{(f)}{0.4}{}{0.5*\w}{0.7}{$\scriptstyle  \\ \hspace{1.3cm}  m-1$}{1cm};
				\coordinate (f) at ($(e)+(3*\w,0)$);
				\coordinate (e) at ($(f)+(1*\w,0)$);
				\pltt{(e)}{(f)}{0.4}{}{-0.5*\w}{0.7}{$\scriptstyle \\ \hspace{-.8cm}  n-1$}{-1cm};
\end{tikzpicture}
\end{equation*}
\
\begin{equation*}
				\begin{array}[H]{c}
								\left(\; \begin{array}[H]{l}
												+ \;\; , \;\; - \\
												m \;\; , \;\; n 
				\end{array}\;\right): \\
				m,n\ge 0
				\end{array}
				\begin{tikzpicture}[baseline,scale=0.8]
				\tikzmath{
								let \w = 1.5cm;
								let \buff = 0.2;
								let \y = 0.35cm;
				};

			\def\mline#1#2#3{
			 \draw[decorate,decoration={markings,mark=at position #3 with {\arrow[color=black]{<}}}] #1--#2;
			 \draw  #1--#2;
			 };
			 \def\plt#1#2#3#4#5#6#7#8{  
							\coordinate (t1) at ($#1$);
							\coordinate (t2) at ($#2$);
							\coordinate (t3) at ($#3*(t2)+(t1)-#3*(t1)$);
							\mline{(t1)}{(t2)}{#3};
							\node at  ($(t3)+(0,1.4*\buff)$) {$#4$};
							\coordinate (u1) at ($#2+(#5,0.5*\w+\y)$);
							\coordinate (u2) at ($#2+(#5,-0.5*\w-\y)$);
							\def\d{#6};
							\draw[-,color=red, dashed] ($\d*(u1)+#2-\d*#2$) to [bend left=#8] ($\d*(u2)+#2-\d*#2$);
							\node at ($\d*(u2)+#2-\d*#2-(\buff,-0.7*\d)$) {#7};
							\mline{#2}{(u1)}{0.5};
							\mline{#2}{(u2)}{0.5};
							\filldraw[color=red] #2 circle [radius=.06cm];
			};
				\coordinate (a) at (0,\y);
				\coordinate (b) at ($(a)+(1.5*\w,0)$);
				\coordinate (c) at ($(b)+(1.2*\w,0)$);
				\coordinate (d) at ($(c)+(1*\w,0)$);
				\coordinate (d1) at ($(d)-(\buff,0)$);
				\coordinate (d2) at ($(d)+(\buff,0)$);
				\coordinate (e) at ($(d)+(1*\w,0)$);
				\coordinate (e1) at ($(e)-(\buff,0)$);
				\coordinate (e2) at ($(e)+(\buff,0)$);
				\coordinate (f) at ($(e)+(1*\w,0)$);
				\coordinate (f1) at ($(f)-(\buff,0)$);
				\coordinate (f2) at ($(f)+(\buff,0)$);
				\plt{(a)}{(b)}{0.2}{}{-0.5*\w}{1.1}{$\scriptstyle \hspace{-.5cm} m$}{-0.9cm};
				\plt{(b)}{(c)}{0.5}{e}{0.5*\w}{0.9}{$\scriptstyle \hspace{0cm} n$}{0.9cm};
				\mline{(c)}{(d1)}{0.5}; 
		    \draw[<-{Hooks[left,length=5,width=6]}] (d2) -- (e1);	

				\coordinate (y1) at ($(e)+(.3*\w,0.3*\w)$);
				\coordinate (y2) at ($(e)+(.3*\w,-0.3*\w)$);
				\coordinate (z1) at ($(e)+(1*\w,0.4*\w)$);
				\coordinate (z2) at ($(e)+(1*\w,-0.4*\w)$);
				\filldraw[color=red] (z1) circle [radius=.06cm]; 
				\filldraw[color=red] (z2) circle [radius=.06cm]; 
				\mline{(y1)}{(z1)}{0.5};
				\mline{(y2)}{(z2)}{0.5};
				\def\curvstart{0.5};
				\draw[-,color=red, dashed] ($\curvstart*(z1)+(y1)-\curvstart*(y1)$) to [bend
				left=-1cm] ($\curvstart*(z2)+(y2)-\curvstart*(y2)$);
				\node at ($0.5*($\curvstart*(z2)+(y2)-\curvstart*(y2)+ (3*\buff,0)$) +
				0.5*($\curvstart*(z1)+(y1)-\curvstart*(y1)$)+(-4*\buff,0)$) {$\scriptstyle m$};

				\coordinate (f) at ($(e)+(1*\w,0)$);
				\coordinate (e) at ($(f)+(1*\w,0)$);
				
				\coordinate (y1) at ($(e)+(.3*\w,0.3*\w)$);
				\coordinate (y2) at ($(e)+(.3*\w,-0.3*\w)$);
				\coordinate (z1) at ($(e)+(1*\w,0.4*\w)$);
				\coordinate (z2) at ($(e)+(1*\w,-0.4*\w)$);
				\filldraw[color=red] (y1) circle [radius=.06cm]; 
				\filldraw[color=red] (y2) circle [radius=.06cm]; 
				\mline{(y1)}{(z1)}{0.5};
				\mline{(y2)}{(z2)}{0.5};
				\def\curvstart{0.5};
				\draw[-,color=red, dashed] ($\curvstart*(z1)+(y1)-\curvstart*(y1)$) to [bend
				left=-1cm] ($\curvstart*(z2)+(y2)-\curvstart*(y2)$);
				\node at ($0.5*($\curvstart*(z2)+(y2)-\curvstart*(y2)+ (3*\buff,0)$) +
				0.5*($\curvstart*(z1)+(y1)-\curvstart*(y1)$)+(-4*\buff,0)$) {$\scriptstyle n$};
\end{tikzpicture}
\end{equation*}
Note that when $n=0$ (resp. $m=0$) we are pruning $G$ near $d_{+}G$ (resp. $d_{-}G$). Note
that $G^{e}$ have one less \myemph{edge} than $G$ (and can have more vertices!). \\
Note that if $f\in\text{Bit}(G,H)$ with $f^{1}$ injective, and $f$ is also a map of graphs
(so that $f^{0}(v)\subseteq C \left( \tilde{f}^{0}(v) \right)$, $\tilde{f}^{0}:
G^{0}\to H^{0}$), then $f$ is the composition of the $0$-faces
$\partial_e, e \in H^{1}\setminus G^{1}$. \\
Finally, we have the \myemph{inner-faces}, where we contract edges,
and they come again in $4$ types: 
\begin{figure}[H]
				\centering
				\begin{tikzpicture}[baseline]
								\tikzmath{
												let \w = 1.2cm;
												let \buff = 0.2cm;
												let \y = 0.2cm;
												let \ax = 2.5cm;
												let \ay = 0.2cm;
												let \bx = \ax+1.5cm;
												let \by = \ay;
								};
							\def\mline#1#2#3{
							 \draw[decorate,decoration={markings,mark=at position #3 with {\arrow[color=black]{<}}}] #1--#2;
							 \draw  #1--#2;
							 };
							\node at (0,0) {$ \begin{array}[H]{c}
																				\left(\begin{array}[H]{l}
																								- \quad, \quad  - \\	
																								m \quad , \quad  n  
																								\end{array}\right)\; : \\
																								m,n\ge 1
																\end{array}$};
							\mline{(1.5cm,0.2cm)}{(2.5cm,0.2cm)}{0.5};
							\coordinate (a) at (\ax,\ay);
							\filldraw[color=red] (a) circle [radius=.06cm];
							\coordinate (b) at (\bx,\by);
							\filldraw[color=red] (b) circle [radius=.06cm];
							\def\e{0.6};
							\mline{(a)}{(b)}{\e};
							\node at ($\e*(b)+(a)-\e*(a)+(0cm,\buff)$) {$e$};
							\coordinate (b1) at ($(b)+(-0.9cm,1cm)$);
							\coordinate (b2) at ($(b)+(-0.6cm,0.6cm)$);
							\coordinate (b3) at ($(b)+(-0.9cm,-1cm)$);
							\mline{(a)}{(b1)}{0.5};
							\mline{(a)}{(b2)}{0.5};
							\mline{(a)}{(b3)}{0.5};
							\node at ($(b3)+(0.4cm,0)$) {$m$};
							\draw [dashed, color=red]  plot [smooth] coordinates {($(b1)+(.1cm,0)$)($(a)+(1.3cm,0)$)($(b3)+(.1cm,0)$)};
							\coordinate (c) at ($(b)+(0cm,0)$);
							\coordinate (d) at ($(c)+(1cm,0)$);
							\coordinate (e) at ($(d)+(1.2cm,0)$);
							\coordinate (f) at ($(e)+(1cm,0)$);
							\coordinate (g) at ($(f)+(1.6cm,0)$);
							\coordinate (c2) at ($(d)+(-0.4cm,0.6cm)$);
							\coordinate (c3) at ($(d)+(-0.4cm,-0.6cm)$);
							\mline{(c)}{(c2)}{0.5};
							\mline{(c)}{(c3)}{0.5};
							\node at ($(c3)+(0.4,0)$) {$n$};
							\draw [dashed, color=red]  plot [smooth] coordinates {($(c2)+(.1cm,0)$)($(c)+(1.1cm,0)$)($(c3)+(.1cm,0)$)};
							\mline{(c)}{(d)}{0.5};
							\draw[<-{Hooks[left,length=5,width=6]}] ($(d)+(.4cm,0)$) -- (e);	
							\node at ($0.5*(d)+0.5*(.4cm,0)+0.5*(e)+(0,.4cm)$) {$\partial_{e}^{\, -}$};
							\mline{($(e)+(.4cm,0)$)}{($(f)+(.5cm,0)$)}{0.5};	
							\filldraw[color=red]  ($(f)+(.5cm,0)$) circle [radius=0.06cm];
							\mline{($(f)+(.5cm,0)$)}{($(g)+(.8cm,0cm)$)}{0.5}; 
							\coordinate (g1) at ($(g)+(-0.6cm,0.6cm)$);
							\coordinate (g2) at ($(g)+(-0.6cm,-0.6cm)$);
							\coordinate (g3) at ($(g)+(-0.3cm,0.4cm)$);
							\coordinate (g4) at ($(g)+(-0.3cm,-0.4cm)$);
							\coordinate (g5) at ($(g)+(-0.1cm,0.2cm)$);
							\mline{($(f)+(0.5cm,0)$)}{(g2)}{0.8};
							\mline{($(f)+(0.5cm,0)$)}{(g1)}{0.5};
							\mline{($(f)+(0.5cm,0)$)}{(g3)}{0.8};
							\mline{($(f)+(0.5cm,0)$)}{(g4)}{0.8};
							\mline{($(f)+(0.5cm,0)$)}{(g5)}{0.8};
							\draw [dashed, color=red]  plot [smooth] coordinates
							{($(g1)+(.1cm,0)$)($(g)+(0.1cm,0)$)($(g2)+(.1cm,0)$)};
							\node at ($(g2)+(0.4cm,-.2cm)$) {$m+n-1$};
				\end{tikzpicture}
				\caption{-- composition face}
				\label{fig:8}
\end{figure}

\begin{figure}[H]
				\centering
				\begin{tikzpicture}[baseline]
								\tikzmath{
												let \w = 1.2cm;
												let \buff = 0.2cm;
												let \y = 0.2cm;
												let \ax = 2.5cm;
												let \ay = 0.2cm;
												let \bx = \ax+1.5cm;
												let \by = \ay;
								};
							\def\mline#1#2#3{
							 \draw[decorate,decoration={markings,mark=at position #3 with {\arrow[color=black]{>}}}] #1--#2;
							 \draw  #1--#2;
							 };
							\node at (0,0) {$ \begin{array}[H]{c}
																				\left(\begin{array}[H]{l}
																								+ \quad, \quad  + \\	
																								m \quad , \quad  n  
																								\end{array}\right)\; : \\
																								m,n\ge 1
																\end{array}$};
							\mline{(2.5cm,0.2cm)}{(1.5cm,0.2cm)}{0.5};
							\coordinate (a) at (\ax,\ay);
							\filldraw[color=red] (a) circle [radius=.06cm];
							\coordinate (b) at (\bx,\by);
							\filldraw[color=red] (b) circle [radius=.06cm];
							\def\e{0.6};
							\mline{(a)}{(b)}{\e};
							\node at ($\e*(b)+(a)-\e*(a)+(-.1cm,\buff)$) {$e$};
							\coordinate (b1) at ($(b)+(-1.9cm,0.6cm)$);
							\coordinate (b3) at ($(b)+(-1.9cm,-0.6cm)$);
							\mline{(a)}{(b1)}{0.5};
							\mline{(a)}{(b3)}{0.5};
							\node at ($(b3)+(-0.4cm,0)$) {$m$};
							\draw [dashed, color=red]  plot [smooth] coordinates
							{($(b1)+(-.1cm,0)$)($(a)+(-0.8cm,0)$)($(b3)+(-.1cm,0)$)};
							\coordinate (c) at ($(b)+(0cm,0)$);
							\coordinate (d) at ($(c)+(1cm,0)$);
							\coordinate (e) at ($(d)+(1.2cm,0)$);
							\coordinate (f) at ($(e)+(1cm,0)$);
							\coordinate (g) at ($(f)+(1.6cm,0)$);
							\coordinate (c2) at ($(d)+(-1.4cm,0.9cm)$);
							\coordinate (c3) at ($(d)+(-1.4cm,-0.9cm)$);
							\coordinate (c4) at ($(d)+(-1.8cm,0.5cm)$);
							\mline{(c)}{(c2)}{0.5};
							\mline{(c)}{(c3)}{0.5};
							\mline{(c)}{(c4)}{0.5};
							\draw [dashed, color=red]  plot [smooth] coordinates {($(c2)+(-.1cm,0)$)($(c)+(-1.1cm,0)$)($(c3)+(-.1cm,0)$)};
							\node at ($(c3)+(-0.4cm,0)$) {$n$};
							\mline{(c)}{(d)}{0.5};
							\draw[<-{Hooks[left,length=5,width=6]}] ($(d)+(.4cm,0)$) -- (e);	
							\node at ($0.5*(d)+0.5*(.4cm,0)+0.5*(e)+(0,.4cm)$) {$\partial_{e}^{\, +}$};
							\mline{(f)}{($(e)+(.4cm,0)$)}{0.5};	
							\filldraw[color=red]  ($(f)+(.5cm,0)$) circle [radius=0.06cm];
							\mline{($(f)$)}{($(g)+(.7cm,0cm)$)}{0.5};
							\coordinate (g1) at ($(g)+(-1.6cm,0.6cm)$);
							\coordinate (g2) at ($(g)+(-1.6cm,-0.6cm)$);
							\coordinate (g3) at ($(g)+(-1.9cm,0.4cm)$);
							\coordinate (g4) at ($(g)+(-1.9cm,-0.4cm)$);
							\coordinate (g5) at ($(g)+(-2cm,0.2cm)$);
							\mline{($(f)+(0.5cm,0)$)}{(g2)}{0.6};
							\mline{($(f)+(0.5cm,0)$)}{(g1)}{0.5};
							\mline{($(f)+(0.5cm,0)$)}{(g3)}{0.7};
							\mline{($(f)+(0.5cm,0)$)}{(g4)}{0.5};
							\mline{($(f)+(0.5cm,0)$)}{(g5)}{0.5};
							\draw [dashed, color=red]  plot [smooth] coordinates
							{($(g1)+(-.1cm,0)$)($(g)+(-2.1cm,0)$)($(g2)+(-.1cm,0)$)};
							\node at ($(g2)+(0.4cm,-.2cm)$) {$m+n-1$};
				\end{tikzpicture}
				\caption{+ composition face}
				\label{fig:9}
\end{figure} 
\begin{figure}[H]
				\centering
				\begin{tikzpicture}[baseline,scale=1]
								\tikzmath{
												let \w = 1.2cm;
												let \buff = 0.2cm;
												let \y = 0.2cm;
												let \ax = 2.5cm;
												let \ay = 0.2cm;
												let \bx = \ax+2cm;
												let \by = \ay;
								};
							\def\mline#1#2#3{
							 \draw[decorate,decoration={markings,mark=at position #3 with {\arrow[color=black]{<}}}] #1--#2;
							 \draw  #1--#2;
							 };
							\node at (0,0) {$ \begin{array}[H]{c}
																				\left(\begin{array}[H]{l}
																								- \quad, \quad  + \\	
																								m \quad , \quad  n  
																								\end{array}\right)\; : \\
																								m\ge n\ge 1
																\end{array}$};
							\mline{(1.5cm,0.2cm)}{(2.5cm,0.2cm)}{0.5};
							\coordinate (a) at (\ax,\ay);
							\filldraw[color=red] (a) circle [radius=.06cm];
							\coordinate (b) at (\bx,\by);
							\filldraw[color=red] (b) circle [radius=.06cm];
							\def\e{0.6};
							\draw[->] (b)--($(a)+(.05cm,0)$);
							\node at ($\e*(b)+(a)-\e*(a)+(-.4cm,0.15)$) {$e_{i}$};
							\node at ($\e*(b)+(a)-\e*(a)+(-.4cm,-0.05)$) {$.$};
							\node at ($\e*(b)+(a)-\e*(a)+(-.4cm,-0.1)$) {$.$};
							\node at ($\e*(b)+(a)-\e*(a)+(-.4cm,-0.15)$) {$.$};
							\node at ($\e*(b)+(a)-\e*(a)+(-.4cm,0.4)$) {$e_{1}$};
							\node at ($\e*(b)+(a)-\e*(a)+(-.4cm,-0.35)$) {$e_{n}$};
							\coordinate (b1) at ($(b)+(-1.8cm,1cm)$);
							\coordinate (b2) at ($(b)+(-1.4cm,0.8cm)$);
							\coordinate (b3) at ($(b)+(-1.8cm,-1cm)$);

							\draw [-> , color=brown]  plot [smooth] coordinates {($(b)$) ($0.5*(a)+0.5*(b)+(0,-0.2cm)$) ($0.7*(a)+0.3*(b)+(0,-0.2cm)$)($(a)+(0.2,-0.1)$) };
							\draw [-> , color=brown]  plot [smooth] coordinates {($(b)$)
							($0.5*(a)+0.5*(b)+(0,0.25cm)$) ($0.7*(a)+0.3*(b)+(0,0.25cm)$) ($(a)+(0.2,0.1)$)};
							\mline{(a)}{(b1)}{0.5};
							\mline{(a)}{(b2)}{0.5};
							\mline{(a)}{(b3)}{0.5};
							\node at ($(b3)+(0.4cm,0)$) {$m$};
							\draw [dashed, color=red]  plot [smooth] coordinates {($(b1)+(.1cm,0)$)($(a)+(1.3cm,0)$)($(b3)+(.1cm,0)$)};
							\coordinate (c) at ($(b)+(0cm,0)$);
							\coordinate (d) at ($(c)+(1cm,0)$);
							\coordinate (e) at ($(d)+(1.2cm,0)$);
							\coordinate (f) at ($(e)+(1cm,0)$);
							\coordinate (g) at ($(f)+(1.6cm,0)$);
							\coordinate (c2) at ($(d)+(-0.4cm,0.6cm)$);
							\coordinate (c3) at ($(d)+(-0.4cm,-0.6cm)$);
							\mline{(c)}{(d)}{0.5};
							\draw[<-{Hooks[left,length=5,width=6]}] ($(d)+(.4cm,0)$) -- (e);	
							\node at ($0.5*(d)+0.5*(.4cm,0)+0.5*(e)+(0,.4cm)$) {$\partial_{e}^{\,
							\leftarrow}$};
							\mline{($(e)+(.4cm,0)$)}{($(f)+(.5cm,0)$)}{0.5};	
							\filldraw[color=red]  ($(f)+(.5cm,0)$) circle [radius=0.06cm];
							\mline{($(f)+(.5cm,0)$)}{($(g)+(.7cm,0cm)$)}{0.5};
							\coordinate (g1) at ($(g)+(-0.6cm,0.6cm)$);
							\coordinate (g2) at ($(g)+(-0.6cm,-0.6cm)$);
							\coordinate (g3) at ($(g)+(-0.3cm,0.4cm)$);
							\coordinate (g4) at ($(g)+(-0.3cm,-0.4cm)$);
							\coordinate (g5) at ($(g)+(-0.1cm,0.2cm)$);
							\mline{($(f)+(0.5cm,0)$)}{(g2)}{0.8};
							\mline{($(f)+(0.5cm,0)$)}{(g1)}{0.5};
							\mline{($(f)+(0.5cm,0)$)}{(g3)}{0.8};
							\draw [dashed, color=red]  plot [smooth] coordinates
							{($(g1)+(.1cm,0)$)($(g)+(0.1cm,0)$)($(g2)+(.1cm,0)$)};
							\node at ($(g2)+(0.4cm,-.2cm)$) {$m-n+1$};
				\end{tikzpicture}
				\caption{- action face}
				\label{fig:10}
\end{figure} 
\begin{figure}[H]
				\centering
				\begin{tikzpicture}[baseline,scale=1]
								\tikzmath{
												let \w = 1.2cm;
												let \buff = 0.2cm;
												let \y = 0.2cm;
												let \ax = 4cm;
												let \ay = 0.2cm;
												let \bx = \ax+2cm;
												let \by = \ay;
								};
							\def\mline#1#2#3{
							 \draw[decorate,decoration={markings,mark=at position #3 with {\arrow[color=black]{<}}}] #1--#2;
							 \draw  #1--#2;
							 };
							\node at (0,0) {$ \begin{array}[H]{c}
																				\left(\begin{array}[H]{l}
																								- \quad, \quad  + \\	
																								m \quad , \quad  n  
																								\end{array}\right)\; : \\
																								1\le m \le  n
																\end{array}$};
							\mline{(1.5cm,0.2cm)}{(2.5cm,0.2cm)}{0.5};
							\coordinate (a) at (\ax,\ay);
							\filldraw[color=red] (a) circle [radius=.06cm];
							\filldraw[color=red] (2.5cm,0.2cm) circle [radius=.06cm];
							\coordinate (q) at (2.6cm,0.2cm);
							\mline{(q)}{(a)}{0.05};
							\coordinate (b) at (\bx,\by);
							\def\e{0.6};
							\mline{(a)}{(b)}{0.5};
							\node at ($\e*(b)+(a)-\e*(a)+(-2cm,0.15)$) {$e_{i}$};
							\node at ($\e*(b)+(a)-\e*(a)+(-2cm,-0.05)$) {$.$};
							\node at ($\e*(b)+(a)-\e*(a)+(-2cm,-0.1)$) {$.$};
							\node at ($\e*(b)+(a)-\e*(a)+(-2cm,-0.15)$) {$.$};
							\node at ($\e*(b)+(a)-\e*(a)+(-2cm,0.4)$) {$e_{1}$};
							\node at ($\e*(b)+(a)-\e*(a)+(-2cm,-0.35)$) {$e_{m}$};
							\coordinate (b1) at ($(b)+(-2.8cm,1.5cm)$);
							\coordinate (b2) at ($(b)+(-3cm,0.9cm)$);
							\coordinate (b3) at ($(b)+(-2.8cm,-1.5cm)$);

							\draw [-> , color=brown]  plot [smooth] coordinates {($(a)$)
							($0.5*(q)+0.5*(a)+(0,-0.2cm)$) ($0.7*(q)+0.3*(a)+(0,-0.2cm)$)($(q)+(-0.1,-0.1)$) };
							\draw [-> , color=brown]  plot [smooth] coordinates {($(a)$)
							($0.5*(q)+0.5*(a)+(0,0.25cm)$) ($0.7*(q)+0.3*(a)+(0,0.25cm)$) ($(q)+(-0.1,0.1)$)};
							\mline{(b1)}{(a)}{0.5};
							\mline{(b2)}{(a)}{0.5};
							\mline{(b3)}{(a)}{0.5};
							\node at ($(b3)+(-0.4cm,0)$) {$n$};
							\draw [dashed, color=red]  plot [smooth] coordinates
							{($(b1)+(-.1cm,0)$)($(a)+(-1.2cm,0)$)($(b3)+(-.1cm,0)$)};
							\coordinate (c) at ($(b)+(0cm,0)$);
							\coordinate (d) at ($(c)+(1cm,0)$);
							\coordinate (e) at ($(d)+(1.2cm,0)$);
							\coordinate (f) at ($(e)+(1cm,0)$);
							\coordinate (g) at ($(f)+(1.6cm,0)$);
							\coordinate (c2) at ($(d)+(-0.4cm,0.6cm)$);
							\coordinate (c3) at ($(d)+(-0.4cm,-0.6cm)$);
							\draw[<-{Hooks[left,length=5,width=6]}] ($(c)+(.4cm,0)$) -- (d);	
							\node at ($0.5*(c)+0.5*(.4cm,0)+0.5*(d)+(0,.4cm)$) {$\partial_{e}^{\,
							\rightarrow}$};
							\mline{($(d)+(.4cm,0)$)}{(f)}{0.5};	
							\filldraw[color=red]  ($(e)+(.5cm,0)$) circle [radius=0.06cm];
							\mline{($(e)+(.5cm,0)$)}{(f)}{0.5};
							\coordinate (f1) at ($(f)+(-1.2cm,0.9cm)$);
							\coordinate (f2) at ($(f)+(-1.2cm,-0.9cm)$);
							\coordinate (f3) at ($(f)+(-1.3cm,0.5cm)$);
							\mline{(f2)}{($(e)+(0.5cm,0)$)}{0.5};
							\mline{(f1)}{($(e)+(0.5cm,0)$)}{0.5};
							\mline{(f3)}{($(e)+(0.5cm,0)$)}{0.5};
							\draw [dashed, color=red]  plot [smooth] coordinates
							{($(f1)+(-.1cm,0)$)($(f)+(-1.5cm,0)$)($(f2)+(-.1cm,0)$)};
							\node at ($(g2)+(-1.7cm,-.2cm)$) {$n-m+1$};
				\end{tikzpicture}
				\caption{+ action face}
				\label{fig:11}
\end{figure}
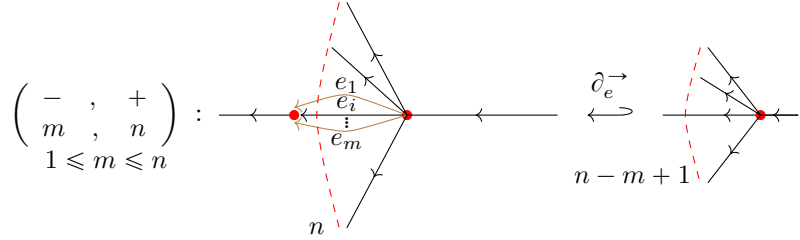 
\noindent For the composition face $\partial_e^{\pm}: G(e)\, \hookedrightarrow{1} G$, $G(e)$ has one less edge
than $G$; for the action faces
$\partial_{e_i}^{\, \overset{\leftarrow}{\rightarrow} }:G(e_{i})\, \hookedrightarrow{1}G$,
$G(e_i)$  has $\min\left\{ m,n \right\}$ less edges than $G$ ! (all these faces have one less verex). Note: $m,n\ge
1$, these are indeed ``inner''. \vspace{.2cm}\\ 
From the inductive construction of the $1$-sub-bits
$\mathscr{S}_{G}^{\pm}$, it follows that every $B\in \mathscr{S}_{G}^{\pm} $ can be replaced by a corolla by
finite compositions of inner faces. We obtain, 
\begin{proposition}
				Every map $f\in \text{Bit}(G,H)$ can be factored into 
				\begin{equation*}
								f: G\xtworightarrow{\pi}{}G/\ker(f)
								\xrightarrow{\overset{i}{\sim}}{}
								f(G)\overset{\partial^{\text{in}}_{f}}{\,\hookedrightarrow{1}}
								\overline{f(G)}
								\overset{\partial^{\text{out}}_{f}}{\,\hookedrightarrow{1}} H
				\end{equation*}
				$\pi$ a composition of degeneracies, $i$ an isomorphism,
				$\partial_{f}^{\text{in}}$ a composition of inner-faces,
				$\partial_{f}^{\text{out}}$ a composition of $0$-faces. 		
				\label{prop:2}
\end{proposition}
We write $\Phi(G)$ for the
				set of all $1$-faces $s_{e}: G^{e}\, \hookedrightarrow{1}G$, and we write
				$\Phi^{\text{in}}(G)\subseteq \Phi(G)$ for the ``inner $1$-faces'' for
				compositions and actions (which can reduce the number of edges by more than one
				$!$). 
				\begin{remark}
								We will only need the full subcategory 
								$ \text{Bit}_{\text{con}}\subseteq\text{Bit} $ consisting of
								$G\in\text{Bit}$ with $G$ a \underline{connected} graph (Indeed,
								$1$-bits are always connected!). The maps in $\text{Bit}_{\text{con}}$
								have similar description via faces and degeneracies, but one has to be
								careful when taking a $0$-face not to disconnect the graph. 
								\label{remark11.6}
				\end{remark}
\section{The categroy of Bital sets $b$Set}
We let  $ \text{bSet}:=\text{Set}^{\text{Bit}^{\op}} $ denote the category of
\underline{Bital-sets}. Its objects are $X=\left\{ X_{G} \right\}_{G\in\text{Bit}}$,\; a
collection of sets one for each $G\in \text{Bit}$, and function 
$\varphi^{\ast}:X_{H}\to X_{G}$ for each $\varphi\in\text{Bit}(G,H)$, $\left(
				\psi\circ\varphi
\right)^{\ast}=\varphi^{\ast}\circ\psi^{\ast}$,
$\text{id}_{G}^{\ast}=\text{id}_{X_{G}}$; \vspace{.1cm}\\
maps \; are \; natural \; transformations \; $f:X\to Y$, \;
$f_{G}\in \text{Set}(X_{G}, Y_{G})$, \;  $f_{G}\circ\varphi^{\ast}=\varphi^{\ast}\circ
f_{H}$ for $\varphi\in \text{Bit}(G,H)$. \vspace{.1cm}\\ 
We call an element $x\in X_{G} \;
\text{a "\underline{bitex} of type $G$" (plural: "bitices")} $ \\
The category $b$Set is complete and cocomplete, and for
$F=\left\{ F_{d} \right\}\in \text{bSet}^{D}$
the (co-)limit is obtained pointwise
$\left(\cocolim\limits_{d\in D} F_{d} \right)_{G}=\cocolim\limits_{d\in D} \left( F_{d}
\right)_{G}$. \vspace{.1cm}\\
Indeed for $G\in \text{Bit}$, the functor $\delta_{G}^{\ast}: \text{bSet}\to \text{Set}$,
$\delta_{G}^{\ast}(X)=X_{G}$, has both a left and a right adjoints. \\
Every $X\in \text{bSet}$ is a canonical colimit of representable functors 
\begin{equation}
				X = \colim\limits_{
				\begin{array}[b]{c}
							\scriptstyle	\xrightarrow{\qquad} \vspace{-2mm}\\
							\scriptstyle (\mathscr{Y}_{G}\to X)
				\end{array}
				\in \text{Bit}/X}
				\mathscr{Y}_{G}
				\label{eq:12.1}
\end{equation}
\begin{equation*}
				\mathscr{Y}:\text{Bit}\, \hookedrightarrow{1} \text{bSet} \; \text{the Yoneda functor,
				$\left( \mathscr{Y}_{H} \right)_{G}$:=\text{Bit}(G,H),} 
\end{equation*}
\begin{equation*}
				\begin{array}[H]{l}
							\text{Bit}/X = 
							\begin{array}[t]{l}
												\text{the "Kan-Grothendieck comma category",}\\ 
												 \begin{array}[H]{l}
																\text{objects:}\;\;\; (G,x),\;\; G\in \text{Bit},\;\; x\in X_{G}
																 \Longleftrightarrow x: \mathscr{Y}_{G}\to X
																 \\\\
																 \text{arrows:}\;\;\; \varphi: (G,x)\to (H,x^{1})\; \text{are
																 $\varphi\in\text{Bit}(G,H)$ with $\varphi^{\ast}(x^{1})=x.$} 
												 \end{array} \\\\
												 \text{there is an obvious functor $\text{Bit}/X\to
												 \text{Bit}\,\overset{y}{\hookedrightarrow{1}}\, \text{bSet}$}.
								\end{array}
				\end{array}
\end{equation*}
By Kan's abstract nonsense we get adjunction
\begin{equation}
				\begin{array}[H]{rlcll}
								\text{\underline{$\tau(X)$ the ``realization`` of $X$:}} & &  & &
								\underline{\text{$N(B)$ the "nerve" of $B$:}} \vspace{.3cm}\\
								X \quad & \in & \text{bSet} & \ni & N(B): N(B)_{G} = \text{Bio}\left(
												\mathscr{S}_{G},B
								\right) \\
								\tikz \draw[|->] (0,0)-- (0,-1cm);\quad \; & & 
								\begin{tikzpicture}[baseline]
												\draw[<-,color=red] (-0.5,0) to [bend left=2cm] ($(-0.5,1)$);
												\draw[->,color=red] (0.1,0) to [bend right=2cm] ($(0.1,1)$);
												\node at ($(0.6,0.5cm)$) {$N$};
												\node at ($(-1.1,0.5cm)$) {$\tau$};
\end{tikzpicture} \;\;
								& & \; \tikz \draw[|->] (0,-1cm)--(0,0); \\
								\tau(X)=\colim\limits_{(\mathscr{Y}_G\to X)\in\text{Bit}/X} \mathscr{S}_G \; :\tau(X) \;\; & \in &
								\text{Bio} & \ni & B
				\end{array}
				\label{eq:12.2}
\end{equation}
\begin{remark}
				The full embeddings 
				\begin{equation}
								\bbDelta \,
								\begin{array}[b]{c} i \vspace{-1.5mm} \\ \hookrightarrow \end{array}
								\Omega
								\,\begin{array}[b]{c} j \vspace{-1.5mm} \\ \hookrightarrow \end{array} \text{Bit}
								\label{eq:12.4}
				\end{equation}
				($\Omega$ the Dendroidal category of \myemph{closed} trees) \\
				give rise to Kan adjoints, compatible with the nerves and realizations functors, 
				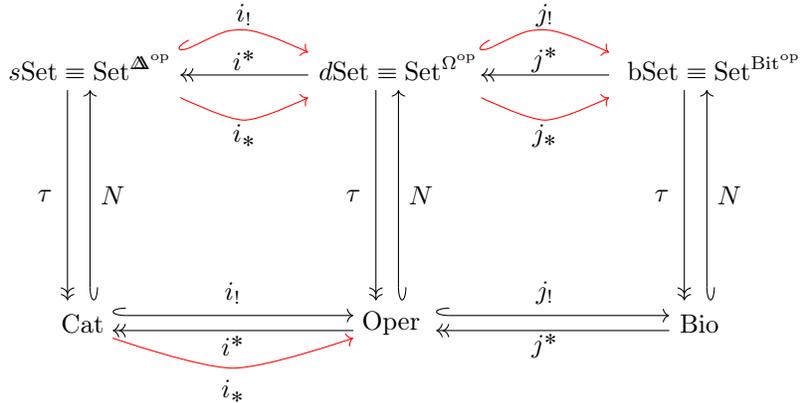
\begin{figure}[H]
								\centering
								\begin{tikzpicture}
												\tikzmath{
																let \buff=0.3cm;
												}

												\node at (1.10cm,5.1cm) {$s\text{Set}\equiv \text{Set}^{\bbDelta^\op}$};
												\draw[<<-] (2.30cm,5.00cm)--(4.00cm,5.00cm);
								\draw [{Hooks[right,length=5,width=6]}->, color=red]  plot [smooth] coordinates {(2.3,5.3)  (3.15,5.6) (4cm,5.3cm)};
								\draw [->, color=red]  plot [smooth] coordinates {(2.3,4.7)  (3.15,4.4) (4cm,4.7cm)};
								\node at (3.15,4.2) {$i_{\ast}$};
								\node at (3.15,5.2) {$i^{\ast}$};
								\node at (3.15,5.8) {$i_!$};
								\node at (5.2cm,5.1cm) {$d\text{Set}\equiv \text{Set}^{{\Omega}^{\op}}$};
								\draw[<<-] (6.3cm,5cm)--(8cm,5cm);
								\node at  (9.4cm,5.1cm) {$\text{bSet}\equiv \text{Set}^{\text{Bit}^{\op}}$};
								\draw [{Hooks[right,length=5,width=6]}->, color=red]  plot [smooth] coordinates {(6.3,5.3)  (7.15,5.6) (8cm,5.3cm)};
								\draw [->, color=red]  plot [smooth] coordinates {(6.3cm,4.7cm) (7.15cm,4.4cm) (8cm,4.7cm)};
								\node at (7.15,4.2) {$j_{\ast}$};
								\node at (7.15,5.2) {$j^{\ast}$};
								\node at (7.15,5.8) {$j_!$};
								\def\doar#1#2{
												\draw[->>] #1--#2;
												\draw[<-{Hooks[left,length=5,width=6]}] ($#1+(3mm,0mm)$)--($#2+(3mm,0mm)$);
												\node at ($0.5*#1+0.5*#2+(6mm,0mm)$) {$N$};
												\node at ($0.5*#1+0.5*#2+(-3mm,0mm)$) {$\tau$};
								};
								\doar{(8mm,48mm)}{(8mm,20mm)};
								\doar{(49mm,48mm)}{(49mm,20mm)};
								\doar{(90mm,48mm)}{(90mm,20mm)};
								\node at (10mm,17mm) {Cat};
								\node at (51mm,17mm) {Oper};
								\node at (92mm,17mm) {Bio};
								\draw[{Hooks[right,length=5,width=6]}->] (14mm,18mm)--(46mm,18mm);
								\draw[<<-] (14mm,16mm)--(46mm,16mm);
								\draw[{Hooks[right,length=5,width=6]}->] (57mm,18mm)--(88mm,18mm);
								\draw[<<-] (57mm,16mm)--(88mm,16mm);
								\node at (30mm,21mm) {$i_!$};
								\node at (30mm,14mm) {$i^{\ast}$};
								\node at (30mm,8mm) {$i_{\ast}$};
								\draw [->, color=red]  plot [smooth] coordinates {(14mm,15mm)(30mm,11mm)(46mm,15mm)};
								;
								\node at (71.5mm,21mm) {$j_!$};
								\node at (71.5mm,14mm) {$j^{\ast}$};
								\end{tikzpicture}
								\caption{Realization functors}
								\label{fig:12}
				\end{figure} \  \\
				The category $b$Set has an involution $X\mapsto X^{\op}$ with
				$(X^{\op})_{G}:=X_{G^{\op}}$, compatible with the involution
				$\biop\mapsto\biop^{\op}$ on $\text{Bio}: N(\biop^{\op})=N(\biop)^\op$,
				$\tau(X^{\op})=\tau(X)^{\op}$.
				\label{remark12.4}
\end{remark}
\section{The symmetric ``monoidal'' structure on $b$Set}
The closed symmetric monoidal structure on Bio induces a closed ``almost-monoidal''
symmetric structure on $b$Set via Day convolution 
\begin{equation}
				X\otimes Y:=\hspace{-11mm}
				\colim\limits_{
								\hspace{10mm} \begin{array}[t]{l}
												\scriptstyle \xrightarrow{\qquad} \vspace{-.1cm}\\ 
												\scriptstyle {(\mathscr{Y}_{G}\to X)} \in \text{Bit}/X \vspace{-.1cm}\\ 
												\scriptstyle (\mathscr{Y}_{H}\to Y) \in \text{Bit}/Y 
								\end{array}
				} \hspace{-9mm}
								N\left( \mathscr{S}_{G}\otimes \mathscr{S}_{H} \right)
								\label{eq:13.1}
\end{equation}
with inner hom
\begin{equation}
				\left( Y^{X} \right)_{G} := \text{bSet}\left( X\otimes \mathscr{Y}_{G},Y \right)
				\label{eq:13.2}
\end{equation}
and adjunction
\begin{equation}
				\text{bSet}(X\otimes Y, Z) \equiv \text{bSet}(X,Z^{Y})
				\label{eq:13.3}
\end{equation}
so that $\_\otimes\_$ commutes with co-limits. This symmetric structure \break $X\otimes Y
\cong Y\otimes X$, is Not associative: we have inclusions that need not be
equalities: 
\begin{equation}
(X\otimes Y)\otimes Z\subseteq X\otimes Y \otimes Z \supseteq X\otimes (Y\otimes Z). 
\label{eq:13.4}
\end{equation}
\begin{remark}
				The tensor product of representable can be described as in \cite{cisinski2014note} via
				``perculations''. We have 
				\begin{equation}
								\mathscr{Y}_{G}\otimes \mathscr{Y}_{H} \equiv N(\mathscr{S}_{G}\otimes \mathscr{S}_{H}) =
								\bigcup\limits_{K} \mathscr{Y}_{K}
								\label{eq:13.6}
				\end{equation}
				the union taken over all ``shuffles`` $K$ of $G$ and $H$, with $K^{1}\subseteq G^{1}\times
				H^{1}$, \break $K^{0}\subseteq G^{1}\times H^{0}\coprod G^{0}\times H^{1}$. For
				$G\in \text{Bit}_{m_{G},n_{G}}$, $H\in \text{Bit}_{m_{H}, n_{H}}$, the shuffles
				come with a partial order, with the biggest shuffle $K_{\max}$ obtained by putting
				$m_{H}$ copies of $G$ on the left of $n_{G}$ copies of $H$, and grafting the
				$m_{H}\times n_{G}$ edges together: $m_{H}\cdot d_{+}G\sim n_{G}\cdot d_{-}{H}$ 
				\begin{equation}
								K_{\max}= G\times d_{-}H \coprod\limits_{d_{+}G\times d_{-}H}
								d_{+}G\times H
								\label{eq:13.7}
				\end{equation}
				Similarly, the smallest shuffle is 
				\begin{equation}
								K_{\min}=d_{-}G\times H\coprod\limits_{d_{-}G\times d_{+}H} G\times
								d_{+}H
								\label{eq:13.8}
				\end{equation}
				We go from $K_{\max}$ down to $K_{\min}$ in this partial order via the
				''perculation moves``, moving $G$ (co)-operations $G^{0}\times H^{1}$ to the right
				of $H$ (co)-operations $G^{1}\times H^{0}$, replacing a copy of $A\times
d_{-} B \hspace{-4mm} \coprod\limits_{d_{+}A\times d_{-}B} \hspace{-4mm} d_{+}A\times B $ by a copy of 
$d_{-}A\times B\hspace{-4mm}\coprod\limits_{d_{-}A\times d_{+}B} \hspace{-4mm} A\times d_{+ }B$  for sub-bits $A\in
				\text{Bit}_{m^{\prime},n}(G)$, $B\in \text{Bit}_{m,n^{\prime}}(H)$. \\ 
				Graphically, with $A_{j}\cong A$, $B_{i}\cong B$, 
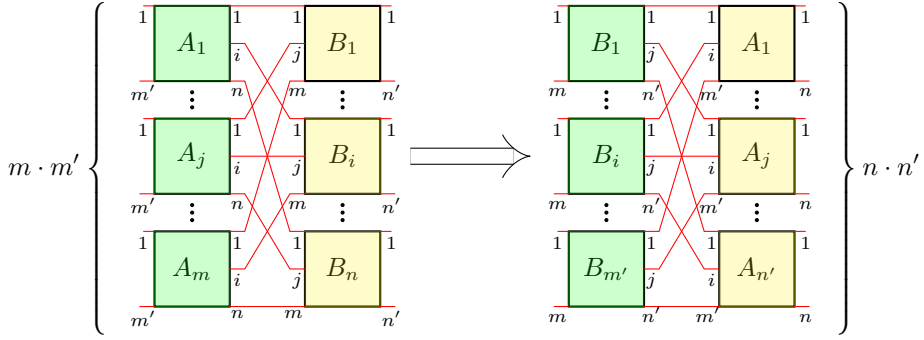
				\begin{equation*}
								m\cdot m^{\prime} \left\{ 
												\begin{array}[H]{lll}
																\begin{tikzpicture}[baseline]
																				\coordinate (a1) at (5mm,35mm);
																				\coordinate (aj) at ($(a1)-(0,15mm)$);
																				\coordinate (am) at ($(a1)-(0,30mm)$);
																				\node at (a1) {$A_{1}$};
																				\node at (aj) {$A_{j}$};
																				\node at (am) {$A_{m}$};
																				\filldraw ($(a1)-(0mm,6.5mm)$) circle  [radius = 0.2mm];
																				\filldraw ($(a1)-(0mm,7.5mm)$) circle  [radius = 0.2mm];
																				\filldraw ($(a1)-(0mm,8.5mm)$) circle  [radius = 0.2mm];
																				\draw[color=black,thick=3mm] (0,30mm)--(0,40mm)--(10mm,40mm)--(10mm,30mm)--cycle ;
																				\filldraw[color=green,opacity=0.2] (0,30mm)--(0,40mm)--(10mm,40mm)--(10mm,30mm)--cycle ;
																				\draw[color=black,thick=3mm] (0,15mm)--(0,25mm)--(10mm,25mm)--(10mm,15mm)--cycle ;
																				\filldraw[color=green,opacity=0.2] (0,15mm)--(0,25mm)--(10mm,25mm)--(10mm,15mm)--cycle ;
																				\draw[color=black,thick=3mm] (0,0mm)--(0,10mm)--(10mm,10mm)--(10mm,0mm)--cycle ;
																				\filldraw[color=green ,opacity=0.2] (0,0mm)--(0,10mm)--(10mm,10mm)--(10mm,0mm)--cycle ;

																				\filldraw ($(a1)-(0mm,21.5mm)$) circle  [radius = 0.2mm];
																				\filldraw ($(a1)-(0mm,22.5mm)$) circle  [radius = 0.2mm];
																				\filldraw ($(a1)-(0mm,23.5mm)$) circle  [radius = 0.2mm];

																				\coordinate (b1) at (25mm,35mm);
																				\coordinate (bi) at ($(b1)-(0,15mm)$);
																				\coordinate (bn) at ($(b1)-(0,30mm)$);
																				\node at (b1) {$B_{1}$};
																				\node at (bi) {$B_{i}$};
																				\node at (bn) {$B_{n}$};
																				\filldraw ($(b1)-(0mm,6.5mm)$) circle  [radius = 0.2mm];
																				\filldraw ($(b1)-(0mm,7.5mm)$) circle  [radius = 0.2mm];
																				\filldraw ($(b1)-(0mm,8.5mm)$) circle  [radius = 0.2mm];
																				\filldraw[color=yellow,opacity=0.2] ($(b1)+(-5mm,-5mm)$)--($(b1)+(5mm,-5mm)$)--($(b1)+(5mm,5mm)$)--($(b1)+(-5mm,5mm)$)--cycle;
																				\draw[color=black,thick=3mm] ($(b1)+(-5mm,-5mm)$)--($(b1)+(5mm,-5mm)$)--($(b1)+(5mm,5mm)$)--($(b1)+(-5mm,5mm)$)--cycle;
																				\draw[color=black,thick=3mm] ($(b1)+(-5mm,-10mm)$)--($(b1)+(5mm,-10mm)$)--($(b1)+(5mm,-20mm)$)--($(b1)+(-5mm,-20mm)$)--cycle;
																				\filldraw[color=yellow,opacity=0.2] ($(b1)+(-5mm,-10mm)$)--($(b1)+(5mm,-10mm)$)--($(b1)+(5mm,-20mm)$)--($(b1)+(-5mm,-20mm)$)--cycle;
																				\draw[color=black,thick=3mm] ($(b1)+(-5mm,-25mm)$)--($(b1)+(5mm,-25mm)$)--($(b1)+(5mm,-35mm)$)--($(b1)+(-5mm,-35mm)$)--cycle;
																				\filldraw[color=yellow,opacity=0.2] ($(b1)+(-5mm,-25mm)$)--($(b1)+(5mm,-25mm)$)--($(b1)+(5mm,-35mm)$)--($(b1)+(-5mm,-35mm)$)--cycle;
																				\filldraw ($(b1)-(0mm,21.5mm)$) circle  [radius = 0.2mm];
																				\filldraw ($(b1)-(0mm,22.5mm)$) circle  [radius = 0.2mm];
																				\filldraw ($(b1)-(0mm,23.5mm)$) circle  [radius = 0.2mm];
																				\draw[color=red] (10mm,0mm)--(20mm,0mm);
																				\node at (11mm,-1mm) {$\scriptstyle n$};
																				\node at (18.5mm,-1mm) {$\scriptstyle m$};
																				\draw[color=red] (10mm,5mm)--(12mm,5mm)--(18mm,15mm)--(20mm,15mm);
																				\node at (11mm,3.5mm) {$\scriptstyle i$};
																				\node at (19mm,13.5mm) {$\scriptstyle m$};
																				\draw[color=red] (10mm,10mm)--(12mm,10mm)--(18mm,30mm)--(20mm,30mm);
																				\node at (11mm,8.5mm) {$\scriptstyle 1$};

																				\draw[color=red] (10mm,15mm)--(12mm,15mm)--(18mm,5mm)--(20mm,5mm);
																				\node at (11mm,13.5mm) {$\scriptstyle n$};
																				\node at (19mm,3.5mm) {$\scriptstyle j$};
																				\node at (11mm,18.5mm) {$\scriptstyle i$};
																				\node at (19mm,18.5mm) {$\scriptstyle j$};
																				\node at (11mm,23.5mm) {$\scriptstyle 1$};
																				\node at (19mm,23.5mm) {$\scriptstyle 1$};
																				\draw[color=red] (10mm,20mm)--(20mm,20mm);
																				\draw[color=red] (10mm,25mm)--(12mm,25mm)--(18mm,35mm)--(20mm,35mm);

																				\draw[color=red] (10mm,30mm)--(12mm,30mm)--(18mm,10mm)--(20mm,10mm);
																				\node at (19mm,28.5mm) {$\scriptstyle m$};
																				\node at (11mm,28.5mm) {$\scriptstyle n$};
																				\node at (19mm,33.5mm) {$\scriptstyle j$};
																				\node at (11mm,33.5mm) {$\scriptstyle i$};
																				\node at (19mm,38.5mm) {$\scriptstyle 1$};
																				\node at (11mm,38.5mm) {$\scriptstyle 1$};
																				\node at (19mm,8.5mm) {$\scriptstyle 1$};
																				\draw[color=red] (10mm,35mm)--(12mm,35mm)--(18mm,25mm)--(20mm,25mm);
																				\draw[color=red] (10mm,40mm)--(20mm,40mm);
																				\node at (-1.5mm,-1.5mm) {$\scriptstyle m^{\prime}$};
																				\draw[color=red] (0mm,0mm)--(-2mm,0mm);
																				\node at (-1.5mm,13.5mm) {$\scriptstyle m^{\prime}$};
																				\node at (-1.5mm,28.5mm) {$\scriptstyle m^{\prime}$};
																				\draw[color=red] (0mm,15mm)--(-2mm,15mm);
																				\draw[color=red] (0mm,30mm)--(-2mm,30mm);
																				\foreach \i in {10mm,25mm,40mm} {
																								\node at (-1.5mm,\i-1.5mm) {$\scriptstyle 1$};
																								\draw[color=red] (-1.5mm,\i)--(0mm,\i);
																								\node at (31.5mm,\i-1.5mm) {$\scriptstyle 1$};
																								\draw[color=red] (31.5mm,\i)--(30mm,\i);
																				}

																				\node at (31.5mm,-1.5mm) {$\scriptstyle n^{\prime}$};
																				\draw[color=red] (30mm,0mm)--(32mm,0mm);
																				\node at (31.5mm,13.5mm) {$\scriptstyle n^{\prime}$};
																				\node at (31.5mm,28.5mm) {$\scriptstyle n^{\prime}$};
																				\draw[color=red] (30mm,15mm)--(32mm,15mm);
																				\draw[color=red] (30mm,30mm)--(32mm,30mm);
																				\draw[double distance=2mm,-{Implies[]}] (34mm,20mm)--(50mm,20mm);
																\end{tikzpicture}  
																\begin{tikzpicture}[baseline]

																				\coordinate (a1) at (5mm,35mm);
																				\coordinate (aj) at ($(a1)-(0,15mm)$);
																				\coordinate (am) at ($(a1)-(0,30mm)$);
																				\node at (a1) {$B_{1}$};
																				\node at (aj) {$B_{i}$};
																				\node at (am) {$B_{m^{\prime}}$};
																				\filldraw ($(a1)-(0mm,6.5mm)$) circle  [radius = 0.2mm];
																				\filldraw ($(a1)-(0mm,7.5mm)$) circle  [radius = 0.2mm];
																				\filldraw ($(a1)-(0mm,8.5mm)$) circle  [radius = 0.2mm];
																				\draw[color=black,thick=3mm] (0,30mm)--(0,40mm)--(10mm,40mm)--(10mm,30mm)--cycle ;
																				\filldraw[color=green,opacity=0.2] (0,30mm)--(0,40mm)--(10mm,40mm)--(10mm,30mm)--cycle ;
																				\draw[color=black,thick=3mm] (0,15mm)--(0,25mm)--(10mm,25mm)--(10mm,15mm)--cycle ;
																				\filldraw[color=green,opacity=0.2] (0,15mm)--(0,25mm)--(10mm,25mm)--(10mm,15mm)--cycle ;
																				\draw[color=black,thick=3mm] (0,0mm)--(0,10mm)--(10mm,10mm)--(10mm,0mm)--cycle ;
																				\filldraw[color=green ,opacity=0.2] (0,0mm)--(0,10mm)--(10mm,10mm)--(10mm,0mm)--cycle ;

																				\filldraw ($(a1)-(0mm,21.5mm)$) circle  [radius = 0.2mm];
																				\filldraw ($(a1)-(0mm,22.5mm)$) circle  [radius = 0.2mm];
																				\filldraw ($(a1)-(0mm,23.5mm)$) circle  [radius = 0.2mm];

																				\coordinate (b1) at (25mm,35mm);
																				\coordinate (bi) at ($(b1)-(0,15mm)$);
																				\coordinate (bn) at ($(b1)-(0,30mm)$);
																				\node at (b1) {$A_{1}$};
																				\node at (bi) {$A_{j}$};
																				\node at (bn) {$A_{n^{\prime}}$};
																				\filldraw ($(b1)-(0mm,6.5mm)$) circle  [radius = 0.2mm];
																				\filldraw ($(b1)-(0mm,7.5mm)$) circle  [radius = 0.2mm];
																				\filldraw ($(b1)-(0mm,8.5mm)$) circle  [radius = 0.2mm];
																				\filldraw[color=yellow,opacity=0.2] ($(b1)+(-5mm,-5mm)$)--($(b1)+(5mm,-5mm)$)--($(b1)+(5mm,5mm)$)--($(b1)+(-5mm,5mm)$)--cycle;
																				\draw[color=black,thick=3mm] ($(b1)+(-5mm,-5mm)$)--($(b1)+(5mm,-5mm)$)--($(b1)+(5mm,5mm)$)--($(b1)+(-5mm,5mm)$)--cycle;
																				\draw[color=black,thick=3mm] ($(b1)+(-5mm,-10mm)$)--($(b1)+(5mm,-10mm)$)--($(b1)+(5mm,-20mm)$)--($(b1)+(-5mm,-20mm)$)--cycle;
																				\filldraw[color=yellow,opacity=0.2] ($(b1)+(-5mm,-10mm)$)--($(b1)+(5mm,-10mm)$)--($(b1)+(5mm,-20mm)$)--($(b1)+(-5mm,-20mm)$)--cycle;
																				\draw[color=black,thick=3mm] ($(b1)+(-5mm,-25mm)$)--($(b1)+(5mm,-25mm)$)--($(b1)+(5mm,-35mm)$)--($(b1)+(-5mm,-35mm)$)--cycle;
																				\filldraw[color=yellow,opacity=0.2] ($(b1)+(-5mm,-25mm)$)--($(b1)+(5mm,-25mm)$)--($(b1)+(5mm,-35mm)$)--($(b1)+(-5mm,-35mm)$)--cycle;
																				\filldraw ($(b1)-(0mm,21.5mm)$) circle  [radius = 0.2mm];
																				\filldraw ($(b1)-(0mm,22.5mm)$) circle  [radius = 0.2mm];
																				\filldraw ($(b1)-(0mm,23.5mm)$) circle  [radius = 0.2mm];
																				\draw[color=red] (10mm,0mm)--(20mm,0mm);
																				\node at (11mm,-1mm) {$\scriptstyle n^{\prime}$};
																				\node at (18.5mm,-1mm) {$\scriptstyle m^{\prime}$};
																				\draw[color=red] (10mm,5mm)--(12mm,5mm)--(18mm,15mm)--(20mm,15mm);
																				\node at (11mm,3.5mm) {$\scriptstyle j$};
																				\node at (19mm,13.5mm) {$\scriptstyle m^{\prime}$};
																				\draw[color=red] (10mm,10mm)--(12mm,10mm)--(18mm,30mm)--(20mm,30mm);
																				\node at (11mm,8.5mm) {$\scriptstyle 1$};

																				\draw[color=red] (10mm,15mm)--(12mm,15mm)--(18mm,5mm)--(20mm,5mm);
																				\node at (11mm,13.5mm) {$\scriptstyle n^{\prime}$};
																				\node at (19mm,3.5mm) {$\scriptstyle i$};
																				\node at (11mm,18.5mm) {$\scriptstyle j$};
																				\node at (19mm,18.5mm) {$\scriptstyle i$};
																				\node at (11mm,23.5mm) {$\scriptstyle 1$};
																				\node at (19mm,23.5mm) {$\scriptstyle 1$};
																				\draw[color=red] (10mm,20mm)--(20mm,20mm);
																				\draw[color=red] (10mm,25mm)--(12mm,25mm)--(18mm,35mm)--(20mm,35mm);

																				\draw[color=red] (10mm,30mm)--(12mm,30mm)--(18mm,10mm)--(20mm,10mm);
																				\node at (19mm,28.5mm) {$\scriptstyle m^\prime$};
																				\node at (11mm,28.5mm) {$\scriptstyle n^\prime$};
																				\node at (19mm,33.5mm) {$\scriptstyle i$};
																				\node at (11mm,33.5mm) {$\scriptstyle j$};
																				\node at (19mm,38.5mm) {$\scriptstyle 1$};
																				\node at (11mm,38.5mm) {$\scriptstyle 1$};
																				\node at (19mm,8.5mm) {$\scriptstyle 1$};
																				\draw[color=red] (10mm,35mm)--(12mm,35mm)--(18mm,25mm)--(20mm,25mm);
																				\draw[color=red] (10mm,40mm)--(20mm,40mm);
																				\node at (-1.5mm,-1.5mm) {$\scriptstyle m$};
																				\draw[color=red] (0mm,0mm)--(-2mm,0mm);
																				\node at (-1.5mm,13.5mm) {$\scriptstyle m$};
																				\node at (-1.5mm,28.5mm) {$\scriptstyle m$};
																				\draw[color=red] (0mm,15mm)--(-2mm,15mm);
																				\draw[color=red] (0mm,30mm)--(-2mm,30mm);
																				\foreach \i in {10mm,25mm,40mm} {
																								\node at (-1.5mm,\i-1.5mm) {$\scriptstyle 1$};
																								\draw[color=red] (-1.5mm,\i)--(0mm,\i);
																								\node at (31.5mm,\i-1.5mm) {$\scriptstyle 1$};
																								\draw[color=red] (31.5mm,\i)--(30mm,\i);
																				}

																				\node at (31.5mm,-1.5mm) {$\scriptstyle n$};
																				\draw[color=red] (30mm,0mm)--(32mm,0mm);
																				\node at (31.5mm,13.5mm) {$\scriptstyle n$};
																				\node at (31.5mm,28.5mm) {$\scriptstyle n$};
																				\draw[color=red] (30mm,15mm)--(32mm,15mm);
																				\draw[color=red] (30mm,30mm)--(32mm,30mm);
																\end{tikzpicture}
												\end{array}
								\right\} n\cdot n^{\prime}	
				\end{equation*}
				\label{remark:13.5}
\end{remark}
\section{The image of $N:\text{Bio}\,\hookrightarrow \, \text{bSet}$}
For a face map in Bit, $\partial_{e}: G^{e}\hookrightarrow G$, we let 
\begin{equation}
				\partial_{e}\mathscr{Y}_{G}=\text{Image}\left\{
								\partial_{e*}:\mathscr{Y}_{G^{e}}\to \mathscr{Y}_{G}
				\right\}\subseteq \mathscr{Y}_{G}
				\label{eq:14.1}
\end{equation}
so
\begin{equation*}
				\left( \partial_{e}\mathscr{Y}_{G} \right)_{K}=\left\{ \partial_{e}\circ \varphi,
				\varphi\in\text{Bit}\left( K, G^{e} \right) \right\}
\end{equation*}
We write the boundary, 
\begin{equation}
				\partial \mathscr{Y}_{G}:= \bigcup_{\partial_{e}\in
				\Phi(G)}\partial_{e}\mathscr{Y}_{G} \subseteq
				\mathscr{Y}_{G}
				\label{eq:14.2}
\end{equation}
and the $\partial_{e}$-horn
\begin{equation}
				\Lambda^{\partial_{e}}\mathscr{Y}_{G}:= \bigcup_{\partial^{\prime}\in
				\Phi(G)\setminus\left\{ \partial_{e} \right\}}
				\partial^{\prime}\mathscr{Y}_{G} \subseteq \partial \mathscr{Y}_{G}. 
				\label{eq:14.3}
\end{equation}
and spine 
\begin{equation}
				\text{Sp}\, \mathscr{Y}_{G}:= \bigcup_{v}\mathscr{Y}_{C(v)},
				\label{eq:14.4}
\end{equation}
the union over all $v\in C^{\pm}(G)$, $C(v)$ the corolla at $v$.
\begin{definition}
				$X\in \text{bSet}$ is a (strict) inner Kan bital set if the diagram has a
				(unique) filling
				\begin{equation}
								\begin{tikzpicture}[scale=1.2]
												\node at (0mm,38mm) {$f:\Lambda^{\partial}\mathscr{Y}_{G}$};
												\draw[->] (8mm,37mm)--(25mm,37mm);
												\draw[{Hooks[left,length=5,width=6]}->] (4mm,34mm)--(4mm,24mm);
												\node at (28mm,37mm) {$X$};
												\draw[->,dashed] (6mm,23mm)--(25mm,35mm);
												\node at (3mm,21mm) {$\mathscr{Y}_{G}$};
								\end{tikzpicture}
								\label{eq:14.6}
				\end{equation}
				for every ($\pm$-composition or $\overset{\leftrightarrows}{o}$-action) inner face
				$\partial$. The inner Kan bital sets will be also called $\infty$-bios and we
				denote them by $\infty\text{Bio}\subseteq \text{bSet}$. \\ 
				For $\biop\in
				\text{Bio}$, $X=N(\biop)$, $X_{G}=\text{Bio}(\mathscr{S}_{G},\biop)$, a bitex $x\in
				X_{G}$ is given by compatible 
				\begin{equation}
								\begin{array}[H]{l}
								f_{x}^{1}: G^{1}\longrightarrow \biop^{0} \\\\
								f_{x}^{\pm}: \mathscr{S}_{G}^{\pm}\longrightarrow \biop^{\pm} \quad , \quad
								f_{x}^{1}(d_{\pm} B) = d_{\pm} f(B), 
								\end{array}
								\label{eq:14.7}
				\end{equation}
				and $f_{x}^{\pm}$ is completely determined (by compositions and actions) by its
				restrictions $f_{x}^{\pm}: C^{\pm}(G)\longrightarrow \biop^{\pm}$ to the corollas of
				$G$, $C^{\pm}(G)$ being the generators of $\mathscr{S}_{G}$. 
				\label{def:14.5}
\end{definition}
A map $f:\Lambda^{e}\mathscr{Y}_{G}\to X=N(\biop)$ again gives such compatible data \break $f^{1}:
G^{1}\to \biop^{0}$, $f^{\pm}: C^{\pm}(G)\to \biop^{\pm}$, and is determined uniquely by
this data, so
\begin{equation}
				\begin{array}[H]{l}
								\text{bSet}\left( \text{Sp} \mathscr{Y}_{G},X \right)=\text{bSet}\left(
								\Lambda^{e}\mathscr{Y}_{G},X \right)\equiv \text{bSet}\left(
												\mathscr{Y}_{G},X \right) \equiv X_{G} 
				\end{array}
				\label{14.8}
\end{equation}
for $X$ (the nerve of) a bio, and $X$ is strict inner Kan.
More precisely we have 
\begin{theorem}
				For $X\in \text{bSet}$, the following are equivalent
				\begin{itemize}
								\item[(1) ] $X$ is strict inner Kan
								\item[(2) ] $X\cong N(\biop)$ for $\biop\in\text{Bio}$
								\item[(3) ] The unit of adjunction $X\xrightarrow{\sim}{} N\tau(X)$ is an
												isomorphism
								\item[(4) ] For every $G\in \text{Bit}$, the restriction 
												\begin{equation*}
																X_{G}=\text{bSet}\left( \mathscr{Y}_{G},X \right) \longrightarrow
																\text{bSet}\left(  \text{Sp}\mathscr{Y}_{G},X \right)
												\end{equation*}
												is a bijection. \\
												And so $N$ gives  a full and faithful embedding of
												$\text{Bio}$ into $\text{bSet}$ with image the strict Kan
												bital sets. 
				\end{itemize}
				\label{thm:2}
\end{theorem}
\section{The homotopy functor $\tau=\text{ho}:\infty\cdot \text{Bio}\rightarrow \text{Bio}$}
For $X\in \infty\text{Bio}$ the associated Bio $\tau(X)$ has a homotopical description as
$\text{ho}(X)$ which we next briefly describe. The objects of $\text{ho}(X)$ are 
\begin{equation}
				\text{ho}(X)^{0} := X_{I} = X\left(
								\overset{0}{\bullet}\,\midarrow{1cm}{0.4cm}{t}{0.05cm} \, \overset{1}{\bullet}
\right).
\label{eq:15.1}
\end{equation}
Thus for a bitex $x\in X_{G}$, and an edge $e\in G^{1}$, the face $\delta_{e}: I
\rightarrow G$ gives $x|_{e}:= \delta_{e}^{*}(x)\in \text{ho}(X)^{0}$. 
\begin{equation*}
				\begin{array}[H]{l}
				\text{For}\;  x\in X_{C_{m}^{-}}=X\left( 
								\begin{tikzpicture}[baseline=-1mm,scale=0.8]
							  \def\mline#1#2#3{
												\draw[decorate,decoration={markings,mark=at position #3 with {\arrow[color=black]{<}}}] #1--#2;
												\draw  #1--#2;
							   };

								 \mline{(0mm,0.5mm)}{(10mm,0.5mm)}{0.5};
								 \node at (6mm,4mm) {$e_0$};
								 \filldraw[color=red]  (11mm,0.5mm) circle [radius=0.8mm];
								 \mline{(12mm,1.2mm)}{(20mm,5mm)}{0.5};
								 \node at (25mm,6mm) {$e_{1}$};
								 \node at (25mm,-6mm) {$e_{m}$};
								 \mline{(12mm,-0.2mm)}{(20mm,-5mm)}{0.5};
								 \filldraw  (19mm,-1.5mm) circle [radius=0.1mm];
								 \filldraw  (19mm,0mm) circle [radius=0.1mm];
								 \filldraw  (19mm,1.5mm) circle [radius=0.1mm];
 \end{tikzpicture} \right), \quad
				\text{for}\; y\in X_{C_{n}^{-}}=X\left( 
								\begin{tikzpicture}[baseline=-1mm,scale=0.8]
							  \def\mline#1#2#3{
												\draw[decorate,decoration={markings,mark=at position #3 with {\arrow[color=black]{<}}}] #1--#2;
												\draw  #1--#2;
							   };

								 \mline{(0mm,0.5mm)}{(10mm,0.5mm)}{0.5};
								 \node at (6mm,4.5mm) {$\ell_0$};
								 \filldraw[color=red]  (11mm,0.5mm) circle [radius=0.8mm];
								 \mline{(12mm,1.2mm)}{(20mm,5mm)}{0.5};
								 \node at (25mm,6mm) {$\ell_{1}$};
								 \node at (25mm,-6mm) {$\ell_{n}$};
								 \mline{(12mm,-0.2mm)}{(20mm,-5mm)}{0.5};
								 \filldraw  (19mm,-1.5mm) circle [radius=0.1mm];
								 \filldraw  (19mm,0mm) circle [radius=0.1mm];
								 \filldraw  (19mm,1.5mm) circle [radius=0.1mm];
				\end{tikzpicture}
				\right),  \\\\
				\text{for}\; z\in X_{C_{m+n-1}^{-}}=X\left( 
								\begin{tikzpicture}[baseline=-1mm,scale=1]
							  \def\mline#1#2#3{
												\draw[decorate,decoration={markings,mark=at position #3 with {\arrow[color=black]{<}}}] #1--#2;
												\draw  #1--#2;
							   };

								 \mline{(0mm,0.5mm)}{(10mm,0.5mm)}{0.5};
								 \node at (6mm,3mm) {$e_0$};
								 \filldraw[color=red]  (11mm,0.5mm) circle [radius=0.8mm];
								 \mline{(12mm,1mm)}{(20mm,7.5mm)}{0.5};
								 \node at (24mm,7.5mm) {$e_{i-1}$};
								 \node at (22mm,15mm) {$e_{1}$};
								 \node at (22.5mm,-7.5mm) {$\ell_{n}$};
								 \node at (22.5mm,0mm) {$\ell_{1}$};
								 \node at (22.5mm,-3mm) {$\vdots$};
								 \node at (22.5mm,12.5mm) {$\vdots$};
								 \mline{(12mm,-0.2mm)}{(20mm,-7.5mm)}{0.5};
								 \mline{(12mm,0mm)}{(20mm,0mm)}{0.5};
								 \mline{(12mm,2mm)}{(20mm,15mm)}{0.5};
								 \mline{(12mm,-1mm)}{(20mm,-15mm)}{0.5};
								 \node at (23mm,-16mm) {$e_{m}$};
								 \node at (22.5mm,-11mm) {$\vdots$};
				\end{tikzpicture}
\right), 
\end{array}
\end{equation*}

define $x\circ_{i} y \approx z$ if and only if there exists a homotopy $h:
x\circ_{i} y\approx z$, so that
\begin{equation}
				h\in X\left(
								\begin{tikzpicture}[baseline=-1mm,scale=1]
							  \def\mline#1#2#3{
												\draw[decorate,decoration={markings,mark=at position #3 with {\arrow[color=black]{<}}}] #1--#2;
												\draw  #1--#2;
							   };
												\mline{(0mm,0mm)}{(10mm,0mm)}{0.5};
												\node at (5mm,2.5mm) {$\scriptstyle e_{0}$};
												\node at (12.5mm,4.5mm) {$\scriptstyle e_{1}$};
												\node at (12.5mm,-4.5mm) {$\scriptstyle e_{m}$};
												\filldraw[color=red] (10mm,0mm) circle [radius=0.6mm];
												\mline{(10mm,0mm)}{(17mm,5mm)}{0.5};
												\mline{(10mm,0mm)}{(17mm,-5mm)}{0.5};
												\mline{(10mm,0mm)}{(25mm,0mm)}{0.5};
												\node at (20mm,2.5mm) {$\scriptstyle \ell_{0}=e_{i}$};

												\mline{(25mm,0mm)}{(32mm,5mm)}{0.5};
												\mline{(25mm,0mm)}{(32mm,-5mm)}{0.5};
												\filldraw[color=red] (25mm,0mm) circle [radius=0.6mm];

												\node at (28mm,4.5mm) {$\scriptstyle \ell_{1}$};
												\node at (28mm,-4.5mm) {$\scriptstyle \ell_{n}$};
												\filldraw[color=red] (30mm,2mm) circle [radius=0.2mm];
												\filldraw[color=red] (30mm,0mm) circle [radius=0.2mm];
												\filldraw[color=red] (30mm,-2mm) circle [radius=0.2mm];
								\end{tikzpicture}
				\right)
				\label{eq:15.2}
\end{equation}
and $h$ restrict to $x,y,z$ under the appropriate  face maps. \vspace{.1cm}\\
For $x,z\in X_{C_{m}^{-}}$, $x\circ_{i} 1_{e_{i}}\approx z$ is independent of $i$, and gives an
equivalence relation on $X_{C_{m}^{-}}$. \vspace{.1cm} \\ 
We define for $x_{i}\in \text{ho}(X)^{0}$ 
\begin{equation}
				\text{ho}(X)^{-}(x_0; x_{1}\cdots x_{n}):= \left\{ x\in X_{C_{n}^{-}},
				x|_{e_{i}} = x_{i} \right\}/\approx
				\label{eq:15.3}
\end{equation} 
and similarly $\text{ho}(X)^{+}(x_{1}\cdots x_{n}; x_{0})$.
For $x\in \text{ho}(X)^{-}(x_0;x_1\cdots x_n)$, \break 
$y\in \text{ho}(X)^{-}(x_i;y_1,\cdots y_m)$ we can
define $x\circ_{i}y:= z$ if $h:x\circ_{i}y\approx z$; the existence of $h$, and $z$,
follows from filling the $\ell_0=e_{i}$ horn in (\ref{eq:15.2}). This is well defined, independent of
the representatives chosen, and makes $\text{ho}(X)^{-}$ into an operad, and similarly for
$\text{ho}(X)^{+}$. For $m\ge n+j-1$ 
\begin{equation}
				\begin{array}[h]{l}
				 x\in X_{C_{m}^{-}}=X\left( 
				\begin{tikzpicture}[baseline=-1mm,scale=1]
							  \def\mline#1#2#3{
												\draw[decorate,decoration={markings,mark=at position #3 with {\arrow[color=black]{<}}}] #1--#2;
												\draw  #1--#2;
							   };

								 \mline{(0mm,0.5mm)}{(10mm,0.5mm)}{0.5};
								 \node at (6mm,3mm) {$e_0$};
								 \filldraw[color=red]  (11mm,0.5mm) circle [radius=0.8mm];
								 \mline{(12mm,1.2mm)}{(20mm,5mm)}{0.5};
								 \node at (16mm,5.5mm) {$e_{1}$};
								 \node at (16mm,-5.5mm) {$e_{m}$};
								 \mline{(12mm,-0.2mm)}{(20mm,-5mm)}{0.5};
								 \filldraw[color=red] (20mm,5mm) circle[radius=0.6mm];
								 \filldraw[color=red] (20mm,-5mm) circle[radius=0.6mm];
								 \filldraw[color=red] (0mm,0.5mm) circle[radius=0.6mm];
								 \filldraw  (19mm,-1.5mm) circle [radius=0.1mm];
								 \filldraw  (19mm,0mm) circle [radius=0.1mm];
								 \filldraw  (19mm,1.5mm) circle [radius=0.1mm];
 \end{tikzpicture} \right), \quad
				y\in X_{C_{n}^{+}}=X\left( 
				\begin{tikzpicture}[baseline=-1mm,scale=1]
							  \def\mline#1#2#3{
												\draw[decorate,decoration={markings,mark=at position #3 with {\arrow[color=black]{<}}}] #1--#2;
												\draw  #1--#2;
							   };

								 \mline{(12mm,0.5mm)}{(20mm,0.5mm)}{0.5};
								 \filldraw[color=red] (20mm,0.5mm) circle [radius=0.6mm];
								 \node at (16mm,4.5mm) {$\ell_0$};
								 \filldraw[color=red]  (11mm,0.5mm) circle [radius=0.8mm];
								 \mline{(0mm,5mm)}{(10mm,1mm)}{0.5};
								 \filldraw[color=red] (0mm,5mm) circle [radius=0.6mm];
								 \filldraw[color=red] (0mm,-5mm) circle [radius=0.6mm];
								 \node at (5mm,6mm) {$\ell_{1}$};
								 \node at (5mm,-6mm) {$\ell_{n}$};
								 \mline{(0mm,-5mm)}{(10mm,0mm)}{0.5};
								 \filldraw  (3mm,-1.5mm) circle [radius=0.1mm];
								 \filldraw  (3mm,0mm) circle [radius=0.1mm];
								 \filldraw  (3mm,1.5mm) circle [radius=0.1mm];
				\end{tikzpicture}
				\right),  \\\\
				 z\in X_{C_{m-n+1}^{-}}=X\left( 
								\begin{tikzpicture}[baseline=9mm,scale=1]
							  \def\mline#1#2#3{
												\draw[decorate,decoration={markings,mark=at position #3 with {\arrow[color=black]{<}}}] #1--#2;
												\draw  #1--#2;
							   };

								 \mline{(0mm,10mm)}{(10mm,10mm)}{0.5};
								 \mline{(10mm,10mm)}{(23mm,10mm)}{0.5};
								 \node at (19mm,11.5mm) {$\scriptstyle \ell_{0}$};
								 \filldraw[color=red]  (23mm,10mm) circle [radius=0.6mm];
								 \node at (6mm,12mm) {$e_0$};
								 \filldraw[color=red]  (10mm,10mm) circle [radius=0.6mm];
								 \filldraw[color=red]  (15mm,20mm) circle [radius=0.6mm];
								 \filldraw[color=red]  (15mm,0mm) circle [radius=0.6mm];
								 \mline{(10mm,10mm)}{(15mm,20mm)}{0.5};

								 \draw[decorate,decoration={markings,mark=at position 0.5 with
								 {\arrow[color=black]{<}}}] (10mm,10mm)--(20mm,15mm);
								 \draw  (10mm,10mm)--(20mm,15mm); 
								 \filldraw[color=red] (20mm,15mm) circle [radius=0.6mm];
								 \node at (13mm,20mm) {$\scriptstyle e_{1}$};
								 \mline{(10mm,10mm)}{(20mm,6.5mm)}{0.5};
								 \mline{(10mm,10mm)}{(15mm,0mm)}{0.5};
								 \node at (19mm,16.5mm) {$\scriptstyle e_{j-1}$};
								 \filldraw[color=red] {(20mm,6.5mm)} circle [radius=0.6mm];
								 \filldraw[color=red] {(0mm,10mm)} circle [radius=0.6mm];
								 \node at (20mm,8.2mm) {$\scriptstyle e_{j+n}$};
								 \node at (12mm,2mm) {$\scriptstyle e_{m}$};
								 \filldraw  (16mm,7mm) circle [radius=0.1mm];
								 \filldraw  (15mm,6mm) circle [radius=0.1mm];
								 \filldraw  (14mm,5mm) circle [radius=0.1mm];
								 \filldraw  (14mm,16mm) circle [radius=0.1mm];
								 \filldraw  (15mm,15mm) circle [radius=0.1mm];
								 \filldraw  (16mm,14mm) circle [radius=0.1mm];
				\end{tikzpicture}
\right), 
				\end{array}
				\label{15.4}
\end{equation}
we say $x \overleftarrow{\circ}_{\hspace{-1mm}j}\, y\approx z$, if there is a homotopy
$h:x \overleftarrow{\circ}_{\hspace{-1mm}j}\, y \approx z$, i.e. if there is an element
$h$, 
\begin{equation}
				 h\in X\left( 
								\begin{tikzpicture}[baseline=9mm,scale=1]
							  \def\mline#1#2#3{
												\draw[decorate,decoration={markings,mark=at position #3 with {\arrow[color=black]{<}}}] #1--#2;
												\draw  #1--#2;
							   };
												\mline{(0mm,10mm)}{(10mm,10mm)}{0.5};
								 \node at (6mm,13mm) {$\scriptstyle e_0$};
								 \filldraw[color=red]  (10mm,10mm) circle [radius=0.6mm];
								 \filldraw[color=red]  (25mm,10mm) circle [radius=0.6mm];
								 \filldraw[color=red]  (35mm,10mm) circle [radius=0.6mm];
								 \mline{(25mm,10mm)}{(35mm,10mm)}{0.5}; 
								 \node at (30mm,13mm) {$\scriptstyle \ell_0$};
								 \mline{(10mm,10mm)}{(15mm,20mm)}{0.5};
								 \mline{(10mm,10mm)}{(15mm,0mm)}{0.5};
								 \draw[decorate,decoration={markings,mark=at position 0.5 with {\arrow[color=black]{<}}}] (10mm,13mm)--(25mm,13mm);
								 \draw plot [smooth] coordinates {(10mm,10mm)(17.5mm,13mm)(25mm,10mm)};
								 \draw[decorate,decoration={markings,mark=at position 0.5 with
								 {\arrow[color=black]{<}}}] (10mm,7mm)--(25mm,7mm);
								 \draw plot [smooth] coordinates {(10mm,10mm)(17.5mm,7mm)(25mm,10mm)};
								 \node at (12mm,20mm) {$\scriptstyle e_{1}$};
								 \node at (12mm,0mm) {$\scriptstyle e_{m}$};
								 \node at (22mm,15mm) {$\scriptstyle e_{j}=\ell_1$};
								 \node at (25mm,5mm) {$\scriptstyle e_{j+n-1}=\ell_n$};
								 \filldraw (15mm,11mm) circle [radius=0.1mm];
								 \filldraw (14.9mm,10mm) circle [radius=0.1mm];
								 \filldraw (14.8mm,9mm) circle [radius=0.1mm];
								 \filldraw (14.8mm,7mm) circle [radius=0.1mm];
								 \filldraw (14.7mm,6mm) circle [radius=0.1mm];
								 \filldraw (14.6mm,5mm) circle [radius=0.1mm];
				\end{tikzpicture}
\right)
\label{eq:15.5}
\end{equation}
that restrict to $x,y,z$ under the obvious face maps. Again, given $x,y$ the existence of
$h$ and $z$, follows from filling the $\ell_k=e_{j+k-1}$ horn in (\ref{eq:15.5}), and putting
$x \overleftarrow{\circ}_{\hspace{-1mm}j}\, y:=z$ is well defined, independent of the
representatives $x,y,z$ chosen, and gives the action of $\text{ho}(X)^{+}$ on
$\text{ho}(X)^{-}$. Similarly we have the action of $\text{ho}(X)^{-}$ on
$\text{ho}(X)^{+}$, and it is easy to see these are linear so that we have a bio
$\text{ho}(X)$, clearly functorial in $X$, so we have a functor 
\begin{equation}
				\text{ho}: \infty\text{Bio}\longrightarrow \text{Bio}
				\label{eq:15.6}
\end{equation}
When $X=N\biop$ we get 
\begin{equation*}
				\begin{array}[h]{l}
								\text{ho}(N\biop)^{0}=(N\biop)_{I}=\text{Bio}(I,\biop)=\biop^{0} \;\; ,
								\;\;
								\text{and for $c_0,c_1,\cdots c_n\in\biop^{0}$} \\\\
								\text{ho}(N\biop)^{-}(c_0;c_1\cdots c_n)=\left\{ f\in \text{Bio}\left(
								C_{n}^{-},\biop \right), f(e_{i})=c_{i}
				\right\}/_{\approx}=\biop^{-}(c_0;c_1,\cdots c_n),
				\end{array}
\end{equation*}
the homotopy relation reduces to identity, and 
\begin{equation}
				\text{ho}(N\biop)\equiv \biop
				\label{15.7}
\end{equation}
For a general $X\in \infty\text{Bio}$, and $\varphi\in \text{bSet}(X,N\biop)$,
$\biop\in\text{Bio}$, we get the adjoint map of bios 
\begin{equation}
				\text{ho}(\varphi): \text{ho}(X)\longrightarrow \text{ho}(N\biop)\equiv \biop. 
				\label{eq:15.8}
\end{equation}
\section{Normal monomorphisms}
The category Bit, like the dendroidal category of trees $\Omega$, and unlike the
simplicial category $\bbDelta$, has automorphisms and so the \myemph{cofibrations} of
$b$Set $f:X\to Y$ are defined to be the monomorphisms that are \myemph{normal}: for each
$G\in \text{Bit}$ the group $\text{Aut}(G)$ acts freely on $Y_{G}\setminus f(X_{G})$. An
object $X\in \text{bSet}$ is \myemph{cofibrant} or \myemph{normal} if $\phi\to X$ is
normal, i.e. for each $G\in \text{Bit}$ the group $\text{Aut}(G)$ acts freely on
$X_{G}$; e.g. every representable $\mathscr{Y}_{G}$ is normal since
$\mathscr{S}_{G}=\sum\mathscr{S}_{\overline{G}}$ has free $S_{n}$-actions. For $X\in
\text{bSet}$, let $\sk_{n}(X)\subseteq X$ the sub-$\text{bSet}$ 
\begin{equation}
				\sk_{n}(X)_{G}:= \left\{ x\in X_G, x=\alpha^{\ast}(\overline{x}),
				\alpha:G\to H, \# H^{1}\le n \right\}	
				\label{eq:16.1}
\end{equation}
generated by biticies coming from bits $H$ with at most $n$ \myemph{edges}. We obtain the
skeltal filteration 
\begin{equation}
				\sk_{1}(X)\subseteq \sk_{2}(X)\subseteq \cdots \subseteq \sk_{n}(X)\subseteq
				\cdots \bigcup_{n}\sk_{n}(X)\equiv X
				\label{eq:16.2}
\end{equation}
Similarly, for a normal mono $f:X\hookrightarrow Y$ define 
\begin{equation}
				\begin{array}[H]{c}
				\sk_{n}(Y,X):= X\cup \sk_{n}(Y) \\\\
				X=\sk_{0}(Y,X)\subseteq \cdots \subseteq \sk_{n}(Y,X)\subseteq\cdots \subseteq
				\bigcup_{n}\sk_{n}(Y,X)\equiv Y
				\end{array}
				\label{eq:16.3}
\end{equation}
We get push out squares
\begin{align}
				\label{eq:16.4} \ \ \\
				\notag \begin{tikzpicture}
								\node at (5mm,35mm) {$\coprod\limits_{G,x}\partial \mathscr{Y}_{G}$};
												\draw[{Hooks[right,length=5,width=6]}->] (15mm,36mm)--(35mm,36mm);
												\node at (50mm,36mm) {$\sk_{n-1}(Y,X)$};
												\draw[{Hooks[left,length=5,width=6]}->] (45mm,30mm)--(45mm,15mm);
												\node at (48mm,10mm) {$\sk_{n}(Y,X)$};
												\draw[{Hooks[left,length=5,width=6]}->] (8mm,30mm)--(8mm,15mm);
												\node at (7mm,10mm) {$\coprod\limits_{G,x} \mathscr{Y}_{G}$};
												\draw[{Hooks[right,length=5,width=6]}->] (15mm,10mm)--(35mm,10mm);
												\draw (42mm,22mm)--(25mm,22mm)--(25mm,13mm);
				\end{tikzpicture}
\end{align}
the sum taken over isomorphism classes of $(G,x)$, $G\in\text{Bit}$, $x\in Y_{G}$
non-degenerate and not in the image of $X_{G}$. \\
Similarly, a normal $X\in \text{bSet}$ can be represented as a union of the
$\sk_{n}(X)$ that are given by push-outs 
\begin{align}
				\label{eq:16.5} \ \ \\
				\notag\begin{tikzpicture}
								\node at (5mm,35mm) {$\coprod\limits_{G,x}\partial \mathscr{Y}_{G}$};
												\draw[{Hooks[right,length=5,width=6]}->] (15mm,36mm)--(35mm,36mm);
												\node at (50mm,36mm) {$\sk_{n-1}(X)$};
												\draw[{Hooks[left,length=5,width=6]}->] (45mm,30mm)--(45mm,15mm);
												\node at (48mm,10mm) {$\sk_{n}(X)$};
												\draw[{Hooks[left,length=5,width=6]}->] (8mm,30mm)--(8mm,15mm);
												\node at (7mm,10mm) {$\coprod\limits_{G,x} \mathscr{Y}_{G}$};
												\draw[{Hooks[right,length=5,width=6]}->] (15mm,10mm)--(35mm,10mm);
												\draw (42mm,22mm)--(25mm,22mm)--(25mm,13mm);
				\end{tikzpicture}
\end{align}
\noindent the sum taken over isomorphism classes of $(G,x)$, $G\in \text{Bit}$, $x\in
X_{G}^{n.d.}\cdot$, $\# G^{1}=n$. It follows that the set of cofibraions of
$\text{bSet}$ 
\begin{equation}
				\mathcal{C}_{\text{bSet}}:=\left\{ f\in \text{bSet}(X,Y), f\; \text{normal mono} \right\}
				\label{eq:16.6}
\end{equation}
is closed under composition, push-outs, sums, retracts, and sequential limits.
Moreover, the push-out squares (\ref{eq:16.4}) show that $\mathcal{C}_{\text{bSet}}$ is the
smallest class that is closed under pushouts, compositions and sequential limits and
contains the generating cofibraions, 
\begin{equation}
				\mathcal{G}\mathcal{C}_{\text{bSet}}=\left\{ \partial \mathscr{Y}_{G} \;\hookedrightarrow{1}
				\mathscr{Y}_{G} \right\}_{G\in\text{Bit}}. 
				\label{eq:16.7}
\end{equation}
The small object argument gives a factorization of any map $f\in b\text{Set(X,Y)}$ as
$f=p\circ j$, with $j$ a normal mono and $p$ having RLP (-right lifting property) with
respect to all normal monos ($p$ a ``trivial fibration''). For $f:\phi\to X$ we get that
every $X\in \text{bSet}$ has a normalization 
\begin{equation}
				p:\tilde{X}\to X
				\label{eq:16.8}
\end{equation}
We have 
\begin{equation}
\mathcal{C}_{\text{bSet}}\equiv \text{\phantom{1}}^{\perp}(\mathcal{C}_{\text{bSet}}^{\perp})
\label{eq:16.9}
\end{equation}
where $\mathcal{C}^{\perp}$ (resp. $\text{\phantom{}}^{\perp}\mathcal{C}$) is the set of
arrows satisfying the right (resp. left) lifting property with respect to
$\mathcal{C}$. The inclusion $\subseteq$ is tautological; for the converse $\supseteq$ use
the retract argument: for
$f\in \text{\phantom{}}^{\perp}\left(\mathcal{C}^{\perp}_{\text{bSet}}\right)$,  factor
$f=p\circ j$ with $p\in \mathcal{C}_{\text{bSet}}^{\perp}$, so there is a lifting 
\begin{equation}
				\begin{tikzpicture}
												\node at (15mm,37mm) {$j$};
												\node at (3mm,27mm) {$f$};
												\node at (30mm,27mm) {$p$};
												\draw[->,dashed] (7mm,20mm)--(25mm,33mm);
												\node at (5mm,35mm) {$X$};
												\draw[->] (8mm,35mm)--(25mm,35mm);
												\node at (28mm,35mm) {$Y^{\prime}$};
												\draw[->] (27mm,33mm)--(27mm,20mm);
												\node at (27mm,18mm) {$Y$};
												\draw[-] (25mm,18.5mm)--(8mm,18.5mm);
												\draw[-] (25mm,17.5mm)--(8mm,17.5mm);
												\draw[->] (5mm,33mm)--(5mm,20mm);
												\node at (5mm,18mm) {$Y$};
				\end{tikzpicture}
				\label{eq:16.11}
\end{equation}
\noindent so $f$ is a retract of $j\in\mathcal{C}_{\text{bSet}}$, and so
$f\in\mathcal{C}_{\text{bSet}}$ too. \\
Given normal monos $f:A\to X$, $g:B\to Y$, the \myemph{pushout-product}
\begin{equation}
				\left\{ f\boxtimes g\; :\; A\otimes Y \coprod\limits_{A\otimes B}X\otimes B
				\longrightarrow X\otimes Y \right\} \in \mathcal{C}_{\text{bSet}}
				\label{eq:16.12}
\end{equation}
is again a normal mono; this follows from the generaing case of boundary inclusions
$f: \partial \mathscr{Y}_{F}\to \mathscr{Y}_{F}$, $g:\partial \mathscr{Y}_{G}\to
\mathscr{Y}_{G}$, $F,G\in\text{Bit}$, which
reduce to $\left( \partial \mathscr{Y}_{F}\otimes \mathscr{Y}_{G} \right)\cap \left(
				\mathscr{Y}_{F}\otimes \partial
\mathscr{Y}_{G} \right)=\partial \mathscr{Y}_{F}\otimes \partial \mathscr{Y}_{G}$ (c.f.
				\cite[4.26 (ii)]{heuts2018trees}). \\ 
Similarly denoting by $\mathscr{A}\subseteq \mathcal{C}_{\text{bSet}}$ the sturation of the inner
(composition or action) faces, the ``\myemph{anodyne extensions}'' (the ``via'' of
\cite{MR4222648}), for $f\in \mathscr{A}$, $g\in \mathcal{C}_{\text{bSet}}$, we have $f\boxtimes
g\in \mathscr{A}$ is again anodyne; this follows from the generating cases
$f:\Lambda^{e}\mathscr{Y}_{F}\hookrightarrow \mathscr{Y}_{F},\, e$  a composition (resp action)
$\text{edge}(s)$ of $F$, and $g:\partial \mathscr{Y}_{G}\hookrightarrow \mathscr{Y}_{G}$
(c.f. 
\cite[6.3]{MR4222648}). \\ 
It follows by adjunction that we have ``Joyal's exponential property``: \break
For $X (=\colim\limits_{\rightarrow}\sk_n(X))$ normal, $\bioq\in \infty\text{Bio}$ every map
$\Lambda^{e} \mathscr{Y}_{G}\otimes X\to \bioq$ extends to $\mathscr{Y}_{G}\otimes X\to \bioq$, so have
$\bioq^{X}\in\infty\text{Bio}$. Similarly, for a normal mono $f:X\hookrightarrow Y$, with
$X,Y$ normal, and any $\bioq\in\infty\text{Bio}$, the map $f^{\ast}:\bioq^{Y}\to
\bioq^{X}$ is an inner fibration between $\infty\text{Bio}$; it is actually a Kan
fibration between the underlying $\infty$categories (cf. \cite[7.2]{MR4222648})!
\section{Quillen Model structure. Ein M\"{a}rchen}
\emph{Notation:} We shall write $\biop^{\pm}(c_0;\overline{c})$  as a short hand for 
'' $\biop^{-}(c_0;c_1,\cdots , c_n)$ and $\biop^{+}(c_{1},\cdots , c_{n};c_{0})$``. \\
The category $\text{sBio}_{C}$, of simplicial bios with a fixes set of objects $C$
(and identify maps on $C$), has the \myemph{projective} model structure, where a map
\break $f:\biop\to \bioq$ is a weak equivalence (resp. fibration) iff for all
$c_0,\overline{c}=c_{1}\cdots c_{n}\in C$, $\biop^{\pm}(c_{0};\overline{c})\to
\bioq^{\pm}(c_0;\overline{c})$ are weak equivalences (resp. fibrations) in
$s\text{Set}$. \vspace{.1cm}\\
There is an equivalent \myemph{Reedy} model structure, with the same weak-equivalences,
and fibrations $f:\biop\to\bioq$ are maps such that for all $c_0;\overline{c}=c_1,\cdots ,
c_n\in C$ the maps
\begin{equation}
				\biop^{\pm}(c_0;\overline{c})\longrightarrow
				\widecheck{\biop}^{\pm}(c_0;\overline{c})\prod\limits_{\widecheck{\bioq}^{\pm}(c_0;\overline{c})}
				\bioq^{\pm}(c_0;\overline{c})
				\label{eq:17.1}
\end{equation}
are fibration in $s\text{Set}$, where 
\begin{equation}
				\widecheck{\biop}^{\pm}(c_{0}; \overline{c}):=\lim\limits_{I\subsetneqq \left\{ 1,\ldots
				, n \right\}}\biop^{\pm}(c_0;c_{I}) \quad, \quad c_{I}=c_{i_{1}},\cdots c_{i_{k}}\; \text{for}\; I=\left\{
				i_1<\cdots < i_{k} \right\}
				\label{eq:17.2}
\end{equation}
For $X\in s\text{Set}$ and $c_{0}; \overline{c}=c_{1},\cdots , c_{n}\in C$, we have the
''Corollas`` simplicial bios (in fact ordinary simplicial operads!) $C^{-}_{c_0;c_1\ldots
c_n}(X)$ and  $C^{+}_{c_0\ldots c_n;c_0}(X)$, to be denoted commonely by
$C^{\pm}_{c_0;c_1\ldots c_n}(X)$, satisfying adjunction 
\begin{equation}
				s\text{Bios}_{C}\left( C^{\pm}_{c_0;\overline{c}}(X),\biop \right)\equiv
				s\text{Set}\left( X,\biop^{\pm}(c_{0};\overline{c}) \right)
				\label{eq:17.3}
\end{equation}
Similarly we have the simplicial bios
\begin{equation}
				\widecheck{C}^{\pm}_{c_0;\overline{c}}(X) := \colimrightarrow\limits_{I\subsetneqq \left\{ 1,\ldots ,
												n
				\right\}} C^{\pm}_{c_0;c_I}(X)
				\label{eq:17.4}
\end{equation}
so that 
\begin{equation}
				\text{sBio}_{C}\left( \widecheck{C}^{\pm}_{c_0,\overline{c}},(X) , \biop \right) \equiv
				s\text{Set}\left( X,\widecheck{\biop}^{\pm}(c_0;\overline{c}) \right)
				\label{eq:17.5}
\end{equation}
The Reedy model structure is cofibrantly generated, with (resp. trivial) fibrations given
by 
\begin{equation}
				\widecheck{\mathcal{F}}_{\text{sBio}_{\scriptstyle c}}\equiv \left\{
				C_{c_{0},\overline{c}}^{\pm}\left( \Lambda_m^{e} \right)
\coprod\limits_{\widecheck{C}^{\pm}_{c_{0},\overline{c}}(\Lambda_{m}^{e})}
\widecheck{C}_{c_{0},\overline{c}}^{\pm}\left( \Delta(m) \right)\to
C^{\pm}_{c_{0},\overline{c}}\left( \Delta(m) \right)  \right\}^{\perp}
\label{eq:17.6}
\end{equation}
resp. 
\begin{equation}
				\mathcal{W}_{\text{sBio}_{\scriptstyle c}}\cap
				\widecheck{\mathcal{F}}_{\text{sBio}_{\scriptstyle c}}\equiv \left\{
								C_{c_{0},\overline{c}}^{\pm}\left( \partial\Delta(m) \right)
								\coprod\limits_{\widecheck{C}^{\pm}_{c_{0},\overline{c}}(\partial\Delta(m))}
								\widecheck{C}_{c_{0},\overline{c}}^{\pm}\left( \Delta(m) \right)\to
C^{\pm}_{c_{0},\overline{c}}\left( \Delta(m) \right)  \right\}^{\perp}
\label{eq:17.7}
\end{equation}
The main lemma here is that for $X\subseteq Y$, and for $c_{0}\,\overline{c}=c_{1}\cdots
c_{n}\in C$, the push-out 
\begin{equation}
				\begin{tikzpicture}
												\node at (5mm,35mm)
												{$C^{\pm}_{c_0;\overline{c}}(X)\coprod\limits_{\widecheck{C}^{\pm}_{c_{0};\overline{c}}(X)}\widecheck{C}^{\pm}_{c_{0};\overline{c}}(Y)$};
												\draw[->] (25mm,36mm)--(55mm,36mm);

												\node at (65mm,36mm) {$C^{\pm}_{c_{0};\overline{c}}(Y)$};
												\draw[->] (65mm,30mm)--(65mm,15mm);
												\node at (65mm,10mm) {$\bioq$};
												\draw[->] (8mm,30mm)--(8mm,15mm);
												\node at (7mm,10mm) {$\biop$};
												\draw[->] (10mm,10mm)--(60mm,10mm);
												\node at (30mm,13mm) {$f$};
												\draw (62mm,22mm)--(40mm,22mm)--(40mm,13mm);
				\end{tikzpicture}
				\label{eq:17.8}
\end{equation}
gives a weak-equivalence $\left\{ f:\biop\rightarrow \bioq \right\}\in
\mathcal{W}_{\text{sBio}_{C}}$. \vspace{.1cm}\\ 
The category of \myemph{open} simplicial bio
$\text{sBio}^{\phi}$ (so for $\biop\in \text{sBio}^{\phi}$, $\biop^{\pm}(c_{0};
\phi)=\phi$), has the cofibrantly generated projective model structure with fibrations 
\begin{equation}
				\mathcal{F}_{\text{sBio}^\phi}= 
				\left\{ 
								\begin{array}[H]{ll}
												f:\biop\to\bioq, \left\{
												\biop^{\pm}(c_0;\overline{c})\to\bioq^{\pm}
												 \left(f^{0}(c_{0});f^{0}(\overline{c}) \right)
												\right\}\in \mathcal{F}_{s\text{Set}}
				\\\\ 
				\text{all}\; c_0,\overline{c}=c_1\cdots c_{n}\in \biop^{0}  \;
				\text{and}\; \left\{ \pi_{0}i^{\ast}j^{\ast}\biop \longrightarrow
				\pi_{0}i^{\ast}j^{\ast}\bioq \right\}\in \mathcal{F}_{\text{Cat}}
								\end{array}
				\right\}
				\label{eq:17.9}
\end{equation}
(here $\mathcal{F}_{\text{Cat}}$ are the fibrations of the naive model structure on Cat), and
weak-equivalences
\begin{equation}
				\mathcal{W}_{\text{sBio}^{\phi}}=\left\{ 
								\begin{array}[H]{l}
												f:\biop\to \bioq, \left\{
																\biop^{\pm}(c_{0};\overline{c})\to\bioq^{\pm}\left(
																				f^{0}(c_{0});f^{0}(\overline{c})
												\right) \right\}\in\mathcal{W}_{s\text{Set}},\;  \\\\ 
												\text{all $c_0,\overline{c}\in \biop$}
																\text{and}\; \pi_0
																i^{\ast}j^{\ast}\biop\to\pi_0i^{\ast}j^{\ast}\bioq
																\;\text{essentially surjective}
								\end{array}
				\right\}
				\label{eq:17.10}
\end{equation}
We have the (resp trivial) generating cofibrations, so that 
\begin{equation}
				\mathcal{W}_{\text{sBio}^\phi}\cap \mathcal{F}_{\text{sBio}^{\phi}}= \left\{
								\left\{ \phi\to g^{\ast}I \right\}\cup \left\{ g^{\ast}
								C_{n}^{\pm}(\partial \Delta(m))\to g^{\ast}C_{n}^{\pm}(\Delta(m))
				\right\}
				\right\}^{\perp}
				\label{eq:17.11}
\end{equation}
\begin{equation}
				\mathcal{F}_{\text{sBio}^\phi} = \left\{ 
								\left\{ g^{\ast}I\rightrightarrows g^{\ast}N(0\leftrightarrow 1) \right\}\cup
								\left\{ g^{\ast}C_{n}^{\pm}(\Lambda_{m}^{e})\to 
								g^{\ast}C_{n}^{\pm}(\Delta(m)) \right\}
				\right\}^{\perp}
				\label{eq:17.12}
\end{equation}
Here $g^{\ast}:\text{sBio}\to \text{sBio}^{\phi}$ is the functor that simply forgets the
constants, \break  $\biop^{\pm}(c_0;\phi)=\left\{ 0_{c_{0}}^{\pm} \right\}\mapsto
				(g^{\ast}\biop)^{\pm} (c_{0};\phi)=\phi$, so that we have the adjunctions 
				\begin{equation}
								\begin{tikzpicture}
												\node at (35mm,35mm) {$\text{sBio}^{\phi}$};
												\draw[<-] (35mm,32mm)--(35mm,15mm);
												\draw[->] plot[smooth ] coordinates {(33mm,31mm)(27mm,23mm)(33mm,15mm) };
												\draw[->] plot[smooth ] coordinates {(37mm,31mm)(43mm,23mm)(37mm,15mm) };
												\node at (35mm,12mm) {$\text{sBio}$};
												\node at (24mm,23mm) {$g_{\text{\phantom{\,}}_!}$};
												\node at (38mm,23mm) {$g^{\ast}$};
												\node at (46mm,23mm) {$g_{\ast}$};
								\end{tikzpicture}
								\label{eq:17.13}
				\end{equation}
				\begin{equation*}
								(g_{_!}\biop)^{\pm}(c_0;\overline{c})=\colimrightarrow\limits_{\left\{ {  c_{1}^{\prime}\ldots
				c_{\ell}^{\prime} }\right\}\subseteq \biop^{0}}
				\biop^{\pm}(c_{0};\overline{c}c^{\prime}_{1}\cdots c_{\ell}^{\prime}) ,
				\end{equation*}
				\begin{equation*}
				\left( g_{\ast}\biop \right)^{\pm}(c_{0};\overline{c})=\prod_{\phi\not= I\subseteq
				\left\{ 1\ldots n \right\}} \biop^{\pm}(c_0;c_{I})
				\end{equation*}
				The category of closed simplicial bios $\text{sBio}$ gets a transfered
				\myemph{projective} model structure via $g^{\ast}$ from $\text{sBio}^{\phi}$, so
				that 
				\begin{equation}
								\mathcal{F}_{\text{sBio}}\equiv \left( g^{\ast} \right)^{-1}\left(
								\mathcal{F}_{\text{sBio}^\phi}\right)\quad , \quad
								\mathcal{W}_{\text{sBio}}=(g^{\ast})^{-1}\left(\mathcal{W}_{\text{sBio}^\phi}
								\right)
								\label{eq:17.14}
				\end{equation}
				The main lemma here is that the transfer condition are satisfied, and this follows
				since $g^{\ast}g_{!} g^{\ast}C_{n}(X)=\coprod\limits_{\phi\not= I \subseteq\left\{ 1\ldots n \right\}} g^{\ast}
				C_{I}(X)$,  and so the functor $g^{\ast}g_{!}$ takes the generating
				trivial cofibrations to trivial cofibrations. \\
				There is an equivalent \myemph{Reedy} model structure on $\text{sBio}$, with the
				same weak-equivalences, and with fibrations 
				\begin{equation}
								\begin{array}[H]{lll}
												\widecheck{\mathcal{F}}_{\text{sBio}} &\equiv&
												\mathcal{F}_{\text{sBio}}\cap \left\{ f:\biop\to\bioq, \left\{
												\biop\to (f^{0})^{\ast}\bioq \right\}\in
								\widecheck{\mathcal{F}}_{\text{sBio}_{\biop^{0}}} \right\} \\\\
								&\equiv& \left\{ \left\{ I\rightrightarrows N(0\leftrightarrow 1) \right\}\cup
								\left\{
C_{n}^{\pm}(\Lambda^{e}_{m})\coprod\limits_{\widecheck{C}^{\pm}_{n}(\Lambda_{m}^{e})}\widecheck{C}_{n}^{\pm}\left(\Delta(m)\right)\to
C_{n}^{\pm}\left( \Delta(m) \right)\right\}\right\}^{\perp}
								\end{array}
								\label{eq:17.15}
				\end{equation}
The main lemma being again that any push-out of the generating trivial
cofibrations is a weak-equivalence. \vspace{.1cm}\\
Recall that (the small object argument gives) a
normalization $\pi:\tilde{X}\rightarrow X$ for any $X\in \text{bSet}$, and moreover, any
map $f:X\rightarrow Y$ in $\text{bSet}$ is covered by a normal map between normal objects
$\tilde{f}:\tilde{X}\rightarrow \tilde{Y}$, (can take
$\tilde{X}=X\prod\limits_{Y}\tilde{Y}$). \\
The category of bital sets $\text{bSet}$ has a
model structure with the cofibrations $\mathcal{C}_{\text{bSet}}\equiv \left\{
				\text{normal monos}
\right\}$, the fibrant objects are $\infty\text{Bio}$, and the weak-equivalences 
\begin{equation}
				\mathcal{W}_{\text{bSet}}\equiv \left\{ f:X\rightarrow Y, \tilde{f}^{\ast}:
				\bioq^{\tilde{Y}}\rightarrow \bioq^{\tilde{X}}\in \mathcal{W}_{s\text{Set}} \;
\text{for all $\bioq\in \infty\text{Bio}$} \right\}
\label{eq:17.16}
\end{equation}
It is cofibrantly generated, so that we have for the (resp. trivial) fibration
\begin{equation}
				\begin{array}[H]{c}
				\mathcal{F}_{\text{bSet}}\equiv \left( \mathcal{C}_{\text{bSet}}\cap
				\mathcal{W}_{\text{bSet}} \right)^{\perp} \equiv \Big\{ \left\{ I\hookrightarrow
								N(0\leftrightarrow 1)\right\}\cup \left\{ \Lambda^{e}\mathscr{Y}_{G}\hookrightarrow
				\mathscr{Y}_{G} \right\}_{G\in\text{Bit}, e\in \Phi^{\text{in}}(G)} \Big\}^{\perp} \\\\
				\mathcal{W}_{\text{bSet}}\cap \mathcal{F}_{\text{bSet}}\equiv
				\mathcal{C}_{\text{bSet}}^{\perp} \equiv \left\{ \partial \mathscr{Y}_{G}\hookrightarrow
				\mathscr{Y}_{G} \right\}^{\perp}_{G\in \text{Bit}}
				\end{array}
				\label{eq:17.17}
\end{equation}
Moreover it seems that the pushout product property (for $f,g\in
\mathcal{C}_{\text{bSet}}$ with $f\in\mathcal{W}_{\text{bSet}}$, $f\boxtimes g\in
\mathcal{W}_{g\text{Set}}^{k}$), as well as the
associativity property (inclusions $(X\otimes Y)\otimes Z\hookrightarrow X\otimes
				Y\otimes Z \hookleftarrow X\otimes (Y\otimes Z)$ are \myemph{trivial}
cofibrations) hold, and so one gets an induced associative symmetric product $\otimes$ on
the homotopy category $\text{Ho}(\text{bSet})$. \vspace{.1cm}\\
A modification of the
Berger-Moerdijk-Boardman-Vogt resolution,, associating with a bit $G \in \text{Bit}$ the
resolution $W\mathscr{S}_G$ of the bio $\mathscr{S}_{G}$, gives a functor
$W:\text{Bit}\to \text{sBio}$, something like
\begin{equation}
				\left( W\mathscr{S}_{G} \right)^{\pm}\left( e_0;e_{i} \right):=
				\coprod_{B\in\mathscr{S}_{G}^{\pm},\, d_{\pm}B=e_{0},\, d_{\mp}B=\left\{
												e_{i}
				\right\}} \Delta[1]^{\Phi^{\text{in}}(B)}
				\label{eq:17.18}
\end{equation}
By Kan extension get the adjunction
\begin{equation}
				\begin{tikzpicture}[baseline=20mm]
												\node at (5mm,20mm) {$W_{!} : \text{bSet}$}; 
												\draw[->] (14mm,20.5mm)--(30mm,20.5mm);
												\draw[<-] (14mm,19mm)--(30mm,19mm);
												\node at (40mm,20mm) {$\text{sBio}: W^{\ast}$};
				\end{tikzpicture}
				\label{eq:17.19}
\end{equation}
\noindent and this should give a Quillen equivalence between the model structures on
$\text{bSet}$ and $\text{sBio}$.
\section{Self-duality}
Having the involution, $\text{Bio}\righttoleftarrow$, $\biop\mapsto\biop^{\text{op}}$, we
can speak of the self-dual bios,  $\text{Bio}^{t}$, with objects the bios with involution,
\begin{equation}
P\mapsto
P^{t}:\biop=(\biop^{-},\biop^{+})\xrightarrow{\sim}\biop^{\text{op}}=(\biop^{+},\biop^{-}),
\label{eq:18.1}
\end{equation}
and with maps of bios commuting with the involution (the involution is always the identity
on $\biop^{0}$). There is a forgetfull functor
$U:\text{Bio}^{t}\to\text{Bio}$. \vspace{.1cm}\\ 
For $\biop=(\biop^{-},\biop^{+})\in\text{Bio}^{t}$, we
can forget $\biop^{+}$ since it is isomorphic to $\biop^{-}$, and we only have to remember
the action $\circleftarrow$ of $\biop^{-}\equiv \biop^{+}$ on $\biop^{-}$. Thus the
objects of $\text{Bio}^{t}$ are (closed, symmetric, colored) operads $\biop$, together
with a \myemph{contraction} operation: For a word $\overline{c}=\overline{c}_{1}\ldots
\overline{c}_{n}$, $\overline{c}_{i}=(c_{i1},\cdots , c_{i m_{i} } )$, $b_0,b_1\cdots
b_n\in\biop^{0}$, we have $S_n$-covariant and $S_{m_{1}}\times\cdots \times
S_{m_{n}}$ - invariant functions
\begin{equation}
				\begin{array}[H]{c}
								// : \biop(b_0;\overline{c})\times \left[
								\biop(b_1;\overline{c}_{1})\times \cdots \times
				\biop(b_n;\overline{c}_{n}) \right]\longrightarrow \biop(b_0;b_1\cdots b_n) \\\\
				(P,P_i)\xmapsto{\qquad} P//(P_i):=P\circleftarrow (P_i^{t}).
				\end{array}
				\label{eq:18.2}
\end{equation}
This contraction is unital
\begin{equation}
				P//_{(\unit_{c_i})}=P
				\label{eq:18.3}
\end{equation}
and associative 
\begin{equation}
				\left( P// (P_{ij}) \right)//(P_{j}) = P//\left( P_j\circ (P_{ij}) \right)
				\label{eq:18.4}
\end{equation}
The linearity axioms read
\begin{equation}
				\begin{array}[H]{l}
				\left( P\circ (P_{i}) \right)// (P_{ij}) = P\circ \left( P_i// P_{ij} \right) \\\\
				P //  \left( P_i// (P_{ij}) \right)=  \left( P\circ (P_{ij}) \right)// (P_i)
				\end{array}
				\label{eq:18.5}
\end{equation}
Note that we get an involution for the underlying category $i^{\ast}\biop \cong
(i^{\ast}\biop)^{\text{op}}$, for $b,c\in\biop^{0}$, 
\begin{equation}
				\begin{array}[H]{c}
								\biop(b,c)\xrightarrow{\quad \sim\quad} \biop(c,b) \\\\
								P\xmapsto{\qquad} P^{t}:= \unit_{c}// P \\\\
								(P_1\circ P_2)^{t} = P_2^{t}\circ P_1^{t}\quad , \quad
								\unit_{c}^{t}= \unit_{c}\quad , \quad P^{tt}=P
				\end{array}
				\label{eq:18.6}
\end{equation}
Similarly we have the category of self-dual bital sets $\text{bSet}^{t}$, its objects are
bital sets $X\in \text{bSet}$ together with an involution $X\cong X^{\text{op}}$, i.e. for
each $G\in \text{Bit}$ we have 
\begin{equation}
x\xmapsto{\quad} x^{t}: \;\; X_{G}\xmapsto{\;\; \sim\;\;} X_{G^{\text{op}}}\;\;\;
				, \;\;\; x^{tt}=x,
				\label{eq:18.7}
\end{equation} 
and natural in $G$; the maps in $\text{bSet}^{t}$ are maps in $\text{bSet}$ that commute
with the involution; we have a forgetfull functor $U:\text{bSet}^{t}\to \text{bSet}$. \\
We get an induced ``self-dual'' picture 
\begin{align}
				\ \ \label{eq:18.8} \\
				\notag \begin{tikzpicture}
												 \node at (5mm,35mm) {$s\text{Set}^{t}$};
												 \node at (50mm,35mm) {$\text{bSet}^{t}$};
												 \draw[{Hooks[right,length=5,width=6]}->] (12mm,36mm)--(43mm,36mm);
												 \draw[<<-] (12mm,32mm)--(43mm,32mm);
												 \node at (28mm,38mm) {$i_{!}$};
												 \node at (30mm,34mm) {$i^{\ast}$};
												 \draw[->>] (3mm,30mm)--(3mm,10mm);
												 \draw[{Hooks[left,length=5,width=6]}->] (6mm,10mm) -- (6mm,30mm);
												 \node at (4mm,8mm) {$\text{Cat}^{t}$};
												 \draw [{Hooks[right,length=5,width=6]}->] (10mm,9mm)--(45mm,9mm);
												 \draw [<<-] (10mm,7mm)--(45mm,7mm);
												 \node at (50mm,8mm) {$\text{Bio}^{t}$};
												 \draw[->>] (48mm,30mm)--(48mm,10mm);
												 \draw[<-{Hooks[left,length=5,width=6]}] (51mm,30mm)--(51mm,10mm); 
												 \node at (1.5mm,20mm) {$\tau$};
												 \node at (8mm,22.5mm) {$N$};
												 \node at (25mm,11mm) {$i_{!}$};
												 \node at (27mm,5mm) {$i^{\ast}$};
												 \node at (53.5mm,22mm) {$N$};
												 \node at (46.5mm,20mm) {$\tau$};
				\end{tikzpicture}
\end{align}
with $\text{Cat}^{t}$ (resp. $s\text{Set}^{t}$) the similarly defined categories of
``categories with an involution'' (resp ``simplicial set with an involution'').
\section{Rigs}
\begin{definition}
				A \myemph{Rig} ($\equiv$ Ring without Negatives) is a set $R$ with two
				operations of addition and multiplication 
				\begin{equation}
								R\times R \rightrightarrows R \quad , \quad (x,y)\xmapsto{\quad}
								\begin{array}[H]{l}
												x+y\\
												x\cdot y
								\end{array},
								\label{def:19.2}
				\end{equation}
				both associative and unital 
				\begin{equation}
								\begin{array}[H]{lll}
												(x+y)+z=x+(y+z)\quad , \quad (x\cdot y)\cdot z = x\cdot (y\cdot z)
												\\\\
												x+0=x=0+x \quad , \quad x\cdot 1=x=1\cdot x
								\end{array}
								\label{def:19.3}
				\end{equation}
				with addition always commutative $x+y=y+x$, and distributive 
				\begin{equation}
								\begin{array}[H]{l}
								(x_1+x_2)\cdot y = (x_1\cdot y)+(x_2\cdot y)\quad , \quad
								x\cdot(y_1+y_2)=(x\cdot y_1)+(x\cdot y_2) \\\\
								x\cdot 0 = 0 = 0\cdot x
								\end{array}
								\label{def:19.4}
				\end{equation}
				\label{def:19.1}
\end{definition}
A rig with involution is a rig $R$ with an involution 
\begin{equation}
				\begin{array}[H]{l}
				x\mapsto x^{t}: R\righttoleftarrow \quad , \quad \text{satisfying}\; x^{tt}=x \;
				\text{and}\; 0^{t}=0\quad , \quad 1^{t}=1 \\\\
				(x+y)^{t}= x^{t}+y^{t}\quad ,\quad (x\cdot y)^{t}=y^{t}\cdot x^{t}
				\end{array}
				\label{def:19.5}
\end{equation}
A rig is commutative $R\in C\text{Rig}$ if multiplication is commutative,
\begin{equation*}
x\cdot y=y\cdot x 
\end{equation*}
A map of rigs $\varphi: R\to R^{\prime}$ is a set map preserving the operations and
constants 
\begin{equation}
				\varphi(0)=0 \quad , \quad \varphi(1)=1\quad ,\quad  
				\varphi(x+y)=\varphi(x)+\varphi(y)\quad , \quad \varphi(x\cdot y)=\varphi(x)\cdot
				\varphi(y)
				\label{def:19.6}
\end{equation}
and in the self-dual case $\text{Rig}^{t}$, $\varphi$ should also preserve the involution, 
\begin{equation*}
				\varphi(x^{t})=\varphi(x)^{t}. 
\end{equation*}
Thus we have categories and functors
\begin{align}
				\begin{array}[H]{l}
				 \begin{tikzpicture}
								 \node at (10mm,30mm) {$\text{Rig}^{t}$};
								 \node at (40mm,30mm) {$\text{Rig}$};
								 \node at (40mm,15mm) {$C\text{Rig}$};
								 \node at (40mm,22.5mm) {\rotatebox{90}{\Huge $\subseteq$}};
								 \draw[->] (15mm,30mm)--(35mm,30mm);
								 \node at (25mm,32mm) {$U$};
								 \draw[<-] (15mm,27mm)--(35mm,17mm);
				\end{tikzpicture}
				\end{array}
				\label{eq:19.7}
\end{align}
(for a commutative rig the identity is an involution). \vspace{.1cm}\\
There is a similar diagram  with
Ring instead of Rig, the inclusion $\text{Ring}\hookrightarrow \text{Rig}$ has the left
adjoint functor $K=$Grothendieck functor localizing addition. \vspace{.1cm}\\ 
E.g. we have the commutative
rigs
\begin{equation}
				\mathscr{B}=\left\{ 0,1 \right\}\hookedrightarrow{1} \mathscr{I}=[0,1]
				\hookedrightarrow{1} \mathscr{R}=[0,\infty]
				\label{def:19.8}
\end{equation}
with the usual multiplication $x\cdot y$, and with the ``tropical'' addition
\begin{equation}
				x+y:\overset{def}{=}\max\left\{ x,y \right\}.
				\label{def:19.9}
\end{equation}
\vspace{.1cm}
For $A\in \text{Rig}$ we have the \myemph{monochromatic} operad $\biop_{A}$ (i.e.
$\biop^{0}=\left\{ \ast \right\}$, and we write $\biop_{A}^{-}(n)\equiv
\biop_{A}(\ast;\underbrace{\ast\cdots\ast}_{n})$)
with $\biop_{A}^{-}(n)\equiv A^{n}$, and with compositions 
\begin{equation}
				o_{i}: A^{n}\times A^{m} \xrightarrow{\qquad} A^{n+m-1}
				\label{def:19.10}
\end{equation}
\begin{equation*}
				(a_1,\cdots , a_{i}, \cdots a_{n})\circ_{i}(b_1,\cdots , b_{m}):= (a_1,\ldots,
				a_{i-1},a_{i}\cdot b_{1},a_{i}\cdot b_2 \cdots a_i\cdot b_{m},a_{i+1}\cdots
a_{m})
\end{equation*}
We have similarly the operad $\biop_{A}^{+}\equiv \biop_{A^{\text{op}}}$ of colume vectors 
\begin{equation}
				\begin{array}[H]{c}
				\biop_{A}^{+}(\underbrace{\ast\cdots \ast}_{n}; \ast)\equiv
				\biop_{A}^{+}(n)\equiv A^{n}\quad , \quad \text{with composition} \\\\
				A^{n}\times A^{m}\xrightarrow{\quad} A^{m+n-1} \;\;
				\begin{pmatrix}
								a_1 \\ \vdots \\ a_n
				\end{pmatrix} \circ_j 
				\begin{pmatrix}
								b_1 \\ \vdots \\ b_{j}\\ \vdots \\ b_m
				\end{pmatrix} := 
				\begin{pmatrix}
								b_1 \\ \vdots \\ a_1\cdot b_{j} \\ \vdots \\ a_n\cdot b_{j} \\  \vdots \\
								b_{m}
				\end{pmatrix}
				\end{array}
				\label{def:19.11}
\end{equation}
Moreover, we have mutual actions, for $1\le k\le m\le n$, 
\begin{equation}
				\begin{array}[H]{c}
								\overleftarrow{\circ}_{k}: A^{n}\times A^{m-k+1} \xrightarrow{\qquad} A^{n-m+k} \\\\
\begin{pmatrix}
				a_1\cdots a_k\cdots a_m\cdots a_n
\end{pmatrix} \circleftarrow_{k} 
\begin{pmatrix}
				b_0\\ b_1 \\ \vdots \\ b_{m-k} 
\end{pmatrix} := 
\\
\begin{pmatrix}
				a_1 \cdots a_{k-1}, a_{k}\cdot b_0+\cdots + a_{k+j}\cdot b_j+\cdots +
				a_{m}\cdot b_{m-k}, a_{m+1}\cdots a_n
\end{pmatrix}
				\end{array}
				\label{def:19.12}
\end{equation}
and similarly for $\circrightarrow$. \\
This construction gives a full and faithful embedding 
\begin{equation}
				\text{Rig}\; \hookedrightarrow{1} \text{Bio}_{ \left\{ \ast \right\} } \quad , \quad
				A\xmapsto{\qquad} \biop_{A}.
				\label{def:19.13}
\end{equation}
Moreover, when the rig $A$ has an involution, the bio $\biop_{A}=\left(
\biop^{-}_{A},\biop^{+}_{A} \right)$ has an involution
\begin{equation}
				(a_1,\cdots , a_{n})^{t} := 
				\begin{pmatrix}
								a_1^{t}\\ \vdots \\ a_{n}^{t}
				\end{pmatrix}
				\label{eq:19.14}
\end{equation}
giving 
\begin{equation}
\text{Rig}^{t} \hookedrightarrow{2}\text{Bio}_{ \left\{ \ast \right\}}^{t}
\label{eq:19.15}
\end{equation}
The initial object of \text{Rig}, and of $\text{Rig}^{t}$, is $\N$ the rig of natural
numbers, \break therefore, the above functor gives $\text{Bios}$ under $\biop_{\N}$. Having the
vectors \break $(1,1)\in \biop_{\N}^{\pm}(2)\to  \biop^{\pm}(2)$   we can define addition on
$\biop(1)$ making it into a rig 
\begin{equation}
				a+b:= (1,1)\circleftarrow (a,b)\circleftarrow 
				\begin{pmatrix}
								1 \\ 1
				\end{pmatrix}.
				\label{19.16}
\end{equation}
We get adjunctions,
\begin{equation}
				\begin{tikzpicture}[baseline=35mm]
								 \node at (5mm,35mm) {$\biop_{A}$};
								 \node at (28mm,35mm) {$\displaystyle \biop_{\N}^{\text{\small $\backslash \text{Bio}^{t}_{ \left\{
								 \ast \right\}}$}}$};
								 \node at (62mm,35mm) {$\displaystyle \biop_{\N}^{\text{ \small $\backslash \text{Bio}_{ \{\ast \} }$}}$};
								 \node at (90mm,35mm) {$\biop$};
								 \node at (5mm,15mm) {$A$};
								 \node at (30mm,15mm) {$\text{Rig}^{t}$};
								 \node at (60mm,15mm) {$\text{Rig}$};
								 \node at (90mm,15mm) {$\biop(1)$};
								 \draw[->] (35mm,35mm)--(54mm,35mm);
								 \draw[->] (35mm,15mm)--(54mm,15mm);
								 \node at (45mm,37mm) {$U$};
								 \node at (45mm,17mm) {$U$};
								 \draw[{Hooks[right,length=5,width=6]}->] (28mm,18mm)--(28mm,31mm); 
								 \draw[<<-] (31mm,18mm)--(31mm,31mm); 
								 \draw[{Hooks[right,length=5,width=6]}->] (58mm,18mm)--(58mm,31mm); 
								 \draw[<<-] (61mm,18mm)--(61mm,31mm); 
								 \draw[|->] (5mm,18mm)--(5mm,32mm);
								 \draw[<-|] (90mm,18mm)--(90mm,32mm);
				\end{tikzpicture}
				\label{eq:19.17}
\end{equation}
\section{The $\ell_{p}$ bio}
Fix $p,q\in [1,\infty]$, with $1/p + 1/q =1$. We have the monochromatic operad
$\biop_{\ell_p}\subseteq \biop_{\R}$, with 
\begin{equation}
				\begin{array}[H]{l}
				\biop_{\ell_p}(n):= \left\{ (x_1,\cdots, x_{n})\in \biop_{\R}(n)\equiv \R^{n},
				x_1^{p}+\cdots + x_{n}^{p}\le 1 \right\} \\\\
				(x_1,\cdots ,x_n)\circ_i (y_1,\cdots , y_m):=(x_1,\cdots, x_{i-1},x_{i}\cdot
				y_1,\ldots , x_i\cdot y_m,x_{i+1},\cdots , x_n)
				\end{array}
				\label{eq:20.1}
\end{equation}
We write the elements of $\biop_{\ell_q}$ as column vector, and we have the mutual actions
induced from $\biop_{\R}$,
\begin{equation}
				(a_1\cdots a_{k}\cdots a_n)\circleftarrow_{k}
				\begin{pmatrix}
								b_0\\b_1\\\vdots\\b_{m-k}				
				\end{pmatrix}
				:= (a_1\cdots a_{k-1},a_k\cdot b_0+\cdots
								+a_{k+j}\cdot b_{j}+\cdots + a_{m}\cdot b_{m-k},a_{m+1},\cdots ,
				a_{n})
				\label{eq:20.2}
\end{equation}
and similarly for $\circrightarrow$. \\
The rig $\R$, being commutative has an (identity) involution, hence $\biop_{\R}$ has an
involution 
\begin{equation}
				(a_1\cdots a_n)^{t}:=
				\begin{pmatrix}
								a_1\\\vdots\\ a_n
				\end{pmatrix}
				\label{20.3}
\end{equation}
The sub-bio $\left( \biop_{\ell_{p}}, \biop_{\ell_{q}} \right)\subseteq \left(
				\biop_{\R}^{-},\biop_{\R}^{+}
\right)$ is stable under this involution iff $p=q=2$, and we are in the self-dual case
\begin{equation}
				\biop_{\ell_2}(n)\equiv \left\{ (x_1,\cdots , x_{n})\in \R^n, x_1^{2}+\cdots
				+x_n^{2}\le 1 \right\}
				\label{eq:20.4}
\end{equation} the $\ell_2$-unit ball. \vspace{.1cm}\\
\noindent The bio $\mathbb{Z}_{\mathbb{R}}:=(\biop_{\ell_2},\biop_{\ell_2})\subseteq \biop_{\R}$ is the ``Real integers'',
analogue at the ``Real prime'' of the $p$-adic integers $\Z_{p}\subseteq
\Q_{p}$, $p$ a prime. We similarly have the ``complex integers'' $\Z_{\C}\subseteq
\biop_{\C}$ given by the unit $\ell_{2}$ complex balls.
\section{Distributive bios}
Let $\biop=(\biop^{-},\biop^{+})$ be a monochromatic bio, $\biop^{0}=\left\{ \ast
\right\}$. \\ 
We say $\biop$ is \myemph{distributive} if the actions and compositions
\myemph{interchange}: \\ 
for $P\in\biop^{-}(n)$, $P^{\prime}\in \biop^{+}(m)$,
$P^{\prime\prime}\in \biop^{-}(\ell)$, 
\begin{equation}
				\left( P \overleftarrow{\circ}_{\hspace{-1mm}j}\, P^{\prime}
				\right){\circ}_{j} P^{\prime\prime}\equiv ( P \circ
								\underbrace{(P^{\prime\prime},\cdots, P^{\prime\prime})}_{m})
								\sigma_{m,\ell}\circleftarrow \underbrace{\left( P^{\prime},\cdots ,
								P^{\prime} \right)}_{\ell}
								\label{eq:21.1}
\end{equation}
and similarly,
\begin{equation}
				P^{\prime\prime}\circ_{j} \left( P^{\prime} \overrightarrow{\circ}_{\hspace{-1mm}j}\, P \right) \equiv
				(\underbrace{P^{\prime},\cdots , P^{\prime}}_{\ell}) \circrightarrow
				\sigma_{\ell,m} ( (\underbrace{P^{\prime\prime},\cdots ,
				P^{\prime\prime}}_{m})\circ P)
				\label{eq:21.2}
\end{equation}
for $P\in \biop^{+}(n)$, $P^{\prime}\in \biop^{-}(m)$, $P^{\prime\prime}\in
\biop^{+}(\ell)$. \\
Pictorially, 
\begin{equation}
				\begin{array}[H]{l}
				\begin{tikzpicture}[baseline]
								  \def\mline#1#2#3{
												\draw[decorate,decoration={markings,mark=at position #3 with {\arrow[color=black]{<}}}] #1--#2;
												\draw  #1--#2;
							   };
								 \mline{(0mm,20mm)}{(5mm,20mm)}{0.5};
								 \draw (8mm,20mm) circle [radius=3mm];
								 \node at (8mm,20mm) {$P$};
								 \mline{(9mm,23mm)}{(15mm,30mm)}{0.5};
								 \mline{(8.5mm,23mm)}{(12mm,30mm)}{0.5};
								 \mline{(11mm,20mm)}{(20mm,20mm)}{0.5};
								 \mline{(9mm,17mm)}{(15mm,10mm)}{0.5};
								 \draw[-] plot[smooth ] coordinates {(10.8mm,21mm)(15.4mm,23mm)(20mm,21mm) };
								 \draw[decorate,decoration={markings,mark=at position 0.5 with {\arrow[color=black] {<}}}] plot[smooth ] coordinates { (10.8mm,21mm)(15.4mm,23mm)(20mm,21mm)};
								 \draw[decorate,decoration={markings,mark=at position 0.5 with {\arrow[color=black] {<}}}] plot[smooth ] coordinates {(10.8mm,19mm)(15.4mm,17mm)(20mm,19mm) };
								 \draw plot[smooth ] coordinates {(10.8mm,19mm)(15.4mm,17mm)(20mm,19mm) };
								 \draw (23mm,20mm) circle [radius=3mm];
								 \node at (23mm,20mm) {$P^{\prime}$};
								 \draw
								 (40mm,17.5mm)--(40mm,22.5mm)--(34mm,22.5mm)--(34mm,17.5mm)--(40mm,17.5mm); 
								 \node at (37mm,20mm) {$P^{\prime\prime}$};
								 \mline{(26mm,20mm)}{(34mm,20mm)}{0.5};
								 \mline{(40mm,22mm)}{(45mm,22mm)}{0.5};
								 \mline{(40mm,18mm)}{(45mm,18mm)}{0.5};
								 \draw[color=red,thick=3mm] (-1mm,9mm)--(46mm,9mm)--(46mm,31mm)--(-1mm,31mm)--(-1mm,9mm);
								 \node at (50mm,20mm) {$\equiv$};
								 \mline{(55mm,20mm)}{(60mm,20mm)}{0.5};
								 \draw (63mm,20mm) circle [radius=3mm];
								 \node at (63mm,20mm) {$P$};
								 \mline{(64mm,22.8mm)}{(70mm,30mm)}{0.5};
								 \mline{(65mm,22mm)}{(73mm,28mm)}{0.5};
								 \mline{(66mm,21mm)}{(75mm,25mm)}{0.5};
								 \mline{(66mm,20mm)}{(75mm,20mm)}{0.5};
								 \def\stag#1#2{
												 \draw
												 ($#1+(2mm,2mm)$)--($#1+(2mm,-2mm)$)--($#1+(-2mm,-2mm)$)--($#1+(-2mm,2mm)$)--($#1+(2mm,2mm)$);
												 \node at #1 {$#2$};
								 }
								 \stag{(77mm,25mm)}{P^{\prime\prime}};
								 \stag{(77mm,20mm)}{P^{\prime\prime}};
								 \stag{(77mm,15mm)}{P^{\prime\prime}};
								 \mline{(65mm,17.8mm)}{(75mm,15mm)}{0.5};
								 \mline{(64mm,17.1mm)}{(73mm,10mm)}{0.5};
								 \draw (95mm,25mm) circle [radius=2.5mm];
								 \draw (95mm,15mm) circle [radius=2.5mm];
								 \node at (95mm,25mm) {$P^{\prime}$};
								 \node at (95mm,15mm) {$P^{\prime}$};
								 \mline{(79mm,26.5mm)}{(93mm,26.5mm)}{0.5};
								 \mline{(79mm,24mm)}{(93mm,16.5mm)}{0.07};
								 \mline{(79mm,21mm)}{(92.5mm,24.5mm)}{0.5};
								 \mline{(79mm,19mm)}{(92.5mm,15.5mm)}{0.5};
								 \mline{(79mm,13.5mm)}{(92.9mm,13.5mm)}{0.5};
								 \mline{(79mm,16mm)}{(92.9mm,23.5mm)}{0.18};
								 \mline{(97.5mm,25mm)}{(102mm,25mm)}{0.5};
								 \mline{(97.5mm,15mm)}{(102mm,15mm)}{0.5};
								 \draw[color=red,thick=3mm]
								 (54mm,9mm)--(104mm,9mm)--(104mm,31mm)--(54mm,31mm)--(54mm,9mm); 
				\end{tikzpicture}
				\end{array}
				\label{eq:21.3}
\end{equation}
\noindent We let $\text{CBio}\subseteq \text{Bio}$ denote the full subcategory of
distributive bios. \vspace{.1cm}\\
For a planer bit $G\in p\text{Bit}$, and $e\in G^{1}$ with
$d_{-} e =v_{-}\in C_{n1}(G)$, \break
$d_{+}e=v_{+}\in C_{1m}(G)$, we define $GZe$, the
\myemph{interchange of $G$ at $e$}, by applying to $G$ the \myemph{perculation move}:
moving $C_{v_{+}}$ to the left of $C_{v_{-}}$, pictorially
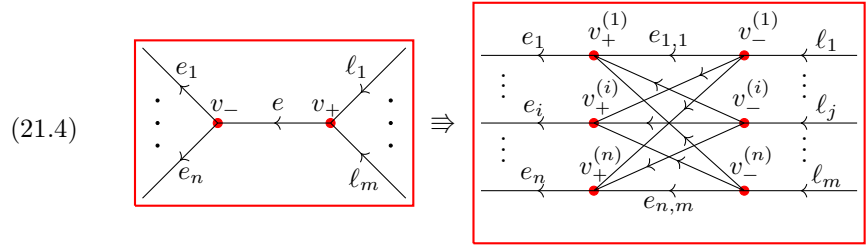
\begin{equation}
				\begin{array}[H]{l}
				\begin{tikzpicture}[baseline=35mm]
								  \def\mline#1#2#3{
												\draw[decorate,decoration={markings,mark=at position #3 with {\arrow[color=black]{>}}}] #1--#2;
												\draw  #1--#2;
							   };
								 \node at (70mm,31mm) {$e_{1,1}$};
								 \node at (70mm,9mm) {$e_{n,m}$};
								 \filldraw[color=red] (10mm,20mm) circle [radius=0.6mm];
								 \node at (11mm,22mm) {$v_{-}$};
								 \filldraw[color=red] (25mm,20mm) circle [radius=0.6mm];
								 \node at (24.5mm,22mm) {$v_{+}$};
								 \mline{(25mm,20mm)}{(10mm,20mm)}{0.5};
								 \node at (18mm,22.5mm) {$e$};
								 \mline{(10mm,20mm)}{(0mm,30mm)}{0.5};
								 \node at (6mm,27mm) {$e_{1}$};
								 \node at (6.5mm,13mm) {$e_{n}$};
								 \mline{(10mm,20mm)}{(0mm,10mm)}{0.5};
								 \filldraw (2mm,20mm) circle [radius=0.2mm];
								 \filldraw (2mm,17mm) circle [radius=0.2mm];
								 \filldraw (2mm,23mm) circle [radius=0.2mm];

								 \mline{(35mm,30mm)}{(25mm,20mm)}{0.6};
								 \mline{(35mm,10mm)}{(25mm,20mm)}{0.6};
								 \filldraw (33mm,20mm) circle [radius=0.2mm];
								 \filldraw (33mm,17mm) circle [radius=0.2mm];
								 \filldraw (33mm,23mm) circle [radius=0.2mm];
								 \node at (28.5mm,27.5mm) {$\ell_{1}$};
								 \node at (29.5mm,12.5mm) {$\ell_{m}$};
								 \draw[color=red,thick=3mm] (-1mm,9mm)--(36mm,9mm)--(36mm,31mm)--(-1mm,31mm)--(-1mm,9mm);
								 \node at (40mm,20mm) {$\Rrightarrow$};

								 \filldraw[color=red] (60mm,11mm) circle [radius=0.6mm];
								 \filldraw[color=red] (60mm,20mm) circle [radius=0.6mm];
								 \filldraw[color=red] (60mm,29mm) circle [radius=0.6mm];
								 \mline{(60mm,11mm)}{(45mm,11mm)}{0.5};
								 \mline{(60mm,20mm)}{(45mm,20mm)}{0.5};
								 \mline{(60mm,29mm)}{(45mm,29mm)}{0.5};
								 \node at (48mm,26mm) {$\vdots$};
								 \node at (48mm,17.5mm) {$\vdots$};
								 \node at (52mm,31mm) {$e_{1}$};
								 \node at (52mm,22mm) {$e_{i}$};
								 \node at (52mm,13mm) {$e_{n}$};
								 \node at (62mm,32.5mm) {$v^{(1)}_{+}$};
								 \node at (61mm,23.5mm) {$v_{+}^{(i)}$};
								 \node at (61mm,14.5mm) {$v_{+}^{(n)}$};
									
								 \coordinate (a) at	(60mm,29mm); 
								 \coordinate (b) at	(60mm,20mm);
								 \coordinate (c) at	(60mm,11mm);
								 \coordinate (a1) at	(80mm,29mm) circle;
								 \coordinate (b1) at	(80mm,20mm) circle;
								 \coordinate (c1) at	(80mm,11mm) circle;
								 \filldraw[color=red] (a1) circle [radius=0.6mm];
								 \filldraw[color=red] (b1) circle [radius=0.6mm];
								 \filldraw[color=red] (c1) circle [radius=0.6mm];
								 \mline{(a1)}{(a)}{0.5};
								 \mline{(a1)}{(b)}{0.3};
								 \mline{(a1)}{(c)}{0.4};
								 \mline{(b1)}{(a)}{0.6};
								 \mline{(c1)}{(a)}{0.4};
								 \mline{(b1)}{(b)}{0.65};
								 \mline{(b1)}{(c)}{0.65};
								 \mline{(c1)}{(b)}{0.45};
								 \mline{(c1)}{(c)}{0.5};

								 \node at (82mm,32.5mm) {$v^{(1)}_{-}$};
								 \node at (81mm,23.5mm) {$v_{-}^{(i)}$};
								 \node at (81mm,14.5mm) {$v_{-}^{(n)}$};

								 \mline{(95mm,11mm)}{(80mm,11mm)}{0.5};
								 \mline{(95mm,20mm)}{(80mm,20mm)}{0.5};
								 \mline{(95mm,29mm)}{(80mm,29mm)}{0.5};
								 \node at (88mm,26mm) {$\vdots$};
								 \node at (88mm,17.5mm) {$\vdots$};
								 \node at (91mm,31mm) {$\ell_{1}$};
								 \node at (91mm,22mm) {$\ell_{j}$};
								 \node at (91mm,13mm) {$\ell_{m}$};
								 \draw[color=red,thick=3mm] (44mm,4mm)--(44mm,36mm)--(96mm,36mm)--(96mm,4mm)--(44mm,4mm);
				\end{tikzpicture}
\end{array}
				\label{eq:21.4}
\end{equation}
Using these percolation moves we can bring all the $C_{v}, v\in C^{-}(G)$, to the left of
all $C_{u},u \in C^{+}(G)$. Thus for $G\in p\text{Bit}_{n,m}(G)$, we obtain after a finite
number of such perculations 
\begin{equation}
				G \Rrightarrow GZ e_1  \Rrightarrow GZe_1Z e_2 \Rrightarrow \cdots GZe_{1}\cdots
				Ze_{\ell}\equiv (T_{i}, \sigma,T_j^{\prime}) 
				\label{eq:21.5}
\end{equation}
with $\left\{ T_{i} \right\}_{i=1\ldots n}$, $\left\{ T_{j}^{\prime} \right\}_{j=1\ldots
m}$ are forest of planer trees, and $\sigma$ is a bijection between their leaves
$\sigma:\coprod\limits_{1\le i \le n}\partial T_i
\xrightarrow[\qquad]{\sim}\coprod\limits_{1\le j\le m} \partial T_{j}^{\prime}$. \\
For a distributive bio $\biop$, $G\in \text{Bit}$, $e\in G^{1}$ with
$v_{-}=d_{-}e \in C^{+}(G)$, \break
$v_{+}=d_{+}e\in C^{-}(G)$, we clearly have 
\begin{equation}
				\begin{array}[H]{lll}
				(N\biop)_{G} &\equiv& \text{Bio}(\mathscr{S}_{G},\biop) \\\\
				&\equiv& \left\{ 
								\begin{array}[H]{ll}
								\varphi\in
								\text{Bio}(\mathscr{S}_{Gze},\biop), \\ 
								\varphi^{\circ}(e_{i,j})=\varphi^{\circ}(e_{i^{\prime},j^{\prime}}), \\
								\varphi(v_{-}^{(i)})=
								\varphi(v_{-}^{(i^{\prime})}),\\ 
								\varphi(v_{+}^{(j)})=\varphi(v_{+}^{(j^{\prime})})
								\end{array}
				\right\} \\\\
				&\subseteq& (N\biop)_{GZe}\subseteq\cdots \subseteq
				(N\biop)_{(T_{i},\sigma,T_{j}^{\prime})}
				\end{array}
				\label{eq:21.6}
\end{equation}
For a distributive bio $\biop$, the underliying (associative, unital)
monoid $\biop(\ast;\ast)=\biop^{\pm}(1)$ is \myemph{commutative}. Indeed, for unary
operations ($n=m=p=1$) interchange reduces to commutativity 
\begin{equation}
				(P\circleftarrow P^{\prime})\circ P^{\prime\prime}\equiv (P\circ
				P^{\prime\prime})\circleftarrow P^{\prime}
				\label{eq:21.7}
\end{equation}
and for $P=1_{\ast}$, $P^{\prime}=Q\circrightarrow 1_{\ast}=\overrightarrow{Q}$, this gives
\begin{equation}
				\begin{array}[H]{lll}
Q\circ P^{\prime\prime}=\overset{\leftrightarrows}{Q}\circ P^{\prime\prime} &\equiv
(1_{\ast}\circleftarrow \overrightarrow{Q})\circ P^{\prime\prime}
\equiv (1_{\ast}\circ P^{\prime\prime})\circleftarrow (Q\circrightarrow 1_{\ast})
\\\\ 
&\equiv \left( P^{\prime\prime}\circ Q \right)\circleftarrow 1_{\ast}
\equiv P^{\prime\prime}\circ Q
				\end{array}
				\label{eq:21.8}
\end{equation}
Moreover, interchange imply this monoid $\biop(1)$ acts \myemph{centrally} on operations
and co-operations, in the sense that for $a\in \biop(1)$, $P\in \biop^{-}(n)$
(resp. $P\in \biop^{+}(n)$) we have 
\begin{equation}
				a\circ P = P\circ (\underbrace{a,\cdots, a}_{n}) \quad , \quad (\text{resp.
				$P\circ a=(\underbrace{a,\cdots , a}_{n})\circ P$})
				\label{21.9}
\end{equation}
For a \myemph{multiplicative subset} $S\subseteq\biop(1)$ 
\begin{equation}
				s_{1},s_{2}\in S \Longrightarrow  s_{1}\circ s_{2}\in S,\;\; \unit\in S, 
				\label{21.10}
\end{equation}
we have the \myemph{localization} of $\biop$ at $S$, $S^{-1}\biop$:
\begin{equation}
				\begin{array}[H]{l}
				S^{-1}\biop \equiv \biop[S^{-1}]\equiv (S^{-1}\biop^{-},S^{-1}\biop^{+}) \\\\
				S^{-1}\biop^{\pm}= \biop^{\pm}\times S /\sim 
				\end{array}
				\label{21.11}
\end{equation}
with $(P^{\prime\prime},s^{\prime\prime})\sim (P^{\prime},s^{\prime})$ iff $s\circ
s^{\prime}\circ P^{\prime\prime} = s\circ
s^{\prime\prime}\circ P^{\prime}$ for some $s\in S$; one writes
$P/s=\frac{1}{s}\circ P$ for the equivalence class $(P,s)/\sim$. \vspace{.2cm}\\
For a commutative rig
$A$ the bio $\biop_{A}$ is distributive; this is true in particular for commutative
rings $A$, and we shall continue and define ideals, primes, spectrum, schemes etc. for
distributive bios making sure that we get the right definitions for
$\biop=\biop_{A}$, $A\in C\text{Ring}$.
\begin{remark}
				For any rig $A$, the associated bio $\biop_{A}$ contains the (distributive,
				self-dual) sub-bio $\mathbb{F}\subseteq \biop_{A}$ with
				\begin{equation}
								\begin{array}[H]{ll}
								\mathbb{F}^{-}(n) &= \left\{ 0=(0,\cdots , 0), \delta_{1}=(1,0,\cdots , 0),
								\cdots , \delta_{n}=(0,\cdots 0,1) \right\} \\\\
								\mathbb{F}^{+}(n) &= \left\{ 0^{t}=
																\left(\begin{matrix}
																								0 \\ \vdots \\ 0
																				\end{matrix}\right) , 
																				\delta_{1}^{t} = 
																								\left(\begin{matrix}
																								1 \\ 0 \\ \vdots \\ 0
																				\end{matrix}\right) ,  \cdots , 
																				\delta_{n}^{t}=
																				\left(\begin{matrix}
																								0 \\ \vdots \\ 0 \\ 1
																				\end{matrix}\right)
								\right\}
								\end{array}
								\label{eq:21.13}
				\end{equation}
				the (co)-vectors with at most one $1$ coordinate (note that $\mathbb{F}$ is closed
				under the composition and action operations); we call $\mathbb{F}$
				``\myemph{the field with one element}''. Although we can develop everything
				using arbitrary distributive bios, we will
				restrict our attention to ``$\mathbb{F}$-algegras'', and always assume our
				distributive bio $\biop$ contains the bio $\mathbb{F}\subseteq \biop$ (note that
				$\mathbb{F}$ is indeed a ``field'': if $\varphi\in \text{Bio}(\mathbb{F},\biop)$
is not an injection then $\varphi(0)=\varphi(1)\in \biop(1)$, and $\biop^{\pm}(n)\equiv
\left\{ \ast \right\}$ reduce to a point, $\biop$ is the zero-bio, the final object of
bio). We will use $\mathbb{F}\subseteq \biop$ beginning in section \ref{sec:29}, to have
the sphere spectrum and ``homological algebra''. This is just for convenience (one can replace
everywhere ``$\phi_{\ast}(V_{+})$'' by the ``free $\biop\text{-set}$ on $V$''). Thus our basic building blocks for
geometry will be the (non-zero) objects of the full subcategory 
\begin{equation}
				C\text{Bio}\subseteq \mathbb{F}^{\displaystyle \backslash \text{Bio}}
				\label{eq:21.14}
\end{equation}
of distributive bios under $\mathbb{F}$.
				\label{rem:21.12}
\end{remark}
\section{The spectrum}
Fix $\biop=(\biop^{-},\biop^{+})\in \text{CBio}$. 
\begin{definition}
				A \myemph{$1$-ideal} is a subset $\oneideal\subseteq\biop(1)$ such that for
				$a_{1}\cdots a_{n}\in \oneideal$, $b\in \biop^{-}(n)$, $d\in\biop^{+}(n)$, the
				\myemph{``linear combination''} 
				\begin{equation}
								\left\{ b,(a_{i}), d  \right\}\overset{\text{def}}{=} \left( b\circ
												(a_{i})
								\right)\circleftarrow d = b\circrightarrow \left( (a_{i})\circ d \right)\in
				\oneideal \;\;\; \text{is in $\oneideal $}
								\label{eq:22.2}
				\end{equation}
				\label{def:22.1}
\end{definition} \vspace{-5mm}
\noindent Note that intersection of ideals is again an ideal, so we can speak of the $1$-ideal $\oneideal$
generated by a set $\left\{ a_{i} \right\}_{i\in I}\subseteq\biop(1);\oneideal$ is also given as
the collection of all linear - combinations of the $a_{i}$'s (with repetition). \\
A proper $1$-ideal $\pid\subseteq\biop(1)$, with $1\not\in \pid$, is called \myemph{prime} if
$S_{\pid}=\biop(1)\setminus \pid$ is multiplicative. The set of primes is denoted
$\text{spec}(\biop)$.  \\
The set $\text{spec}(\biop)$ is not empty: a maximal proper ideal $1\not\in
\pid\subseteq\biop(1)$, which exists by Zorn's lemma, is always prime. Indeed, if
$a,a^{\prime}\in \biop(1)\setminus \pid$, then for some linear combinations 
\begin{equation*}
				\left\{ b,(a_{i})_{i\le n}, d \right\}=\unit = \left\{
				b^{\prime},(a^{\prime}_{j})_{j\le m},d^{\prime} \right\},b,b^{\prime}\in
				\biop^{-}, d,d^{\prime}\in\biop^{+}, 
\end{equation*}
$a_{i}\in \pid$ or $a_{i}=a$, $a_{j}^{\prime}\in \pid$ or 
$a_{j}^{\prime}=a^{\prime}$ and by distributivity we have
\begin{equation*}
				\begin{array}[H]{lll}
								1 = 1\circ 1 &=& \left(\left(b\circ (a_{i}) \right)\circleftarrow
								d\right)\circ \left(\left(
												b^{\prime}\circ (a_{j}^{\prime})
				\right)\circleftarrow d^{\prime}\right) \\\\
								&=& \left( b\circ (\underbrace{b^{\prime},\cdots b^{\prime}}_{n})
								\right)\sigma \circ
								(a_{i}\circ a_{j}^{\prime})\circleftarrow \sigma^{\prime}\left(
												(\underbrace{d,\cdots , d}_{m})\circ d^{\prime}
								\right)
				\end{array}
\end{equation*}
with $a_{i}\circ a_{j}^{\prime}\in \pid$ or $a_{i}\circ a_{j}^{\prime}= a\circ a^{\prime}$,
and so $a\circ a^{\prime}\not\in \pid$. \\
Moreover, $\text{spec}(\biop)$ is a (compact,
sober$\equiv$Zariski) topological space with respect to the \myemph{Zariski topology} with
closed sets
\begin{equation}
				\begin{array}[H]{l}
				V(\oneideal):= \left\{ \pid\in \text{spec}(\biop),\pid\supseteq \oneideal \right\} \quad ,
				\quad \oneideal\subseteq \biop(1)\; \text{ideal}, \\\\
				V(\oneideal)\cap V(\oneideal^{\prime})=V(\oneideal\cdot \oneideal^{\prime})\quad ,
				\quad \oneideal\cdot\oneideal^{\prime}=\text{the ideal generated by $\left\{
												a\cdot a^{\prime},a\in \oneideal, a^{\prime}\in
												\oneideal^{\prime}
				\right\}$}; \\\\
				\bigcup\limits_{i}V(\oneideal_{i})=V(\sum\limits_{i}\oneideal_{i}) \quad , \quad
				\sum\limits_{i}\oneideal_{i}\;\;\text{the ideal generated by
				$\cup_{i} \oneideal_{i}$}; \\\\
				V(0)=\text{spec}(\biop)\quad , \quad V(1)=\phi
				\end{array}
				\label{eq:22.3}
\end{equation}
We have a basis for the topology by \myemph{basic open sets}
\begin{equation}
				\begin{array}[H]{c}
				D(f):=\left\{ \pid\in \text{spec}(\biop), \pid\not\ni f \right\}\quad , \quad
				f\in\biop(1). \\\\
				D(f_1)\cap D(f_2)=D\left(f_1\circ f_2\right) \quad , \quad D(1)=\text{spec}(\biop)\quad ,
				\quad D(0)=\phi, \\\\
				\text{spec}(\biop)\setminus V(\oneideal)=\bigcup_{f\in \mathfrak{a}}D(f)
				\end{array}
				\label{eq:22.4}
\end{equation}
For a map of distributive bios $\varphi\in \text{CBio}(\biop,\bioq)$, the pull back of a
(prime) ideal is a (prime) ideal, and we obtain
\begin{equation}
				\begin{array}[H]{c}
				\text{spec}(\varphi)=\varphi^{\ast}: \text{spec}(\bioq)\longrightarrow
				\text{spec}(\biop) \\\\
				\varphi^{\ast}(\mathfrak{q})=\varphi^{-1}(\mathfrak{q})=\left\{
				a\in\biop(1),\varphi(a)\in \mathfrak{q} \right\}.
				\end{array}
				\label{eq:22.5}
\end{equation}
It is a continuous map: 
\begin{equation}
				\varphi^{\ast^{-1}}\big(V(\oneideal)\big) = \left\{ \mathfrak{q}\in
				\text{spec}(\bioq),\varphi^{-1}(\mathfrak{q})\supseteq \oneideal \right\} =
				V\big(\varphi(\oneideal)\big)
				\label{eq:22.6}
\end{equation}
and similarly the pull back of basic open set is again basic open 
\begin{equation}
				\varphi^{\ast^{-1}}\left( D(f) \right) = \left\{ \mathfrak{q}\in \text{spec}(\bioq),
				\varphi^{-1}(\mathfrak{q})\not\ni f \right\} = D\big(\varphi(f)\big). 
				\label{eq:22.7}
\end{equation}
Thus we have a functor 
\begin{equation}
				\text{spec}: (\text{CBio})^{\text{op}} \longrightarrow \text{Top} \quad
				\text{(compact, sober)}
				\label{eq:22.8}
\end{equation}
We have a Galois correspondence ($\equiv$ adjunction of order sets)
\begin{equation}
				\begin{tikzpicture}[baseline=5mm]
								 \node at (0mm,5mm) {$\left\{ \oneideal \subseteq \biop(1)\;\text{ideal} \right\}$};
								 \draw[->] plot[smooth ] coordinates {(15mm,6mm)(25mm,8mm)(35mm,6mm)};
								 \draw[<-] plot[smooth ] coordinates {(15mm,4mm)(25mm,2mm)(35mm,4mm)};
								 \node at (25mm,10mm) {$V$};
								 \node at (25mm,0mm) {$I$};
								 \node at (53mm,5mm) {$\left\{ Z\subseteq \text{spec}(\biop)\;\text{closed} \right\}$};
				\end{tikzpicture}
				\label{eq:22.9}
\end{equation}
\begin{equation*}
				\bigcap\limits_{\pid\in Z} \pid =: I(Z) \; \tikz \draw[<-|] (0mm,0mm)--(20mm,0mm);\; Z
\end{equation*}
We have $VI(Z)=Z$, and 
\begin{equation}
				IV(\oneideal)=\bigcap\limits_{\mathfrak{p}\supseteq \mathfrak{a}} \pid = \left\{ a\in \biop(1),
								a^{n}=\underbrace{a\circ\ldots \circ a }_{n} \in \oneideal\;\text{for $n>>0$}
\right\}=\sqrt{\oneideal}
\label{eq:22.10}
\end{equation}
We get induced bijections
\begin{equation}
				\begin{tikzpicture}[baseline=5mm]
								 \node at (35mm,6mm) {$\sim$};
								 \node at (10mm,5mm) {$\text{spec}(\biop)$}; 
								 \node at (80mm,5mm) {$\left\{ Z\subseteq\text{spec}(\biop)\;\text{closed irreducible} \right\}$}; 
								 \node at (10mm,15mm) {\rotatebox{90}{\Huge $\subseteq$}};
								 \node at (75mm,15mm) {\rotatebox{90}{\Huge $\subseteq$}};
								 \node at (80mm,23mm) {$\left\{ Z\subseteq\text{spec}(\biop)\;\text{closed } \right\}$}; 
								 \node at (10mm,23mm) {$\left\{ \oneideal\subseteq\biop(1),\oneideal=\sqrt{\oneideal} \right\}$}; 
								 \draw[->] (30mm,24mm)--(60mm,24mm);
								 \draw[arrows={<-}] (30mm,21mm)--(60mm,21mm);
								 \draw[arrows={<->}] (20mm,5mm)--(50mm,5mm);
								 \node at (45mm,26mm) {$V$};
								 \node at (45mm,22mm) {$\sim$};
								 \node at (45mm,19mm) {$I$};
				\end{tikzpicture}
				\label{eq:22.11}
\end{equation}
For a multiplicative set $S\subseteq \biop(1)$, let $\phi_{S}:\biop\rightarrow
S^{-1}\biop$, $\phi_{S}(P)=P/1$, be the canonical map, it induces a homeomorphism 
\begin{equation}
				\begin{tikzpicture}[baseline=5mm]
								 \node at (0mm,5mm) {$\left\{ \pid\in \text{spec}(\biop), \pid\cap S=\phi \right\}$};
								 \draw[<-] plot[smooth ] coordinates {(20mm,6mm)(32.5mm,8mm)(45mm,6mm)};
								 \draw[->] plot[smooth ] coordinates {(20mm,4mm)(32.5mm,2mm)(45mm,4mm)};
								 \node at (32mm,11mm) {$\phi_{S}^{\ast}$};
								 \node at (32mm,0mm) {$S^{-1}$};
								 \node at (32mm,4mm) {$\sim$};
								 \node at (56mm,5mm) {$\text{spec}(S^{-1}\biop)$};
				\end{tikzpicture}
				\label{eq:22.12}
\end{equation}
In particular, for $S_{f}=\left\{ f^{\N} \right\}=\left\{ f^{n},n\ge 0 \right\}$,
$f\in\biop(1)$, we get 
\begin{equation}
				\text{spec}(\biop)\supseteq
				D(f)\xleftrightarrow[\qquad]{\sim}\text{spec}(\biop_{f}), \qquad
				\biop_{f}:= \left\{ f^{\N} \right\}^{-1}\biop
				\label{eq:22.13}
\end{equation}
and for $S_{\pid}=\biop(1)\setminus \pid$, $\pid$ prime we get 
\begin{equation}
				\left\{ q\in\text{spec}(\biop),q\subseteq \pid \right\}
				\xleftrightarrow[\qquad]{\sim}\text{spec}(\biop_{\pid}),\;\biop_{\pid}:=S_{\pid}^{-1}\biop
				\label{eq:22.14}
\end{equation}
i.e. $\biop_{\pid}\in \text{CBio}_{\text{loc}}$ is a \myemph{local} bio in the sense that
\begin{equation}
				m_{\pid}:=\biop_{\pid}(1)\setminus \left\{ x\in \biop_{\pid}(1),\exists x^{-1}\;\text{with}\;
				x\circ x^{-1}=1=x^{-1}\circ x \right\}=S_{\pid}^{-1}(\pid)
				\label{eq:22.15}
\end{equation}
is the unique maximal ideal of $\biop_{\pid}$. \\
We make $\text{CBio}_{\text{loc}}$ into a category by 
\begin{equation}
				\text{CBio}_{\text{loc}}(\biop,\bioq)=\left\{ \varphi\in
				\text{CBio}(\biop,\bioq),\varphi\;\text{is local:
$\varphi^{\ast}(m_{\bioq})=m_{\biop}$} \right\}
\label{eq:22.16}
\end{equation}
\section{The structure sheaf $\mathscr{O}_{\biop}$}
Fix $\biop\in \text{CBio}$. For an open set $\mathscr{U}\subseteq \text{spec}(\biop)$ define 
\begin{equation}
				\mathscr{O}^{\pm}_{\biop}(n):=\left\{ f:\mathscr{U}\to\coprod\limits_{\pid\in
								\mathscr{U}}\biop^{\pm}_{\mathfrak{p}}(n),
								f(\mathfrak{p})\in\biop^{\pm}_{\mathfrak{p}}(n), \;\text{and $f$ is
\myemph{locally a fraction}} \right\}	
\label{eq:23.1}
\end{equation}
$f$ \myemph{locally a fraction}: for all $\mathfrak{p}\in \mathscr{U}$, there exists open
$\mathfrak{p}\in
U_{\mathfrak{p}}\subseteq U$, and $P\in \biop^{\pm}(n)$, $s\in \biop(1)\setminus
\bigcup\limits_{\mathfrak{q}\in U_{\mathfrak{p}}} \mathfrak{q}$, such that for all
$\mathfrak{q}\in \mathscr{U}_{\mathfrak{p}}$ we have $f(\mathfrak{q})\equiv
P/ s $ in $\biop_{\mathfrak{q}}$. \\
Note that $\mathscr{O}_{\biop}(\mathscr{U})=\left(
\mathscr{O}_{\biop}^{-}(\mathscr{U}),\mathscr{O}_{\biop}^{+}(\mathscr{U}) \right)$ is a distributive bio, the
operations of compositions and actions are defined pointwise (in each $\biop_{\mathfrak{p}}$), and the
``local fraction condition'' is preserved. For open sets $\mathscr{V}\subseteq \mathscr{U}\subseteq
\text{spec}(\biop)$ we have the restriction maps 
\begin{equation}
				\rho_{\mathscr{V}}^{\mathscr{U}}: \biop(\mathscr{U})\to \biop(\mathscr{V})
				\label{eq:23.2}
\end{equation}
making $\mathscr{U}\mapsto \biop(\mathscr{U})$ a pre-sheaf of bios over $\text{spec}(\biop)$; by the local
nature of the ``locally-fraction-condition'' it is clearly a sheaf. \\
For $\mathfrak{p}\in \text{spec}(\biop)$, the \myemph{stalk} of $\mathscr{O}_{\biop}$ at
$\mathfrak{p}$ is given by 
\begin{equation}
				\begin{array}[H]{c}
								\mathscr{O}_{\biop, \mathfrak{p}}:=
								\limrightarrow\limits_{\mathscr{U}\ni
								\mathfrak{p}}\mathscr{O}_{\biop}(\mathscr{U})\xrightarrow{\sim}\biop_{\mathfrak{p}} \\\\
								(\mathscr{U},f)_{/\approx} \xmapsto{\quad} f(\mathfrak{p})
				\end{array}
				\label{eq:23.3}
\end{equation}
it is well defined, surjective, and injective. \\
The global sections of $\mathscr{O}_{\biop}$ are given by 
\begin{theorem}
				For a basic open set $D(s)\subseteq \text{spec}(\biop)$, $s\in \biop(1)$,
				\begin{equation*}
								\begin{array}[H]{c}
								\Psi:\biop_{s}=\biop\left[ \frac{1}{s} \right]\xrightarrow{\quad\sim\quad}
								\mathscr{O}_{\biop}\left( D(s) \right) \\\\
								\Psi\left( P/s^{n} \right):= \left\{ f(\mathfrak{p})\equiv P/s^{n} \; \text{in
								$\biop_{\mathfrak{p}}$ for all $\mathfrak{p}\in D(s)$} \right\}.
								\end{array}
				\end{equation*}
				\label{thm:3}
\end{theorem}
For $s=1$ we obtain the global sections 
\begin{equation*}
				\biop\xrightarrow{\quad \sim \quad}\mathscr{O}_{\biop}(\text{spec}(\biop))
\end{equation*}
\begin{proof}
				The map $\Psi$ which takes $P/s^{n}\in \biop_{s}$ to the constant section
				$f$, $f(\mathfrak{p})\equiv P/s^{n}$, is a well defined map of bios. Take $P\in
				\biop^{-}(k)$, the case $P\in\biop^{+}(k) $ is similar. \\
				\myemph{$\Psi$ is injective}: Assume $\Psi\left( P_{1}/s^{n_1}
				\right)=\Psi(P_2/s^{n_2})$, and define 
				\begin{equation*}
								\oneideal := \text{An}\left( s^{n_2}\circ P_{1} ; s^{n_1}\circ P_2
								\right)=\left\{ a\in \biop(1), a\circ s^{n_2}\circ P_{1}=a\circ
								s^{n_{1}}\circ P_2 \right\}; 
				\end{equation*}
				by distributivity it follows $\oneideal $ is a $1$-ideal. We have
				\begin{equation*}
								P_1/s^{n_1}=P_2/s^{n_2}\;\; \text{in}\;\; \biop_{\mathfrak{p}} \;\;\text{for all}\;\; \mathfrak{p}\in D(s)
				\end{equation*}
				\begin{description}
								\item[$\Rightarrow$ ]  $s_{\mathfrak{p}}\circ s^{n_2}\circ
												P_1=s_{\mathfrak{p}}\circ
												s^{n_1}\circ P_2$ for some $s_{\mathfrak{p}}\in \biop(1)\setminus
												\mathfrak{p}$,
												all $\mathfrak{p}\in D(s)$. 
								\item[$\Rightarrow$ ]  $\oneideal\not\subseteq\mathfrak{p}$, all
												$\mathfrak{p}\in
												D(s)$
								\item[$\Rightarrow$ ]  $V(\oneideal)\cap D(s)=\emptyset\Rightarrow
												V(\oneideal)\subseteq V(s)\Rightarrow s\in
												IV(\oneideal)=\sqrt{\oneideal}$ 
								\item[$\Rightarrow$ ] $s^{n}\in \oneideal$ for $n>>0 \Rightarrow
												s^{n+n_{2}}\circ P_{1}=s^{n+n_1}\circ P_{2}$ for $n>>0$
								\item[$\Rightarrow$ ] $P_1/s^{n_1}=P_2/s^{n_2}$ in $\biop_{s}$.
				\end{description}
				\myemph{$\Psi$ is surjective}: Fix $f\in
				\mathscr{O}_{\biop}(D(s))^{-}(k)$. Since $D(s)$ is compact we can cover it by
				finite collection of basic open sets 
				\begin{equation*}
								D(s)=D(s_{1})\cup\cdots\cup D(s_{N})
				\end{equation*}
				with 
				\begin{equation*}
								f(\mathfrak{p})\equiv P_{i}/t_{i} \;\text{for}\; \mathfrak{p} \in D(s_{i})\quad ;\quad
								t_{i}\in \biop(1)\setminus\bigcup\limits_{\mathfrak{p}\in D(s_{i})}\pid.
				\end{equation*}
				We have $V(t_{i})\subseteq V(s_{i})$, so $s_{i}^{n_i}=c_{i}\circ t_{i}$ for some
				$c_{i}\in \biop(1)$, and for $\mathfrak{p}\in D(s_{i})$, $f(\mathfrak{p})=P_{i}/t_{i}=c_{i}\circ
				P_{i}/s_{i}^{n_i}$. We can replace $s_{i}$ by $s_{i}^{n_i}$,
				($ D(s_{i})=D(s_{i}^{n_{i}})$), and $P_{i}$ by $c_{i}\circ P_{i}$, so 
				\begin{equation*}
								f(\mathfrak{p})\equiv P_{i}/s_{i}\quad \text{for}\; \mathfrak{p}\in D(s_{i}).
				\end{equation*}
				On the set $D(s_{i})\cap D(s_{j})=D(s_{i}\circ s_{j})$, the section $f$ is given
				by both $P_{i}/s_{i}$ and $P_{j}/s_{j}$, and by the injectivity  of $\Psi$ 
				\begin{equation*}
								(s_{i}\circ s_{j})^{n}\circ s_{j}\circ P_{i} = (s_{i}\circ
								s_{j})^{n}\circ s_{i}\circ P_{j}
				\end{equation*}
				By finiteness we may assume one $n$ works for all $i,j$. \vspace{.1cm}\\ 
				Replacing $s_{i}$ by
				$s_{i}^{n+1}$, and replacing $s_{i}^{n}\circ P_{i}$ by $P_{i}$, we may assume 
				\begin{equation*}
								\begin{array}[H]{l}
												f(\mathfrak{p})\equiv P_{i}/s_{i} \quad \text{for all
																$\mathfrak{p}\in
												D(s_{i})$} \\\\
												s_{j}\circ P_{i}= s_{i}\circ P_{j} \quad \text{all $i,j$.}
								\end{array}
				\end{equation*}
				Since $D(s)\subseteq \bigcup\limits_{i}D(s_{i})$ we have that some power
				$s^{M}$ is a linear-combination of the $s_{i}$
				\begin{equation*}
								\begin{array}[H]{l}
								s^{M}=\left\{ b,(c_{j}),d \right\} = b\circ (c_{j})\circleftarrow d \quad
								, \quad c_{j}=s_{i(j)}\quad ,\quad b,d\in \biop^{\pm}(\ell)\quad ,\\\\
								j\in \left\{ 1,\cdots , \ell \right\}\xrightarrow{i}\left\{
												1,\cdots, N
								\right\}.
								\end{array}
				\end{equation*}
				Define $P=b\circ (P_{i(j)})\circleftarrow (\underbrace{d,\ldots ,d}_{k})$. \\
				We have, 
				\begin{equation*}
								\begin{array}[H]{lll}
												s_{j}\circ P &=&  b\circ \left( s_{j}\circ P_{i(j)}
												\right)\circleftarrow (d,\cdots , d)\equiv  b\circ \left(
												s_{i(j)}\circ P_{j} \right)\circleftarrow (d,\cdots , d) \\\\
												&=& (b\circ (c_{j}\circ P_{j}))\circleftarrow (d,\cdots , d) \\\\
												&=& ((b\circ(c_{j}))\circleftarrow d)\circ P_{j}\qquad
												\text{by distributivity} \\\\
												&=& s^{M}\circ P_{j}
								\end{array}
				\end{equation*}
				Thus $f(\mathfrak{p})\equiv P_{j}/s_{j}\equiv P/s^{M}$ is constant. 
\end{proof}
\section{Generalized Schemes}
For a topological space $\xspace$, we have the category $\text{CBio}/_{\xspace}$ of
sheaves of $\text{CBio}$ over $\xspace$, its maps are natural transformations
$\varphi=\left\{ \varphi_{\mathscr{U}} \right\}$, $\varphi_{\mathscr{U}}\in
\text{CBio}(\biop(\mathscr{U}),\biop^{\prime}(\mathscr{U}))$. Putting all these categories together we have
the category $\text{CBio}/\text{Top}$: its object are pairs $(\xspace,\biop)$,
$\xspace\in \text{Top}$, $\biop\in \text{CBio}/\xspace$, and its maps
$f:(\xspace,\biop)\to(\xspace^{\prime},\biop^{\prime})$ are pairs $f\in
\text{Top}(\xspace,\xspace^{\prime})$ and $f^{\natural}\in
\text{CBio}/_{\xspace^{\prime}} (\biop^\prime,f_{\ast}\biop) $; explicitly, $f$ is a
continuous function, and for $\mathscr{U}\subseteq \xspace^{\prime}$ open, we have the map of bios
$f_{\mathscr{U}}^{\natural}\in \text{CBio}\left( \biop^{\prime}(\mathscr{U}),\biop\left(
f^{-1}(\mathscr{U}) \right) \right)$,
these maps being compatible with restrictions. 
\begin{remark}
				For $f\in \text{Top}(\xspace,\xspace^{\prime})$ we have adjunction
\begin{figure}[H]
				\centering
				\begin{tikzpicture}[baseline=0mm]
								\node at (41.5mm,22mm) {$\text{CBio}/\xspace$};
								 \draw[<-] plot[smooth ] coordinates {(38mm,20mm)(35mm,14mm)(38mm,7mm)};
								 \draw[->] plot[smooth ] coordinates {(45mm,20mm)(48mm,14mm)(45mm,7mm)};
								 \node at (41.5mm,4mm) {$\text{CBio}/\xspace^{\prime}$};
								 \node at (32mm,14mm) {$f^{\ast}$};
								 \node at (51mm,14mm) {$f_{\ast}$};
								 \node at (0mm,12mm) {$f^{\ast}\biop^{\prime} = $ sheaf associated};
								 \node at (75mm,12mm) {$f_{\ast}\biop(\mathscr{U})=\biop(f^{-1}\mathscr{U})$ };
								 \node at (10mm,4mm) {to the pre-sheaf};
				\end{tikzpicture}
\end{figure}
\vspace{-0.8cm}
\begin{equation*}
				\hspace{-6cm}  \xspace \supseteq \mathscr{V} \xmapsto{\qquad} \hspace{-10mm} \limrightarrow\limits_{
								\begin{array}[H]{l}
												\hspace{10mm} \mathscr{U}\subseteq \xspace^{\prime}\,\text{open} \\
												\hspace{10mm} f(\mathscr{V})\subseteq \mathscr{U}
								\end{array}
				} 
				\hspace{-7mm}\biop^{\prime}(\mathscr{U})
\end{equation*}
For a map $f\in \text{CBio}/{\text{Top}} \left(
(\xspace,\biop),(\xspace^{\prime},\biop^{\prime}) \right)$, and for a point $x\in
\xspace$, we get an induced map of stalks 
\begin{equation}
				f_{x}^{\natural}:\biop^{\prime}_{f(x)}= \limrightarrow\limits_{f(x)\in \mathscr{V}\subseteq
				\xspace^{\prime}}
				\biop^{\prime}(\mathscr{V})\xrightarrow{\qquad}\limrightarrow\limits_{f(x)\in
								\mathscr{V}\subseteq
				\xspace^{\prime}} \biop(f^{-1}\mathscr{V})\to \limrightarrow\limits_{x\in
								\mathscr{U}\subseteq
				\xspace}\biop(\mathscr{U})=\biop_{x} 
				\label{eq:24.2}
\end{equation}
\label{remark:24.1}
\end{remark}
\begin{definition}
				The category of \myemph{locally-bio-spaces}
				$\text{CBio}_{\text{loc}}/_{\text{Top}}$, is the category with object
				$(\xspace,\biop)\in \text{CBio}/_{\text{Top}}$, such that for all $x\in \xspace$
				the stalk $\biop_{x}$ is a local bio with the unique maximal proper $1$-ideal
				$m_{x}\subseteq\biop_{x}(1)$; the maps 
				$f\in \text{CBio}_{\text{loc}}/_{\text{Top}} \left(
				(\xspace,\biop),(\xspace^{\prime},\biop^{\prime}) \right) $ are maps $f\in
				\text{CBio}/_{\text{Top}}\left( (\xspace,\biop),(\xspace^{\prime},\biop^{\prime})
				\right)$ such that for all $x\in X$, $f^{\natural}_{x}\in \text{CBio}_{\text{loc}}\left(
				\biop^{\prime}_{f(x)},\biop_{x} \right)$ is a local map:
				$f_{x}^{\natural}(m_{f(x)})\subseteq m_{x}$.
				\label{def:24.3}
\end{definition}
\begin{theorem}
				We have the adjunction
\begin{figure}[H]
				\centering
				\begin{tikzpicture}[baseline=0mm]
								 \node at (41.5mm,23mm) {$\left(\text{CBio}_{\text{loc}}/_{\text{Top}}\right)^{\text{op}}$};
								 \draw[<-] plot[smooth ] coordinates {(38mm,20mm)(35mm,14mm)(38mm,7mm)};
								 \draw[->] plot[smooth ] coordinates {(45mm,20mm)(48mm,14mm)(45mm,7mm)};
								 \node at (41.5mm,4mm) {$\text{CBio}$};
								 \node at (6mm,14mm) {$\text{spec}(\biop):=\left(
								 \text{spec}(\biop),\mathscr{O}_{\biop} \right)$ \hspace{6mm} spec};
								 \node at (68mm,14mm) {$\Gamma$ \hspace{4mm} $\Gamma(\xspace, \biop):=\biop(\xspace)$};
								 \node at (70mm,9mm) {the global sections};
				\end{tikzpicture}
\end{figure}
\begin{equation*}
				\text{CBio}_{\text{loc}}/_{\text{Top}} \left(
				(\xspace,\biop),\text{spec}(\biop^{\prime}) \right)\equiv \text{CBio}\left(
								\biop^{\prime},\biop(\xspace)
				\right)
\end{equation*}
				\label{thm:adjunction}
\end{theorem}
\begin{proof}
				For a point $x\in \xspace$, we have canonical map $\phi_x\in \text{CBio}\left(
				\biop(\xspace),\biop_{x} \right)$, $\phi_{x}P=P|_{x}$ the stalk of the global
				section $P$ at the point $x$, and we get a prime
				$\pid_{x}:=\phi_{x}^{-1}(m_{x})\in \text{spec}\left(\biop(\xspace)\right)$, $m_{x}\subseteq
				\biop_{x}(1)$ the maximal ideal. For a basic open set $D(s)\subseteq
				\text{spec}\left(\biop(\xspace)\right)$, $s\in \biop(\xspace)(1)$, we have $\left\{ x\in
								X,\pid_{x}\in D(s)
				\right\}=\left\{ x\in X,\phi_{x}(s)\not\in m_{x} \right\}$ is \myemph{open} in
				$\xspace$: if $\phi_{x}(s)\not\in m_{x}$, it is invertible in $\biop_{x}$,
				$a_{x}\circ \phi_{x}(s)=1$, $a_{x}\in\biop_{x}$ and there is an open
				$\mathscr{U}_{x}\ni x$,
				$a\in \biop(\mathscr{U}_{x})$ with $a_{x}=a|_{x}$; taking $\mathscr{U}_{x}$ smaller we have $a\circ
				s|_{\mathscr{U}_{x}}=1$ already in $\biop(\mathscr{U}_x)$, and so $\phi_{x^{\prime}}(s)\not\in
				m_{x^{\prime}}$ for all $x^{\prime}\in \mathscr{U}_{x}$. Thus the map
				$x\mapsto\pid_{x}$ is a continuous map
				$\pid:\xspace\xmapsto{\quad}\to\text{spec}\left(\biop(\xspace)\right)$.
				The uniqueness of the inverse $\left( s|_{u_{x}} \right)^{-1}$, shows these
				local inverses glue to a global inverse $s^{-1}\in \biop\left( \pid^{-1}\left(
				D(s) \right) \right)$ and we get the map of $\text{CBio}$
				\begin{equation}
								\pid_{D(s)}^{\natural}: \biop(\xspace)_{s}\equiv \left\{ s^{\N}
								\right\}^{-1}\biop(\xspace)\xrightarrow{\qquad}\biop\left(\pid^{-1}\left(
																D(s)
								\right)\right)
								\label{eq:24.4}
				\end{equation}
				These maps are compatible on intersections, $D(s_{1})\cap D(s_{2})=D(s_{1}\circ
				s_{2})$, so by the sheaf property give $\pid_{u}^{\natural}\in \text{CBio}\left(
								\mathscr{O}_{\text{spec}\; \biop(\xspace)}(\mathscr{U}),
								\biop(\pid^{-1}\mathscr{U})
				\right)$ for any open $\mathscr{U}\subseteq \text{spec}\left(\biop(\xspace)\right)$, compatible with
				restrictions, so 
				\begin{equation}
								\pid = (\pid,\pid^{\natural})\in \text{CBio}/_{\text{Top}}\left(
								(\xspace,\biop),\text{spec}\left(\biop(\xspace)\right) \right)
								\label{eq:24.5}
				\end{equation}
				For $x\in\xspace$, we get 
				\begin{equation}
								\pid_{x}^{\natural}=\limrightarrow\limits_{\phi_{x}(s)\not\in m_{x}}
								\pid^{\natural}_{D(s)}\in \text{CBio}_{\text{loc}}\left(
								\biop(\xspace)_{\pid_{x}},\biop_{x} \right)
								\label{eq:24.6}
				\end{equation}
				is \myemph{local}, so $\pid\in
				\text{CBio}_{\text{loc}/\text{Top}}\left((\xspace,\biop),\text{spec}\left(\biop( \xspace )\right)\right)$
is the co-unit of adjunction. \vspace{.1cm}\\
				Given $\varphi\in \text{CBio}\left(
				A,\biop(\xspace) \right)$ we get $\text{spec}(\varphi)\circ\pid\in
				\text{CBio}_{\text{loc}/\text{Top}}\left( (\xspace,\biop),\text{spec} (A) \right)$.
				\vspace{.1cm}\\
				Given  $f\in \text{CBio}_{\text{loc}/\text{Top}}\left(
				(\xspace,\biop),\text{spec}(A) \right)$ we get the global sections 
				\begin{equation}
								\Gamma(f)=f_{\text{spec}(A)}^{\natural}\in \text{CBio}(A,\biop(\xspace)).
								\label{eq:24.7}
				\end{equation}
				One  checks these are inverse bijections.
\end{proof}
\noindent Following the footsteps of Grothendieck
we can define the category of \break
\myemph{Generalized Schemes} $\text{GSch} \subseteq
\text{CBio}_{\text{loc}}/_{\text{Top}}$ to be the full subcategory of
				$\text{CBio}_{\text{loc}}/_{\text{Top}}$ consisting of the object
				$(\xspace,\mathscr{O}_{\xspace})$ which are locally affine: we have some open
				cover $\xspace=\bigcup
				\mathscr{U}_{i}$, and $(\mathscr{U}_{i},\mathscr{O}_{\xspace}|_{\mathscr{U}_{i}})\cong
				\text{spec}\left(\mathscr{O}_{\xspace}(\mathscr{U}_{i})\right)$. \vspace{.1cm}\\
An open subset of a scheme is again a scheme. \vspace{.1cm}\\
				Schemes can be glued along open subsets and consistent glueing data.
				\vspace{.1cm}\\
				Since ordinary commutative rings $A$ give distributive bios $\biop_{A}\in
				\text{CBio}$ and since all our definitions reduce to their classical analogues for
				$\biop=\biop_{A}$, $A\in C\text{Ring}$, ordinary schemes embeds fully faithfully in generalized
				schemes, 
				$$\text{Sch}\hookrightarrow \text{GSch},
(\xspace,\mathscr{O}_{\xspace})\mapsto (\xspace,\biop_{\mathscr{O}_{\xspace}})$$
\begin{theorem}
				The category $\text{GSch}$ has fiber products, for  $f_{i}\in\text{GSch}(X_{i},Y)$:
\begin{figure}[H]
				\centering
				\begin{tikzpicture}[baseline=0mm]
								 \node at (50mm,30mm) {$X_{0}\prod\limits_{Y}X_{1}$};
								 \node at (35mm,15mm) {$X_{0}$};
								 \node at (65mm,15mm) {$X_{1}$};
								 \node at (50mm,0mm) {$Y$};
								 \draw[->] (54mm,28mm)--(64mm,18mm);
								 \draw[->] (46mm,28mm)--(36mm,18mm);
								 \draw[-] (43mm,23mm)--(50mm,18mm)--(57mm,23mm);
								 \draw[->] (36mm,13mm)--(47mm,2.5mm);
								 \draw[->] (63mm,13mm)--(52mm,2.5mm);
								 \node at (38mm,8mm) {$f_{0}$};
								 \node at (61mm,8mm) {$f_{1}$};

				\end{tikzpicture}
\end{figure}
				\label{thm:category}
\end{theorem}
\begin{proof}
				Exactly as for ordinary schemes. Write
				$Y=\bigcup\limits_{i}\text{spec}(B_{i})$, 
								$X_{\varepsilon}=\bigcup\limits_{i,j}
								\text{spec}(A_{j,i}^{\varepsilon})$, $\varepsilon=0,1$
								with $f_{\varepsilon}\left(\text{spec}(A_{j,i}^{\varepsilon})\right)\subseteq
								\text{spec}\; (B_{i})$,  and glue:
								\begin{equation}
								X_0\prod\limits_{Y} X_{1}\equiv
								\coprod_{i,j_{0},j_{1}}\text{spec}\left(
								A^{0}_{j_{0},i}\coprod\limits_{B_{i}}A_{j_{1},i}^{1} \right).
												\label{eq:24.8}
								\end{equation}
								Note that one uses the push-out
								$A^{0}_{j_0,i}\coprod\limits_{B_{i}}A_{j_{1},i}^{1}$ in $\text{CBio}$
								(Not the tensor product!).
\end{proof}
\begin{remark}
				For (the initial object) $\Z\in C\text{Ring}\subseteq C\text{Rig}$, the
				``arithmetical surface''
				\begin{equation*}
								\text{spec}\left(\biop_{\Z}\coprod \biop_{\Z}\right) \equiv
								\text{spec}\left( \biop_{\Z} \right)\prod \text{spec}\left( \biop_{\Z} \right)
				\end{equation*}
				does not reduce to its diagonal, as it does in classical algebraic geometry where 
				\begin{equation*}
								\text{spec}(\Z)\prod \text{spec}(\Z) = \text{spec}(\Z\otimes\Z) =
								\text{spec}(\Z) = \left\{ (0); (2),(3),(5),(7)\cdots \right\}.
				\end{equation*}
				\label{remark:24.9}
\end{remark}
\section{$A$-Sets}
We will denote our new ``generalized rings'', the distributive bios under
$\mathbb{F}$,  $\text{CBio}$, by the
letters $A,B,\cdots ,$
\begin{definition}
				For $A\in \text{CBio}$ an $A\text{-set}$ is a pointed set $M\in
				\text{Set}_{0}$ together with an $A\text{-action}$: for $n\ge 1$ we have maps
				\begin{equation*}
								\begin{array}[H]{l}
								A^{-}(n)\times M^{n}\times A^{+}(n) \xrightarrow{\qquad} M \\\\
								(b,(m_{j}),d) \xmapsto{\qquad} \left<b,(m_j)d\right>_{n}
								\end{array}
				\end{equation*}
				These maps are required to satisfy the following axioms: \vspace{.1cm}\\
				\begin{description}
								\item[\myemph{zero} ]
												$\left<b,(m_{j}),d\right>_{n}=\left<b\circ_{j_{0}} 0,
												(m_{j})_{j\not = j_{0}}, 0 \circ_{j_{0}}d\right>_{n-1}$ if $m_{j_{0}}=0$.
								\item[\myemph{$S_{n}$-variance} ] $\left<b,(m_{j}),d\right>_{n}=\left<
												b\sigma, (m_{\sigma(j)}),\sigma^{-1}d\right>_{n}$
								\item[\myemph{Associativity} ] For $b\in A^{-}(n)$, $d\in A^{+}(n)$ 
																				$b_{i}\in A^{-}(k_{i}),d_{i}\in A^{+}(k_{i})$,  \break
																				$m=(\overline{m}_{i})\in M^{\sum k_{i}}$, 
												\begin{equation*}
												\left<b\circ (b_{i}), m, (d_{i})\circ d \right>_{\sum k_{i}}=
																				\left<b,\left<b_i,\overline{m}_{i},d_{i}\right>_{k_{i}},d\right>_{n}
												\end{equation*}
								\item[\myemph{Unit}: ] $\left< 1,m,1\right>_{1}=m$
								\item[\myemph{Distributivity}: ] For \hspace{2mm} $b\in A^{-}(n)$,
												\hspace{2mm} $d\in A^{+}(\ell),$ \hspace{2mm}
												$a\in A^{\pm}(|n-\ell|+1)$,  \hspace{2mm}
												$m=(m_{i})\in M^{\min\left\{ m,\ell
												\right\}}$,  \hspace{2mm}
												$\tilde{m}=( m_{1},\cdots m_{i-1}
												\underbrace{m_{i}\cdots m_{i}}_{|n-\ell|+1},\cdots )$ 
												\begin{equation*}
																\begin{array}[H]{c}
																				\left< b\circleftarrow_{i}a, m, d
																				\right>_{\ell}= \left<
																				b,\tilde{m},a\circ_{i}d\right>_{n}, \quad n\ge
																				\ell \\\\
																				\left< b,m,a\circrightarrow_{i}d \right>_{n}= \left<
																				b\circ_{i}a,\tilde{m},d\right>_{\ell} , \quad n\le
																				\ell
																\end{array}
												\end{equation*}
												For $a\in A(1)$, $m\in M$, we write 
												\begin{equation*}
																a\cdot m := \left< a, m, 1\right>_{1} =
																\left<1,m,a\right>_{1}
												\end{equation*}
												giving an action of the monoid $A(1)$ on $M$. \\
												A map $\varphi: M\to N$ of $A\text{-Set}$ is a set map $\varphi\in
												\text{Set}_{0}(M,N)$, preserving the $A\text{-action}$
												\begin{equation*}
																\varphi\left( \left<b,(m_{j}),d
																\right>_{n}\right)=\left<b,\left( \varphi(m_j)
																\right),d\right>_{n}
												\end{equation*}
												Thus we have a category $A$-Set. 
				\end{description}
				\label{def:a-set}
\end{definition}
\subsection*{Example of $A$-Sets}
\begin{equation}
				\begin{array}[H]{l}
					\text{Given $\varphi\in \text{CBio}\left( A,B \right)$, the sets $B^{-}(m)$ and $B^{+}(m)$ are
									$A$-Sets} \\\\
				\left<b,(x_{i}),d\right>_n :=\varphi(b)\circ (x_{i})\circ\sigma_{n,m}\circ\left(
				\underbrace{\varphi(d),\cdots ,\varphi(d)}_{m} \right)\quad , \quad x_{i}\in
				B^{-}(m).
				\end{array}
				\label{eq:25.2}
\end{equation}
\begin{equation}
				\ \hspace{-3.5cm} \text{A sub-$A$-set of $A(1)$ is just an ideal. }
				\label{eq:25.3}
\end{equation}
\begin{equation}
\begin{array}[H]{c}
				\text{For $A=\mathbb{F}$ we have $\mathbb{F}\text{-Set}\equiv \text{Set}_{0}$ and for
								$M\in\text{Set}_{0}$ there is a unique $\mathbb{F}$-action:} \\\\
								\left< b,m_{j},d\right>_{n} =
								\begin{cases}
												m_{i} & b=\delta_{i},\; d=\delta_{i}^t \\
												0 & \text{otherwise}
								\end{cases}
				\end{array}
\label{eq:25.4}
\end{equation}
\begin{equation}
				\begin{array}[H]{l}
								\text{For a commutative ring $A$ we have $A\text{-Set}\equiv A\text{-mod}$} \\ \text{the category of $A$ modules. }
				\end{array}
				\label{eq:25.5}
\end{equation}
\begin{equation}
				\begin{array}[H]{l}
								\text{For $A=\Z_{\R}$ (resp. $\Z_{\C}$), the sub-$A$-sets of $\R^{n}$
				(resp $\C^{n}$)} \\ 
				\text{ are the convex symmetric subsets $M\subseteq \R^{n}$ (resp
				$M\subseteq\C^{n}$)} \\ 
				\text{ (since the pointwise product of two vectors in the
				$\ell_{2}$-unit} \\ 
\text{ball is a vector in the $\ell_{1}\text{-unit}$ ball).}
				\end{array}
				\label{eq:25.6}
\end{equation}
\noindent The category $A$-Set is complete and co-complete. Inverse limits, and filtered
co-limits are formed in $\text{Set}_{0}$, co-limits and sums are more
complicated. Given a set $V$ the free $A$-Set on $V$, $A^{V}$ is given by 
\begin{equation}
				A^{V}=\coprod\limits_{n\ge 1} A^{-}(n)\times V^{n}\times
				A^{+}(n)/_{\sim}
				\label{eq:25.7}
\end{equation}
where $\sim$ is the equivalence relation generated by zero, $S_{n}$-invariance, and
distributivity. The element of $A^{V}$ can be written, non-uniquely, as
$\left<b,v,d\right>_{n}$, $b\in A^{-}(n)$, $v\in V^{n}$, $d\in A^{+}(n)$. For $M\in
A\text{-Set}$, we have
\begin{equation}
				\begin{array}[H]{l}
								\text{Set}(V,M) \equiv A\text{-Set}(A^{V},M) \\
								\varphi\xmapsto{\qquad} \tilde{\varphi}\left(
								\left<b,v_{j},d\right>_{n} \right)=\left< b,\varphi(v_{j}),d\right>_{n}
				\end{array}
				\label{eq:25.8}
\end{equation}
When $V=\left\{ v_{0} \right\}$ is a singleton $A^{\left\{ v_{0} \right\}}\equiv
A(1)\cdot v_{0}$, and the free $A$-Set on one generator is just $A(1)$ \\
Given a homomorphism $\varphi\in \text{CBio}(B,A)$, we have an adjunction (we use
``geometric'' notation):
\begin{equation}
				\begin{tikzpicture}
								 \node (a) at (4,4) {$\text{$A$-set}$};
								 \node (b) at (4,2) {$\text{$B$-set}$};
								 \draw [->] (a) to [out=330,in=30,right] node{$\varphi_{*}$} (b);
								 \draw [<-] (a) to [out=210,in=150,left] node{$\varphi^{*}$} (b);
				\end{tikzpicture}
				\label{eq:25.9}
\end{equation}
where $\varphi_* N\equiv N$ with the $B$-action $\langle
b,m_j,b^{\prime}\rangle_n := \langle
\varphi(b),m_j,\varphi(b^{\prime})\rangle_n$. \vspace{.1cm}\\ 
The left adjoint is
$\varphi^*M=\left( \coprod\limits_{n} A_n\times M^n\times A_n
\right)\big/\sim$, where $\sim$ is the equivalence relation generated by
zero, $S_{n}$-invariance, distributivity,  and $B$-linearity: 
\begin{equation}
				\langle a, \langle b_{j},m_i,b^{\prime}_{j}\rangle, a^{\prime}\rangle_{n} =
				\langle a \circ \varphi(b),m,\varphi(b^{\prime})\circ a\rangle_m 
				\label{eq:25.10}
\end{equation}
$a\in A^{-}(n)$, $a^{\prime}\in A^{+}(n)$, $b_{j}\in B^{-}(k_{j})$, $b_{j}^{\prime}\in
B^{+}(k_{j})$, $m=\overline{m}_{1},\cdots , \overline{m}_{n}\in M^{\sum k_{i}}$. \\
In particular, for $B=\mathbb{F}$ and $\phi\in
\text{CBio}(\mathbb{F},A)$ the unique homomorphism, $\phi_*$ is just
the functor forgetting the $A$-action, and for $V\in \text{Set}_0$, 
$\phi^{*} V= A^{V\setminus \langle 0 \rangle}$ is the free
$A$-set on $V\setminus \left\{ 0 \right\}$. \vspace{.2cm}\\
More generally, for a simplicial distributive bio $A=(A_{n})_{n\ge 0}\in
s\text{CBio}\equiv (\text{CBio})^{\bbDelta^{\text{op}}}$ we have the category of
simplicial $A\text{-Sets}$, $A\text{-sSet}$, with objects the pointed simplicial sets
$M=(M_{m})_{m\ge 0}$, with compatible $A_{m}-\text{action}$ in dimension $m\ge 0$, 
\begin{equation}
				\begin{array}[H]{c}
								A_{m}^{-}(n)\times (M_{m})^{n}\times A_{m}^{+}(n) \xrightarrow{\qquad}
								M_{m} \\\\
								(b, (m_{i}), d) \xrightarrow{\qquad} \left< b,(m_{i}),d\right>_{n}
				\end{array}
				\label{eq:25.11}
\end{equation}
and for $\ell\in\bbDelta(m,m^{\prime})$
\begin{equation}
				\ell^{\ast}\left( \left< b,(m_{i}),d\right>_{n} \right) = \left<
				\ell^\ast(b),(\ell^{\ast}m_{i}),\ell^{\ast}(d)\right>_{n}
				\label{eq:25.12}
\end{equation}
The category $A\text{-sSet}$ has a simplicial, cellular, Quillen model structure given by
\begin{equation}
		 \begin{array}[h]{l}
		 \begin{array}[h]{lll}
						 (\romannumeral 1) \text{ \underline{Fibrations}: } &
						 \mathcal{F}_{A} \equiv \phi_{*}^{-1}(\mathcal{F}_{\mathbb{F}}), &
						 \mathcal{F}_{\mathbb{F}}\equiv\text{Kan fibrations.}  \\\\
						 (\romannumeral 2) \text{ \underline{Weak-equivalence}: } &
						 \mathcal{W}_{A} \equiv
						 \phi_{*}^{-1}(\mathcal{W}_{\mathbb{F}}), &
						 \mathcal{W}_{\mathbb{F}}\equiv 
						 \hspace{-0.1cm} \left[ \hspace{-0.1cm}
						 \begin{array}[H]{l}
										 \parbox{2cm}{Weak-equivalence of (pointed) simplicial sets}
		 \end{array}\hspace{-0.25cm}\right] \\\\

(\romannumeral 3) 
								\text{ \underline{Cofibrations}: } &
												 \mathcal{C}_A\equiv
												\phantom{}^{\perp}(\mathcal{W}_{A}\cap \mathcal{F}_{A}), \\
												\hspace{2cm} 
		\end{array}
						\vspace{-.2cm}\\
												\hspace{2cm} \parbox[t][][t]{6.9cm}{the maps satisfying the left lifting property
																						 with respect to the trivial fibrations.}
		 \end{array}
		 \label{eq:25.13}
\end{equation}
The cofibrations can also be characterized as the retracts of the \underline{free
maps}, where a map $f:M_{.}\to N_{.}$ is \underline{free} if there exists subsets
$V_n\subseteq N_n$ with $\ell^{*}(V_{n^{\prime}})\subseteq V_n$ for all surjective
$\ell\in\BDelta(n,n^{\prime})$, and such that $f_n$ induces isomorphism 
\begin{equation}
		 M_n\coprod \phi_{A}^{*}(V_n)\xrightarrow{\sim}N_n
		 \label{eq:25.14}
\end{equation}
where $\phi_{A}^{*}(V_n)$ is the free $\aset$ on $V_n$. \vspace{.1cm}\\
This model structure is cofibrantly generated 
\begin{equation}
		 \begin{array}[H]{ll}
						 (\romannumeral  1) & \mathcal{F}_A\equiv
						 \left\{\phi_{A}^{*}(\Lambda_{n+}^k)\to
						 \phi_{A}^{*}\big(\Delta(n)_{+}\big) \right\}^{\perp}_{0\le k\le n >0} \\\\
		 (\romannumeral  2) & \mathcal{W}_{A}\cap\mathcal{F}_A\equiv
		 \left\{\phi_{A}^{*}(\partial\Delta(n)_{+})\to\phi_A^{*}\left(
										 \Delta(n)_{+}
		 \right)\right\}^{\perp}_{n>0}
		 \end{array}
		 \label{eq:25.15}
\end{equation}
\section{Commutativity}
\begin{definition}
				For $A\in \text{CBio}$, $M\in A\text{-set}$ will be called \myemph{commutative} if for
				$b\in A^{-}(n)$, $d\in A^{+}(n)$, $b^{\prime}\in A^{-}(k)$, $d^{\prime}\in
				A^{+}(k)$, $m=(m_{ij})\in M^{n\times k}$, we have \\
				\underline{interchange:} \\
				\begin{equation}
								\big< b\circ \underbrace{(b^{\prime}\cdots
								b^{\prime})}_{n}\sigma_{n,k},m,\underbrace{(d^{\prime},\cdots ,
								d^{\prime})}_{n}\circ d \big>_{n\cdot k} = \big< b^{\prime}\circ
								\underbrace{(b,\cdots , b)}_{k},m,\sigma_{n,k}\underbrace{(d,\cdots , d)}_{k}\circ
								d\big>_{n\cdot k}
				\end{equation}
				We let $CA\text{-Set}\subseteq A\text{-Set}$ (resp. $CA\text{-sSet}\subseteq
				A\text{-sSet}$) denote the full subcategory of commutative $\text{A-Sets}$
				(resp. in each dimension $n\ge 0$). 
				\label{def:commuative-A-set}
\end{definition}
\begin{definition}
				We say $A\in \text{CBio}$ is \myemph{$n$-commutative} if the sets $A^{-}(n)$ and
				$A^{+}(n)$ are commutative, and we let $C_{n}\text{Bio}\subseteq \text{CBio}$
				denote the full subcategory of $n$-commutative distributive bios. We let
				$C_{\infty}\text{Bio}=\bigcap\limits_{n\ge 1}C_n\text{Bio}$. We say $A$ is
				totally-commutative if for $b\in A^{-}(n)$, $b^{\prime}\in A^{-}(k)$ (resp.
				$b\in A^{+}(n)$, $b^{\prime}\in A^{+}(k)$), we have 
				\begin{equation*}
								b\circ \underbrace{\left(b^{\prime},\cdots , b^{\prime}
								\right)}_{n} \sigma_{n,k} = b^{\prime}\circ\underbrace{(b,\cdots , b)}_{k} \qquad
								\text{resp.} \; (b^{\prime}\cdots b^{\prime})\circ b = \sigma_{n,k} (b,\cdots , b)\circ
								b^{\prime}. 
				\end{equation*}
				We let $C_{\text{tot}}\text{Bio}\subseteq C_{\infty}\text{Bio}$ denote the full
				subcategory of totally-commutative bios; note that $C\text{Rig}\subseteq
				C_{\text{tot}}\text{Bio}$, and so also $\mathbb{F}$, $\mathbb{Z}_{R}$,
				$\mathbb{Z}_{C}\in C_{\text{tot}}\text{Bio}$. 
				\label{def:c-bio-n-commutative}
\end{definition}
\noindent For $A\in C_{\text{tot}}\text{Bio}$ we have $CA\text{-Set}\equiv A\text{-Set}$. \vspace{.1cm}\\
The inclusion $CA\text{-Set}\hookrightarrow A\text{-Set}$ has a left adjoint $M\mapsto
M_{C}$; in particular, we have the free-commutative-$A\text{-set}$ generated by a set
$V$, $A_{C}^{V}$.
\section{Symmetric monoidal structure}
The category $A$-Set (resp. $A\text{-sSet}$) has a symmetric monoidal structure. \\ 
We only describe the
symmetric monoidal structure on $CA\text{-Set}$ (resp. $CA\text{-sSet}$), which is closed.
\\
\noindent For $M,N,K\in CA\text{-set}$, let the ``bilinear maps'' be defined by 
\begin{equation}
				\text{Bil}_{A}(M,N;K)= 
				\left\{ \begin{array}[H]{ll}
								\varphi: M\wedge N\to K, & 
								\begin{array}[H]{l}
												\varphi\left( \left< a,m_{j},a^{\prime}\right>,n \right) =
												\left<a,\varphi\left( m_{j}\wedge n \right),a^{\prime}\right>, \\
												\varphi\left( m,\left<a,n_{j},a^\prime\right> \right) = \left<
												a,\varphi\left( m\wedge n_{j} \right),a^{\prime}\right>
								\end{array}
\end{array}\right\}
\label{eq:27.1}
\end{equation}
It is a functor in $K$, and as such it is representable 
\begin{equation}
				\text{Bil}_{A}\left( M,N;K \right)\equiv CA\text{-set}\left( M\otimes_{A}N,K \right)
				\label{eq:27.2}
\end{equation}
Where $M\otimes_{A} N$ is the free-commutative-A-Set on $M\wedge N$ modulo the equivalance
relation generated by the $A$-bilinear relations, and where $\otimes:M\wedge N\to
M\otimes_{A}N$ is the universal $A$-bilinear map. The elements of $M\otimes_{A} N$ can be
written, non-uniquely, as $\left<a,m_{j}\otimes n_{j},a^{\prime}\right>_{n}$. We thus get a
bi-functor
\begin{equation}
				\text{\phantom{k}}_{\text{\phantom{k}}_{-}}\otimes_{A_{-}}:CA\text{-Set}\times CA\text{-Set}\to
				CA\text{-Set}
				\label{eq:27.3}
\end{equation}
giving a symmetric monoidal structure on $CA\text{-Set}$, with unit $A(1)_{C}$.
\vspace{.1cm} \\This
symmetric monoidal structure on $CA\text{-Set}$ is closed,
\begin{equation}
				CA\text{-Set}\left( M\otimes_{A}N,K \right)\equiv CA\text{-Set}\left(
								M,\text{Hom}_{A}(N,K)
				\right)
				\label{eq:27.4}
\end{equation}
with the internal Hom functor 
\begin{equation}
				\begin{array}[H]{c}
				\text{Hom}_{A}(_{-},_{-}): \left( A\text{-Set} \right)^{\text{op}}\times
				CA\text{-Set}\to CA\text{-Set} \\
				\text{Hom}_{A}(M,N):=A\text{-Set}(M,N)
				\end{array}
				\label{eq:27.5}
\end{equation}
where the $A$-action on $\text{Hom}_{A}(M,N)$ is given by
\begin{equation*}
				\left< b, \varphi_{j},d\right>_{n}(m):= \left<
				b,\varphi_{j}(m),d\right>_{n}, \quad b\in A^{-}(n),\; d\in A^{+}(n)\quad
				\varphi_{j}\in \text{Hom}_{A}(M,N)\quad m\in M.
\end{equation*}
This in itself is a map of $A$\text{-Set} because of commutativity, and all the other
properties: (associativity, unit, distributivity commutativity) follow from their validity
in $N$. \\
The tensor product commute with extension of scalars: for $\varphi \in
\text{CBio}(B,A)$, and for $M,N\in CB\text{-Set}$, we have
\begin{equation}
				\varphi^{\ast}\left( M\otimes_{B} N \right)\cong \varphi^{\ast}(M)\otimes_{A}\varphi^{\ast}(N)\quad , \quad
				\varphi^{\ast}B(1)_{C}=A(1)_{C}
				\label{eq:27.6}
\end{equation}
We have therefore the adjunction formula 
\begin{equation}
				\varphi_{\ast}\text{Hom}_{A}\left( \varphi^{\ast}M,N \right)\cong \text{Hom}_{B}\left( M,\varphi_{\ast}N \right)
				\label{eq:27.7}
\end{equation}
The tensor product is distributive over sums, 
\begin{equation}
				M\otimes_{A}\left( \coprod\limits_{i}N_{i} \right)\cong
				\coprod_{i}(M\otimes_{A}N_{i}).
				\label{eq:27.8}
\end{equation}
and more generally commutes with colimits. \vspace{.1cm}\\
The category $CA\text{-sSet}$ inherit a closed symmetric monoidal structure 
\begin{equation}
				\begin{array}[H]{c}
				\text{\phantom{k}}_{-}\otimes_{A-}:\, CA\text{-sSet}\otimes
				CA\text{-sSet}\to{}
				CA\text{-sSet} \vspace{.2cm}\\
				(M\otimes_{A}N)_{n}:= M_{n}\otimes_{A_{n}}N_{n} \\\\
				\text{Hom}_{A}(_{-},_{-}):\left( CA\text{-sSet} \right)^{\text{op}}\otimes
				CA\text{-sSet}\xmapsto{} CA\text{-sSet} \\\\
				\text{Hom}_{A}(M_{\cdot},N_{\cdot})_{n}:= CA\text{-sSet}\left(
				M_{\cdot}\otimes \phi_{A}^{\ast}\left( \Delta(n)_{+} \right), N_{\cdot} \right)
				\end{array}
				\label{eq:27.9}
\end{equation}
Here $\phi_{A}^{\ast}\left( \Delta(n)_{+} \right)$ is the free-\myemph{commutative}
$A$-sSet
on $\Delta(n)$. \vspace{.1cm}\\
The category $CA\text{-sSet}$ has a simplicial, cellular, Quillen model structure, exactly
as in (\ref{eq:25.13}); for the free maps, and for the generating (trivial) cofibrations we take sums,
and free objects $\phi^{\ast}_{A}(_{-})$, in the category $CA\text{-sSet}$. \\
The model structure is compatible with the symmetric monoidal structure: \\
For $\left\{ i_{\cdot}\; :\; N\to N^{\prime}_{\cdot} \right\}$, $\left\{ j_{\cdot}\; :\;
M_{\cdot}\to M_{\cdot}^{\prime} \right\}$ in $\mathcal{C}_{A}$, we have the
push-out-product 
\begin{equation}
				i_{\cdot}\boxtimes  j_{\cdot}  \, : \,  \left( N_{\cdot} \otimes_{A} M_{\cdot}^{\prime} \right)
				\coprod\limits_{N_{\cdot}\otimes_{A} M_{\cdot}} \left(
				N_{\cdot}^{\prime}\otimes_{A}M_{\cdot} \right)\to
				N^{\prime}_{\cdot}\otimes_{A}M_{\cdot}^{\prime}
				\label{eq:27.10}
\end{equation}
It is also a cofibration in $\mathcal{C}_{A}$, and moreover, if $i_{\cdot}$ is in $\mathcal{W}_{A}$, also
$i_{\cdot}\boxtimes j_{\cdot}$ is in $\mathcal{W}_{A}$. \vspace{.1cm}\\
To see this compatibility of the
monoidal and model structures, we may assume $i_{\cdot}$ and $j_{\cdot}$ are free 
\begin{equation*}
				N_{n}\coprod\phi^{\ast}_{A_{n}}(V_{n})\xrightarrow{\;\sim\;}{} N^{\prime}_{n},
				\quad  M_{n}\coprod\phi^{\ast}_{A_{n}}\left( W_{n}
				\right)\xrightarrow{\;\sim\;}{} M^{\prime}_n
\end{equation*}
and then 
\begin{equation}
				\left( N_{n}\otimes_{A_{n}}
								M_{n}^{\prime}\coprod\limits_{N_{n}\otimes_{A}M_{n}}
								N_{n}^{\prime}\otimes_{A_{n}} M_{n} \right) \coprod\phi^{\ast}_{A_{n}}(V_{n})
								\otimes_{A_{n}}\phi^{\ast}_{A_{n}}(W_{n})\xrightarrow{\;\;\sim\;\;} N^{\prime}_{n}\otimes_{A_{n}}
M^{\prime}_{n}
\label{eq:27.11}
\end{equation}
and since $\phi^{\ast}_{A_{n}}(V_{n})\otimes_{A_{n}}\phi^{\ast}_{A_{n}}(W_{n})\equiv
\phi^{\ast}_{A_{n}} (V_{n}\prod W_{n})$, we see that $i_{\cdot}\boxtimes j_\cdot$ is also
free.\vspace{.1cm}\\
A map $\varphi\in s\text{CBio}(B_{\cdot},A_{\cdot})$ induces a Quillen adjunction 
\begin{equation}
				\begin{tikzpicture}[baseline=20mm]

								\node (a) at (4,4) {$\text{$CA_{\cdot}$-sSet}$};
								\node (b) at (4,2) {$\text{$CB_{\cdot}$-sSet}$};
								 \draw [->] (a) to [out=330,in=30,right] node{$\varphi_{*}$} (b);
								 \draw [<-] (a) to [out=210,in=150,left] node{$\varphi^{*}$} (b);
				\end{tikzpicture}
				\label{eq:27.12}
\end{equation}
When $A$ is an ordinary commutative ring, the bio $\biop_{A}$ is totally commutative, and
we have by the Dold-Kan correspondence 
\begin{equation}
				C\biop_{A}\text{-sSet}\equiv \biop_{A}\text{-sSet} \equiv \left(
				\biop_{A}\text{-Set} \right)^{\bbDelta^{\text{op}}} \equiv Ch_{\ge 0}\left(
								A\text{-mod}
				\right).
				\label{eq:27.13}
\end{equation}
The model structure on $C\biop_{A}\text{-sSet}$ corresponds to the projective model
structure on $Ch_{\ge 0}\left( A\text{-mod} \right)$, which embeds in the
\myemph{stable} model structure $Ch_{+}(A\text{-mod})$. For general $A_{\cdot}\in
s\text{CBio}$, the model structure on $CA_{\cdot}\text{-sSet}$ is not stable, and we shall
stabilize it, preserving the symmetric monoidal structure by using the symmetric spectra
of \cite{MR1860878}, \cite{MR1695653}.
\section{Symmetric Sequences}
Let $S_n$ denote the symmetric group on $n$ letters, and let
$\sum=\coprod\limits_{n\ge 0} S_{n}$ denote the category of finite bijection; it is
equivalent to the category $\text{Iso}(\text{Fin})$ of bijections of finite sets. The
category of symmetric sequences in $CA\text{-sSet}$ is 
\begin{equation}
				\sum(A_{\cdot}):= \left( CA_{\cdot}\text{-sSet} \right)^{\sum}\cong \left(
								CA_{\cdot}\text{-sSet}
				\right)^{\text{Iso}(\text{Fin})}
				\label{eq:28.1}
\end{equation}
It has objects $M^{\cdot}=\left\{ M_{\cdot}^{n} \right\}_{n\ge 0} $ with
$M_{\cdot}^{n}\in \left( CA.\text{-sSet} \right)^{S_{n}}$ a simplicial commutative
$A_{\cdot}\text{-set}$ with an action of $S_{n}$. The category $\sum(A_{\cdot})$ is
complete and co-complete. \\
The category $\sum(A)$ has a closed symmetric monoidal structure 
\begin{equation}
				\begin{array}[H]{c}
								\phantom{\;}_{-}\otimes_{\sum(A)^{-}}: \sum(A)\times \sum(A)\to \sum(A) \\\\	
				\left( M^{\cdot}\otimes_{\sum(A)}N^{\cdot} \right)^{n} :=
				\coprod\limits_{p+q=n} S_{n} 
				\begin{array}[t]{c}
								\times \vspace{-.2cm}\\ \scriptstyle S_{p}\times S_{q}
				\end{array} 
				\left( M^{p}\otimes_{A} N^{q} \right)
				\end{array}
				\label{eq:28.2}
\end{equation}
Here the induction functor is the left adjoint of the forgetful functor
\begin{equation}
				\left(CA_{\cdot}\text{-sSet}\right)^{S_{n}}\to \left(
				CA_{\cdot}\text{-sSet} \right)^{S_{p}\times S_{q}}
				\label{eq:28.3}
\end{equation}
and is given by 
\begin{equation}
				S_{n}
				\begin{array}[t]{c}
								\times \vspace{-.2cm}\\ \scriptstyle S_{p}\times S_{q}
				\end{array} (M) := 
\coprod\limits_{S_{n}/S_{p}\times S_{q}} M
\label{eq:28.4}
\end{equation}
Equivalently, writing $M^{\cdot},N^{\cdot}\in \sum(A)$ as functor
$\text{Iso}(\text{Fin})\to CA\text{-sSet}$ we have 
\begin{equation}
				\left( M^{\cdot}_{\cdot}\otimes_{\sum(A)} N^{\cdot}_{\cdot} \right)^{n}
				= \coprod_{n=n_{0}\coprod n_{1}} M^{n_{0}}_{\cdot}\otimes_A N^{n_{1}}. 
				\label{eq:28.9}
\end{equation}
the sum over all decomposition of $n$ as a disjoint union of subsets $n_{0},n_{1}\subseteq
n$. The unit of this monoidal structure is the symmetric sequence 
\begin{equation}
				\unit_{A}:= \left( A(1)_{C},0,0,0,\cdots  \right)
				\label{eq:28.10}
\end{equation}
\begin{remark}
				Note that this monoidal structure is symmetric, 
				\begin{equation*}
								\mathfrak{T}_{M^{\cdot},N^{\cdot}} : M^{\cdot}\otimes_{\sum(A)}
								N^{\cdot}\cong N^{\cdot}\otimes_{\sum(A)} M^{\cdot}
				\end{equation*}
				This symmetry is clear in the formula (\ref{eq:28.9}) $M^{n_{0}}\otimes_{A} N^{n_{1}}\cong
				N^{n_{1}}\otimes_{A}M^{n_0}$  but in formula (\ref{eq:28.2}) the symmetry isomorphisms 
				\begin{equation*}
								S_{n} 
												\begin{array}[t]{c}
																\times \vspace{-.2cm}\\ \scriptstyle S_{p}\times S_{q}
												\end{array}
								\left( M^{p}\otimes_{A} N^{q} \right) \cong S_{n}
												\begin{array}[t]{c}
																\times \vspace{-.2cm}\\ \scriptstyle S_{q}\times S_{p}
												\end{array}
												\left( N^{q}\otimes_{A} M^{p} \right)
				\end{equation*}
				involves the $(p,q)$-shuffle $\omega_{p,q}\in S_n$ that conjugates $S_{p}\times
				S_{q}$ to $S_{q}\times S_{p}$. \vspace{.2cm}\\
				The internal Hom is given by 
				\begin{equation}
								\begin{array}[H]{c}
												\text{Hom}_{\sum(A)}(_{-},_{-}) : \sum(A)^{\text{op}}\times
												\sum(A)\rightarrow \sum(A) \\\\
												\text{Hom}_{\sum(A)}(M^{\cdot},N^{\cdot})^{n} := \prod_{k\ge 0}
												\text{Hom}_{A}\left(M^{k},N^{k+n}\right) \\\\
												\sum(A)\left( M^{\cdot}\otimes_{\sum(A)} N^{\cdot},K^{\cdot}
												\right)\cong \sum(A)\left(
												M^{\cdot},\text{Hom}_{\sum(A)}(N^{\cdot},K^{\cdot}) \right)
								\end{array}
								\label{eq:28.12}
				\end{equation}
				For a homomorphism $\varphi\in s\text{CBio}(B,A)$, we have adjunction
				\begin{equation}
								\begin{tikzpicture}[baseline=20mm]
												\node (a) at (4,4) {$\sum(A)$};
												\node (b) at (4,2) {$\sum(B)$};
												 \draw [->] (a) to [out=330,in=30,right] node{$\varphi_{*}$} (b);
												 \draw [<-] (a) to [out=210,in=150,left] node{$\varphi^{*}$} (b);
								\end{tikzpicture}
								\label{eq:28.13}
				\end{equation}
and $\varphi^{*}$ is strict-monoidal 
\begin{equation}
				\varphi^{*}\left( M^{\cdot}\otimes_{\Sigma(B)}N^{\cdot} \right)\cong
				\varphi^{*}(M^{\cdot})\otimes_{\Sigma(A)}\varphi^{*}(N^{\cdot}), \quad
				\varphi^{*}\left( \characteristic_{B} \right)\cong \characteristic_{A}.
				\label{eq:28.14}
\end{equation}
\label{rem:28.7}
\end{remark}
\begin{remark}
				One can think of the elements of $\sum(A)$ as ``Fourier-coefficients'', and
				associate with $M^{\cdot}\in \sum(A)$ the ``Fourier-transform'' given by the
				``analytic'' functor
				\begin{equation*}
								\begin{array}[H]{l}
								\widehat{M^{\cdot}}\, :\, CA\text{-sSet}\longrightarrow
								CA_{\cdot}\text{-sSet} \\\\
								\widehat{M}^{\cdot}(X):=\coprod\limits_{n\ge 0}
								M^{n}\otimes_{A}^{S_{n}}X^{\otimes n} \\\\
								M^{n}\otimes_{A}^{S_{n}} X^{\otimes n}:=
								M^{n}\otimes_{A}(\underbrace{X_{\cdot}\otimes\cdots \otimes
								X_{\cdot}}_{n})/_{\sigma m\otimes x_{1}\otimes \cdots \otimes x_{n} \sim
								m\otimes x_{\sigma(1)}\otimes \cdots \otimes x_{\sigma(n)} }
								\end{array}
				\end{equation*}
				The Fourier-transform converts ``convolution'' ($\equiv$ the symmetric product in
				$\Sigma(A)$) to ``multiplication''
				\begin{equation}
								\begin{array}[H]{lll}
												\left( M^{\cdot}\otimes_{\Sigma(A)}N^{\cdot} \right)^{\Lambda}(X) &=
												\coprod\limits_{n\ge 0} \left( M^{\cdot}\otimes_{\Sigma(A)}
												N^{\cdot} \right)^{n} \otimes_{A}^{S_{n}} X^{\otimes n}	 \\\\
												&= \coprod\limits_{p,q\ge> 0} \left( M^{p}\otimes_{A} N^{q}
												\right) 
																\begin{array}[H]{c}
																								\scriptstyle \times \vspace{-.2cm}\\
																								\scriptstyle S_{p}\times S_{q} 
																\end{array}
												 S_{p+q}
												\otimes_{A}^{S_{p+q}} X^{\otimes(p+q)} \\\\
												&= \coprod_{p,q\ge 0} M^{p}\otimes_{A} N^{q}
												\otimes_{A}^{S_{p}\times S_{q}} X^{\otimes (p+q)} \\\\
												&= \left( \coprod\limits_{p\ge 0} M^{p}\otimes_{A}^{S_{p}}
												X^{\otimes p} \right)\otimes_{A} \left( \coprod\limits_{q\ge 0}
												N^{q}\otimes_{A}^{S_{q}} X^{\otimes q} \right) \\\\
												&= \widehat{M^{\cdot}}(X) \otimes_{A} \widehat{N^{\cdot}}(X)
								\end{array}
								\label{eq:28.16}
				\end{equation}
				\label{rem:28.11}
\end{remark}
\section{The sphere spectrum}
\label{sec:29}
The categogry $\Sigma(\mathbb{F})\equiv (\text{Set}_0)^{\Sigma_{.}\times \BDelta^{\op}}$ is the
usual category of symmetric sequence of pointed simplicial sets, and in particular contains
the \underline{sphere-spectrum}: 
\begin{equation}
				S^{\cdot}_{\mathbb{F}} := \left\{ S^{n}=\underbrace{S^{1}\wedge\dots \wedge
				S^{1}}_{n} \right\}_{n\ge 0} 
				\label{eq:29.1}
\end{equation}
with the permutation action of $S_{n}$ on $S^{n}$. \vspace{.1cm}\\
The sphere-spectrum is a monoid
object of $\Sigma(\mathbb{F})$, with multiplication
\begin{equation}
				\begin{array}[H]{c}
								m:
								S^{\cdot}_{\mathbb{F}}\otimes_{\Sigma(\mathbb{F})}S^{\cdot}_{\mathbb{F}}\to
								S^{\cdot}_{\mathbb{F}} \\\\
								m(S^{n}\otimes_{\mathbb{F}} S^{m}) \equiv m(S^n\wedge S^m)
								\equiv S^{n+m} 
				\end{array}
				\label{eq:29.2}
\end{equation}
Note that it is a \underline{commutative} monoid,
$m=m\circ\mathfrak{T}_{{S_{\mathbb{F}}^{\cdot}},S^{\cdot}_{\mathbb{F}}}$, cf. remark
(\ref{rem:28.7}). \\
The unit is given by the embedding 
\begin{equation}
				\varepsilon: \characteristic_{\mathbb{F}}\equiv \left(
				\mathbb{F}(\unit),0,0,\dots \right) \equiv \left( S^{0},0,0,\dots
				\right)\hookrightarrow S^{\cdot}_{\mathbb{F}}, \; \text{using}\;
				\mathbb{F}(1)=\left\{ 0,1 \right\} \equiv S^{0}.
				\label{eq:29.3}
\end{equation}
We write $S^{\cdot}_{A}=\phi^{*}_{A}S^{\cdot}_{\mathbb{F}}$ for the corresponding
commutative monoid object of $\Sigma(A)$. We let $S^{\cdot}_{A}\mymod\subseteq
\Sigma(A)$ denote the sub-category of $S_{A}^{\cdot}$-modules, this is the category of
``symmetric spectra''. It has objects the symmetric sequences \linebreak $M^{\cdot}=\left\{
				M^{n}
\right\}\in \Sigma(A)$, together with associative unital 
 $S_{A}^{\cdot}$-action \linebreak
$S_{A}^{\cdot}\otimes_{\Sigma(A)}M^{\cdot}\xrightarrow{m}M^{\cdot}$, or equivalently,
associative unital, $S_{p}\times S_{q}\hookrightarrow S_{p+q}$ covariant,
action $S^{p}\wedge M^{q}\to M^{p+q}$. The maps in $S_{A}^{\cdot}\mymod$ are the maps in
$\Sigma(A)$ that preserve the $S_{A}^{\cdot}$-action. The category $S_{A}^{\cdot}\mymod$ is
complete and co-complete. \vspace{.1cm}\\
The category $S_{A}^{\cdot}\mymod$ has a closed symmetric monoidal structure 
\begin{equation}
				\begin{array}[H]{c}
								\_\otimes_{S^{\cdot}_{A}}\_: S^{\cdot}_{A}\mymod\times
								S_{A}^{\cdot}\mymod\to S^{\cdot}_{A}\mymod \\\\
								M^{\cdot}\otimes_{S_{A}^{\cdot}}N^{\cdot} := \text{Cok}\left\{
												M^{\cdot}\otimes_{\Sigma(A)}S_{A}^{\cdot}\otimes_{\Sigma(A)}N^{\cdot}\overset{\xrightarrow{m\otimes\text{id}_{N^{\cdot}}}}{
								\xrightarrow[\text{id}_{M^{\cdot}}\otimes m]{}} M^{\cdot}\otimes_{\Sigma(A)}N^{\cdot} \right\}
								\end{array}
								\label{eq:29.4}
\end{equation}
The unit is
\begin{equation}
				S_{A}^{\cdot} := \left\{ \phi_{A}^{*}S^{0},\phi_{A}^{*}S^{1},\ldots ,
				\phi_{A}^{*}S^{n},\ldots \right\}. 
				\label{eq:29.5}
\end{equation}
The internal $\myhom$ is given by 
\begin{equation}
				\begin{array}[H]{cc}
				\myhom_{S_{A}}(\_,\_): \left( S_{A}^{\cdot}\mymod \right)^{\op}\times
				S_{A}^{\cdot}\mymod\to S_{A}^{\cdot}\mymod \\\\
				\myhom_{S_A}(M^{\cdot},N^{\cdot}):= \kernel \left\{
				\myhom_{\Sigma(A)}(M^{\cdot},N^{\cdot})\rightrightarrows 
				\myhom_{\Sigma(A)}\left( S^{\cdot}_{A}\underset{{\Sigma(A)}}{\otimes} M^{\cdot},N^{\cdot}
\right)\right\}.
				\end{array}
				\label{eq:29.6}
\end{equation}
\begin{equation}
				S_A\mymod\left(M^{\cdot}\otimes_{S^{\cdot}_{A} }N^{\cdot},K^{\cdot}  \right)
				\equiv S_{A}\mymod\left(
				M^{\cdot},\myhom_{S^{\cdot}_{A}}(N^{\cdot},K^{\cdot}) \right)
				\label{eq:29.7}
\end{equation}
The category $S_{A}^{\cdot}\mymod$ is tensored, co-tensored, and enriched over pointed
simplicial sets  
\begin{equation}
				s\text{Set}_{0} \equiv \left( \text{Set}_{0} \right)^{\bbDelta^{\text{op}}} \equiv
				\mathbb{F}\text{-sSet}.
				\label{eq:29.8}
\end{equation}
The enrichment is given via the mapping space
\begin{equation}
				\left({N^{\cdot}}^{M^{\cdot}}\right)_{n} \equiv S_{A}\text{-mod}\left( M^{\cdot}
				\otimes_{A}\phi_{A}^{*}\left( \Delta(n)_{+} \right),N^{\cdot} \right)
				\label{eq:29.9}
\end{equation}
For $M\in CA\text{-sSet}$, the ``frea-$S_{A}$-module of level $n$'' on
$M_{\cdot}$ 
\, ,
$F_{n}(M_{\cdot})$, is the value at $M_{\cdot}$ of the left-adjoint $F_{n}$ of the
forgetfull functor
\begin{equation}
				\begin{array}[H]{c}
				CA_{\cdot}\text{-sSet}\longleftarrow \left( CA_{\cdot}\text{-sSet}
				\right)^{S_{n}} \longleftarrow S_{A}^{\cdot}\text{-mod} \\\\
				F_{n}(M_{\cdot}) = S^{\cdot}_{A}\otimes_{\sum(A)} \left( S_{n}\times M_{\cdot}
				\right)[n] = \left( 0,\dots , 0 , S_{n}\times M_{\cdot},\cdots ,
				S_{n+p}\times_{S_{p}} S^{p}\otimes_{A} M_{\cdot}, \cdots \right)
				\end{array}
				\label{eq:29.10}
\end{equation}
We say $M^{\cdot}\in S^{\cdot}_{A}\mymod$ is \underline{an $\Omega$-spectrum}, or a
fibrant object $M^{\cdot}\in \left( S^{\cdot}_{A}\mymod \right)_{\mathcal{F}}$ if
$M^{\cdot}$ is levelwise fibrant, and if the adjoint of the action map  
\begin{equation}
				m^{1,n}: S_{A}^{1} \otimes_{A} M^{n}\to M^{n+1} 
				\label{eq:29.11}
\end{equation}
is a weak equivalence
\begin{equation}
				\left( m^{1,n} \right)^{\natural} : M^{n}\xrightarrow{\sim}\myhom_{A}\left(
				S_{A}^{1},M^{n+1} \right) = (M^{n+1})^{S^{1}}:= \Omega M^{n+1}
				\label{eq:29.12}
\end{equation}
The \underline{stable} model structure on $S^{\cdot}_{A}\mymod$ is a Bousfield localization
of the projective model structure, having the same cofibrations, but with the fibrant
objects being $\left( S^{\cdot}_{A}\mymod \right)_{\mathcal{F}}$ 
\begin{equation}
				\begin{array}[H]{ll}
								\text{\underline{Cofibrations}:}			&
								\mathcal{C}_{S^{\cdot}_{A}}= \text{\phantom{\ }}^{\perp}\left\{
												\mathcal{W}_{S^{\cdot}_{A}}^{\text{lev}}\cap
				\mathcal{F}_{S^{\cdot}_{A}}^{\text{lev}} \right\}												 \\\\
				\text{\underline{Weak equivalences}:} &	 \mathcal{W}_{S^{\cdot}_{A}}=\left\{ 
									\begin{array}[H]{l}
													f^{\cdot}\in S^{\cdot}_{A}\mymod\left( M^{\cdot},N^{\cdot} \right), \\
													\left( f^{*}: X^{N^{\cdot}}\to X^{M^{\cdot}} \right) \in
													\mathcal{W}_{\mathbb{F}}  \\
													\text{for all } X^{\cdot}\in \left( S^{\cdot}_{A}\mymod \right)_{\mathcal{F}} 
									\end{array} 
								\right\}  \\\\
				\text{\underline{Fibrations}:} &	\mathcal{F}_{S^{\cdot}_{A}} =  \left\{
												\mathcal{C}_{{S}^{\cdot}_{A}} \cap \mathcal{W}_{{S}^{\cdot}_{A}}
								\right\}^{\perp}
				\end{array}
				\label{eq:29.13}
\end{equation}
It is left proper, simplicial, cellular, Quillen Model structure compatible with its
symmetric monoidal structure. It is moreover \underline{stable}. We let
\begin{equation}
				\mathbb{D}(A)=\text{Ho}\left( S^{\cdot}_{A}\mymod \right)\cong
				S^{\cdot}_{A}\mymod\left[ \mathcal{W}_{S^{\cdot}_{A}}^{-1} \right]
				\label{eq:29.14}
\end{equation}
denote the associated homotopy category, this is the {\underline{derived}} category of
$\aset$. It is a triangulated symmetric monoidal category. \vspace{.1cm}\\ 
Given a homomorphism $\varphi\in s\text{CBio}(B_{\cdot},A_{\cdot})$ we get an induce
adjunction
\begin{equation}
				\begin{tikzpicture}[baseline=20mm]
								\node (a) at (4,4) {$\mathbb{D}(A)$};
								\node (b) at (4,2) {$\mathbb{D}(B)$};
								\draw [->] (a) to [out=330,in=30,right] node{$\mathbb{R}\varphi_{*}$} (b);
								\draw [<-] (a) to [out=210,in=150,left] node{$\mathbb{L}\varphi^{*}$} (b);
				\end{tikzpicture}
				\label{eq:29.15}
\end{equation}
and $\mathbb{L}\varphi^{\ast}$ is a monoidal functor. We have the adjunction formula
for $M^{\cdot}\in \mathbb{D}(B)$, $N^{\cdot}\in \mathbb{D}(A)$, 
\begin{equation}
				\mathbb{R}\varphi_{*}\left( \mathbb{R}\text{Hom}_{S^{\cdot}_{A}} \left(
												\mathbb{L}\varphi^{\ast}M^{\cdot},N^{\cdot}
				\right) \right) \equiv \mathbb{R}\text{Hom}_{S^{\cdot}_{B}}\left(
				M^{\cdot},\mathbb{R}\varphi_{*}N^{\cdot} \right)
				\label{eq:29.16}
\end{equation}
\section{Quasi-coherent $\mathscr{O}_{X}$-modules}
\begin{definition}
				For $X\in \text{GSch}$, a pre-$\mathscr{O}_{X}$-Set (resp. pre-$C\mathscr{O}_{X}\text{-Set}$,
				pre-$C\mathscr{O}_{X}\text{-sSet}$) is an assignment $M$ for each open
				$\mathscr{U}\subseteq X$ of an 
				$\mathscr{O}(\mathscr{U})\text{-Set} \; M(\mathscr{U})$  (resp.
								$C\mathscr{O}_{X}(\mathscr{U})\text{-Set}$,
				$C\mathscr{O}_{X}(\mathscr{U})\text{-sSet}$), and for $\mathscr{V}\subset \mathscr{U}$, on
				$\mathscr{O}_{X}(\mathscr{U})\text{-map}$ $r_{\mathscr{V},\mathscr{U}}:
				M(\mathscr{U})\to
				M(\mathscr{V})$, and $r_{\mathscr{W},\mathscr{V}}\circ
				r_{\mathscr{V},\mathscr{U}}=r_{\mathscr{W},\mathscr{U}}$,
				$r_{\mathscr{U},\mathscr{U}}=\text{id}_{M(\mathscr{U})}$. The maps are natural transformation commuting
				with the $\mathscr{O}_{X}$-action (resp. and the simplicial action
				$\bbDelta^{\text{op}}$); We have the full subcategories of sheaves, denoted by omitting the prefix
				``pre''. Thus we have full-embeddings and adjunctions (using the left adjoint of
				the inclusion of sheaves in pre-sheaves, $M\mapsto M^{\natural}$ the shiftication functor)
\begin{equation}
				\begin{tikzpicture}[baseline=0mm]
								\node  at (10mm,40mm) {pre-$C\mathscr{O}_{X}$-Set};
								\node  at (40mm,40mm) {pre-$\mathscr{O}_{X}$-Set};
								\node  at (10mm,20mm) {pre-$C\mathscr{O}_{X}$-sSet};
								\node  at (40mm,20mm) {$C\mathscr{O}_{X}$-Set};
								\node  at (70mm,20mm) {$\mathscr{O}_{X}$-Set};
								\node  at (40mm,5mm) {$C\mathscr{O}_{X}$-sSet};
								\coordinate (a) at (21mm,42mm);
								\coordinate (b) at (31mm,42mm);
								\draw[<<-] plot[smooth ] coordinates {(a)(26mm,43mm)(b)};
								\draw[{Hooks[right,length=5,width=6]}->] ($(a)-(0,3mm)$)--($(b)-(0,3mm)$); 
								\node at (27mm,45mm) {$\scriptstyle (\;)_{c}$};
								\draw[ arrows={[harpoon, left]<->>} ] (4mm,23mm)--(4mm,38mm);
								\draw[<-{Hooks[right,length=5,width=6]}] (10mm,23mm)--(10mm,38mm);
								\node at (1mm,30mm) {$\pi_{0}$};
								\draw[{Hooks[left,length=5,width=6]}->] (35mm,23mm)--(21mm,38mm);
								\draw[<<-] (32mm,23mm)--(18mm,38mm);
								\node at (21mm,30mm) {$\scriptstyle (\;)^{\natural}$};

								\draw[ arrows={[harpoon, left]<->>} ] (38mm,7mm)--(38mm,18mm);
								\draw[<-{Hooks[right,length=5,width=6]}] (41mm,7mm)--(41mm,18mm);
								\node at (35.5mm,12mm) {$\pi_{0}$};
				\draw[<<-] plot[smooth ] coordinates {(48mm,22mm)(55mm,23mm)(63mm,22mm)};
				\draw[{Hooks[right,length=5,width=6]}->] (48mm,19mm)--(63mm,19mm);
				\node at (55mm,25mm) {$\scriptstyle (\;)_{c}$};
				\node at (50mm,30mm) {$\scriptstyle (\;)^{\natural}$};

				\draw[{Hooks[left,length=5,width=6]}->] (65mm,23mm)--(48mm,38mm);
				\draw[<<-] (62mm,23mm)--(45mm,38mm);

				\draw[{Hooks[left,length=5,width=6]}->] (34mm,6mm)--(20mm,18mm);
				\draw[<<-] (31mm,6mm)--(17mm,18mm);
				\node at (22mm,10mm) {$\scriptstyle (\;)^{\natural}$};

				\end{tikzpicture}
				\label{eq:30.2}
				\end{equation}
All these categories are complete and co-complete. Limits and colimits for pre-sheaves
(resp. sheaves) are created sectionwise over open subset (resp. for co-limit of sheaves use
shification). They have a symmetric monoidal structure, but it is closed only in the
commutative case: for pre-sheaves (resp. sheaves) the symmetric monoidal structure denoted
$\overset{\sim}{\otimes}_{\mathscr{O}_{X}}$ (resp. $\otimes_{\mathscr{O}_{X}}$) is defined
sctiowise 
\begin{equation}
				\left( M\overset{\sim}{\otimes}_{\mathscr{O}_{X}} N \right) (\mathscr{U}) :=
				M(\mathscr{U})\otimes_{\mathscr{O}_{X}(\mathscr{U})} N(\mathscr{U}) \quad , \quad \Big(\text{resp.} \;
				M\otimes_{\mathscr{O}_{X}}N:= \left(
M\overset{\sim}{\otimes}_{\mathscr{O}_{X}} N \right)^{\natural}\Big).
\label{eq:30.3}
\end{equation}
\label{def:30.4}
\end{definition}
The internal Hom, defined only in the commutative case, will be denoted by
$\text{Hom}_{\mathscr{O}_{X}}$, and is defined for pre-sheaves or sheaves, so that
$\text{Hom}_{\mathscr{O}_{X}}(M,M^{\prime})(\mathscr{U})=\mathscr{O}_{U}\text{-set}(M|_{\mathscr{U}},M^{\prime}|_{\mathscr{U}})$
is the equalizer 
\begin{equation}
				\begin{array}[H]{c}
								\ker \left\{ \prod\limits_{\mathscr{V}\subseteq \mathscr{U}}
												\text{Hom}_{\mathscr{O}_{X}/_{\mathscr{V}}}, \; \left(
				M(\mathscr{V}),M^{\prime}(\mathscr{V}) \right) \rightrightarrows
				\prod\limits_{\mathscr{W}\subseteq\mathscr{V}\subseteq \mathscr{U}}
				\text{Hom}_{\mathscr{O}_{X}(\mathscr{V})} \left(
				M(\mathscr{V}),M^{\prime}(\mathscr{W}) \right) \right\} \\\\
				\begin{tikzpicture}[baseline=0]
								\node at (0mm,0mm) {\normalsize{$\Big\{ \varphi_{_\mathscr{V}} \Big\}$}};
								\draw[|->] (6mm,2mm)-- (30mm,2mm);
								\draw[|->] (6mm,-2mm)-- (30mm,-2mm);
								\node at (40mm,2mm) {$\displaystyle \varphi_{\mathscr{W}}\circ r_{\mathscr{W},\mathscr{V}}$};
								\node at (40mm,-3mm) {$\displaystyle
								r_{\mathscr{W},\mathscr{V}}\circ \varphi_{\mathscr{V}}$};
				\end{tikzpicture}
				\end{array}
				\label{eq:30.4}
\end{equation}
Given $A\in \text{CBio}$, $S\subseteq A(1)$ a multiplicative subset, and $M\in
A\text{-Set}$, (commutative or not!), we can localize $M$ to obtain $S^{-1}M\in
S^{-1}A\text{-Set}$ in the usual way: $S^{-1}M:= (M\times S)/{\simeq}$, with the
equivalence relation $\simeq$ given by 
\begin{equation}
				(m_{1},s_{1})\simeq (m_{2},s_{2})\; \text{if} \; s\cdot s_{2}\cdot m_{1} = s\cdot
				s_{1}\cdot m_{2} \; \text{for some $s\in S$}.
				\label{eq:30.5}
\end{equation}
We denote by $m/s$ the equivalence class of $(m,s)$. \\
The functor $M\mapsto S^{-1}M$ commutes with colimits, finite-limits, tensor products and
Hom's. We obtain the localization functors
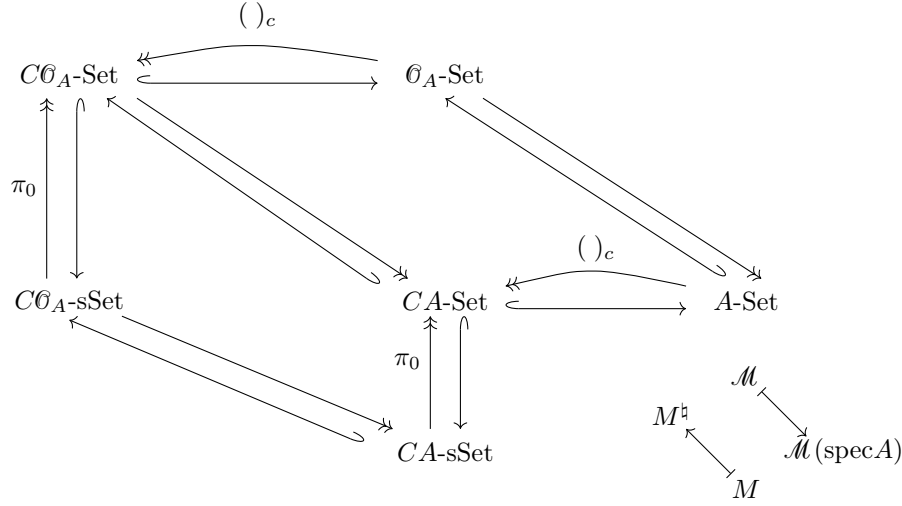
\begin{equation}
				\begin{tikzpicture}[baseline=0mm]

								\node  at (10mm,60mm) {$C\mathscr{O}_{A}\text{-Set}$};
								\node  at (60mm,60mm) {$\mathscr{O}_{A}\text{-Set}$};
								\node  at (10mm,30mm) {$C\mathscr{O}_{A}\text{-sSet}$};
								\node  at (60mm,30mm) {$CA\text{-Set}$};
								\node  at (100mm,30mm) {$A\text{-Set}$};
								\node  at (60mm,10mm) {$CA\text{-sSet}$};
								\node at (90mm,15mm) {$M^{\natural}$};
								\node at (100mm,20mm) {$\mathscr{M}$};
								\node at (113mm,10mm) {$\mathscr{M}(\text{spec}A)$};
								\node at (100mm,5mm) {$M$};
								\draw[<<-] plot[smooth ] coordinates {(19mm,62mm)(35mm,64mm)(51mm,62mm)};
								\node at (35mm,68mm) {$(\;)_{c}$};
								\draw[{Hooks[right,length=5,width=6]}->] (19mm,59mm)--(51mm,59mm); 
								\draw[<<-] (7mm,57mm)--(7mm,33mm);
								\draw[{Hooks[right,length=5,width=6]}->] (11mm,57mm)--(11mm,33mm);
								\draw[->>] (19mm,57mm)--(55mm,33mm);
								\draw[<-{Hooks[left,length=5,width=6]}] (15mm,57mm)--(51mm,32mm);
								\draw[<<-] plot[smooth ] coordinates {(68mm,32mm)(80mm,34mm)(92mm,32mm)};
								\draw[{Hooks[right,length=5,width=6]}->] (68mm,29mm)--(92mm,29mm);
								\node at (80mm,37mm) {$(\;)_{c}$};
								\draw[->>] (17mm,28mm)--(53mm,13mm);
								\draw[<-{Hooks[left,length=5,width=6]}] (10mm,27.5mm)--(49mm,11mm);
								\draw[->>] (58mm,13mm)--(58mm,28mm);
								\draw[{Hooks[right,length=5,width=6]}->] (62mm,28mm)--(62mm,13mm);
								\node at (55mm,22mm) {$\pi_{0}$};
								\draw[|->] (102mm,18mm)--(108mm,12mm);
								\draw[|->] (98mm,7mm)--(92mm,13mm);

								\coordinate (a) at (21mm,42mm);
								\coordinate (b) at (31mm,42mm);
								
								\draw[->>] (65mm,57mm)--(102mm,33mm);
								\draw[<-{Hooks[left,length=5,width=6]}] (60mm,57mm)--(97mm,33mm);
								\node at (4mm,45mm) {$\pi_{0}$};								
				\end{tikzpicture}
				\label{eq:30.6}
\end{equation}
With the localization functor $M\mapsto M^{\natural}$ given by 
\begin{equation}
				M^{\natural}(\mathscr{U}):= \left\{ \sigma\; :\; \mathscr{U}\to
								\coprod_{\mathfrak{p}\in \mathscr{U}} S_{\mathfrak{p}}^{-1}M, \;
				\sigma(\mathfrak{p})\in S_{\mathfrak{p}}^{-1} M,\; \sigma \; \text{locally constant} \right\}
				\label{eq:30.7}
\end{equation}
where $\sigma$ is \myemph{locally constant} if for all $\mathfrak{p}\in \mathscr{U}$, there is
neighbourhood $\mathfrak{p}\in D(s)\subseteq \mathscr{U}$, and $m\in M$ such that for $q\in
D(s), \sigma(q)\equiv m/s \in S_{q}^{-1}M$. \\
We have an identification of the stalk of $M^{\natural}$ at $\plusprimeideal\in
\text{spec}(A)$, 
\begin{equation}
								M^{\natural}\Big|_{\plusprimeideal} =
								\lim\limits_{\xrightarrow[\mathscr{U}\ni\plusprimeideal]{}}
								M^{\natural}(\mathscr{U})
								\xrightarrow{\;\sim\;} S^{-1}_{\plusprimeideal}M := M_{\plusprimeideal},
								\qquad 
								\left( \sigma,\mathscr{U} \right)/\sim \xmapsto{\quad\;} \sigma(\plusprimeideal)
								\label{eq:30.8}
\end{equation}
Moreover, we have the identification of the global sections of $M^{\natural}$ over an affine basic
open subset $D(s)\subseteq\text{spec}(A)$, 
\begin{equation}
				\begin{array}[H]{rl}
								\Psi : M\left[ 1/s \right]:= \left\{ s^{\mathbbm{N}}\right\}^{-1}M 
								&\xrightarrow{\;\sim\;}M^{\natural}\left( D(s) \right) \\
								\frac{m}{s^n} &\xmapsto{\quad\;}\sigma(\plusprimeideal)\equiv
								\frac{m}{s^{n}}\in S^{-1}_{\plusprimeideal}M
				\end{array}
				\label{eq:30.9}
\end{equation}
(the proof that $\Psi$ is a bijection is a repetition of the  proof of  Theorem (\ref{thm:3})). \vspace{.1cm}\\
Given a generalized scheme $X$, we have the full subcategory of ``quasi-coherent'' sheaves of
$\mathscr{O}_{X}\text{-Sets}$ (resp. $C\mathscr{O}_{X}\text{-Set}$),
where an object $M$ is quasi-coherent if there is a covering of $X$ by open affines
$X=\cup_{i}\text{spec}(A_i)$, and there are $M_i\in A_i\text{-Set}$
(resp. $CA_{i}\text{-Set}$), with 
$M\big|_{\text{spec}(A_i)}\cong M_{i}^{\natural}$. Equivalently, for all open affine
$\text{spec}(A)\subseteq X$, we have 
\begin{equation}
M\big|_{{\text{spec}}(A)}\cong M\left( \text{spec}(A) \right)^{\natural}.
\label{eq:30.10}
\end{equation}
We denote the subcategories of quasi-coherent sheaves by the prefix ``q.c.'', so that for
affine scheme $X=\text{spec}(A)$, q.c. $\mathscr{O}_{X}\text{-Set}\cong A\text{-Set}$.
\vspace{.1cm}\\
The subcategories of quasi-coherent sheaves are closed under colimits and finite-limits in
the associated categories of sheaves, and are therefore co-complete and finite-complete.
Moreover they are closed under the symmetric monoidal structure, so that for
quasi-coherent (resp. commutative) $\mathscr{O}_{X}\text{-Sets}$ $M,N$, and for any affine
open subset $\mathscr{U}=\text{spec}(A)\subseteq X$ we have
\begin{equation}
				\begin{array}[H]{c}
								M\otimes_{\mathscr{O}_{X}} N\big|_{\mathscr{U}} \cong \left(
												M(\mathscr{U})\otimes_{A} N(\mathscr{U})
				\right)^{\natural} \\\\
				\text{resp. }\; \text{Hom}_{\mathscr{O}_{X}}\left( M,N \right)\big|_{\mathscr{U}} \cong
				\text{Hom}_{A}\left( M(\mathscr{U}), N(\mathscr{U}) \right)^{\natural}
				\end{array}
				\label{eq:30.11}
\end{equation}
For a mapping of generalized schemes $f\in \text{GSch}(X,Y)$, the functor
$f^{\ast}$ takes quasi-coherent $\mathscr{O}_{Y}\text{-sets}$ to quasi-coherent
$\mathscr{O}_{X}\text{-sets}$.
Moreover, for affine open subsets $U=\text{spec}(A)\subseteq X$,
$\mathcal{V}=\text{spec}(B)\subseteq Y$, with  $f(\mathscr{U})\subseteq \mathcal{V}$ and with
$\varphi=f^{\#}:B\to A$ the associated homomorphism, we have for quasi-coherent
sheaf $M$ on $Y$, 
\begin{equation}
				f^{*}M\Big|_{\mathscr{U}} \cong \left( \varphi^{*} M(\mathscr{U}) \right)^{\natural}.
				\label{eq:30.12}
\end{equation}

\begin{definition} \label{def:30.13}
				A mapping of generalized schemes $f:X\to Y$ is called a
				$c\text{\normalfont-map}$
				(``compact and quasi-separated'') if we can cover $Y$ by open affine
				subsets (or equivalently, if for all open affine)
				$\mathscr{U}=\text{\normalfont spec}(B)\subseteq Y$, we have that
				$f^{-1}(\mathscr{U})$ is compact, so that
				\[
								f^{-1}(\mathscr{U})=\bigcup\limits_{i=1}^{N}\mathcal{V}_{i}, \qquad
								\mathcal{V}_{i}= \text{\normalfont spec}(A_i),
				\]
				is a \underline{finite} union of open affine subsets, and if moreover for
				$i,j\le N$, we have $\mathcal{V}_{i}\cap \mathcal{V}_{j}$ also compact, so that 
												\[
																\mathcal{V}_{i}\cap \mathcal{V}_j = \bigcup_{k=1}^{N_{i
																j}}\mathcal{W}_{i j k}, \qquad \mathcal{W}_{i j k} =
																\text{\normalfont spec}\left( A_{i j k} \right). 
												\]
\end{definition}
For a $c\text{\normalfont -map}$ $f:X\to Y$, the functor $f_{*}$ takes quasi-coherent
sheaves to quasi-coherent sheaves, so we get adjuctions 
\begin{equation}
				\begin{tikzpicture}[baseline=40mm]
							\node (a) at ({4*5mm},{13*5mm}) {q.c. $C\mathscr{O}_{X}\text{-sSet}$};
							\node (b) at ({4*5mm},{9*5mm}) {q.c. $C\mathscr{O}_{Y}\text{-sSet}$};
							\draw[<-] plot[smooth ] coordinates {($(a)+(-6mm,-3mm)$)($0.5*(a)+0.5*(b)+(-9mm,0mm)$)($(-6mm,3mm)+(b)$)};
							\draw[->] plot[smooth ] coordinates {($(a)+(6mm,-3mm)$)($0.5*(a)+0.5*(b)+(9mm,0mm)$)($(6mm,3mm)+(b)$)};
							\node at ($0.5*(a)+0.5*(b)+(-12mm,0mm)$) {$f^{\ast}$};
							\node at ($0.5*(a)+0.5*(b)+(12mm,0mm)$) {$f_{\ast}$};
							\node (a) at ({12.5*5mm},{13*5mm}) {q.c. $C\mathscr{O}_{X}\text{-Set}$};
							\node (b) at ({12.5*5mm},{9*5mm}) {q.c. $C\mathscr{O}_{Y}\text{-Set}$};
							\draw[<-] plot[smooth ] coordinates {($(a)+(-6.5mm,-3mm)$)($0.5*(a)+0.5*(b)+(-9.5mm,0mm)$)($(-6.5mm,3mm)+(b)$)};
							\draw[->] plot[smooth ] coordinates {($(a)+(6.5mm,-3mm)$)($0.5*(a)+0.5*(b)+(9.5mm,0mm)$)($(6.5mm,3mm)+(b)$)};
							\node at ($0.5*(a)+0.5*(b)+(-12mm,0mm)$) {$f^{\ast}$};
							\node at ($0.5*(a)+0.5*(b)+(12mm,0mm)$) {$f_{\ast}$};
							\node (a) at ({21*5mm},{13*5mm}) {q.c. $\mathscr{O}_{X}\text{-Set}$};
							\node (b) at ({21*5mm},{9*5mm}) {q.c. $\mathscr{O}_{Y}\text{-Set}$};
							\draw[<-] plot[smooth ] coordinates {($(a)+(-6mm,-3mm)$)($0.5*(a)+0.5*(b)+(-9mm,0mm)$)($(-6mm,3mm)+(b)$)};
							\draw[->] plot[smooth ] coordinates {($(a)+(6mm,-3mm)$)($0.5*(a)+0.5*(b)+(9mm,0mm)$)($(6mm,3mm)+(b)$)};
							\node at ($0.5*(a)+0.5*(b)+(-12mm,0mm)$) {$f^{\ast}$};
							\node at ($0.5*(a)+0.5*(b)+(12mm,0mm)$) {$f_{\ast}$};
				\end{tikzpicture}
				\label{eq:30.14}
\end{equation}
Indeed, for a quasi-coherent sheaf $M$ on $X$, we have 
(using the above notations)  that 
$f_{*}M|_\mathscr{U}=f_{*}M|_{\text{spec}(B)}$ is the equalizer
\begin{equation}
				\begin{array}[h]{l}
								f_{*}M\big|_{\mathscr{U}}\xrightarrow{\qquad}\prod\limits_{i=1}^{N}f_{*}\left(
								M\big|_{\mathcal{V}_{i}} \right)\rdarrow{}{}
								\prod\limits_{i,j\le N}\prod\limits_{k=1}^{N_{i j}}f_{*}\left(
								M\big|_{\mathcal{W}_{i j k}} \right) \\\\
								\text{and } 
								\begin{array}[t]{l}
								f_{*}\left( M\big|_{\mathcal{V}_i}
								\right) = \left( f_{i}^{\#}M(\mathcal{V}_{i})
								\right)^{\natural}, \qquad
								 f_{*}\left( M\big|_{\mathcal{W}_{i j k}}
								\right) = \left( f_{i j k}^{\#}M(\mathcal{W}_{i j k})
								\right)^{\natural}
								\end{array}
				\end{array}
				\label{eq:30.15}
\end{equation}
are quasi-coherent. \vspace{.1cm}\\
The adjunctions (\ref{eq:30.14}) satisfy the usual properties, such as the existence of canonical equivalences
\begin{equation}
				\begin{array}[H]{c}
								\left( g\circ f \right)_{*}\xrightarrow{\;\sim\;} g_{*}\circ f_{*} \qquad
								\left( g\circ f \right)^{*}\xrightarrow{\;\sim\;} f^{*}\circ g^{*} \\\\
								f^{*}\left( M\otimes_{\bigo_{Y}} N \right) \xrightarrow{\;\sim\;}
								f^{*} M\otimes_{\bigo_{X}} f^{*}N \\\\
								f_{*}\myhom_{\bigo_{X}}\left( f^{*}M,N \right)\xrightarrow{\;\sim\;}
								\myhom_{\bigo_{Y}}\left( M,f_{*}N \right). 
				\end{array}
				\label{eq:30.16}
\end{equation}
\section{The global derived category}
The language of model categories is notoriously difficult for categories of sheaves, and
to describe the derived category of $\mathscr{O}_{X}$-modules one usually uses the
``hairy'' injective resolutions (rather than the straightforward projective or even free
resolutions). Using the language of $\infty$ categories one can overcome these
difficulties and \myemph{define} $\mathbb{D}(X)$, the derived category of quasi-coherent
$\mathscr{O}_{X}\text{-modules}$, as the homotopy category $\mathbb{D}(X)=\text{Ho}
D(X)_{\infty}$, of the $\infty$ categorical limit $D(X)_{\infty}$ of the $\infty$
 categories $D(A)_{\infty}$ of fibrant co-fibrant $S_{A}$-modules, the limit
 taken over the category
\begin{equation*}
				\text{Aff}/_{X} = \left\{ \text{spec}(A)\subseteq X \; \text{open} \right\}
\end{equation*}
We can do the same for a generalized scheme $X\in \text{GSch}$, and define the derived
category of quasi-coherent commutative $\mathscr{O}_{X}\text{-sets}$ 
\begin{equation*}
				\mathbb{D}(X)= \pi_{0}D(X)_{\infty}\quad , \quad D(X)_{\infty}=
				\infty\limleftarrow\limits_{\text{Aff}/_{X}} D(A)_{\infty}
\end{equation*}
Using Lurie's straightening/ unstraightening functors \cite{MR2522659}, corollary 3.3.3.2, this
limit $D(X)_{\infty}$ has an equivalent, and more concrete description as the cartesian
sections of an inner, Cartesian and co-Cartesian, fibrantion of $\infty$ categories 
\begin{equation}
				\begin{array}[h]{l}
								N_{\bbDelta}(\rho): N_{\bbDelta}\left( \mathscr{D}(X)
								\right)\xtworightarrow{\qquad}
				N(\text{Aff}/_{X}), \\\\
				\rho : \mathscr{D}(X)=\text{Aff}/_{X}\ltimes \left( S^{\cdot}_{A}\text{-mod}\right)\to
				\text{Aff}/_{X} \;\;\; \text{the forgetful functor.}
				\end{array}
				\label{eq:5.36}
\end{equation}
Here for a generalized scheme $X$, 
$\mathcal{D}(X)=\text{Aff}/X \ltimes (S^\cdot\mymod)$ is the simplicial
category with objects pairs $(\text{spec}(A),M^{\cdot})$, with
$\text{spec}(A)\subseteq X$ open affine, and $M^{.}\in
S^{\cdot}_{A}\mymod$. The maps are given by
\begin{equation}
				\mathcal{D}(X)\left(
				(\text{\normalfont spec}(A_0),M_{0}^{\cdot}),
(\text{\normalfont spec}(A_1),M_{1}^{\cdot})\right) =
S_{A_{0}}\text{-mod}\left(
\mathbb{L}\phi^{*}M^{\cdot}_{1},M^{\cdot}_{0} \right)
\end{equation}
where
$\phi:\text{spec}(A_0)\xhookrightarrow{\;\;}\text{spec}(A_1)$
denotes the inclusion (and are empty if $\text{spec}\left( A_0
\right)\not\subseteq\text{spec}(A_1)$). 
The coherent nerve of $\mathcal{D}(X)$ is the simplicial set
\begin{equation}
				N_{\BDelta}\left( \mathcal{D}(X) \right)_{n} =
				\text{Cat}_{\BDelta}\left(\mathcal{C}(\Delta^{n}),\mathcal{D}(X) \right)  
				\label{eq:31.3}
\end{equation}
Explicitly, its elements are given by the data of 
\begin{equation*}
				\begin{array}[H]{lll}
								\text{(\romannumeral 1)} &
												(U_0\subseteq U_1 \subseteq
												\dots \subseteq U_n)\in N_n\left( \text{Aff}/X \right), \quad
												U_i=\text{spec}(A_i)\subseteq X \text{ open}, \\\\
								\text{(\romannumeral 2)} & 
												M^{\cdot}_{i}\in S^{\cdot}_{A_{i}}\mymod \\\\
								\text{(\romannumeral 3)} &
												\text{For $0\le i< j\le
												n$}, \quad \varphi_{ij}\in
												S^{\cdot}_{A_{i}}\mymod\left(
												\mathbb{L}\phi_{i
												j}^{*}M^{\cdot}_{j},M_{i}^{\cdot}
												\right) 
				\end{array}
\end{equation*}
where $\phi_{i j}: U_i\hookrightarrow U_j$ denotes the inclusion, and
{\underline{higher coherent homotopies}} given by
\begin{equation}
				\begin{array}[h]{l}
								\varphi_\tau\in S^{\cdot}_{A_i}\mymod\left(
												\mathbb{L}\phi^{*}_{i j} M^{\cdot}_{j}\otimes
								\Delta(\ell)_{+},M_{i}^{\cdot} \right) \\\\
								\ \qquad\text{for all
												$\tau=(\tau_0\subseteq \tau_1\subseteq \dots \subseteq \tau_\ell)\in
								N_{\ell}(\mathcal{P}_{i j})$}
				\end{array}
				\label{eq:31.4}
\end{equation}
(Where $P_{i j}$ is the set of subsets of $[i,j]$ containing the end
points). \vspace{.1cm}\\
These are required to satisfy:
\begin{equation}
				\displaystyle d_{k}\varphi_{\tau}=
\varphi_{(\tau_0\subseteq\dots
\hat{\tau}_k\ldots \subseteq \tau_{\ell})}
\label{eq:31.5}
\end{equation} 
and 
\begin{equation}
\begin{array}[h]{l}
				\text{for $\ell=0$, $\tau=(\tau_{0})=\left\{
				i,k_1,\dots,k_{q},j \right\}$}, \\\\
				 \varphi_{\tau_{0}} =  \varphi_{i k_1}\circ 
				 \mathbb{L}\phi_{i k_1}^{*}\left( \varphi_{k_1,k_2}\circ
								\mathbb{L}\phi_{k_1,k_2}^{*}\left( \ldots
												\varphi_{k_{q-1},k_q}\circ \mathbb{L}
												\phi_{k_{q-1},k_{q}}^{*}\left( \varphi_{k_q,j}
								\right)\dots
\right) \right)
\end{array}
\label{eq:31.6}
\end{equation}
\begin{equation*}
				\begin{array}[h]{ll}
								\text{A section }  m: N\left( \text{Aff}/X \right)\xrightarrow{\qquad} N_{\BDelta}\left(
												\mathcal{D}(X)
								\right)\;\; , \quad
								 N_{\BDelta }(\rho)\circ m =\text{id}
				\end{array}
\end{equation*}
is {\underline{Cartesian}} iff for all $n$-simplex
$(U_0\subseteq U_1\subseteq \dots \subseteq U_n)\in N_{n}\left( \text{Aff}/X \right)$,
$U_i=\text{spec}(A_i)$, the data $m(U_{\cdot})$  satisfies that all the arrows 
$\varphi_{i j}\in S^{\cdot}_{A_{i}}\mymod(\mathbb{L}\phi^{*}_{i j}
M_{j}^{\cdot}, M_{i}^{\cdot})$, $0\le i < j\le n$, are weak-equivalences. \\ 
The category $D(X)_{\infty}$ is symmetric monoidal stable quasi-category, so
$\mathbb{D}(X)$ is symmetric monoidal triangulated category. Where $X$ is an ordinary
scheme, $\mathbb{D}(X)$ is equivalent to the usual derived category of quasi-coherent
$\mathscr{O}_{X}$-modules. A $c\text{-map}\; f\in \text{GSch}(X,Y)$ induces adjunction
$f_{\ast}:\mathbb{D}(X) \rightleftarrows\mathbb{D}(Y): f^{\ast}$  satisfying the derived
version of (\ref{eq:30.16}). 
\nocite{*} 			
\bibliography{From_bi-operads_to_generalized_schemes}{}

\begin{thebibliography}{10}

\bibitem{MR0245577}
M.~Artin and B.~Mazur.
\newblock {\em Etale homotopy}.
\newblock Lecture Notes in Mathematics, No. 100. Springer-Verlag, Berlin-New
  York, 1969.

\bibitem{MR3459031}
Ilan Barnea and Tomer~M. Schlank.
\newblock A projective model structure on pro-simplicial sheaves, and the
  relative \'{e}tale homotopy type.
\newblock {\em Adv. Math.}, 291:784--858, 2016.

\bibitem{MR2016697}
Clemens Berger and Ieke Moerdijk.
\newblock Axiomatic homotopy theory for operads.
\newblock {\em Comment. Math. Helv.}, 78(4):805--831, 2003.

\bibitem{MR2248514}
Clemens Berger and Ieke Moerdijk.
\newblock The {B}oardman-{V}ogt resolution of operads in monoidal model
  categories.
\newblock {\em Topology}, 45(5):807--849, 2006.

\bibitem{MR2342815}
Clemens Berger and Ieke Moerdijk.
\newblock Resolution of coloured operads and rectification of homotopy
  algebras.
\newblock In {\em Categories in algebra, geometry and mathematical physics},
  volume 431 of {\em Contemp. Math.}, pages 31--58. Amer. Math. Soc.,
  Providence, RI, 2007.

\bibitem{MR2860274}
Clemens Berger and Ieke Moerdijk.
\newblock On an extension of the notion of {R}eedy category.
\newblock {\em Math. Z.}, 269(3-4):977--1004, 2011.

\bibitem{MR1422371}
Jonathan Block and Andrey Lazarev.
\newblock Homotopy theory and generalized duality for spectral sheaves.
\newblock {\em Internat. Math. Res. Notices}, (20):983--996, 1996.

\bibitem{MR0420609}
J.~M. Boardman and R.~M. Vogt.
\newblock {\em Homotopy invariant algebraic structures on topological spaces}.
\newblock Lecture Notes in Mathematics, Vol. 347. Springer-Verlag, Berlin-New
  York, 1973.

\bibitem{MR2402406}
Dennis~V. Borisov and Yuri~I. Manin.
\newblock Generalized operads and their inner cohomomorphisms.
\newblock In {\em Geometry and dynamics of groups and spaces}, volume 265 of
  {\em Progr. Math.}, pages 247--308. Birkh\"{a}user, Basel, 2008.

\bibitem{MR0365573}
A.~K. Bousfield and D.~M. Kan.
\newblock {\em Homotopy limits, completions and localizations}.
\newblock Lecture Notes in Mathematics, Vol. 304. Springer-Verlag, Berlin-New
  York, 1972.

\bibitem{MR341469}
Kenneth~S. Brown.
\newblock Abstract homotopy theory and generalized sheaf cohomology.
\newblock {\em Trans. Amer. Math. Soc.}, 186:419--458, 1973.

\bibitem{MR2294028}
Denis-Charles Cisinski.
\newblock Les pr\'{e}faisceaux comme mod\`eles des types d'homotopie.
\newblock {\em Ast\'{e}risque}, (308):xxiv+390, 2006.

\bibitem{MR3931682}
Denis-Charles Cisinski.
\newblock {\em Higher categories and homotopical algebra}, volume 180 of {\em
  Cambridge Studies in Advanced Mathematics}.
\newblock Cambridge University Press, Cambridge, 2019.

\bibitem{MR2805991}
Denis-Charles Cisinski and Ieke Moerdijk.
\newblock Dendroidal sets as models for homotopy operads.
\newblock {\em J. Topol.}, 4(2):257--299, 2011.

\bibitem{MR3100887}
Denis-Charles Cisinski and Ieke Moerdijk.
\newblock Dendroidal {S}egal spaces and {$\infty$}-operads.
\newblock {\em J. Topol.}, 6(3):675--704, 2013.

\bibitem{MR3100888}
Denis-Charles Cisinski and Ieke Moerdijk.
\newblock Dendroidal sets and simplicial operads.
\newblock {\em J. Topol.}, 6(3):705--756, 2013.

\bibitem{cisinski2014note}
Denis-Charles Cisinski and Ieke Moerdijk.
\newblock Note on the tensor product of dendroidal sets.
\newblock {\em arXiv preprint arXiv:1403.6507}, 2014.

\bibitem{dhillon2018infty}
Ajneet Dhillon and P{\'a}l Zs{\'a}mboki.
\newblock On the $\infty$-stack of complexes over a scheme.
\newblock {\em arXiv preprint arXiv:1801.06701}, 2018.

\bibitem{MR688240}
V.~Drinfeld.
\newblock Hamiltonian structures on {L}ie groups, {L}ie bialgebras and the
  geometric meaning of classical {Y}ang-{B}axter equations.
\newblock {\em Dokl. Akad. Nauk SSSR}, 268(2):285--287, 1983.

\bibitem{MR934283}
V.~Drinfeld.
\newblock Quantum groups.
\newblock In {\em Proceedings of the {I}nternational {C}ongress of
  {M}athematicians, {V}ol. 1, 2 ({B}erkeley, {C}alif., 1986)}, pages 798--820.
  Amer. Math. Soc., Providence, RI, 1987.

\bibitem{MR579087}
W.~G. Dwyer and D.~M. Kan.
\newblock Simplicial localizations of categories.
\newblock {\em J. Pure Appl. Algebra}, 17(3):267--284, 1980.

\bibitem{MR1361887}
W.~G. Dwyer and J.~Spali\'{n}ski.
\newblock Homotopy theories and model categories.
\newblock In {\em Handbook of algebraic topology}, pages 73--126.
  North-Holland, Amsterdam, 1995.

\bibitem{MR1392221}
Emmanuel~Dror Farjoun.
\newblock {\em Cellular spaces, null spaces and homotopy localization}, volume
  1622 of {\em Lecture Notes in Mathematics}.
\newblock Springer-Verlag, Berlin, 1996.

\bibitem{MR3643404}
Benoit Fresse.
\newblock {\em Homotopy of operads and {G}rothendieck-{T}eichm\"{u}ller groups.
  {P}art 1}, volume 217 of {\em Mathematical Surveys and Monographs}.
\newblock American Mathematical Society, Providence, RI, 2017.
\newblock The algebraic theory and its topological background.

\bibitem{getzler1994operads}
Ezra Ge\phantom{s}\hspace{-.13cm}tzler and John~DS Jones.
\newblock Operads, homotopy algebra and iterated integrals for double loop
  spaces.
\newblock {\em arXiv preprint hep-th/9403055}, 1994.

\bibitem{MR1601666}
E.~Getzler and M.~M. Kapranov.
\newblock Modular operads.
\newblock {\em Compositio Math.}, 110(1):65--126, 1998.

\bibitem{MR1301191}
Victor Ginzburg and Mikhail Kapranov.
\newblock Koszul duality for operads.
\newblock {\em Duke Math. J.}, 76(1):203--272, 1994.

\bibitem{MR1711612}
Paul~G. Goerss and John~F. Jardine.
\newblock {\em Simplicial homotopy theory}, volume 174 of {\em Progress in
  Mathematics}.
\newblock Birkh\"{a}user Verlag, Basel, 1999.

\bibitem{MR3075000}
A.~Grothendieck and J.~A. Dieudonn\'{e}.
\newblock {\em \'{E}l\'{e}ments de g\'{e}om\'{e}trie alg\'{e}brique. {I}},
  volume 166 of {\em Grundlehren der mathematischen Wissenschaften [Fundamental
  Principles of Mathematical Sciences]}.
\newblock Springer-Verlag, Berlin, 1971.

\bibitem{MR3408444}
Philip Hackney, Marcy Robertson, and Donald Yau.
\newblock {\em Infinity properads and infinity wheeled properads}, volume 2147
  of {\em Lecture Notes in Mathematics}.
\newblock Springer, Cham, 2015.

\bibitem{MR3924179}
Philip Hackney, Marcy Robertson, and Donald Yau.
\newblock Higher cyclic operads.
\newblock {\em Algebr. Geom. Topol.}, 19(2):863--940, 2019.

\bibitem{haran1989}
Shai Haran.
\newblock Index theory, potential theory, and the {R}iemann hypothesis.
\newblock In {\em {$L$}-functions and arithmetic ({D}urham, 1989)}, volume 153
  of {\em London Math. Soc. Lecture Note Ser.}, pages 257--270. Cambridge Univ.
  Press, Cambridge, 1991.

\bibitem{MR1872029}
Shai Haran.
\newblock {\em The mysteries of the real prime}, volume~25 of {\em London
  Mathematical Society Monographs. New Series}.
\newblock The Clarendon Press, Oxford University Press, New York, 2001.

\bibitem{MR2330442}
Shai Haran.
\newblock Non-additive geometry.
\newblock {\em Compos. Math.}, 143(3):618--688, 2007.

\bibitem{MR2433635}
Shai Haran.
\newblock {\em Arithmetical investigations}, volume 1941 of {\em Lecture Notes
  in Mathematics}.
\newblock Springer-Verlag, Berlin, 2008.
\newblock Representation theory, orthogonal polynomials, and quantum
  interpolations.

\bibitem{haran2017geometry}
Shai Haran.
\newblock Geometry over ${F}_1$.
\newblock {\em arXiv preprint arXiv:1709.05831}, 2017.

\bibitem{MR3605614}
Shai Haran.
\newblock New foundations for geometry---two non-additive languages for
  arithmetical geometry.
\newblock {\em Mem. Amer. Math. Soc.}, 246(1166):x+200, 2017.

\bibitem{haran2020algebra}
Shai Haran.
\newblock Algebra over generalized rings.
\newblock {\em arXiv preprint arXiv:2006.15613}, 2020.

\bibitem{MR3545944}
Gijs Heuts, Vladimir Hinich, and Ieke Moerdijk.
\newblock On the equivalence between {L}urie's model and the dendroidal model
  for infinity-operads.
\newblock {\em Adv. Math.}, 302:869--1043, 2016.

\bibitem{heuts2018trees}
Gijs Heuts and Ieke Moerdijk.
\newblock Trees in algebra and topology.
\newblock {\em preprint}, 2018.

\bibitem{MR1465117}
Vladimir Hinich.
\newblock Homological algebra of homotopy algebras.
\newblock {\em Comm. Algebra}, 25(10):3291--3323, 1997.

\bibitem{hirschhorn2009model}
Philip~S Hirschhorn.
\newblock {\em Model categories and their localizations}.
\newblock Number~99. American Mathematical Soc., 2009.

\bibitem{MR3802235}
Eric Hoffbeck and Ieke Moerdijk.
\newblock Shuffles of trees.
\newblock {\em European J. Combin.}, 71:55--72, 2018.

\bibitem{MR1860878}
Mark Hovey.
\newblock Spectra and symmetric spectra in general model categories.
\newblock {\em J. Pure Appl. Algebra}, 165(1):63--127, 2001.

\bibitem{MR1695653}
Mark Hovey, Brooke Shipley, and Jeff Smith.
\newblock Symmetric spectra.
\newblock {\em J. Amer. Math. Soc.}, 13(1):149--208, 2000.

\bibitem{MR906403}
J.~F. Jardine.
\newblock Simplicial presheaves.
\newblock {\em J. Pure Appl. Algebra}, 47(1):35--87, 1987.

\bibitem{MR927763}
Andr\'{e} Joyal.
\newblock Foncteurs analytiques et esp\`eces de structures.
\newblock In {\em Combinatoire \'{e}num\'{e}rative ({M}ontreal, {Q}ue.,
  1985/{Q}uebec, {Q}ue., 1985)}, volume 1234 of {\em Lecture Notes in Math.},
  pages 126--159. Springer, Berlin, 1986.

\bibitem{MR1935979}
Andr\'{e}. Joyal.
\newblock Quasi-categories and {K}an complexes.
\newblock volume 175, pages 207--222. 2002.
\newblock Special volume celebrating the 70th birthday of Professor Max Kelly.

\bibitem{joyal2008theory}
Andr{\'e} Joyal.
\newblock The theory of quasi-categories and its applications.
\newblock 2008.

\bibitem{joyal2011feynman}
Andr{\'e} Joyal and Joachim Kock.
\newblock Feynman graphs, and nerve theorem for compact symmetric
  multicategories.
\newblock {\em Electronic Notes in Theoretical Computer Science},
  270(2):105--113, 2011.

\bibitem{MR2342834}
Andr\'{e} Joyal and Myles Tierney.
\newblock Quasi-categories vs {S}egal spaces.
\newblock In {\em Categories in algebra, geometry and mathematical physics},
  volume 431 of {\em Contemp. Math.}, pages 277--326. Amer. Math. Soc.,
  Providence, RI, 2007.

\bibitem{MR79762}
Daniel~M. Kan.
\newblock Abstract homotopy. {I}.
\newblock {\em Proc. Nat. Acad. Sci. U.S.A.}, 41:1092--1096, 1955.

\bibitem{kan1957css}
Daniel~M Kan.
\newblock On css complexes.
\newblock {\em American Journal of Mathematics}, 79(3):449--476, 1957.

\bibitem{kan1958adjoint}
Daniel~M Kan.
\newblock Adjoint functors.
\newblock {\em Transactions of the American Mathematical Society},
  87(2):294--329, 1958.

\bibitem{kan1958combinatorial}
Daniel~M Kan.
\newblock A combinatorial definition of homotopy groups.
\newblock {\em Annals of Mathematics}, pages 282--312, 1958.

\bibitem{kan1958homotopy}
Daniel~M Kan.
\newblock On homotopy theory and css groups.
\newblock {\em Annals of Mathematics}, pages 38--53, 1958.

\bibitem{MR1423619}
Jean-Louis Loday.
\newblock La renaissance des op\'{e}rades.
\newblock Number 237, pages Exp. No. 792, 3, 47--74. 1996.
\newblock S\'{e}minaire Bourbaki, Vol. 1994/95.

\bibitem{MR2954392}
Jean-Louis Loday and Bruno Vallette.
\newblock {\em Algebraic operads}, volume 346 of {\em Grundlehren der
  mathematischen Wissenschaften [Fundamental Principles of Mathematical
  Sciences]}.
\newblock Springer, Heidelberg, 2012.

\bibitem{lurie2004derived}
Jacob Lurie.
\newblock {\em Derived algebraic geometry}.
\newblock PhD thesis, Massachusetts Institute of Technology, 2004.

\bibitem{MR2522659}
Jacob Lurie.
\newblock {\em Higher topos theory}, volume 170 of {\em Annals of Mathematics
  Studies}.
\newblock Princeton University Press, Princeton, NJ, 2009.

\bibitem{algebra2017available}
Jacob Lurie.
\newblock {\em Higher Algebra}.
\newblock 2017.

\bibitem{MR1712872}
Saunders Mac~Lane.
\newblock {\em Categories for the working mathematician}, volume~5 of {\em
  Graduate Texts in Mathematics}.
\newblock Springer-Verlag, New York, second edition, 1998.

\bibitem{MR1898414}
Martin Markl, Steve Shnider, and Jim Stasheff.
\newblock {\em Operads in algebra, topology and physics}, volume~96 of {\em
  Mathematical Surveys and Monographs}.
\newblock American Mathematical Society, Providence, RI, 2002.

\bibitem{MR0420610}
J.~P. May.
\newblock {\em The geometry of iterated loop spaces}.
\newblock Lecture Notes in Mathematics, Vol. 271. Springer-Verlag, Berlin-New
  York, 1972.

\bibitem{MR1436912}
J.~P. May.
\newblock Definitions: operads, algebras and modules.
\newblock In {\em Operads: {P}roceedings of {R}enaissance {C}onferences
  ({H}artford, {CT}/{L}uminy, 1995)}, volume 202 of {\em Contemp. Math.}, pages
  1--7. Amer. Math. Soc., Providence, RI, 1997.

\bibitem{MR4222648}
Ieke Moerdijk.
\newblock Closed dendroidal sets and unital operads.
\newblock {\em Theory Appl. Categ.}, 36:Paper No. 5, 118--170, 2021.

\bibitem{MR2797154}
Ieke Moerdijk and Bertrand To\"{e}n.
\newblock {\em Simplicial methods for operads and algebraic geometry}.
\newblock Advanced Courses in Mathematics. CRM Barcelona.
  Birkh\"{a}user/Springer Basel AG, Basel, 2010.
\newblock Edited by Carles Casacuberta and Joachim Kock.

\bibitem{MR2366165}
Ieke Moerdijk and Ittay Weiss.
\newblock Dendroidal sets.
\newblock {\em Algebr. Geom. Topol.}, 7:1441--1470, 2007.

\bibitem{moerdijk2009inner}
Ieke Moerdijk and Ittay Weiss.
\newblock On inner kan complexes in the category of dendroidal sets.
\newblock {\em Advances in Mathematics}, 221(2):343--389, 2009.

\bibitem{MR258031}
Daniel Quillen.
\newblock Rational homotopy theory.
\newblock {\em Ann. of Math. (2)}, 90:205--295, 1969.

\bibitem{MR0257068}
Daniel Quillen.
\newblock On the (co-) homology of commutative rings.
\newblock In {\em Applications of {C}ategorical {A}lgebra ({P}roc. {S}ympos.
  {P}ure {M}ath., {V}ol. {XVII}, {N}ew {Y}ork, 1968)}, pages 65--87. Amer.
  Math. Soc., Providence, R.I., 1970.

\bibitem{MR0338129}
Daniel Quillen.
\newblock Higher algebraic {$K$}-theory. {I}.
\newblock In {\em Algebraic {$K$}-theory, {I}: {H}igher {$K$}-theories ({P}roc.
  {C}onf., {B}attelle {M}emorial {I}nst., {S}eattle, {W}ash., 1972)}, pages
  85--147. Lecture Notes in Math., Vol. 341, 1973.

\bibitem{MR0223432}
Daniel~G. Quillen.
\newblock {\em Homotopical algebra}.
\newblock Lecture Notes in Mathematics, No. 43. Springer-Verlag, Berlin-New
  York, 1967.

\bibitem{R1974}
C.~L. Reedy.
\newblock Homotopy theory of model categories.

\bibitem{MR353298}
Graeme Segal.
\newblock Categories and cohomology theories.
\newblock {\em Topology}, 13:293--312, 1974.

\bibitem{MR2074990}
Christophe Soul\'{e}.
\newblock Les vari\'{e}t\'{e}s sur le corps \`a un \'{e}l\'{e}ment.
\newblock {\em Mosc. Math. J.}, 4(1):217--244, 312, 2004.

\bibitem{MR2320654}
Bruno Vallette.
\newblock A {K}oszul duality for {PROP}s.
\newblock {\em Trans. Amer. Math. Soc.}, 359(10):4865--4943, 2007.

\bibitem{weil1939analogie}
Andr{\'e} Weil.
\newblock Sur l’analogie entre les corps de nombres alg{\'e}briques et les
  corps de fonctions alg{\'e}briques.
\newblock {\em Revue Scient}, 77:104--106, 1939.

\bibitem{weiss2007dendroidal}
Ittay Weiss.
\newblock {\em Dendroidal sets}.
\newblock PhD thesis, Utrecht University, 2007.

\bibitem{MR2742425}
Ittay Weiss.
\newblock From operads to dendroidal sets.
\newblock In {\em Mathematical foundations of quantum field theory and
  perturbative string theory}, volume~83 of {\em Proc. Sympos. Pure Math.},
  pages 31--70. Amer. Math. Soc., Providence, RI, 2011.

\bibitem{MR3329226}
Donald Yau and Mark~W. Johnson.
\newblock {\em A foundation for {PROP}s, algebras, and modules}, volume 203 of
  {\em Mathematical Surveys and Monographs}.
\newblock American Mathematical Society, Providence, RI, 2015.

\end{thebibliography}
\bibliographystyle{plain}
\end{document}